\documentclass[a4paper,10pt]{amsart}
\usepackage[hmarginratio={1:1},vmarginratio={1:1},lmargin=80.0pt,tmargin=90.0pt]{geometry}
\usepackage{currvita, amssymb, amsmath, amsthm, 
graphicx, subfigure, pifont, wrapfig, pinlabel, amscd,
latexsym, exscale, enumerate, amsfonts, times, color, hyphenat, pinlabel, 
mathtools, a4wide, fullpage, array, bbm, url}
\usepackage{float}
\usepackage[table]{xcolor}
\usepackage{stmaryrd}
\SetSymbolFont{stmry}{bold}{U}{stmry}{m}{n}
\usepackage{multirow}
\usepackage[all]{xy}
\usepackage{young}
\usepackage[vcentermath]{youngtab}

\definecolor{orchid}{RGB}{143,40,194}
\definecolor{lava}{RGB}{207,16,32}
\definecolor{mydarkblue}{RGB}{10,10,150}

\numberwithin{equation}{section}
\usepackage[hypertexnames=false]{hyperref}
\usepackage{cleveref,bookmark}
\hypersetup{
pdftoolbar=true,        
pdfmenubar=true,        
pdffitwindow=false,     
pdfstartview={FitH},    
pdftitle={\texorpdfstring{$\mathfrak{gl}_n$}{gln}-webs, categorification and Khovanov--Rozansky homologies},    
pdfauthor={Daniel Tubbenhauer},     
pdfsubject={},   
pdfcreator={Daniel Tubbenhauer},   
pdfproducer={Daniel Tubbenhauer}, 
pdfkeywords={}, 
pdfnewwindow=true,      
colorlinks=true,       
linkcolor=mydarkblue,          
citecolor=teal,        
filecolor=magenta,      
urlcolor=orchid,          
linkbordercolor=lava,
citebordercolor=teal,
urlbordercolor=orchid,  
linktocpage=true
}

\DeclareMathOperator{\Mor}{Mor}

\DeclareMathOperator{\KR}{KR}

\newcommand{\bN}{\mathbb{N}}
\newcommand{\bZ}{\mathbb{Z}}
\newcommand{\bQ}{\bar{\mathbb{Q}}}

\newcommand{\bC}{\mathbb{C}}

\newcommand{\onel}{{\textbf 1}_{\vec{k}}}

\def\ui{{\text{$\underline{i}$}}}
\def\idm{\text{\bfseries 1}}

\newcommand{\refequal}[1]{\xy {\ar@{=}^{#1}
(-1,0)*{};(1,0)*{}};
\endxy}

\newcommand{\spid}[1]{\textbf{Sp}(\Uu_q(\mathfrak{gl}_{{#1}}))}
\newcommand{\spidn}[2]{\textbf{Sp}^{{#2}}(\Uu_q(\mathfrak{gl}_{{#1}}))}
\newcommand{\fspidn}[1]{\textbf{Sp}_f(\Uu_q(\mathfrak{gl}_{{#1}}))}
\newcommand{\fspidnn}[2]{\textbf{Sp}^{{#2}}_f(\Uu_q(\mathfrak{gl}_{{#1}}))}

\newcommand{\Um}{\dot{{\Uu}}_q(\mathfrak{gl}_m)}

\newcommand{\Ucatm}{\mathcal{U}(\mathfrak{gl}_m)}
\newcommand{\Ucatt}{\mathcal{U}(\mathfrak{gl}_2)}
\newcommand{\Ucattd}{\dot{\mathcal{U}}(\mathfrak{gl}_2)}
\newcommand{\Ucattc}{\check{\mathcal{U}}(\mathfrak{gl}_2)}
\newcommand{\Ucatmc}{\check{\mathcal{U}}(\mathfrak{gl}_m)}
\newcommand{\Ucatmd}{\dot{\mathcal{U}}(\mathfrak{gl}_m)}
\newcommand{\Ucatmm}{\mathcal{U}^-(\mathfrak{gl}_m)}

\newcommand{\qbin}[2]{
\left[
\begin{array}{c}
#1 \\
#2 \\
\end{array}
\right]
}

\DeclareMathOperator{\Uu}{\textbf{U}}
\newcommand{\Uun}[1][n]{\textbf{U}_q(\mathfrak{gl}_{#1})}
\DeclareMathOperator{\Rep}{\textbf{Rep}}
\DeclareMathOperator{\Ob}{Ob}

\DeclareMathOperator{\RPMOD}{R_{\Lambda}-p\textbf{Mod}_{\mathrm{gr}}}
\DeclareMathOperator{\RcPMOD}{\check{R}_{\Lambda}-p\textbf{Mod}_{\mathrm{gr}}}

\DeclareMathOperator{\KAR}{\textbf{Kar}}

\DeclareMathOperator{\Kom}{\textbf{Kom}}

\newtheorem{prop}{Proposition}[section]
\newtheorem{thm}[prop]{Theorem}
\newtheorem{lem}[prop]{Lemma}
\newtheorem{cor}[prop]{Corollary}

\theoremstyle{remark}
\newtheorem{rem}[prop]{Remark}

\theoremstyle{remark}
\newtheorem{ex}[prop]{\textbf{Example}}

\theoremstyle{remark}

\theoremstyle{definition}
\newtheorem{defn}[prop]{Definition}

\numberwithin{equation}{subsection}

\newcommand{\F}{\mathcal{F}}

\newcommand{\overcrossing}{\xy
(0,0)*{\includegraphics[height=.02\textheight]{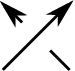}};
\endxy}
\newcommand{\undercrossing}{\xy
(0,0)*{\includegraphics[height=.02\textheight]{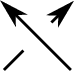}};
\endxy}

\newcommand{\cupcap}{\xy
(0,0)*{\includegraphics[height=.085\textheight]{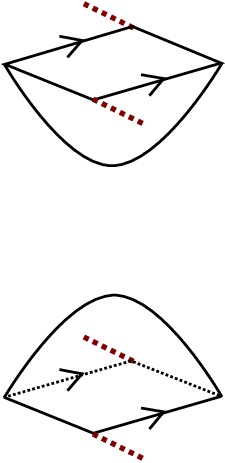}};
\endxy}
\newcommand{\saddleup}{\xy
(0,0)*{\includegraphics[height=.085\textheight]{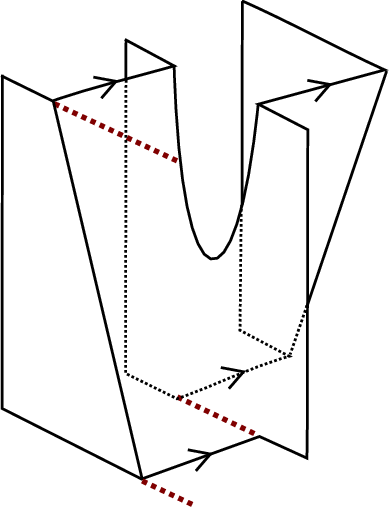}};
\endxy} 
\newcommand{\saddledown}{\xy
(0,0)*{\includegraphics[height=.085\textheight]{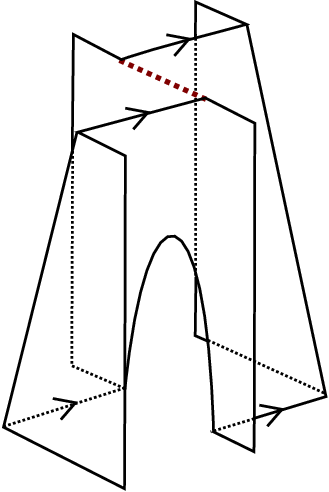}};
\endxy} 

\newcounter{In}

\title{$\mathfrak{gl}_n$-webs, categorification and Khovanov--Rozansky homologies}
\author{Daniel Tubbenhauer}
\thanks{The author was supported by the ``Centre for 
Quantum Geometry of Moduli Spaces'' of the Aarhus 
University that is granted by the Danish National Research Foundation (DNRF)}

\begin{document}
\begin{abstract}
In this paper we define an explicit basis for the 
$\mathfrak{gl}_n$-web algebra $H_n(\vec{k})$ (the 
$\mathfrak{gl}_n$ generalization of Khovanov's arc 
algebra) using categorified $q$-skew Howe duality.

Our construction is a $\mathfrak{gl}_n$-web version 
of Hu--Mathas' graded cellular basis and has two major applications: 
it gives rise to an explicit isomorphism between a certain 
idempotent truncation of a thick calculus cyclotomic KLR 
algebra and $H_n(\vec{k})$, and it gives an explicit graded 
cellular basis of the $2$-hom space between two $\mathfrak{gl}_n$-webs. 
We use this to give a (in principle) computable 
version of colored Khovanov--Rozansky $\mathfrak{gl}_n$-link 
homology, obtained from a complex defined purely combinatorially 
via the (thick cyclotomic) KLR 
algebra and needs only $F$.
\end{abstract}

\maketitle

\tableofcontents

\section{Introduction}\label{sec-intro}
\setcounter{subsection}{1}
\subsubsection{The framework}\label{subsub-frame}
The arc algebra $H_2(m)$ was introduced by Khovanov in his influential paper \cite{kh4} in order to extend his celebrated categorification of the Jones polynomial \cite{kh1} to tangles. The algebra realizes the homology of a tangle with $2m$ top boundary points and $2m^{\prime}$ bottom boundary points as certain $H_2(m)\text{-}H_2(m^{\prime})$-bimodules. His algebra consists of $\mathfrak{sl}_2$-cobordisms in the sense of Bar-Natan \cite{bn2} and has a beautiful diagrammatic calculus.

In the same vein, the so-called $\mathfrak{sl}_3$-web algebra $K_S$, introduced in \cite{mpt}, consists of $\mathfrak{sl}_3$-foams in the sense of Khovanov \cite{kh3} and is related to the $\mathfrak{sl}_3$-version of Khovanov homology from \cite{kh3}. Shortly after the definition of $K_S$, Mackaay introduced \cite{mack1} the $\mathfrak{gl}_n$-version of the arc algebra, denoted by $H_n(\vec{k})$. (It is more convenient for us to work with 
the general linear Lie algebra and not with the special linear Lie algebra; the difference for us is not crucial and the reader, following history, is invited to think about $\mathfrak{sl}_n$ instead of $\mathfrak{gl}_n$.) These algebras use the matrix factorization framework introduced to the field of link homologies by Khovanov--Rozansky \cite{kr1}. We should note that, using results of Queffelec--Rose \cite{qr1} (their results became available shortly after the first preprint of this paper appeared. But everything stated in this paper is also true using $\mathfrak{gl}_n$-foams instead of matrix factorizations), $H_n(\vec{k})$ could also be described using $\mathfrak{gl}_n$-foams introduced to the field by Mackaay--Sto\v{s}i\'{c}--Vaz \cite{msv}.

These algebras can be seen as the underlying algebraic structure for $2$-categories of cobordisms or foams 
or matrix factorizations in the sense that these $2$-categories are equivalent to certain bimodule categories of these algebras, see in the literature cited above for details.

Moreover, the work of Brundan--Stroppel on generalizations of the arc algebra, intensively studied in the series of papers \cite{bs1}, \cite{bs2}, \cite{bs3}, \cite{bs4} and \cite{bs5} (and additionally studied e.g. \cite{chkh}, \cite{kh5}, \cite{st2} and \cite{sw}), suggested that these algebras, in addition to their relations to knot theory, also have an interesting underlying representation theoretical and combinatorial structure. After their influential work the study of these algebras was carried out in great detail, e.g. the type $A_2$ variant was studied \cite{mpt}, \cite{rob2}, \cite{rob1}, \cite{tub2} and \cite{tub3} as well as the type $A_n$-web algebra \cite{mack1}. There is also a type $D$ version of the arc algebra, see \cite{ehst1}, \cite{ehst2} and \cite{ehst3}, and a $\mathfrak{gl}(1|1)$ variant \cite{sar}.

In this paper we consider the $\mathfrak{gl}_n$-web algebra $H_n(\vec{k})$ from both sides: we study its combinatorial and representation theoretical structure and discuss its relation to the $\mathfrak{gl}_n$-link polynomials/link homologies. And, although we restrict ourselves to $\bC$, everything should work over $\bZ$ as well.

\subsubsection{Some history}\label{subsub-hist}
In order to get more precise let us recall that these algebras categorify the $\mathfrak{gl}_n$-web spaces $W_n(\vec{k})$. These spaces consist of $\mathfrak{gl}_n$-webs which give a diagrammatic presentation of the representation category $\Rep(\Uun)$ of $\Uu_q(\mathfrak{gl}_n)$. In the case $n=2$ this is well-known and already appeared in work of Rumer--Teller--Weyl \cite{rtw} (in the non-quantum setting, of course). For $n=3$ the diagrammatic calculus was introduced by Kuperberg \cite{kup}, but, in the $n>3$ case, it was only proven much later case by Cautis--Kamnitzer--Morrison \cite{ckm}, using $q$-skew Howe duality, that the $\mathfrak{gl}_n$-webs give rise to a diagrammatic presentation of $\Rep(\Uun)$. (To be precise, the papers \cite{rtw}, \cite{kup} and \cite{ckm} work in the special linear setting.)

These $\mathfrak{gl}_n$-webs are also related to the MOY-calculus, introduced by 
Murakami--Ohtsuki--Yamada \cite{moy}. Therefore, these $\mathfrak{gl}_n$-webs can also be used in the context of the colored (we always mean $k$-colored with $\Lambda^k\bC^n$ (we usually write $\bC$ instead of $\bC(q)$ for simplicity of notation), i.e. colored with the fundamental $\Uu_q(\mathfrak{gl}_n)$-representations) Reshetikhin--Turaev $\mathfrak{gl}_n$-link polynomials. 
The uncolored polynomials were categorified by Khovanov--Rozansky \cite{kr1} using the language of matrix factorizations. Later Wu \cite{wu} and independently Yonezawa \cite{yo2} have categorified the colored version. Thus, the $\mathfrak{gl}_n$-web algebras $H_n(\vec{k})$ have a direct connection to (colored) $\mathfrak{gl}_n$-link polynomials and $\mathfrak{gl}_n$-link homologies.

It is worth noting that matrix factorizations are not the only way to define the $\mathfrak{gl}_n$-link homologies. In fact, there are many, e.g. using $\mathfrak{gl}_n$-foams \cite{msv}, there is an approach using category $\mathcal{O}$, see \cite{mast}, \cite{st1} and \cite{su1}, while another approach uses derived categories associated to certain projective varieties, see \cite{ck1} and \cite{ck2}. Cautis--Kamnitzer's $\mathfrak{gl}_n$-link homologies are related to constructions by Manolescu \cite{mano1} and Seidel and Smith \cite{sesm1} via mirror symmetry. And there is a version for $n=2,3$ by Lauda--Queffelec--Rose \cite{lqr1} that uses $q$-skew Howe duality and higher representation theory of $\dot{\Uu}_q(\mathfrak{gl}_m)$. Moreover, the approach of Webster \cite{web1} to categorify the Reshetikhin--Turaev $\mathfrak{g}$-polynomial for arbitrary simple Lie algebra $\mathfrak{g}$, is another example.

But in all cases, including Khovanov--Rozansky's approach, calculations seem to be (very) hard for $n>3$, see \cite{cm1},  \cite{ras1} 
and \cite{wed1} for some approaches. Moreover, the calculations in the $n=2$, see \cite{bn3}, and $n=3$, see \cite{le1}, cases are based on the $\mathfrak{sl}_2$-cobordism or $\mathfrak{sl}_3$-foam framework respectively, where it has been known for some time (see \cite{mv2}) that the matrix factorization and the $\mathfrak{sl}_2$-cobordism or $\mathfrak{sl}_3$-foam approach give the same result.

\subsubsection{Our motivation and approach}\label{subsub-moti}
Our approach is to obtain the Khovanov--Rozansky $\mathfrak{gl}_n$-link homologies using (thick) cyclotomic KLR algebras and categorified $q$-skew Howe duality. Since these algebras have an explicit basis, one can write down the differentials explicitly with respect to these bases. Moreover, our complex is completely combinatorial in nature: neither the matrix factorization framework nor $\mathfrak{gl}_n$-foams or any of the other techniques 
mentioned above are needed.

Our motivation originated from the viewpoint of the combinatorial and representation theoretical structure of the $\mathfrak{gl}_n$-web algebra $H_n(\vec{k})$. To be more precise, it is known that the $\mathfrak{gl}_n$-web algebras are graded cellular algebras for any $n>1$, see \cite{mpt} and \cite{my}. But only an explicit graded cellular basis would make it (in principle) possible to write down the set of graded projective indecomposables which, under the identification mentioned above, correspond to indecomposable $\mathfrak{gl}_n$-web modules which categorify the dual canonical basis of $W_n(\vec{k})$.

But only in the $n=2$ case there was a construction of an explicit graded cellular basis by Brundan and Stroppel \cite{bs1}. That was the reason why the author used categorified $q$-skew Howe duality \cite{tub3}, loosely called $\mathfrak{sl}_3$-foamation, to define an explicit graded cellular basis of the $\mathfrak{sl}_3$-web algebra by giving a foamy version of Hu--Mathas' \cite{hm} graded cellular basis 
(HM basis) of the cyclotomic KLR algebra $R_{\Lambda}$ (see Khovanov--Lauda \cite{kl1}, \cite{kl3} or Rouquier \cite{rou}), where $\Lambda$ denotes a dominant $\mathfrak{gl}_m$-weight. (Note that Hu--Mathas results depend on Dipper--James--Mathas standard basis \cite{djm} of the cyclotomic Hecke algebra and thus, on \cite{bk1}.)

It is worth noting that the construction \cite{tub3} can be easily adopted to the $\mathfrak{sl}_2$-cobordism framework using the $\mathfrak{sl}_2$-foamation of Lauda--Queffelec--Rose \cite{lqr1} (and Blanchet's $\mathfrak{gl}_2$ foams \cite{bla1} due to sign issues). Moreover, it turns out that the relation between $\mathfrak{sl}_3$-webs and the multitableaux language is surprisingly useful to study for example dual canonical bases of the $\mathfrak{sl}_3$-web spaces.

Thus, the starting motivation of the author was to extend this explicit basis to the $\mathfrak{gl}_n$-web algebras. In order to do so, we follow the approach already indicated for $n=3$ in \cite{tub3}, i.e. the usage of categorified, diagrammatic quantum skew Howe duality studied independently in the $\mathfrak{gl}_n$-web framework in for example \cite{ckm}, \cite{lqr1} and \cite{mpt} (and later extended to all $n>1$ cf. \cite{my}).

\subsubsection{\texorpdfstring{$\mathfrak{gl}_n$}{gln}-webs, \texorpdfstring{$q$}{q}-skew Howe duality and combinatorics}\label{subsub-introwebs}
Let $\Lambda$ denote $n$-times the $\ell$-th fundamental $\dot{\Uu}_q(\mathfrak{gl}_m)$-weight. The point is now that the $q$-skew Howe duality realizes the $\mathfrak{gl}_n$-web space $W_n(\Lambda)$ as the $\dot{\Uu}_q(\mathfrak{gl}_m)$-module of highest weight $\Lambda$. In Lemma \ref{lem-webasF} we show something stronger, i.e. we give an explicit way to write any $\mathfrak{gl}_n$-web $u\in W_n(\vec{k})$ as a ($\bC(q)$-multiple of a) certain string of only $F_i^{(j)}$ acting as elements of $\dot{\Uu}_q(\mathfrak{gl}_m)$ under $q$-skew Howe duality: $\dot{\Uu}_q^-(\mathfrak{gl}_m)$ suffices (in fact, all $\mathfrak{gl}_n$-web relations follow only from the Serre relations) and we can see the $\mathfrak{gl}_n$-web spaces $W_n(\vec{k})$ as instances of $\dot{\Uu}_q(\mathfrak{gl}_m)$-highest weight theory.

Using this explicit description in terms of $F_i^{(j)}$, it was not too hard to extend the relations between $3$-multiparti\-tions and $\mathfrak{sl}_3$-webs, $3$-multitableaux and $\mathfrak{sl}_3$-flows, and Brundan, Kleshchev and Wang's degree of $3$-multitableaux (that comes from their work on graded Specht modules \cite{bkw}) and weights of $\mathfrak{sl}_3$-flows (as the authors has worked out in detail \cite{tub3}) to all $n>3$.
Moreover, recall that the $\mathfrak{gl}_n$-webs $u\in W_n(\vec{k})$ diagrammatically represent the invariant tensors $\mathrm{Inv}_{\Uu_q(\mathfrak{gl}_n)}(\bigotimes_{i}\Lambda^{k_i}\bC^n)\cong\mathrm{hom}_{\Uun}(\bC,\bigotimes_{i}\Lambda^{k_i}\bC^n)$ and the $\mathfrak{gl}_n$-flows and their weights are a combinatorial way to express these vectors explicitly in terms of the elementary tensors. Thus, since the $n$-multipartition and $n$-multitableaux framework comes naturally when working with some kind of Hecke algebras, one can loosely say that the Hecke algebra ``knows'' the $\mathfrak{gl}_n$-web framework.

It is clear, using $\mathrm{hom}_{\Uun}(A,B)\cong\mathrm{hom}_{\Uun}(\bC,A^*\otimes B)$ for $A,B\in\Ob(\Rep(\Uun))$ and the bijection between $n$-multitableaux and flows on $\mathfrak{gl}_n$-webs, cf. Section \ref{sec-tabwebs}, that the $\Uun$-intertwiners can be explained completely combinatorial using tableaux combinatorics.

Note now that for a closed $\mathfrak{gl}_n$-web $w$ these $\mathfrak{gl}_n$-flows give the decomposition into elementary tensors of the trivial $\Uu_q(\mathfrak{gl}_n)$-representation $\bC$, i.e. a certain quantum number. This number is the evaluation (up to a shift) of the $\mathfrak{gl}_n$-web $w$ using the relations found in \cite{ckm} - something that cannot be done directly by an algorithm yet. But we state in Theorem \ref{thm-evaluation} an \textit{inductive evaluation algorithm} for arbitrary closed $\mathfrak{gl}_n$-webs by using only $F$. Our algorithm uses the $q$-skew Howe duality and can be either stated in the combinatorial language of $n$-multitableaux (as we do) or in the algebraic language as the actions of the $F_i^{(j)}$ of $\dot{\Uu}_q(\mathfrak{gl}_m)$ on a highest weight vector $v_h$. As an almost direct consequence we are able to prove an explicit if-and-only-if condition for a $\mathfrak{gl}_n$-web $u\in W_n(\vec{k})$ to be a dual canonical basis element, see Theorem \ref{thm-dualcanwebs}.

We discuss another application of our algorithm in Section \ref{sec-linkswebs}\,: the evaluation of $\mathfrak{gl}_n$-webs is connected to colored Reshetikhin--Turaev $\mathfrak{gl}_n$-link polynomial $\langle L_D\rangle_n$ (see e.g. \cite{wu}), but the usual translation of an $a,b$-colored crossing $\overcrossing$ into sums of $\mathfrak{gl}_n$-webs would use $E$ and $F$, e.g.
\[
\left\langle
\xy
(0,0)*{\includegraphics[scale=.75]{figs/linkpoly/positive.eps}};
(2,-4.1)*{\scriptstyle b};
(-2,-4.35)*{\scriptstyle a};
\endxy
\right\rangle_n=\sum_{k=0}^b\underbrace{(-1)^{k+(a+1)b}q^{-b+k}}_{\alpha(k)}\cdot
\scalebox{.7}{$\xy(0,0)*{\includegraphics[scale=.75]{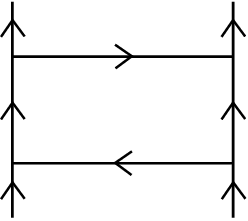}};(-11.5,-11)*{\scriptstyle a};(11.5,-11)*{\scriptstyle b};(6,8)*{\scriptstyle a+k-b};(2,-5)*{\scriptstyle k};(-9.8,0)*{\scriptstyle a+k};(10,0)*{\scriptstyle b-k};(-11.5,10)*{\scriptstyle b};(11.5,10)*{\scriptstyle a}\endxy$}
\leftrightsquigarrow
\sum_{k=0}^b \alpha(k)\cdot F^{(a+k-b)}_iE^{(k)}_iv_{\dots a,b \dots}
\]
Thus, we had to rearrange it (this corresponds to an embedding of $\dot{\Uu}_q(\mathfrak{gl}_i)$ into $\dot{\Uu}_q(\mathfrak{gl}_{i+1})$ and then use the relations in $\dot{\Uu}_q(\mathfrak{gl}_{i+1})$ to rewrite $F^{(a+k-b)}_iE^{(k)}_i$ in $\dot{\Uu}_q(\mathfrak{gl}_{i+1})$), using the observation that any $\mathfrak{gl}_n$-web can be obtained by a string of $F_i^{(j)}$, to
\[
\sum_{k=0}^b \alpha(k)\cdot 
\scalebox{.7}{$\xy
(0,0)*{\includegraphics[scale=.75]{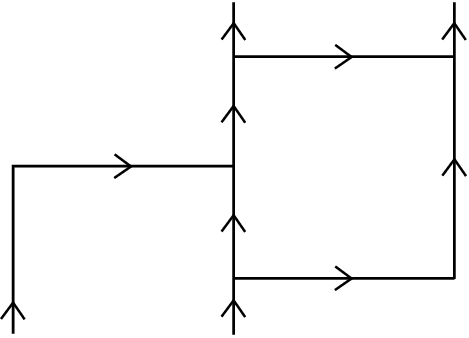}};
(17,18.5)*{F_{i+1}^{(a+k-b)}};
(-26,-20)*{a};
(2,-19.5)*{b};
(30,-19.5)*{0};
(-12,4)*{F_{i}^{(a)}};
(-26,-6.5)*{a};
(2.5,-7)*{k};
(33.5,-6.5)*{b-k};
(15,-9.5)*{F_{i+1}^{(b-k)}};
(-26,6.5)*{0};
(6,6.5)*{a+k};
(33.5,6)*{b-k};
(-26,19.5)*{0};
(2,19.5)*{b};
(30,19)*{a};
\endxy$}\leftrightsquigarrow \sum_{k=0}^b \alpha(k)\cdot F_{i+1}^{(a+k-b)}F_{i}^{(a)}F_{i+1}^{(b-k)}v_{\dots a,b \dots}.
\]
A neat fact is that the invariance under the Reidemeister moves, as we sketch in the proof of Theorem \ref{thm-evoflinks}, are then just instances of the higher quantum Serre relations (which can be found e.g. in \cite[Chapter 7]{lu}).

Using this, we give, as we explain in Section \ref{sec-linkswebs}, an explicit algorithm to compute the colored $\mathfrak{gl}_n$-MOY graph polynomials $\langle\cdot\rangle_{\mathrm{MOY}}$, and thus, the colored Reshetikhin--Turaev $\mathfrak{gl}_n$-polynomials.

Our version is completely combinatorial in nature and has the nice upshot that there is no conceptual difference between different $n$ and between the uncolored and colored setting.

\subsubsection{Categorified \texorpdfstring{$q$}{q}-skew Howe duality}\label{subsub-introcatwebs}
Categorified $q$-skew Howe duality in the $\mathfrak{gl}_n$ case means that there is a strong $\mathfrak{gl}_m$-$2$-representation $\Gamma_{m,n\ell,n}\colon\Ucatm\to\mathcal W^p_{\Lambda}$ of Khovanov--Lauda's \cite{kl5} categorification of $\dot{\Uu}_q(\mathfrak{gl}_m)$, that we denote by $\Ucatm$, on a certain category of matrix factorizations (see \cite{my} Definition 9.1) equivalent to a (suitable) module category $\mathcal W^p_{\Lambda}$ of the $\mathfrak{gl}_n$-web algebra $H_n(\Lambda)$ (see \cite{mack1} Definition 7.1). This functor was used \cite{mack1} to show that $H_n(\Lambda)$ is Morita equivalent to a certain block of the cyclotomic KLR algebra $R_{\Lambda}$. (We note that we follow \cite{mpt}, \cite{my} and \cite{tub3} with our notation for $\Ucatm$, $\Gamma_{m,n\ell,n}$ and $R_{\Lambda}$.)

Roughly speaking, on the categorified level the observations above allow us to extend the construction of the foamy version of HM graded cellular basis to the $\mathfrak{gl}_n$ setting. We do this by giving a growth algorithm for homomorphisms (modulo null-homotopic maps) of matrix factorizations in Definition \ref{defn-growthfoam}. These form a graded cellular basis, see Theorem \ref{thm-cellular}. The procedure is explicit and two immediate advantages are that the growth algorithm gives a basis of $\mathrm{HOM}_{\mathrm{nh}}(\widehat{u},\widehat{v})$ for any $u,v\in W_n(\vec{k})$ (here $\widehat{u},\widehat{v}$ are certain associated matrix factorizations) and computations can be done completely locally using the cyclotomic KLR relations, see \cite{kl1} or \cite{hm} for a list of these relations in terms of diagrams or multitableaux. Another direct advantage of using only the cyclotomic quotients is that everything is finite dimensional and can be done using explicit bases.
And, as before, our construction is completely combinatorial 
and one does not need the matrix factorization (or $\mathfrak{gl}_n$-foams).

\subsubsection{Divided powers and extended graphical calculus}\label{subsub-divthick}
A main difference between the $\mathfrak{gl}_n$-web setting and the categorified quantum groups $\Ucatm$ is that the first is closer to its Karoubi envelope. That is, it is possible to use divided powers in the $\mathfrak{gl}_n$-web setting, but not directly for $\Ucatm$. For $\Ucatm$ one has to go to a full $2$-subcategory of the Karoubi envelope $\Ucatmd$, denoted by $\Ucatmc$, which we briefly recall in Section \ref{sec-higherstuff}. Diagrammatically $\Ucatmc$ is given by a version, called thick calculus, of the extended graphical calculus from \cite{klms} where the reader can find more details.

In order to work with it, we extend Mackaay--Yonezawa's $2$-functor to $\Ucatmc$, see Theorem \ref{thm-thick}. Moreover, using Lemma \ref{lem-webasF} and Corollary \ref{cor-foambasis}, we show in Theorem \ref{thm-iso} explicitly (by giving a thick version of the HM basis) that the extended $2$-functor gives rise to an equivalence between the categories of modules over a certain block of a thickened cyclotomic KLR algebra, that we denote by $\check{R}_{\Lambda}$, and a suitable category of $H_n(\Lambda)$-modules. 

In fact, we show in Theorem \ref{thm-iso} that the $\mathfrak{gl}_n$-web algebra $H_n(\vec{k})$ is isomorphic to a (certain idempotent truncation) of $\check{R}_{\Lambda}$. Since $\check{R}_{\Lambda}$ can be studied completely combinatorially using thick KLR calculus and the thick  combinatorics of the HM basis, we can see this as a categorification of the corresponding results from the $\mathfrak{gl}_n$-web framework: elements of $\mathrm{HOM}_{\mathrm{nh}}(\widehat{u},\widehat{v})$ are parameterized by pairs of $n$-multitableaux of a certain shape.

An interesting remark is that working with $\Ucatmc$ (which is combinatorially not much more complicated than $\Ucatm$) suffices. That is, we can avoid working in the full Karoubi envelope $\Ucatmd$ where no diagrammatic or combinatorial definition is available for $n>2$ yet.

\subsubsection{\texorpdfstring{$\mathfrak{gl}_n$}{gln}-link homologies using combinatorics}\label{subsub-introlinkhom}
For the $\mathfrak{gl}_n$-link homologies this means that, using a complex as for example
\[
\begin{xy}
\xymatrix{\scalebox{.7}{$\xy
(0,0)*{\includegraphics[scale=.75]{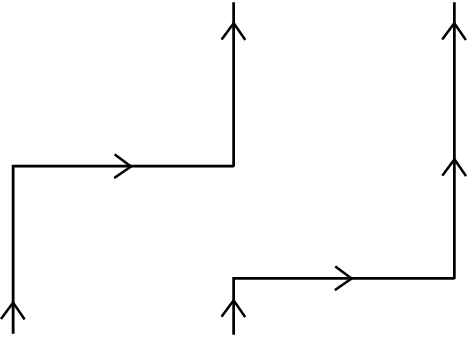}};
(-26,-20)*{1};
(2,-19.5)*{1};
(30,-19.5)*{0};
(-14,3.5)*{F_{i}};
(-26,-6.5)*{1};
(2.5,-7)*{0};
(30,-6.5)*{1};
(14,-10.5)*{F_{i+1}};
(-26,6.5)*{0};
(2.5,6.5)*{1};
(30,6)*{1};
(-26,19.5)*{0};
(2,19.5)*{1};
(30,19)*{1};
(0,-26)*{k=0; a=b=1};
\endxy$}
\{-1\}\ar[r]^/.4em/{d}   &    \scalebox{.7}{$\xy
(0,0)*{\includegraphics[scale=.75]{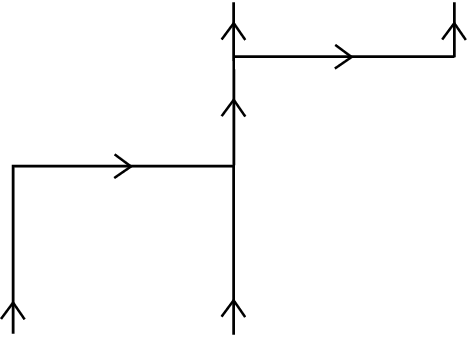}};
(14,17.5)*{F_{i+1}};
(-26,-20)*{1};
(2,-19.5)*{1};
(30,-19.5)*{0};
(-14,3.5)*{F_{i}};
(-26,-6.5)*{1};
(2.5,-7)*{1};
(30,-6.5)*{0};
(-26,6.5)*{0};
(2.5,6.5)*{2};
(30,6)*{0};
(-26,19.5)*{0};
(2,19.5)*{1};
(30,19)*{1};
(0,-26)*{k=1; a=b=1};
\endxy$}}
\end{xy}
\]
with differential $d=\xy(0,0)*{\Gamma_{m,n\ell,n}(\raisebox{-0.15em}{\includegraphics[width=10px]{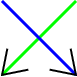}})};\endxy\colon F_iF_{i+1}v_{\dots 1,1,0 \dots}\{-1\}\to F_{i+1}F_iv_{\dots 1,1,0 \dots}$, we can define a complex that only uses the lower part $\Ucatmm$. Since categorified $q$-skew Howe duality descends to the cyclotomic KLR algebra, we can define the complex using only the cyclotomic KLR algebra with $d=\tilde{\Gamma}(\xy(0,0)*{\raisebox{-0.15em}{\includegraphics[width=10px]{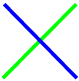}}};\endxy)\colon F_iF_{i+1}v_{\dots 1,1,0 \dots}\{-1\}\to F_{i+1}F_iv_{\dots 1,1,0 \dots}$. Thus, we obtain in this way Khovanov--Rozansky's $\mathfrak{gl}_n$-link homology using categorified $\dot{\Uu}_q(\mathfrak{gl}_m)$-highest weight theory.

The same works in the colored setup using thick calculus and the ($n$-multitableaux combinatorics of the) thick cyclotomic KLR algebra. And, as before for the colored Reshetikhin--Turaev $\mathfrak{gl}_n$-polynomials, everything is completely combinatorial in nature and there is no conceptual difference between different $n$ and between the uncolored and colored setting.

The explicit calculation of this complex is then a straightforward application of linear algebra: use the $\mathfrak{gl}_n$-web version of the HM basis to write an explicit basis for all resolutions. The differential is then just given by applying a thick cyclotomic KLR diagram from the left (stacking it on top) to the basis elements of the source. Then pairing the result with the dual of the thick HM basis for the target gives the differentials as a matrix. This gives an explicit way to compute the homology. It is worth noting again that for these calculations, due to the local properties of the construction, the matrix factorizations framework is not really needed: the homology is governed by the combinatorics of the (thick) HM basis and the (thick) cyclotomic KLR algebra. We explain how this works in Section \ref{sec-catpart2}.

\subsubsection{A remark about foams}\label{subsub-foam}
While typing this paper, the author was informed by Queffelec--Rose about their paper \cite{qr1} where the authors have independently obtained similar results for the $\mathfrak{gl}_n$-link homologies (especially, they independently discovered that the $\mathfrak{gl}_n$-link homology can be obtained in the KLR setting), but using $\mathfrak{gl}_n$-foams instead of matrix factorizations.

Note that Section \ref{sec-catpart2}, by similar arguments as \cite{lqr1}, can be extended to show that some of the aforementioned link homologies are the same. But this is not our purpose and is discussed \cite{qr1}. In fact, I like to thank Queffelec--Rose to point out to me that Chuang--Rouquier's version of the Rickard complex and the $F$-braiding complex I use (based on the observations above) are the same when passing to the (thick) Schur quotient (see \cite{msv2} for the definition of the $2$-Schur algebra). 

Moreover, everything in this paper can be done with their $\mathfrak{gl}_n$-foams too, since the combinatorics of the (thick) cyclotomic KLR and $n$-multitableaux suffices. In fact, as before with the Serre relations on the uncategorified level, all the $\mathfrak{gl}_n$-foam relations are consequences of the (thick cyclotomic) KLR relations. Although formally one would not need $\mathfrak{gl}_n$-foams: some facts are easier to see using $\mathfrak{gl}_n$-foams (e.g. the isotopies) and others using $n$-multitableaux (e.g. the combinatorics). So we claim that both perspectives are worthwhile.

A neat fact about the $\mathfrak{gl}_n$-foam framework is that Brundan--Kleshchev--Wang's degree of multitableaux (which originated from their work on graded Specht theory \cite{bkw}) is, under the translation we discuss in Section \ref{sec-tabwebs} together with the $\mathfrak{gl}_n$-foamation of Queffelec--Rose and their Definition 3.3, then nothing else than a (slightly adjusted) Euler characteristic of foams.
~\newline
\paragraph*{Acknowledgments} I especially would like to 
thank Anna Beliakova, Nils Carqueville, Lukas Lewark, 
Marco Mackaay, Jean-Baptiste Meilhan, Weiwei Pan, Hoel Queffelec, 
Louis-Hadrien Robert, David Rose, Antonio Sartori, 
Marko Sto\v{s}i\'{c}, Catharina Stroppel, Anne-Laure Thiel, 
Pedro Vaz and Paul Wedrich for helpful comments, questions 
and discussions about (higher) $q$-skew Howe duality, (cycl.) 
Hecke/KLR algebras, matrix factorizations and $\mathfrak{gl}_n$-link 
homologies. Special thanks to Marco Mackaay, an anonymous referee, Pedro Vaz and Paul Wedrich for numerous helpful comments on a draft of this paper.

I have also benefited from a lot of support from all 
members of the QGM who created a working atmosphere 
that encouraged me to continue my research. Moreover, 
I want to thank the TIFR in Mumbai for their 
hospitality - a big part of typing this paper has taken place at their Institute.

Which leaves open the question of what my personal contribution to this paper is.
\section{A short summary of the paper}\label{sec-summary}
\setcounter{subsection}{1}
\setcounter{subsubsection}{0}
\subsubsection{Summary of our notation}\label{sec-notation}
We start by summarizing our notation to avoid confusion due to the fact that we are working in the overlap of different worlds, i.e. the diagrammatic framework of $\Ucatm$ that consists of string diagrams, the combinatorial framework of the cyclotomic Hecke algebra that consists of multipartitions or multitableaux and the $\mathfrak{gl}_n$-web/matrix factorization framework that uses pictures (that is, the $\mathfrak{gl}_n$-webs) and the algebraic notion of matrix factorizations.

Since we tend to use highest and not lowest weight theory and $F$ and not $E$, we think of a $\Uu_q(\mathfrak{gl}_2)$-representation $V(N)$ of highest weight $N$ as
\begin{equation}\label{eq-hweight}
\xymatrix{
V_{-N}  \ar@<3pt>[r]^/-.2em/E  &  V_{-N+2}  \ar@<3pt>[l]^/.25em/F \ar@<3pt>[r]^/-.01em/E &V_{-N+4}\ar@<3pt>[l]^/.01em/F\ar@<3pt>[r]^/.3em/E&\dots\ar@<3pt>[l]^/-.35em/F\ar@<3pt>[r]^/-.2em/E & V_{N-4}\ar@<3pt>[l]^/.3em/F \ar@<3pt>[r]^/.01em/E& V_{N-2} \ar@<2pt>[l]^/-.01em/F \ar@<3pt>[r]^/.25em/E  &  V_N,  \ar@<3pt>[l]^/-.25em/F  \\
}
\end{equation}
That is, we usually read from right to left. This is our reading convention for all diagrams of $\Ucatm$ and the cyclotomic KLR algebra (thick ones as well): we think of them as being a sequence of $E$ and $F$ ordered from right to left. Moreover, we read them from bottom to top, i.e.
\[
\psi_3=\xy
(0,1)*{\includegraphics[width=25px]{figs/higherstuff/downcross}};
(-5,-3)*{\scriptstyle i};
(5,-3)*{\scriptstyle j};
(5.5,0)*{\scriptstyle\vec{k}};
\endxy\hspace*{0.25cm}\text{is}\hspace*{0.25cm}\psi_3\colon\mathcal{F}_{i}\mathcal{F}_{j}\onel\Rightarrow\mathcal{F}_{j}\mathcal{F}_{i}\onel\{\alpha^{ij}\}.
\]
However, we read $\mathfrak{gl}_n$-webs from right to left such that a turn of the diagrams by $\frac{\pi}{2}$ in clockwise direction matches the conventions before.

For example we read the string $F_1F_2F_1v_{(4,0,0)}$ as a $\mathfrak{gl}_4$-web (here the numbers on the grid correspond to the labels of the closed edges with the convention that we do not draw edges labeled $0$ and the edges labeled $n$ are pictured as a Bordeaux colored dotted line) as
\[
\scalebox{.7}{$\xy
(0,-1.5)*{\includegraphics[scale=0.65]{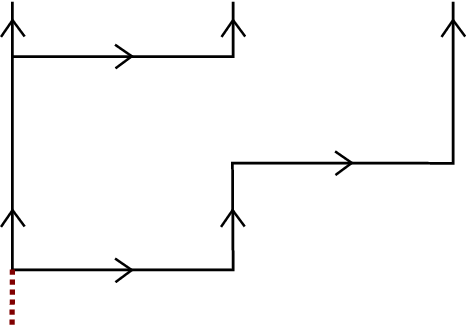}};
(-12,-11)*{F_1};
(-12,13)*{F_1};
(12,1)*{F_2};
(-22,-19)*{4};
(-22,-7.3)*{3};
(-22,5.7)*{3};
(-22,16)*{2};
(2,-19)*{0};
(2,-7.3)*{1};
(2,5.7)*{0};
(2,16)*{1};
(26,-19)*{0};
(26,-7.3)*{0};
(26,5.7)*{1};
(26,16)*{1};
\endxy$}
\]
Note that the labels of the middle and horizontal edges can easily be read off, since they are just the difference of the top right (left) and bottom right (left) numbers for the $F$ (the $E$).

Thus, since we can see a $\mathfrak{gl}_n$-web $u$ as a certain matrix factorization $\widehat{u}$ (see for example Section 5.4 in \cite{my}), we can read a $\Ucatm$ diagram as a certain (equivalence class of) homomorphisms of matrix factorizations from the bottom $\mathfrak{gl}_n$-web $u_b$ to the top $\mathfrak{gl}_n$-web $u_t$. Here the two $\mathfrak{gl}_n$-webs are obtained by letting the $E$ and $F$ for the bottom and top act on the weight vector $\vec{k}$.

We use the highest weight notation for the cyclotomic Hecke algebra too, i.e. reading multipartitions and multitableaux from right to left (the first entry is the rightmost etc.). Moreover, the elements of the $\mathfrak{gl}_n$-web algebra ${}_{u_t}H_n(\vec{k})_{u_b}$ are certain (equivalence classes of) homomorphisms of matrix factorizations $\mathcal{F}=\widehat{u}_b\to\widehat{u}_t$ that we inductively build from right to left. As an example, we decompose the whole morphism into steps
\begin{equation}\label{eq-decom}
u_b=u_1\to u_2\to u_3\to\dots\to u_{k-2}\to u_{k-1}\to u_k=u_t.
\end{equation}
Then we use stepwise certain homomorphisms of matrix factorizations $\phi_i\colon\widehat{u}_i\to\widehat{u}_{i+1}$ and we set $\mathcal{F}=\phi_k\circ\dots\circ\phi_1$. For example
\[
\widehat{CR}_{1,12}\colon
\scalebox{.7}{$\xy
(0,0)*{\includegraphics[scale=.75]{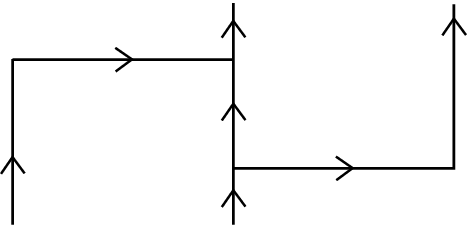}};
(-26,-13)*{1};
(2.5,-13)*{2};
(30,-13)*{0};
(14,-4)*{F_2};
(-26,0)*{1};
(2.5,0)*{1};
(30,0)*{1};
(-14,10)*{F_1};
(-26,13)*{0};
(2.5,13)*{2};
(30,13)*{1};
\endxy$}
\to
\scalebox{.7}{$\xy
(0,0)*{\includegraphics[scale=.75]{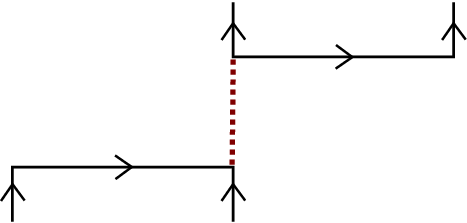}};
(-26,-13)*{1};
(2.5,-13)*{2};
(30,-13)*{0};
(14,10)*{F_2};
(-26,0)*{0};
(2.5,0)*{3};
(30,0)*{0};
(-14,-4)*{F_1};
(-26,13)*{0};
(2.5,13)*{2};
(30,13)*{1};
\endxy$}
\]
is such a local step. Here $n=3$. The reader familiar with the $\mathfrak{sl}_2$ or $\mathfrak{sl}_3$ framework (see for example \cite{kh4} or \cite{mpt}) may think of it as building a $\mathfrak{sl}_2$-cobordism or $\mathfrak{sl}_3$-foam by composing (in a certain way) basic pieces such as saddles, zips, unzips and dotted identities. Roughly the same works for $\mathfrak{gl}_n$-foams and the reader can always think in terms of foams - if (s)he prefers foams.

\subsubsection{A rough sketch of our approach: the uncategorified world}\label{subsub-sketch}
We start by giving a short summary of the relations between the three worlds mentioned above on the uncategorified level. The crucial diagram is
\[
\xymatrix{
n\text{-multitableaux}  \ar@{<->}[rrr]^/.3em/{\mbox{Sections \ref{sec-slnwebs} and \ref{sec-tabwebs}}}  & & & \Rep(\Uun)  \ar@{<->}[rrr]^/.3em/{\mbox{Section \ref{sec-slnwebs}}}  & & &  \mathfrak{gl}_n\text{-webs}  \ar@{<->}@/^0.5cm/[llllll]^/-3em/{\mbox{Section \ref{sec-tabwebs}}}  \\
}
\]
We call the three worlds loosely combinatorics, representation theory and topology. In our opinion all of them have their own advantages:
\begin{itemize}
\item For $n$-multitableaux everything is very explicit and can be done inductively/algorithmically by certain operations on $n$-multitableaux motivated by the classical story of the representation theory of the symmetric group.
\item $\Rep(\Uun)$ is the category that we want to understand.
\item The third one is the category of $\mathfrak{gl}_n$-webs. Here it is easy to see the topology, e.g. isotopies and the connection to $\mathfrak{gl}_n$-link polynomials. In fact, it is non-trivial that the rather ``rigid'' $n$-multitableaux framework is isotopy invariant and on the other hand the $\mathfrak{gl}_n$-link polynomials are completely determined by this ``rigid'' combinatorics. This follows from the non-trivial translations in Sections \ref{sec-tabwebs} and \ref{sec-linkswebs}.
\end{itemize}
Let us focus on $n=2$ for now. The following table summarizes the uncategorified story.

Note that $\mathrm{hom}_{\Uun[2]}(\bC,\bC^2\otimes\bC^2)\cong \mathrm{Inv}_{\Uun[2]}(\bC^2\otimes\bC^2)\subset \bC^2\otimes\bC^2$. Fix for the $\Uun[2]$-vector representation $\bC^2$ the basis $x_{\{1\}}$ and $x_{\{2\}}$ with the first vector in the $+1$- and the second in the $-1$-weight space of $\bC^2$. We write $x_{12}=x_{\{1\}}\otimes x_{\{2\}}$ and $x_{21}=x_{\{2\}}\otimes x_{\{1\}}$.

The dotted line (leash) represents $\Lambda^2\bC^2$, the empty space $\Lambda^0\bC^2$ and the line $\Lambda^1\bC^2\cong \bC^2$. Since the first two spaces are trivial (the reader should think of the leash as non-existing: it encodes certain signs): the bottom/top of the right column is the source/target of the hom-space in the middle.
\begin{center}
\renewcommand{\arraystretch}{1.7}
\scalebox{.9}{$\begin{tabular}{ p{5cm} | p{5cm} | p{5cm} }
\hline
Combinatorics & Representation theory & Topology \\ \hline
$r(\left(\;\emptyset\;,\;\xy (0,0)*{\begin{Young}1 \cr\end{Young}}\endxy\;\right))=r(\left(\;\xy (0,0)*{\begin{Young}1 \cr\end{Young}}\endxy\;,\;\emptyset\;\right))$ & $u\in\mathrm{hom}_{\Uun[2]}(\bC,\bC^2\otimes\bC^2)$ & $u=\xy(0,0)*{\includegraphics[width=40px]{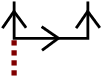}};(-4,-4)*{\scriptstyle 2};(6.5,-4)*{\scriptstyle 0};(-4,4)*{\scriptstyle 1};(6.5,4)*{\scriptstyle 1};(2,1.5)*{\scriptstyle F};\endxy\in W_2((1,1))$ \\ \hline
$\left(\;\emptyset\;,\;\xy (0,0)*{\begin{Young}1 \cr\end{Young}}\endxy\;\right)\neq\left(\;\xy (0,0)*{\begin{Young}1 \cr\end{Young}}\endxy\;,\;\emptyset\;\right)$ & $u=x_{21}-qx_{12}\in\bC^2\otimes\bC^2$ & $\xy(0,0)*{\includegraphics[width=40px]{figs/intro/intro-web}};(-2.5,4)*{\scriptstyle \{2\}};(8,4)*{\scriptstyle \{1\}};(1.5,2.5)*{\scriptstyle \{1\}};(-1.5,-4)*{\scriptstyle \{2,1\}};\endxy\neq \xy(0,0)*{\includegraphics[width=40px]{figs/intro/intro-web}};(-2.5,4)*{\scriptstyle \{1\}};(8,4)*{\scriptstyle \{2\}};(1.5,2.5)*{\scriptstyle \{2\}};(-1.5,-4)*{\scriptstyle \{2,1\}};\endxy$ \\ \hline
degree 0 and degree 1 & coefficients $q^0$ and $q^1$ & weight 0 and weight 1\\
\hline
\end{tabular}$}
\end{center}
To summarize, $2$-multitableaux of the same residue sequence $r(\cdot)$ represent 1:1 certain $\mathfrak{gl}_2$-webs, $2$-multitableaux represent 1:1 flows on these $\mathfrak{gl}_2$-webs and the degree of the $2$-multitableaux gives the weight of the flows. It follows from the middle and the right columns that one can hope to get information about dual canonical bases (for $\mathfrak{gl}_n$-webs a dual canonical basis in our notation is a ``good basis for $q\to 0$'', i.e. having a positive exponent property) and about $\mathfrak{gl}_n$-link polynomials using the left column. This is what we show in Sections \ref{sec-tabwebs} and \ref{sec-linkswebs}.

\subsubsection{A rough sketch of our approach: the categorified world}\label{subsub-sketch2}
From the categorified viewpoint one can hope that the $n$-multitableaux framework can be used to define cellular bases (since they give rise to a method to obtain the indecomposable modules that decategorify to the dual canonical basis) and an explicit method to obtain the $\mathfrak{gl}_n$-link homologies. 
This is what we show in Sections \ref{sec-catpart1} and \ref{sec-catpart2}.

The crucial question is how to generate the string in \eqref{eq-decom}. To do this we use (a thick) 
HM basis. This works roughly as follows. Fix two $\mathfrak{gl}_n$-webs $u,v\in W_n(\vec{k})$. There is a homogeneous $\bC$-basis (that, even from the cyclotomic Hecke side, also works over $\bZ$, see Theorem 3.14 in \cite{hm1} or \cite{li1}) of $\mathrm{HOM}_{\mathrm{nh}}(\widehat{u},\widehat{v})$ (or alternatively of $\mathfrak{gl}_n$-foam spaces) where each basis element is determinated by two $n$-multitableaux $\vec{T},\vec{T}^{\prime}$, one for $u$ and one for $v$, with a certain fixed number of boxes $c(\vec{k})$. The string in \eqref{eq-decom} is generated by actions $\sigma,\sigma^{\prime}$ of elements of $S_{c(\vec{k})}$ by permuting nodes. The different basic pieces then depend on the difference of the residue of the permuted nodes. This can be seen as an analog of classical Specht theory.

The actions are roughly obtained as follows. The $n$-multitableaux $\vec{T},\vec{T}^{\prime}$ are of the same shape, since the shape only depends on the boundary of the $\mathfrak{gl}_n$-webs. Then there is a $n$-multitableaux $T_{\vec{\lambda}}$ in between of the same shape with all its nodes filled in an ordered way. The actions are then given by applying a suitable sequence of transpositions $\tau_k(i,j),\tau_k^{\prime}(i,j)$ from $T_{\vec{\lambda}}$ to $\vec{T},\vec{T}^{\prime}$.

Let us sketch in a diagram how the ``higher'' Specht basis works. Here we focus on $n=2$ and, as in Example \ref{ex-cheating}, use Bar-Natan's $\mathfrak{sl}_2$-cobordisms \cite{bn2}. (They are useful to illustrate the concepts, although we do not work with them due to sign issues.) In general one works with $n$-multitableaux, thick calculus and $\mathfrak{gl}_n$-matrix factorizations or $\mathfrak{gl}_n$-foams. Below we read again from bottom to top, i.e. the reader may think of the $\mathfrak{sl}_2$-web $u$ sitting at the bottom and the $\mathfrak{sl}_2$-web $v$ at the top (the colors in the middle column indicate the different residues of the nodes, e.g. $r(T_{\vec{\lambda}})=(2,2,3,1)$). The element below is in $\mathrm{hom}_{\check{R}(\Lambda)}(F_1F_2F_3F_2,F_3F_1F_2F_2)$.
\[
\begin{xy}
\xymatrix@!C{
\text{cycl. Hecke algebra}\ar@{.}[d]|/-.35em/{2\text{-multitableaux}} & \text{cycl. KLR algebra}\ar@{.}[d]|/-.35em/{\text{string diagrams}} & \mathfrak{sl}_2\text{-webs}\ar@{.}[d]|/-.35em/{\mathfrak{sl}_2\text{-cobordisms}}\\
\vec{T}^{\prime}=\left(\;\xy(0,0)*{\begin{Young}1\cr\end{Young}}\endxy\;,\xy(0,0)*{\begin{Young}2&4\cr 3\cr\end{Young}}\endxy\;\right)\ar@{<->}[r] & \raisebox{0.75ex}{\xy(0,0)*{\includegraphics[width=62.5px]{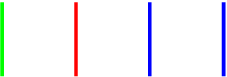}}\endxy}\ar@{<->}[r] & \raisebox{0.75ex}{\xy(0,0)*{\includegraphics[width=62.5px]{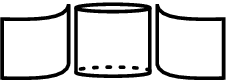}}\endxy}\\
T_{\vec{\lambda}}=\left(\;\xy(0,0)*{\begin{Young}1\cr\end{Young}}\endxy\;,\xy(0,0)*{\begin{Young}2&3\cr 4\cr\end{Young}}\endxy\;\right)\ar[u]|{\tau_3(1,3)}\ar@{<->}[r] & \raisebox{0.75ex}{\xy(0,0)*{\includegraphics[width=62.5px]{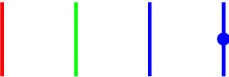}};(9.5,-2)*{\scriptstyle 2};(4.5,-2)*{\scriptstyle 2};(-4.5,-2)*{\scriptstyle 3};(-9.75,-2)*{\scriptstyle 1};\endxy}\ar[u]|{\raisebox{0.5ex}{\xy(0,0)*{\includegraphics[width=62.5px]{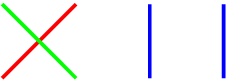}}\endxy}}\ar@{<->}[r] & \raisebox{0.75ex}{\xy(0,0)*{\includegraphics[width=62.5px]{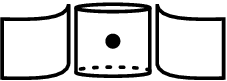}}\endxy}\ar[u]|{\raisebox{0.5ex}{\xy(0,0)*{\includegraphics[width=62.5px]{figs/intro/cob-top}}\endxy}}\\
\vec{T}^{\prime\prime}=\left(\;\xy(0,0)*{\begin{Young}2\cr\end{Young}}\endxy\;,\xy(0,0)*{\begin{Young}1&3\cr 4\cr\end{Young}}\endxy\;\right)\ar[u]|{\tau_1(2,2)}\ar@{<->}[r] & \raisebox{0.75ex}{\xy(0,0)*{\includegraphics[width=62.5px]{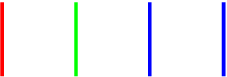}}\endxy}\ar[u]|{\raisebox{0.5ex}{\xy(0,0)*{\includegraphics[width=62.5px]{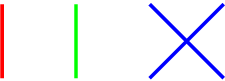}}\endxy}}\ar@{<->}[r] & \raisebox{0.75ex}{\xy(0,0)*{\includegraphics[width=62.5px]{figs/intro/cob-top}}\endxy}\ar[u]|{\raisebox{0.5ex}{\xy(0,0)*{\includegraphics[width=62.5px]{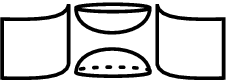}}\endxy}}\\
\vec{T}=\left(\;\xy(0,0)*{\begin{Young}3\cr\end{Young}}\endxy\;,\xy(0,0)*{\begin{Young}1&2\cr 4\cr\end{Young}}\endxy\;\right)\ar[u]|{\tau_2(3,2)}\ar@{<->}[r] & \raisebox{0.75ex}{\xy(0,0)*{\includegraphics[width=62.5px]{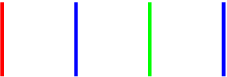}}\endxy}\ar[u]|{\raisebox{0.5ex}{\xy(0,0)*{\includegraphics[width=62.5px]{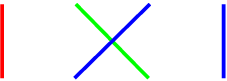}}\endxy}}\ar@{<->}[r] & \raisebox{0.75ex}{\xy(0,0)*{\includegraphics[width=62.5px]{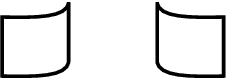}}\endxy}\ar[u]|{\raisebox{0.5ex}{\xy(0,0)*{\includegraphics[width=62.5px]{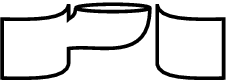}}\endxy}}
}
\end{xy}
\]
We stress again: given$\tau_k(i,j)$ one uses a certain $\mathfrak{sl}_2$-cobordism whose position depends on $k$ and whose shape depends on the difference between the residues $|i-j|$. From bottom to top we see a saddle (difference $1$), a cup-cap (difference $0$) and a shift (difference $>1$). The shift is hard to illustrate here but it just shifts the relative positions of the right and left arc, see also third diagram in Example \ref{ex-cheating}. Another important fact is that all possible dots are just given by $T_{\vec{\lambda}}$. This corresponds to an identity with dots that determines the cell in the cellular basis.

We can use this basis for the colored $\mathfrak{gl}_n$-link homologies as follows. In the language of Bar-Natan from \cite{bn2}: the Khovanov chain complex has chain groups consisting of certain $\mathfrak{sl}_2$-webs and the differentials are $\mathfrak{sl}_2$-cobordisms between them. Thus, using the approach indicated above, we can formulate a chain complex whose chain groups are strings of $F_i^{(j)}$ and whose differentials are $\Ucatmc$ diagrams between them. For example (compare to Example \ref{ex-honestcalc})
\[
\begin{xy}
\xymatrix{
F_2F_1F_2F_1v_{(2^1)}\{-2\}\ar[rr]^{\xy(0,0)*{\raisebox{-0.15em}{\includegraphics[width=10px]{figs/higherstuff/downcross}}};\endxy\colon F_1F_2\to F_2F_1} & & F_2F_2F_1F_1v_{(2^1)}\{-1\} \\    
}
\end{xy}
\]
would be such a complex. This can be thought of as the local version of colored $\mathfrak{gl}_n$-link homology.

In Bar-Natan's picture: in order to do calculations one applies $\mathrm{hom}_{\check{R}(\Lambda)}(\emptyset,\cdot)$ and the chain groups are then given by (possibly dotted) cups and the differential $d$ is just given by gluing the $\mathfrak{sl}_2$-cobordism $d$ on top of the cups. Then use the dual (possibly dotted) cap basis of the target, evaluate the closed $\mathfrak{sl}_2$-cobordism and obtain numbers $\bC$. This gives $d$ as a matrix.

We do literally the same: we apply $\mathrm{hom}_{\check{R}(\Lambda)}(F^c_{(n^{\ell})},\cdot)$ (where $F^c_{(n^{\ell})}$ is a certain canonical string of leash-shifts that can be thought of as non-existent). Now the chain groups are certain $\check{R}_{\Lambda}$-modules and the differentials are $\check{R}_{\Lambda}$-module maps given by composition from the right (gluing to the top).

The rest is also the same as in Bar-Natan's picture. That is, write a thick HM basis element $m_s$ (the cup basis) of the source, glue the differential $d$ to its top and pair it with a thick HM dual basis element $m_t$ (the cap basis which is literally obtained by reading everything backwards) of the target (here $F^c_{(n^{\ell})}=F_2^{(2)}F_1^{(2)}$):
\[
\xy
(0,0)*{\includegraphics[scale=.75]{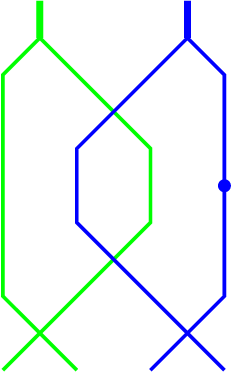}};
(20,0)*{\leftrightsquigarrow};
(40,0)*{\includegraphics[scale=.75]{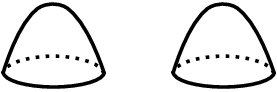}};
(0,-29)*{\includegraphics[scale=.75]{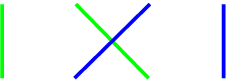}};
(20,-29)*{\leftrightsquigarrow};
(40,-29)*{\includegraphics[scale=.75]{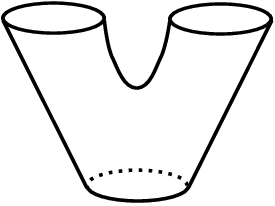}};
(0,-50)*{\includegraphics[scale=.75]{figs/catcell/dual-example1}};
(20,-50)*{\leftrightsquigarrow};
(40,-50)*{\includegraphics[scale=.75]{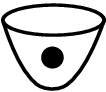}};
(0,-24)*{\circ};
(0,-35)*{\circ};
(10,-40)*{m_s};
(10,-29)*{d};
(10,0)*{m_t};
(-30,-60)*{\left(\;\xy(0,0)*{\begin{Young}1&2\cr\end{Young}}\endxy\;,\;\xy(0,0)*{\begin{Young}1&2\cr\end{Young}}\endxy\;\right)};
(-30,-50)*{\left(\;\xy(0,0)*{\begin{Young}1&3\cr\end{Young}}\endxy\;,\;\xy(0,0)*{\begin{Young}2&4\cr\end{Young}}\endxy\;\right)};
(-30,-40)*{\left(\;\xy(0,0)*{\begin{Young}1&2\cr\end{Young}}\endxy\;,\;\xy(0,0)*{\begin{Young}3&4\cr\end{Young}}\endxy\;\right)};
(-30,-55)*{\uparrow};
(-37.5,-45)*{\tau_2(2,1)};
(-30,-45)*{\uparrow};
(-40,-55)*{\text{unthickening}};
(-31,20)*{\left(\;\xy(0,0)*{\begin{Young}1&2\cr\end{Young}}\endxy\;,\;\xy(0,0)*{\begin{Young}1&2\cr\end{Young}}\endxy\;\right)};
(-31,10)*{\left(\;\xy(0,0)*{\begin{Young}2&4\cr\end{Young}}\endxy\;,\;\xy(0,0)*{\begin{Young}1&3\cr\end{Young}}\endxy\;\right)};
(-31,0)*{\left(\;\xy(0,0)*{\begin{Young}3&4\cr\end{Young}}\endxy\;,\;\xy(0,0)*{\begin{Young}1&2\cr\end{Young}}\endxy\;\right)};
(-31,-10)*{\left(\;\xy(0,0)*{\begin{Young}2&4\cr\end{Young}}\endxy\;,\;\xy(0,0)*{\begin{Young}1&3\cr\end{Young}}\endxy\;\right)};
(-31,-20)*{\left(\;\xy(0,0)*{\begin{Young}1&3\cr\end{Young}}\endxy\;,\;\xy(0,0)*{\begin{Young}2&4\cr\end{Young}}\endxy\;\right)};
(-31,-15)*{\downarrow};
(-49,-15)*{\tau_1(1,1)\text{ and }\tau_3(2,2)};
(-31,-5)*{\downarrow};
(-38.5,-5)*{\tau_2(2,1)};
(-31,5)*{\downarrow};
(-38.5,5)*{\tau_2(1,2)};
(-31,15)*{\downarrow};
(-41,15)*{\text{unthickening}};
\endxy
\]
The elements of the source are elements of the $\check{R}_{\Lambda}$-module $\mathrm{hom}_{\check{R}(\Lambda)}(F_2^{(2)}F_1^{(2)},F_2F_1F_2F_1)$, the elements of the target are elements of the $\check{R}_{\Lambda}$-module $\mathrm{hom}_{\check{R}(\Lambda)}(F_2F_2F_1F_1,F_2^{(2)}F_1^{(2)})$ and the differential is an $\check{R}_{\Lambda}$-module map in $\mathrm{hom}_{\check{R}(\Lambda)}(F_2F_1F_2F_1,F_2F_2F_1F_1)$. Thus, the composite is an element of the $1$-dimensional $\check{R}_{\Lambda}$-module $\mathrm{hom}_{\check{R}(\Lambda)}(F_2^{(2)}F_1^{(2)},F_2^{(2)}F_1^{(2)})$: it is just a number in $\bC$. This can be seen as the evaluation of closed $\mathfrak{gl}_n$-foams that categorifies our algorithm to evaluate closed $\mathfrak{gl}_n$-webs. This number can be obtained explicitly by using rules from thick calculus (see \cite{klms} or \cite{sto1}) that can also be stated directly in terms of $n$-multitableaux. In fact, one can (if one likes) say that the evaluation of closed $\mathfrak{gl}_n$-foams is already inside of at least work by HM. Although the combinatorics go back even further, see the references in Section 6 of \cite{hm}.

\subsubsection{A calculation example}\label{subsub-examplecalc}

We sketch by an example our approach to calculate the (colored) Khovanov--Rozansky $\mathfrak{gl}_n$-link homologies. We want to stress three things again before we start: the possibility for calculations is just one application of our translation. Moreover, it follows from Rouquier's universality theorem (see \cite[Corollary 5.7]{rou}) that all link homologies using the MOY-calculus as underlying uncategorified framework and analogs of Khovanov's original differentials have to give the same result (very, very roughly: the $\mathfrak{gl}_n$-web space $W_n(\Lambda)$ is the $\dot{\Uu}_q(\mathfrak{gl}_m)$-representation of highest weight $\Lambda$ and ``there is only one categorification'' of this). Thus, we do not need neither matrix factorizations nor $\mathfrak{gl}_n$-foams (we need them to show that everything \textit{works}). Another point we would like to add: our framework has enough local properties to perform an analogue of Bar-Natan's ``divide and conquer'' algorithm from \cite{bn3}. His local simplifications seem to correspond on our side to the categorification of the higher quantum Serre relations by Sto\v{s}i\'{c}, see Sections 4 and 5 in \cite{sto1}. Life is short, but this paper is not: we only sketch how this should work in Remark \ref{rem-betterway}.

Now the example: this is the Hopf link example that also appears in the Examples \ref{ex-Hopf} and \ref{ex-hopfnew} where the reader can find the pictures. We set $n=3, m=6$ and we have colored the two positive crossings with the colors $1$ (left component) and $2$. The presentation via $F_i^{(j)}$ is
\[
\mathrm{Hopf}=F_4^{(3)}F_5^{(2)}F_3^{(2)}F_2^{(2)}F_1^{(2)}T_{2,1,3}T_{1,2,2}F_5F_4F_3F_1F_2^{(3)}v_{(3,3,0,0,0,0)}.
\]
where the $T$ represent the braiding and the right and left strings of $F_i^{(j)}$ (that we shortly denote by $F_b$ and $F_t$) correspond to the bottom and top closure respectively. The braid complex $T_{2,1,3}T_{1,2,2}\tilde v=T_{2,1,3}T_{1,2,2}v_{\dots,1,2,\dots}$ (that technically takes place in a Schur quotient of $\check{\mathcal{U}}(\mathfrak{gl}_6)$) is
\[
\begin{xy}
\xymatrix{
& {\color{red}F^{(2)}_4F^{(2)}_3}{\color{green}F_2F_3}\tilde v\{-1\}\ar[dr]|{\phantom{-}\xy
(0,0)*{\raisebox{-0.15em}{\includegraphics[width=10px]{figs/higherstuff/HM-strings1}}};\endxy\colon F_2F_3\to F_3F_2}\ar@{.}[dd]|{\bigoplus} & \\
{\color{red}F_4F^{(2)}_3F_4}{\color{green}F_2F_3}\tilde v\{-2\}\ar[ur]|{\phantom{-}\xy
(0,0)*{\raisebox{-0.15em}{\includegraphics[width=20px]{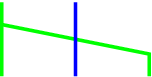}}};\endxy\colon F_4F^{(2)}_3F_4\to F^{(2)}_4F^{(2)}_3}\ar[dr]|{\phantom{-}\xy
(0,0)*{\raisebox{-0.15em}{\includegraphics[width=10px]{figs/higherstuff/HM-strings1}}};\endxy\colon F_2F_3\to F_3F_2} &  & {\color{red}F^{(2)}_4F^{(2)}_3}{\color{green}F_3F_2}\tilde v\{0\}\\
& {\color{red}F_4F^{(2)}_3F_4}{\color{green}F_3F_2}\tilde v\{-1\}\ar[ur]|{-\xy
(0,0)*{\raisebox{-0.15em}{\includegraphics[width=20px]{figs/linkhom/HM-diff}}}\endxy\colon F_4F^{(2)}_3F_4\to F^{(2)}_4F^{(2)}_3} &    
}
\end{xy}
\]
with leftmost part in homological degree zero. In the rightmost part we see $F_3^{(2)}F_3$ that is isomorphic (given by an explicit isomorphism) to $[3]F^{(3)}_3$ (this is a shorthand notation for a shifted direct sum) in $\check{\mathcal{U}}(\mathfrak{gl}_6)$, see Theorem 5.1.1 in \cite{klms}. By using one of Sto\v{s}i\'{c}'s categorifications of the higher quantum Serre relations (Theorem 3 in \cite{sto1}), we see that $F_3^{(2)}F_2F_3$ is (in the Schur quotient) isomorphic (given by an explicit isomorphism) to $[2]F_3^{(3)}F_2$. Using a Gauss elimination (induced differential $\tilde d$!) we see that the middle top and the non-top degree part of the rightmost component will cancel and the complex simplifies to (with $d=\xy
(0,0)*{\raisebox{-0.15em}{\includegraphics[width=10px]{figs/higherstuff/HM-strings1}}};\endxy\colon F_2F_3\to F_3F_2$ as before)
\[
\begin{xy}
\xymatrix{
F_4F^{(2)}_3F_4F_2F_3\tilde v\{-2\}\ar[r]^{d} & F_4F^{(2)}_3F_4F_3F_2\tilde v\{-1\}\ar[r]^/0.3em/{\tilde d} & F^{(2)}_4F^{(3)}_3F_2\tilde v\{2\}.
}
\end{xy}
\]
We now close it with $F_t,F_b$. By using $\mathrm{hom}_{\check{R}}(F^c_{(3^{2})},\cdot)$ and calculate the HM basis for the left two $\check{R}$-modules and the dual for the right two $\check{R}$-modules, we get, using the approach sketched above, the two differentials as matrices. Thus, calculating the homology is just linear algebra.
\subsubsection{Paper structure}\label{subsub-structure}
Before we summarize the paper let us note that Section \ref{sec-basics} (mostly) introducing notations and can be skipped by readers who feel safe using the language of $\mathfrak{gl}_n$-webs and categorified quantum groups. We try to illustrate everything with plenty of examples to help the reader on his/her way through this (too?) long paper. One can always go back to Section \ref{sec-basics} and look for the explicit definitions.

The summary of the uncategorified picture in this paper is as follows.

We start in Sections \ref{sec-combi} and \ref{sec-slnwebs} by recalling some notions and fixing notations, e.g. the notions of $n$-multitableaux and $\mathfrak{gl}_n$-webs. Most parts in those sections are known, but we have also included new results related to our framework, e.g. in Theorem \ref{thm-kkele} we show how the flows and their weights correspond to the decomposition into elementary tensors (we think this should be known, but we were unable to find the result in the literature).

In Section \ref{sec-tabwebs}, among other things, we give a detailed discussion of the relation between the $\mathfrak{gl}_n$-webs and the $n$-multitableaux language.

The combinatorial heart of Section \ref{sec-tabwebs} is the extended growth algorithm from Definition \ref{defn-exgrowth} that gives a bijection between $\mathfrak{gl}_n$-webs with flows and $n$-multitableaux (see Proposition \ref{prop-extgrowth}). This bijection can be extended to match Brundan--Kleshchev--Wang's degree of $n$-multitableaux with weights of flows (see Proposition \ref{prop-extgrowthdeg}).

We use this to give an evaluation algorithm in Theorem \ref{thm-evaluation} and its application to the dual canonical basis in Theorem \ref{thm-dualcanwebs}. Note that Lemma \ref{lem-webasF} implies that all relations from \cite{ckm} follow from the higher Serre relations (see e.g. in Chapter 7 in \cite{lu}).

Section \ref{sec-linkswebs} contains the application of the evaluation algorithm for calculations of the colored Reshetikhin--Turaev $\mathfrak{gl}_n$-polynomials in detail. That is, after showing in Lemma \ref{lem-linkasF} how links can be explicitly seen as strings of $F_i^{(j)}$, we show in Theorem \ref{thm-evoflinks} how to use $n$-multitableaux to compute the colored Reshetikhin--Turaev $\mathfrak{gl}_n$-polynomials. A neat fact (although we only sketch how it works): the invariance under the Reidemeister moves is a consequence of  the higher Serre relations. Afterwards we give two explicit examples (see Section \ref{subsub-twoexamples}).

The summary of the categorified picture in this paper is as follows.

We start in Section \ref{sec-higherstuff} by recalling some notions and fixing notations. Most parts in those sections are known, but we have also included new results, e.g. thick categorical q-skew Howe duality, see Theorem \ref{thm-thick}.

In Section \ref{sec-catpart1} we give the $\mathfrak{gl}_n$-web version of the HM basis by a growth algorithm, see Definition \ref{defn-growthfoam} (for the dual HM basis see Remark \ref{rem-othermethod}), and show that it is a graded cellular basis in Theorem \ref{thm-cellular}. Moreover, we relate our construction to the thick KL-R algebra in Theorem \ref{thm-iso}.

And in the last section, i.e. Section \ref{sec-catpart2}, we define our version of the colored $\mathfrak{gl}_n$-link homology in Definition \ref{defn-khshoweF2} and show in Theorem \ref{thm-same2} that it agrees with the colored Khovanov--Rozansky $\mathfrak{gl}_n$-link homology. Moreover, we discuss some local properties related to the Rickard complex in Lemma \ref{lem-rickardisf}. Afterwards we show (Definition \ref{defn-hmcomplex} and Theorem \ref{thm-algworks}) how to use the $\mathfrak{gl}_n$-web version of the HM basis for calculations.

Note that this shows that the Khovanov--Rozansky $\mathfrak{gl}_n$-link homologies are completely combinatorial in nature. Thus, everything is down to earth and can be made explicit.

We note again that, in order to illustrate that everything is explicit, we give numerous examples. We hope these will help the reader to understand the sometimes very confusing combinatorics. 
\section{Basic notions}\label{sec-basics}
\subsection{Combinatorics, (multi)partitions and (multi)tableaux}\label{sec-combi}
In this section we define/recall the combinatorial notions about multitableaux that we use in this paper.

For an integer 
$m\geq 1$ let 
\[
\Lambda(m,d)=\big\{\lambda=(\lambda^1,\dots,\lambda^m)\in \mathbb{N}^m\mid \sum_{j=1}^m\lambda^j=d\big\}
\]
be the set of \textit{compositions} of $d$ of length $m$. 
By $\Lambda^+(m,d)\subset\Lambda(m,d)$ we denote the 
subset of \textit{partitions}, i.e. all $\lambda\in\Lambda(m,d)$ such that 
\[
\lambda^1\geq\lambda^2\geq \dots\geq \lambda^m\geq 0.
\]

Let $\Lambda^{(+)}(m,d)_{I}\subset \Lambda^{(+)}(m,d)$ be the subset of compositions (or partitions)
whose entries are all in $I\subset\bN$. (Here we use a notation that we will use throughout, i.e. $(+)$ means both versions, with or without the $+$,
with the appropriate adaption of the notions in question.)
In particular, for some fixed $M\in\bN$ we use $\Lambda^{(+)}(m,d)_{M}$ as a notation for
\[
\Lambda^{(+)}(m,d)_M=\big\{\lambda=(\lambda^1,\dots,\lambda^m)\in \mathbb{N}^m\mid \sum_{j=1}^m\lambda^j=d,\;\lambda^j\in\{0,\dots,M\}\big\}.
\]

Recall that we can associate to each $\lambda\in\Lambda^+(m,d)$ a \textit{diagram for $\lambda$}
\[
\lambda=
\big\{(r(ow),c(olumn))\mid 1\leq c\leq \lambda^j, 0\leq r\leq m,j=1,\dots,m
\big\},
\]
which we denote by the same symbol $\lambda$. 
The elements of a diagram are called \textit{nodes $N$}. For example, if $\lambda=(4,2,1)$, that is $d=6,m=3$, then
\[
\lambda=
\xy(0,0)*{\begin{Young} & & & \cr & \cr \cr\end{Young}}\endxy.
\]
(We use the English notation to denote 
our partitions/diagrams.) We associate, by convention, 
all partitions of zero to the empty diagram $\emptyset$.

A \textit{tableau $T$ of shape $\lambda$} is a filling of $\lambda$ with (possibly repeating) numbers from a chosen, fixed set $\{1,\dots,k\}$. Such a tableau $T$ is said to be \textit{semistandard} if its entries increase along its rows (weakly) and columns (strictly), and \textit{column strict} if its entries increase along its columns (strictly) with no conditions on rows. For example
\[
T_1=\xy(0,0)*{\begin{Young} 1& 2& 2\cr 2& 3\cr 4\cr\end{Young}}\endxy,
\quad\quad
T_2=\xy(0,0)*{\begin{Young} 1& 2& 1\cr 2& 3\cr 4\cr\end{Young}}\endxy
.
\]
The tableau $T_1$ is semistandard, but $T_2$ is only column strict. We denote the set of all 
semistandard tableaux of shape $\lambda$ by $\mathrm{Std}^s(\lambda)$ and the 
set of all column strict tableaux of shape $\lambda$ by $\mathrm{Col}(\lambda)$.

Finally, let $\lambda\in\Lambda^+(m,d)$ be a partition. Then we associate to each node $N=(r,c)\in\lambda$ of $\lambda$ a \textit{residue} $r(N)$ by the rule $r(N)=c-r+\ell$ where $\ell$ is the number of non-zero entries of $\lambda$. 
(We see $\ell$ as being fixed by $\lambda$, even if we speak later about addable or removable nodes.
Moreover, the convention for the shift of the residue by $\ell$  is a normalization that ensures that the lowest residue for nodes is $1$.)

In the same vein, 
for an integer $n\geq 1$ a 
\textit{$n$-multipartition} $\vec{\lambda}\in\Lambda^+(m,d,n)$ of $d$ of length $m$ is an $n$-tuple of partitions $\vec{\lambda}=(\lambda_n,\dots,\lambda_1)$. Each of its components $\lambda_{i}=(\lambda_{i}^1,\dots,\lambda_{i}^{|\lambda_{i}|})$ is of length $|\lambda_{i}|$ such that their total length is $m$ and their total sum is $d$. We can associate to each $\vec{\lambda}\in\Lambda^+(m,d,n)$ a \textit{diagram for $\vec{\lambda}$}
\[
\vec{\lambda}=
\big\{(r(ow),c(olumn),i(entry))\mid 1\leq c\leq \lambda_{i}^j,\;0\leq r\leq
|\lambda_{i}|,\;i=n,\dots,1,\;j=0,\dots,|\lambda_{i}|
\big\},
\]
which we denote by the same symbol $\vec{\lambda}$. 
For example, if we have $\vec{\lambda}=(\lambda_4=(3,2,1),\lambda_3=(0),\lambda_2=(4),\lambda_1=(3,1))$, 
that is $d=14,m=6$ and $n=4$, then
\[
\lambda=\left(\;\xy(0,0)*{\begin{Young} & & \cr & \cr \cr\end{Young}}\endxy
\;,\;
\emptyset
\;,\;
\xy(0,0)*{\begin{Young} & & &\cr\end{Young}}\endxy\;,\;\xy(0,0)*{\begin{Young} & &\cr \cr\end{Young}}\endxy\;\right).
\]

Similarly as before, an \textit{$n$-multitableau $\vec{T}=(T_n,\dots,T_1)$ of shape $\vec{\lambda}$} is a filling of $\vec{\lambda}$ with (possible repeating) numbers from a chosen, fixed set $\{1,\dots,k\}$. Such a tableau $\vec{T}$ is said to be \textit{standard}, if its entries increase along its rows and columns (both strictly) and all repeating numbers appear at most once in $T_i$, and all nodes with the same number are of the same residue. 
(This is actually a slight generalization of the notion $n$-multitableau in the literature.)
The \textit{residue} of a node is defined verbatim as for tableaux, with $\ell$ being the maximal number of non-zero entries of the components of $\vec{\lambda}$. 

We denote the set of all standard tableaux $\vec{T}$ of shape $\vec{\lambda}$ by $\mathrm{Std}(\vec{\lambda})$.
If not stated otherwise, then all appearing $n$-multitableaux are assumed to be standard for the duration.

There are two natural embeddings $\iota_{n_1}^{n_2},\kappa_{n_1}^{n_2}\colon\Lambda^+(m,d,n_1)\to\Lambda^+(m,d,n_2)$ for $n_2\geq n_1$, i.e.
\[
\iota_{n_1}^{n_2}(\vec{\lambda})=(\underbrace{(0),\dots,(0)}_{n_2-n_1},\lambda_{n_1},\dots,\lambda_{1})
\quad\text{and}\quad
\kappa_{n_1}^{n_2}(\vec{\lambda})=(\lambda_{n_1},\dots,\lambda_{1},\underbrace{(0),\dots,(0)}_{n_2-n_1}).
\]

\begin{defn}\label{defn-tabcomb}
An \textit{addable node $N$ of residue $r(N)=k$} is a node $N$ that can be added to the diagram of $\lambda$ such that the new diagram is still the diagram of a partition and the residue is $r(N)=k$. We denote the \textit{set of addable nodes of residue $k$ of $\lambda$} by $\mathsf{A}^{k}(\lambda)$.
Similarly, a \textit{removable node $N$ of residue $r(N)=k$} is a node that can be removed from the diagram of $\lambda$ such that the new diagram is still the diagram of a partition and the residue of $N$ is $r(N)=k$. We denote the \textit{set of removable nodes of residue $k$ of $\lambda$} by $\mathsf{R}^{k}(\lambda)$.

Again, we can use the same notions for $n$-multipartitions $\vec{\lambda}\in\Lambda^+(m,d,n)$.

Moreover, we say a node $N_1=(r_1,c_1,i_1)$ of $\vec{\lambda}=(\lambda_{i})_{i=n}^1$ comes \textit{before/left of (or after/right of)} another node $N_2=(r_2,c_2,i_2)$ of $\vec{\lambda}$, denoted by $N_1\preceq N_2$ (or $N_1\succeq N_2$), if $i_1>i_2$ or $i_1=i_2$ and $r_1\leq r_2$ (or $i_1<i_2$ or $i_1=i_2$ and $r_1\geq r_2$). We use the evident definitions for the notions \textit{strictly} before $\prec$ and \textit{strictly} after $\succ$.

For a fixed node $N$, we denote the set of addable nodes of $\lambda$ before $N$ with the same residue $r(N)=k$ by $\mathsf{A}^{k\prec N}(\lambda)$ and we denote the set of addable nodes of $\lambda$ after $N$ with the same residue $r(N)=k$ by $\mathsf{A}^{k\succ N}(\lambda)$.

Similarly, for a fixed node $N$, we denote the set of removable nodes of $\lambda$ before $N$ with the same residue $r(N)=k$ by $\mathsf{R}^{k\prec N}(\lambda)$ and we denote the set of removable nodes of $\lambda$ after $N$ with the same residue $r(N)=k$ by $\mathsf{R}^{k\succ N}(\lambda)$.
\end{defn}

\begin{ex}\label{ex-resi}
Let $\vec{\lambda}=(\lambda_3,\lambda_2,\lambda_1)$ be the following $3$-multipartition (we have $\ell=3$), filled with its residues.
\[
\lambda_3=\xy(0,0)*{\begin{Young} 3&  4&  5&6\cr  2& 3\cr 1\cr\end{Young}}\endxy\;,\;\;\;\;
\lambda_2=\xy(0,0)*{\begin{Young} 3&  4\cr  2\cr\end{Young}}\endxy\;,\;\;\;\;
\lambda_1=\xy(0,0)*{\begin{Young} 3&  4&  5&6\cr  2& 3&4&5\cr 1&2 &3\cr\end{Young}}\endxy\;.
\]
Note that the residues are constant along the diagonals.

The set of addable nodes $\cdot$ of residue $4$ for $\vec{\lambda}$ and the set of removable nodes 
$\times$ of residue $2$ for $\vec{\lambda}$ are given by
\[
\lambda_3=\xy(0,0)*{\begin{Young} &  &  &\cr  & &$\cdot$\cr \cr\end{Young}}\endxy\;,\;\;\;\;\lambda_2=\xy(0,0)*{\begin{Young} &  \cr  $\times$\cr\end{Young}}\endxy\;,\;\;\;\;\lambda_1=\xy(0,0)*{\begin{Young} &  &  &\cr  & & &\cr & & &$\cdot$\cr\end{Young}}\endxy\;.
\]
The removable node is after/right of the left addable and before/left of the right addable node. Moreover, in the following we demonstrate all nodes strictly before $\prec$ and strictly after $\succ$ a fixed node marked $-$.
\[
\lambda_3=\xy(0,0)*{\begin{Young} $\prec$&  $\prec$&  $\prec$&$\prec$\cr $\prec$ &$\prec$\cr $\prec$\cr\end{Young}}\endxy\;,\;\;\;\;\lambda_2=\xy(0,0)*{\begin{Young} $\prec$&  $\prec$\cr  $\prec$\cr\end{Young}}\endxy\;,\;\;\;\;\lambda_1=\xy(0,0)*{\begin{Young} & $-$ &  &\cr  $\succ$& $\succ$& $\succ$& $\succ$\cr $\succ$& $\succ$& $\succ$\cr\end{Young}}\endxy\;.
\]
\end{ex}

Let us recall Brundan, Kleshchev and Wang's definition of the degree of a $n$-multitableau \cite{bkw}, slightly 
generalized to our setup.

\begin{defn}\label{defn-combinatorics1}
Let $\vec{T}\in\mathrm{Std}(\vec{\lambda})$ be an $n$-multitableau $\vec{T}=(T_n,\dots,T_1)$. 
We associate to $\vec{T}$ a sequence of $n$-multitableaux $(\vec{T}^j)$ for each $j\in\{0,1,\dots,k\}$ where $\vec{T}^j=(T^j_n,\dots,T^j_1)$ and $T^j_{n,\dots,1}$ is obtained from $T_{n,\dots,1}$ by deleting all nodes with numbers strictly bigger than $j$.

Moreover, we associate to it a 
sequence of $n$-multipartitions $(\vec{\lambda}^j)$ by removing the entries of the nodes of $(\vec{T}^j)$.
\end{defn}

\begin{ex}\label{ex-combinatorics1}
For the $4$-multitableau
\[
\vec{T}=\left(\;\xy (0,0)*{\begin{Young}1&2 \cr 3 \cr
\end{Young}}\endxy\;,\;\xy (0,0)*{\begin{Young}4 \cr\end{Young}}\endxy\;,\;\xy (0,0)*{\begin{Young}1&2 \cr\end{Young}}\endxy\;,\;\xy (0,0)*{\begin{Young}4 \cr\end{Young}}\endxy
\;\right),
\]
we obtain the following sequence. First note that, 
by definition, $\vec{T}^0=(\emptyset,\emptyset,\emptyset,\emptyset)$ and $\vec{T}^4=\vec{T}$. The intermediate $4$-multitableaux are
\[
\vec{T}^1=\left(\;\xy (0,0)*{\begin{Young}1 \cr\end{Young}}\endxy\;,\;\emptyset\;,\;\xy (0,0)*{\begin{Young}1 \cr\end{Young}}\endxy,\;\emptyset\;\right),\;\;\vec{T}^2=\left(\;\xy (0,0)*{\begin{Young}1&2 \cr\end{Young}}\endxy\;,\;\emptyset\;,\;\xy (0,0)*{\begin{Young}1 & 2 \cr\end{Young}}\endxy\;,\;\emptyset\;\right),\;\;\vec{T}^3=\left(\;\xy (0,0)*{\begin{Young}1&2 \cr 3 \cr\end{Young}}\endxy\;,\;\emptyset\;,\;\xy (0,0)*{\begin{Young}1 &2 \cr\end{Young}}\endxy\;,\;\emptyset\;
\right).
\]
\end{ex}

For repeating entries we very often add $0<\varepsilon\ll 1$ (being strictly smaller than one over max number of repeated boxes is sufficient) from left to right, e.g. letting $\varepsilon=0.1$:
\[
\vec{T}=\left(\;
\xy(0,0)*{\begin{Young}1&4\cr 3\cr\end{Young}}\endxy\;,\;\xy(0,0)*{\begin{Young}2\cr\end{Young}}\endxy\;,\;\xy(0,0)*{\begin{Young}1&4\cr 5\cr\end{Young}}\endxy
\;,\;\xy(0,0)*{\begin{Young}1\cr\end{Young}}\endxy
\;\right)
\rightsquigarrow
\left(\;
\xy(0,0)*{\begin{Young}1&4\cr 3\cr\end{Young}}\endxy\;,\;\xy(0,0)*{\begin{Young}2\cr\end{Young}}\endxy\;,\;\xy(0,0)*{\begin{Young}1.1&4.1\cr 5\cr\end{Young}}\endxy
\;,\;
\xy(0,0)*{\begin{Young}1.2\cr\end{Young}}\endxy
\;\right).
\]
Then we apply the definitions from the non-repeating setup (extended in the evident sense to non-integral entries; only the order matters).

\begin{defn}\label{defn-degpart1}
Let $\vec{T}\in\mathrm{Std}(\vec{\lambda})$ be a $n$-multitableau. For $j\in\{1,\dots,k\}$ let $N^j$ denote the set 
of all nodes that are filled with the number $j$ and let $\vec{T}^j$ denote as before the 
$n$-multitableau obtained from $\vec{T}$ by removing all nodes with entries $>j$.

First assume that there are no repeating numbers.
The \textit{degree of $\vec{T}^j$}, denoted by $\mathrm{deg}(\vec{T}^j)$, is defined to be
\[
\mathrm{deg}(\vec{T}^j)=|\mathsf{A}^{k\succ N}(\vec{T}^j)|-|\mathsf{R}^{k\succ N}(\vec{T}^j)|-a\;\;\text{ with }\;\;a=\sum_{i=0}^{N^j-1}i.
\]
If there are repeating numbers, then replace these by adding a small amount $0<\varepsilon\ll 1$ to each repeating number, increasing from 
left to right, and apply the definition from the non-repeating case.

The \textit{degree} of the $n$-multitableau $\vec{T}=(T_n,\dots,T_1)$, denoted by $\mathrm{deg}_{\mathrm{BKW}}(\vec{T})$, is then defined by
\[
\mathrm{deg}_{\mathrm{BKW}}(\vec{T})=\sum_{j=1}^k \mathrm{deg}(\vec{T}^j).
\]
\end{defn}

\begin{ex}\label{ex-degreea}
All of the following four standard $4$-multitableaux have degree zero.
\begin{gather*}
\vec{T}_1 =\left(\;\emptyset\;,\;\emptyset\;,\;\emptyset\;,\;\xy(0,0)*{\begin{Young} 1\cr\end{Young}}\endxy\;\right)\;,\; 
\vec{T}_2=\left(\;\emptyset\;,\;\emptyset\;,\;\xy(0,0)*{\begin{Young} 1\cr\end{Young}}\endxy\;,\;\xy(0,0)*{\begin{Young} 1\cr\end{Young}}\endxy\;\right)\;,
\;\vec{T}_3 =\left(\;\emptyset\;,\;\xy(0,0)*{\begin{Young} 1\cr\end{Young}}\endxy\;,\;\xy(0,0)*{\begin{Young} 1\cr\end{Young}}\endxy\;,\;\xy(0,0)*{\begin{Young} 1\cr\end{Young}}\endxy\;\right)\;,\;
\\
\vec{T}_4=\left(\;\xy(0,0)*{\begin{Young} 1\cr\end{Young}}\endxy\;,\;\xy(0,0)*{\begin{Young} 1\cr\end{Young}}\endxy\;,\;\xy(0,0)*{\begin{Young} 1\cr\end{Young}}\endxy\;,\;\xy(0,0)*{\begin{Young} 1\cr\end{Young}}\endxy\;\right)
\rightsquigarrow
\left(\;\xy(0,0)*{\begin{Young} 1\cr\end{Young}}\endxy\;,\;\xy(0,0)*{\begin{Young} 1.1\cr\end{Young}}\endxy\;,\;\xy(0,0)*{\begin{Young} 1.2\cr\end{Young}}\endxy\;,\;\xy(0,0)*{\begin{Young} 1.3\cr\end{Young}}\endxy\;\right),\quad\text{i.e. }\varepsilon=0.1.
\end{gather*}
To see this, we note that in the first case there is no node after $\succ$ the unique node $N^1$. Hence, $\mathrm{deg}(\vec{T}_1)=0$. In the second case we have to calculate two steps. In the first step, i.e.
\[
\left(\;\emptyset\;,\;\emptyset\;,\;\xy(0,0)*{\begin{Young} 1\cr\end{Young}}\endxy\;,\;\xy(0,0)*{\begin{Young} $\cdot$\cr\end{Young}}\endxy\;\right),
\]
we count one addable node of the same residue which we have marked with a $\cdot$, but for the second step there is again no node after $\succ$ the last node anymore. Hence, $\mathrm{deg}(\vec{T}_2)=0$, since we have to take the shift from Definition \ref{defn-degpart1} into account. For the third case we have to calculate three steps, i.e. the first and the second are
\[
\left(\;\emptyset\;,\;\xy(0,0)*{\begin{Young} 1\cr\end{Young}}\endxy\;,\;\xy(0,0)*{\begin{Young} $\cdot$\cr\end{Young}}\endxy\;,\;\xy(0,0)*{\begin{Young} $\cdot$\cr\end{Young}}\endxy\;\right)\;\;\text{ and }\;\;\left(\;\emptyset\;,\;\xy(0,0)*{\begin{Young} 1\cr\end{Young}}\endxy\;,\;\xy(0,0)*{\begin{Young} 1\cr\end{Young}}\endxy\;,\;\xy(0,0)*{\begin{Young} $\cdot$\cr\end{Young}}\endxy\;\right),
\]
where we have again indicated the addable nodes of the same residue with a $\cdot$. The third step is as before. Hence, $\mathrm{deg}(\vec{T}_3)=0$, because of the shift. The last case works similarly with a shift by $6$.

Note that the degree (total or local) can be negative. For example the last step of
\[
\vec{T}_5=\left(\;\xy(0,0)*{\begin{Young} 1 & 2 &3\cr 8 & 9\cr\end{Young}}\endxy\;,\;\xy(0,0)*{\begin{Young} 5 & 6\cr 10\cr 11\cr\end{Young}}\endxy\;,\;\xy(0,0)*{\begin{Young} 1 & 2 & 3\cr 4& 9\cr 7\cr\end{Young}}\endxy\;\right)
\]
has no addable nodes after $\succ$ the node $N^{11}$ with the same residue, but one removable node, namely the node filled with the entry $7$. Hence, $\mathrm{deg}(\vec{T}_5^{11})=-1$. The total degree in this case is
\[
\mathrm{deg}_{\mathrm{BKW}}(\vec{T}_5)=1+0+0+0+1+0+0+1+0+1-1=3.
\]
\end{ex}

\begin{defn}\label{defn-dominnancelambda}
Let $\vec{\lambda}=(\lambda_n,\dots,\lambda_1)$ and $\vec{\mu}=(\mu_n,\dots,\mu_1)$ be $n$-multipartitions in $\Lambda^+(m,d,n)$. Recall that $\lambda_{i}=(\lambda_{i}^1,\dots,\lambda_{i}^{|\lambda_i|})$ and $\mu_{i}=(\mu_{i}^1,\dots,\mu_{i}^{|\mu_i|})$ for $i\in\{n,\dots,1\}$.

We say \textit{$\vec{\mu}$ dominates $\vec{\lambda}$,} denoted by $\vec{\lambda}\trianglelefteq\vec{\mu}$, if
\[
\sum_{i^{\prime}=1}^{i-1}|\lambda_{n+1-i^{\prime}}|+\sum_{j=1}^{|\lambda_{n+1-i}|}\lambda_{n+1-i}^j\leq \sum_{i^{\prime}=1}^{i-1}|\mu_{n+1-i^{\prime}}|+\sum_{j=1}^{|\mu_{n+1-i}|}\mu_{n+1-i}^j
\]
for all $1\leq i\leq n$. We write $\vec{\lambda}\lhd\vec{\mu}$, if $\vec{\lambda}\unlhd\vec{\mu}$ and $\vec{\lambda}\neq\vec{\mu}$. It is easy to check that $\unlhd$ is a partial ordering of the set of all $n$-multipartitions $\Lambda^+(m,d,n)$, called the \textit{dominance order}.
This order can be extended to $n$-multitableaux in the following way. Suppose we have two standard $n$-multitableaux $\vec{T}_1\in\mathrm{Std}(\vec{\lambda})$ and $\vec{T}_2\in\mathrm{Std}(\vec{\mu})$ filled with numbers from $\{1,\dots,k\}$. As in Definition \ref{defn-combinatorics1}, we denote the corresponding $n$-multipartitions after removing all nodes with entries higher than $j\in\{1,\dots,k\}$ by $\vec{\lambda}^j$ and $\vec{\mu}^j$. Then
\[
\vec{T}_1\unlhd\vec{T}_2\Longleftrightarrow \vec{\lambda}^j\unlhd\vec{\mu}^j\;\;\text{ for all }\;\;j\in\{1,\dots,k\}.
\]

Given $\vec{\lambda}\in\Lambda^+(m,d,n)$, we can associate to it two \textit{unique standard $n$-multitableaux $T_{\vec{\lambda}}\in\mathrm{Std}(\vec{\lambda})$ and $T^*_{\vec{\lambda}}\in\mathrm{Std}(\vec{\lambda})$} with the property
\[
\vec{T}\in\mathrm{Std}(\vec{\lambda})\Rightarrow \vec{T}\unlhd T_{\vec{\lambda}}\;\text{ and }\;\vec{T}\in\mathrm{Std}(\vec{\lambda})\Rightarrow T^*_{\vec{\lambda}}\unlhd \vec{T}.
\]
The $n$-multitableaux $T_{\vec{\lambda}}$ is easily seen to be the $n$-multitableau with all entries in order from top to bottom, filling up rows before columns, and left to right and its so-called \textit{dual} $T^*_{\vec{\lambda}}$ has its entries ordered also from top to bottom, but filling up columns before rows, and from right to left.

To use the definitions above for repeating entries we, by convention, 
use the same notions as above after adding $0<\varepsilon\ll 1$ from left to right as before. 
\end{defn}

\begin{ex}\label{ex-morecomb1}
Intuitively $\vec{T}_1\triangleleft \vec{T}_2$ means the numbers in $\vec{T}_1$ appear earlier to the right than in $\vec{T}_2$. For example, given the $3$-multipartition
\[
\;\;\;\vec{\lambda}=\left(\;\xy(0,0)*{\begin{Young}&\cr\cr\end{Young}}\endxy\;,\;\xy(0,0)*{\begin{Young}\cr\end{Young}}\endxy\;,\;\xy(0,0)*{\begin{Young}&\cr\cr\end{Young}}\endxy\;\right),
\]
we see that
\[
T_{\vec{\lambda}}=\left(\;\xy(0,0)*{\begin{Young}1&2\cr3\cr\end{Young}}\endxy\;,\;\xy(0,0)*{\begin{Young}4\cr\end{Young}}\endxy\;,\;\xy(0,0)*{\begin{Young}5&6\cr7\cr\end{Young}}\endxy\;\right)\;\text{ and }\;T^*_{\vec{\lambda}}=\left(\;\xy(0,0)*{\begin{Young}5&7\cr 6\cr\end{Young}}\endxy\;,\;\xy(0,0)*{\begin{Young}4\cr\end{Young}}\endxy\;,\;\xy(0,0)*{\begin{Young}1&3\cr2\cr\end{Young}}\endxy\;\right)
.
\]
The left tableau will dominate all $\vec{T}\in\mathrm{Std}(\vec{\lambda})$. For example
\[
\vec{T}=\left(\;\xy(0,0)*{\begin{Young}1&2\cr3\cr\end{Young}}\endxy\;,\;\xy(0,0)*{\begin{Young}5\cr\end{Young}}\endxy\;,\;\xy(0,0)*{\begin{Young}4&6\cr7\cr\end{Young}}\endxy\;\right)
\]
will be dominated, since
\[
\vec{T}^4=\left(\;\xy(0,0)*{\begin{Young}1&2\cr3\cr\end{Young}}\endxy\;,\;\emptyset,\;\xy(0,0)*{\begin{Young}4\cr\end{Young}}\endxy\;\right)\unlhd T_{\vec{\lambda}}^4=\left(\;\xy(0,0)*{\begin{Young}1&2\cr3\cr\end{Young}}\endxy\;,\;\xy(0,0)*{\begin{Young}4\cr\end{Young}}\endxy\;,\;\emptyset\;\right).
\]
The dual one the other hand is dominated by all the others.
\end{ex}

\begin{defn}\label{defn-rsequence}
Let $\vec{T}\in\mathrm{Std}(\vec{\lambda})$ be a $n$-multitableau. The \textit{residue sequence} of $\vec{T}$, denoted by $r(\vec{T})$, is the $k$-tuple whose $j\in\{1,\dots,k\}$ entry is the residues of the node with number $j$. Moreover, the \textit{residue sequence} of a $n$-multipartition $\vec{\lambda}$, denoted by $r(\vec{\lambda})$, is defined to be $r(\vec{\lambda})=r(T_{\vec{\lambda}})$.

If the $n$-multitableau $\vec{T}$ has multiple entries with label $j$ and all of them are of the same residue, then we use the same definition.
\end{defn}
\subsection{The \texorpdfstring{$\mathfrak{gl}_n$}{gln}-spiders and the \texorpdfstring{$\mathfrak{gl}_n$}{gln}-web spaces}\label{sec-slnwebs}
\subsubsection{Definition of the \texorpdfstring{$\mathfrak{gl}_n$}{gln}-spider}\label{subsub-spider}
In this section we are going to define the $\Uu_q(\mathfrak{gl}_n)$\textit{-spider category or $\mathfrak{gl}_n$-web-category} $\spid{n}$, 
following \cite{ckm}.

Our convention for reading diagrams is from bottom to top and left to right. By diagram we mean oriented, planar graphs with labeled edges, where all vertices are either part of the boundary or $3$-valent. The boundary in our case are lines at either the bottom or the top of the diagrams with a certain number of fixed points ordered from left to right. Moreover, in the whole section let the letters $a,b,c,d$ and $e$ denote elements of $\{0,\dots,n\}$.

Furthermore, we use the convention that $[a]$ denotes the \textit{quantum integer} (with $[0]=1$), $[a]!$ denotes the \textit{quantum factorial}, 
and we also use the \textit{quantum binomial}:
\[
[a]=\frac{q^a-q^{-a}}{q-q^{-1}}=q^{a-1}+q^{a-3}+\dots+q^{-a+3}+q^{-a+1},\quad[a]!=[0][1]\dots [a-1][a],
\quad
\qbin{a}{b}=\frac{[a]!}{[a-b]![b]!}.
\]

\begin{defn}\label{def-freespid}(\textbf{Free \texorpdfstring{$\Uu_q(\mathfrak{gl}_n)$}{\textbf{U}q(sln)}-spider}) 
The \textit{free $\Uu_q(\mathfrak{gl}_n)$-spider category}, which we denote by $\fspidn{n}$, 
is the the $\bC(q)$-linear, monoidal category consisting of:
\begin{itemize}
\item The objects of $\fspidn{n}$, denoted by $\Ob(\fspidn{n})$, are tuples $\vec{k}$ with entries in $\{0,\dots,n\}$. We display their entries ordered from left to right according to their appearance in $\vec{k}$.
\item The $1$-morphisms of $\fspidn{n}$ between $\vec{k}$ and $\vec{l}$, denoted by $\Mor_{\fspidn{n}}(\vec{k},\vec{l})$, are diagrams between $\vec{k}$ and $\vec{l}$ freely (monoidally) generated by the following basic pieces,
\begin{align}\label{eq-genspider}
\scalebox{.7}{$\text{split}\colon\xy(0,0)*{\includegraphics[scale=.75]{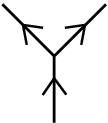}};(-6.5,5)*{\scriptstyle a};(6.3,5)*{\scriptstyle b};(-3.5,-5)*{\scriptstyle a+b}\endxy,\quad\text{merge}\colon
\xy(0,0)*{\includegraphics[scale=.75]{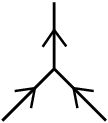}};(-6.5,-5)*{\scriptstyle a};(6.3,-5)*{\scriptstyle b};(3.5,5)*{\scriptstyle a+b}\endxy$},
\end{align}
called \textit{split (up)} and \textit{merge (up)}.
The boundary objects, by convention, should be the same as the label of the edge next to it. Therefore, we usually do not picture the objects directly as e.g. in \eqref{eq-genspider}.
\end{itemize}

We usually do not draw edges labeled $0$ and use edges labeled $n$, drawn as dotted \textit{leashes} (see also Remark \ref{rem-leash}). These conventions are illustrated in \eqref{eq:convention} below. We think of $0$ and $n$ labeled edges as non-existing. And, by convention, all diagrams with strictly smaller or bigger labels than $0$ or $n$ are defined to be $0$.

Moreover, we use shorthand notations for \textit{ladders}. 
For example, we use the following diagrams (and similar ones for other ladders) as a shorthand notation.
\begin{gather}\label{eq:convention}
\scalebox{.7}{$\xy(0,0)*{\includegraphics[scale=.75]{figs/slnwebs/trivalenta.eps}};(-6.5,5)*{\scriptstyle a};(6.5,5)*{\scriptstyle b};(-5,-5)*{\scriptstyle a+b=n}\endxy=\xy(0,2.5)*{\includegraphics[scale=.75]{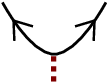}};(-6.5,4.5)*{\scriptstyle a};(6.3,4.5)*{\scriptstyle b}\endxy\;\text{and}\;\xy(0,0)*{\includegraphics[scale=.75]{figs/slnwebs/trivalentb.eps}};(-6.5,-5)*{\scriptstyle a};(6.3,-5)*{\scriptstyle b};(5.5,5)*{\scriptstyle a+b=n}\endxy=\xy(0,-2.9)*{\includegraphics[scale=.75]{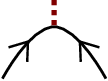}};(-6.5,-4.9)*{\scriptstyle a};(6.3,-4.9)*{\scriptstyle b}\endxy$}
\;\text{and}\;
\scalebox{.7}{$\xy(0,0)*{\includegraphics[scale=.75]{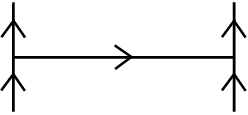}};(-11.5,-4)*{\scriptstyle a};(11.5,-4)*{\scriptstyle b};(3,2)*{\scriptstyle c+d};(-8,4)*{\scriptstyle a-c-d};(8,4)*{\scriptstyle b+c+d}\endxy=\xy(0,0)*{\includegraphics[scale=.75]{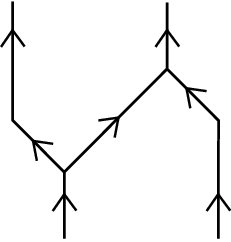}};(-4,-10.5)*{\scriptstyle a};(15.5,-10.5)*{\scriptstyle b};(3.5,-1)*{\scriptstyle c+d};(-6.5,10.5)*{\scriptstyle a-c-d};(13,10.5)*{\scriptstyle b+c+d}\endxy$}.
\end{gather}
\end{defn}

\begin{defn}\label{def-spid}(\textbf{$\Uu_q(\mathfrak{gl}_n)$-spider}) Let $n>1$. The $\Uu_q(\mathfrak{gl}_n)$-\textit{spider category}, which we denote by $\spid{n}$, is defined as the additive Karoubi closure (taking direct sums and direct summands, the latter in the abstract sense of the Karoubi envelope recalled below) of a quotient of $\fspidn{n}$ by the following relations.

The \textit{(co)associativity relations},
\begin{equation}\label{eq-tripod}
\scalebox{.7}{$
\xy(0,0)*{\includegraphics[scale=.75]{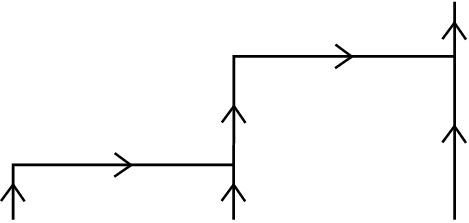}};(-26,-12)*{\scriptstyle a};(2,-12)*{\scriptstyle b};(30,-12)*{\scriptstyle c};(4,-1)*{\scriptstyle a+b};(22.5,11)*{\scriptstyle a+b+c}\endxy\;\;\;=\;\;\;\xy(0,0)*{\includegraphics[scale=.75]{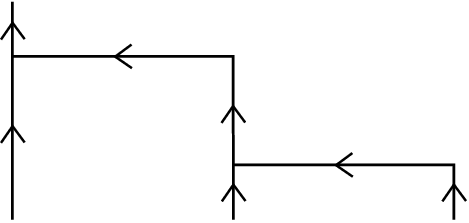}};(-26,-12)*{\scriptstyle a};(2,-12)*{\scriptstyle b};(30,-12)*{\scriptstyle c};(4,-1)*{\scriptstyle b+c};(-22.5,11)*{\scriptstyle a+b+c}\endxy$},
\end{equation}
including the evident coassociativity version as well,
the \textit{digon removals}
\begin{equation}\label{eq-digon1}
\scalebox{.7}{$\xy(0,0)*{\includegraphics[scale=.75]{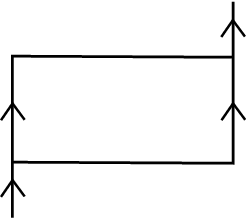}};(9,10)*{\scriptstyle a+b};(-9.5,-10)*{\scriptstyle a+b};(-12,0)*{\scriptstyle a};(16,0)*{\scriptstyle b};\endxy\;\;\;=\qbin{a+b}{b}\xy(0,0)*{\includegraphics[scale=.75]{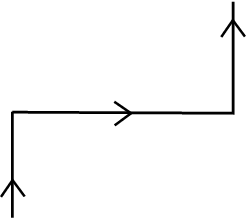}};(6.5,2)*{\scriptstyle a+b}\endxy$},
\end{equation}
the \textit{square removals}
\begin{equation}\label{eq-square1}
\scalebox{.7}{$\xy(0,0)*{\includegraphics[scale=.75]{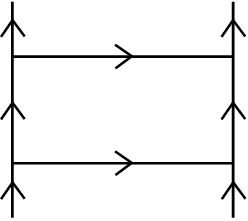}};(-11.5,-11)*{\scriptstyle a};(16.5,-11)*{\scriptstyle b};(2,8)*{\scriptstyle c};(2,-5)*{\scriptstyle d};(-10,0)*{\scriptstyle a-d};(18,0)*{\scriptstyle b+d};(-8,11)*{\scriptstyle a-c-d};(19.5,11)*{\scriptstyle b+c+d}\endxy\;\;\;=\qbin{c+d}{c}\;\xy(0,0)*{\includegraphics[scale=.75]{figs/slnwebs/laddera.eps}};(-11.5,-4)*{\scriptstyle a};(16.5,-4)*{\scriptstyle b};(3,2)*{\scriptstyle c+d};(-8,4)*{\scriptstyle a-c-d};(20,4)*{\scriptstyle b+c+d}\endxy$},
\end{equation}
and the \textit{square switches}
\begin{equation}\label{eq-square2}
\scalebox{.7}{$\xy(0,0)*{\includegraphics[scale=.75]{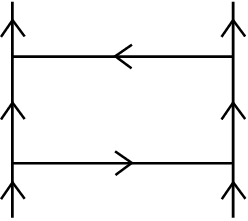}};(-11.5,-11)*{\scriptstyle a};(16.5,-11)*{\scriptstyle b};(2,8)*{\scriptstyle c};(2,-5)*{\scriptstyle d};(-10,0)*{\scriptstyle a-d};(18,0)*{\scriptstyle b+d};(-8,11)*{\scriptstyle a+c-d};(19.5,11)*{\scriptstyle b-c+d}\endxy\;\;\;=\sum_e \qbin{a-b+d-c}{e}\;\xy(0,0)*{\includegraphics[scale=.75]{figs/slnwebs/squarec.eps}};(-11.5,-11)*{\scriptstyle a};(16.5,-11)*{\scriptstyle b};(4,8)*{\scriptstyle d-e};(4,-5)*{\scriptstyle c-e};(-8,0)*{\scriptstyle a+c-e};(20,0)*{\scriptstyle b-c+e};(-8,11)*{\scriptstyle a+c-d};(19.5,11)*{\scriptstyle b-c+d}\endxy$}.
\end{equation}
\end{defn}

Moreover, for $0\leq n^{\prime}\leq n$ we also consider the full subcategory $\spidn{n}{n^{\prime}}$ consisting of objects with labels in $\{0,\dots,n^{\prime}\}$ only.

\subsubsection{Some \texorpdfstring{$\mathfrak{gl}_n$}{gln}-representation theoretical notions}\label{subsub-reptheo}
Let us briefly recall some of the representation theory of $\Uun$. Much more details that are related to our framework can be found in \cite{ckm} or \cite{mack1}. Moreover, we often use ``$\mathfrak{gl}_n$-webs'', ``$\mathfrak{gl}_n$-weights'' etc. instead of ``$\Uu_q(\mathfrak{gl}_n)$-webs'', ``$\Uu_q(\mathfrak{gl}_n)$-weights'' etc. and also omit to put a $q$ in the notation if no confusion can arise.

Recall that the $\mathfrak{gl}_n$-weight 
lattice is isomorphic to $\bZ^n$. Let 
$\epsilon_i=(0,\dots,1,\dots,0)\in\bZ^n$, with $1$ being on the $i$-th 
coordinate, and $\alpha_i=\epsilon_i-\epsilon_{i+1}
=(0,\dots,1,-1,\dots,0)\in\bZ^{n}$, for 
$i=1,\dots,n-1$. Recall that the Euclidean inner product on $\bZ^n$ is defined by  
$(\epsilon_i,\epsilon_j)=\delta_{i,j}$.

\begin{defn}\label{def-serre} For $n\in\bN_{>1}$ the \textit{quantum general linear algebra} 
$\Uu_q(\mathfrak{gl}_n)$ is 
the associative, unital $\bC(q)$-algebra generated by $K_i$ and $K_i^{-1}$, for $1,\dots, n$, 
and $E_{i},F_i$, for $i=1,\dots, n-1$, subject to the relations
\begin{gather*}
K_iK_j=K_jK_i,\quad K_iK_i^{-1}=K_i^{-1}K_i=1,
\\
E_iF_{j} - F_{j}E_i = \delta_{i,j}\dfrac{K_iK_{i+1}^{-1}-K_i^{-1}K_{i+1}}{q-q^{-1}},
\quad
K_iE_{j}=q^{ (\epsilon_i,\alpha_j)}E_{j}K_i,
\quad
K_iF_{j}=q^{- (\epsilon_i,\alpha_j)}F_{j}K_i,
\\
E_{i}^2E_{j}-[2]E_{i}E_{j}E_{i}+E_{j}E_{i}^2=0,
\qquad\text{if}\quad |i-j|=1,
\\
E_{i}E_{j}-E_{j}E_{i}=0,\qquad\text{else},
\\
F_{i}^2F_{j}-[2]F_{i}F_{j}F_{i}+F_{j}F_{i}^2=0,
\qquad\text{if}\quad |i-j|=1,
\\
F_{i}F_{j}-F_{j}F_{i}=0,\qquad\text{else}.
\end{gather*}
The last four relations are the so-called (quantum) \textit{Serre relations}.
\end{defn}

It is worth noting that $\Uu_q(\mathfrak{gl}_n)$ is a Hopf algebra with coproduct $\Delta$ given by
\[
\Delta(E_i)=E_i\otimes K_i+1\otimes E_i,\;\;\Delta(F_i)=F_i\otimes 1+K^{-1}_i\otimes F_i\;\;\text{and}\;\;\Delta(K_i^{\pm 1})=K_i^{\pm 1}\otimes K_i^{\pm 1}.
\]
The antipode $S$ and the counit $\varepsilon$ are given by
\[
S(E_i)=-E_iK^{-1}_i,\;\; S(F_i)=-K_iF_i,\;\;S(K_i^{\pm 1})=K_i^{\mp 1},\;\;\varepsilon(E_i)=\varepsilon(F_i)=0\;\;\text{and}\;\;\varepsilon(K_i^{\pm 1})=1.
\]
Recall that the Hopf algebra structure allows to extend actions to tensor products of representations, to duals of representations and there is a trivial representation.
We denote the standard basis of the $\Uun$-representation $\bC(q)^n$ 
(written $\bC^n$ for simplicity) by $\{x_1,\dots,x_n\}$, where the action is given by
\[
E_i(x_j)=\begin{cases}x_{j-1}, &\text{if }i=j-1,\\0 &\text{else,}\end{cases}\;F_i(x_j)=\begin{cases}x_{j+1}, &\text{if }i=j,\\0 &\text{else,}\end{cases}\;K_i(x_j)=\begin{cases}qx_{j}, &\text{if }i=j,\\ x_j &\text{else.}\end{cases}
\]
Then we consider the following quotient of the tensor algebra $\mathcal{T}\bC^n$:
\[
\Lambda^{\bullet}\bC^n=\mathcal{T}\bC^n/\mathcal S^2\bC^n,
\]
where $\mathcal S^2\bC^n$ is the symmetric square of $\bC^n$ spanned by $x_ix_j+qx_jx_i$ for all pairs $i<j$ and by $x_ix_i$ for all $i$, cf. \cite{bezw}.
Recall that $\Lambda^{\bullet}\bC^n$ is a graded algebra with product $\wedge$ and we denote by $\Lambda^k\bC^n$ its $k$-th direct summand, i.e.
\[
\Lambda^{\bullet}\bC^n=\bigoplus_{k=0}^n\Lambda^k\bC^n.
\]
These summands are irreducible $\Uun$-representations and the $k$-th one is called the \textit{$k$-th fundamental} $\Uun$-representation. We note that the $(n{-}k)$-th $\Uun$-representation is isomorphic to the dual of the $k$-th one. Moreover, the two cases $k=0,n$, which are duals, are called the \textit{trivial} $\Uun$-representation.

Given an $>$-ordered $k$-element subset $S=\{s_1,\dots,s_k\}$ of $\{n,\dots,1\}$ (we follow Cautis, Kamnitzer and Morrison, i.e. the sets $S$ are ordered decreasing; we write all involved sets decreasing), the \textit{tensor basis} of $\Lambda^k\bC^n$ is given by
\[
\{x_S=x_{s_1}\wedge\dots\wedge x_{s_k}\in\Lambda^k\bC^n\mid S\subset \{n,\dots,1\},|S|=k\}
\]
and its elements are called \textit{elementary tensors}. Moreover, as in \cite{mack1}, let $\vec{k}=(k_1,\dots,k_m)$ be an $m$-tuple with $0\leq k_i<n$ and define
\[
\Lambda^{\vec{k}}\bC^n=\Lambda^{k_1}\bC^n\otimes\dots\otimes\Lambda^{k_m}\bC^n.
\]
The tensor basis can be extended to a basis of $\Lambda^{\vec{k}}\bC^n$, which we also call \textit{tensor basis} and its elements $x_{\vec{S}}$ the \textit{elementary tensors} of $\Lambda^{\vec{k}}\bC^n$. Here we have $\vec{S}=(S_1,\dots,S_m)$ with $S_j\subset \{n,\dots, 1\},|S_j|=k_j$ for $j=1,\dots,m$.

\subsubsection{Relation to the representation category \texorpdfstring{$\Rep(\mathfrak{gl}_n)$}{\textbf{Rep}(gln)}}\label{subsub-spidrep}
By definition, $\Rep(\Uun)$ is the additive, Karoubi closure (taking direct sums and direct summands) of the full subcategory of all $\Uun$-representations 
generated by $\Lambda^k\bC^n$. Furthermore, recall that the $\Uun$-spider $\spid{n}$ is a monoidal category due to the Hops algebra structure of $\Uun$.

Given two subsets $S,T\subset\{n,\dots,1\}$ define $\ell(S,T)=|\{(i,j)\in S\times T\mid i<j\}|$. For any $a,b\in\{1,\dots,n-1\}$ with $a+b<n$ define the following (generating) intertwiners.
\begin{itemize}
\item[(a)] The intertwiner $M^{a,b}_s$ called \textit{split} is given by
\[
M_s^{a,b}\colon \Lambda^{a+b}\bC^n\to\Lambda^{a}\bC^n\otimes \Lambda^{b}\bC^n,\;M^{a,b}_s(x_S)=\sum_{T\subset S}(-q)^{\ell(S,T)}x_{T}\otimes x_{S-T}.
\]
\item[(b)] The intertwiner $M^{a,b}_m$ called \textit{merge} is given by
\[
M_m^{a,b}\colon \Lambda^{a}\bC^n\otimes \Lambda^{b}\bC^n\to \Lambda^{a+b}\bC^n,\;M^{a,b}_s(x_S\otimes x_T)=\begin{cases}(-q)^{-\ell(T,S)}x_{S\cup T}, &\text{if }S\cap T=\emptyset,\\ 0 &\text{else.}\end{cases}
\]
\end{itemize}

\begin{defn}\label{defn-pifunctor}
Define a monoidal functor $\Psi\colon\spid{n}\to\Rep(\Uun)$, given on objects by
\[
\vec{k}=(k_1^{\pm 1},\dots,k_m^{\pm 1})\mapsto (\Lambda^{k_1}\bC^n)^{\pm 1}\otimes\dots\otimes (\Lambda^{k_m}\bC^n)^{\pm 1},
\]
where a minus should indicate the dual $\Uun$-representation. On the morphisms the functor $\Psi$ is defined by
\begin{align}
\scalebox{.7}{$\xy(0,0)*{\includegraphics[scale=.75]{figs/slnwebs/trivalenta.eps}};(-6.5,5)*{\scriptstyle a};
(6.3,5)*{\scriptstyle b};(-3.5,-5)*{\scriptstyle a+b}\endxy$}
\mapsto M^{a+b}_s\;\;\text{ and }\;\;
\scalebox{.7}{$\xy(0,0)*{\includegraphics[scale=.75]{figs/slnwebs/trivalentb.eps}};(-6.5,-5)*{\scriptstyle a};
(6.3,-5)*{\scriptstyle b};(3.5,5)*{\scriptstyle a+b}\endxy$}\mapsto M^{a+b}_m.
\end{align}
\end{defn}

\begin{thm}\label{thm-pifunctor}(\cite[Theorem 3.3.1]{ckm})
The functor $\Psi$ from above is a well-defined equivalence of monoidal categories $\spid{n}$ to $\Rep(\Uun)$.\qed
\end{thm}

One can actually upgrade Theorem \ref{thm-pifunctor} into 
an equivalence of braided categories, with $\Rep(\Uun)$ being braided by the $R$-matrix.

\subsubsection{Ladder moves and \texorpdfstring{$q$}{q}-skew Howe duality}\label{subsub-qhowe}

Adjoin an idempotent $1_{\vec{k}}$ for ${\Uu}_q(\mathfrak{gl}_m)$
for each $\vec{k}\in\bZ^{m}$ and add the relations
\begin{gather*}
1_{\vec{k}}1_{\vec{l}}=\delta_{\vec{k},\vec{l}}1_{\vec{k}},   
\quad
E_{i}1_{\vec{k}}=1_{\vec{k}+\alpha_i}E_{i},
\quad
F_{i}1_{\vec{k}}=1_{\vec{k}-\alpha_i}F_{i},
\quad
K_i1_{\vec{k}}=q^{\vec{k}_i}1_{\vec{k}}.
\end{gather*}

Following \cite{blm} we define:

\begin{defn} 
The idempotented quantum general linear algebra is defined by 
\[
\dot{\Uu}_q(\mathfrak{gl}_m)=\bigoplus_{\vec{k},\vec{l}\in\bZ^{m}}1_{\vec{k}}{\Uu}_q(\mathfrak{gl}_m)1_{\vec{l}}.
\]
\end{defn}

The morphisms of the algebra (or $1$-category) are generated for $i=1,\dots,m-1$ by the \textit{divided powers}
\[
E^{(j)}_i=\frac{E^j_i}{[j]!}\;\;\text{ and }\;\;F^{(j)}_i=\frac{E^j_i}{[j]!}.
\]
(Over $\mathbb{C}(q)$ the usual 
powers of $E_i$ and $F_i$ are sufficient and being generated by divided powers or usual powers is the same. But since we in principle could work integrally we prefer the above definition.)

We now briefly recall the $q$-skew Howe duality from \cite{ckm}.
To define $q$-skew Howe duality on the level of $\mathfrak{gl}_n$-webs with $m$ boundary points we restrict to certain weights $\vec{k}$ that we call \textit{$n$-bounded}. These weights have only entries $0\leq k_i\leq n$. Denote by a superscript $n$ the subalgebras with only these weights.

\begin{prop}\label{prop-qhowe}(\textbf{Pictorial $q$-skew Howe duality} - \cite[Section 5]{ckm})
The functor
\[
\gamma_{m,n}\colon 
\dot{{\Uu}}^n_q(\mathfrak{gl}_m)\to\spid{n}
\]
determined on morphisms by
\begin{align*}
1_{\vec{k}}\mapsto
\scalebox{.7}{$\xy
(-20,0)*{\includegraphics[scale=.75]{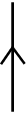}};
(-22,-5)*{\scriptstyle k_1};
(-10,0)*{\includegraphics[scale=.75]{figs/slnwebs/pointup.eps}};
(-12,-5)*{\scriptstyle k_2};
(0,0)*{\dots};
(10,0)*{\includegraphics[scale=.75]{figs/slnwebs/pointup.eps}};
(6.5,-5)*{\scriptstyle k_{m-1}};
(20,0)*{\includegraphics[scale=.75]{figs/slnwebs/pointup.eps}};
(18,-5)*{\scriptstyle k_m};
\endxy$}\quad
E_i1_{\vec{k}},F_i1_{\vec{k}}\mapsto
\scalebox{.7}{$\xy
(-45,0)*{\includegraphics[scale=.75]{figs/slnwebs/pointup.eps}};
(-47,-5)*{\scriptstyle k_1};
(-35,0)*{\dots};
(-25,0)*{\includegraphics[scale=.75]{figs/slnwebs/pointup.eps}};
(-28.5,-5)*{\scriptstyle k_{i-1}};
(0,0)*{\includegraphics[scale=.75]{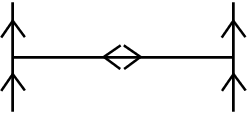}};
(-16.5,-5)*{\scriptstyle k_i};
(-17.75,5)*{\scriptstyle k_i \pm 1};
(2,2)*{\scriptstyle 1};
(8,5)*{\scriptstyle k_{i+1}\mp 1};
(10,-5)*{\scriptstyle k_{i+1}};
(25,0)*{\includegraphics[scale=.75]{figs/slnwebs/pointup.eps}};
(21.5,-5)*{\scriptstyle k_{i+2}};
(35,0)*{\dots};
(45,0)*{\includegraphics[scale=.75]{figs/slnwebs/pointup.eps}};
(43,-5)*{\scriptstyle k_m};
\endxy$}
\end{align*}
where the orientation of the arrow in the middle of the ladder is to the left for $E$ and to the right for $F$, is well-defined, pivotal and full. This defines an $\Uu_q(\mathfrak{gl}_m)$-action on the $\mathfrak{gl}_n$-spider.
\end{prop}

We note that the image of the divided powers is easy to write down, i.e. for $E_i^{(j)}$ and $F_i^{(j)}$ the middle arrow will have a label $j$ and the two shifts at the top will also be by $j$ instead of $1$.

\begin{rem}\label{rem-leash}
In order to work with the ladders in a pictorial convenient way we have to use the following convention, which we call \textit{leash-convention}.
\begin{itemize}
\item Edges labeled $0$ are not pictured.
\item Edges labeled are pictured using \textit{dotted leashes} that we tend to picture as Bordeaux colored edges. We do not 
illustrate orientation for leashes.
\end{itemize}
This has the advantage that ladders corresponding to $F$ (the ones we mostly use) will always point upwards. An example with $n=5$ is the following.
\[
\scalebox{.7}{$\xy(0,0)*{\includegraphics[scale=.75]{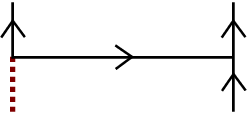}};(-12,-4)*{\scriptstyle 5};(16.5,-4)*{\scriptstyle 2};(1,2)*{\scriptstyle 1};(-12,4)*{\scriptstyle 4};(17,4)*{\scriptstyle 3}\endxy$}.
\]
\end{rem}

\subsubsection{The \texorpdfstring{$\mathfrak{gl}_n$}{gln}-web space}\label{subsub-webspace}
Now we are going to define the \textit{$\mathfrak{gl}_n$-web space} and afterwards in Section \ref{subsub-flows} the \textit{$\mathfrak{gl}_n$-flow lines} in the spirit of \cite{kk}. We only use $n$-bounded $\vec{k}$, i.e. $k_i\in\{0,\dots,n\}$, and we tend to omit the ``$n$-bounded'' from our notation. Moreover, we write $(n^{\ell})=(n,\dots,n,0,\dots,0)\in\Lambda(m,n\ell)_n$.

\begin{defn}(\textbf{The $\mathfrak{gl}_n$-web space})\label{defn-slnwebspace}
Given a fixed $\vec{k}\in\Lambda(m,n\ell)_n$ for some $\ell\in\bN$, the \textit{$\mathfrak{gl}_n$-web space for $\vec{k}$}, denoted by $W_n(\vec{k})$, is defined by
\[
W_n(\vec{k})=\Mor_{\spidn{n}{n}}((n^{\ell}),\vec{k})\cong\mathrm{Inv}_{\dot{\Uu}_q(\mathfrak{gl}_n)}(\Lambda^{\vec{k}}\bC^n).
\]
The \textit{$\mathfrak{gl}_n$-web space} $W_n(\Lambda)$ ($\Lambda$ denotes $n$-times the $\ell$-th fundamental $\mathfrak{gl}_m$-weight) is defined by
\[
W_n(\Lambda)=\bigoplus_{\vec{k}\in \Lambda(m,n\ell)_n}W_n(\vec{k})=\bigoplus_{\vec{k}\in \Lambda(m,n\ell)_n}\Mor_{\spidn{n}{n}}((n^{\ell}),\vec{k}).
\]
\end{defn}

Note that $q$-skew Howe duality gives $W_n(\Lambda)$ the structure of the irreducible $\dot{\Uu}_q(\mathfrak{gl}_m)$-module 
of highest weight $\Lambda$ (see \cite[Corollary 4.10]{my}).

Boundaries of $\mathfrak{gl}_n$-webs consist of univalent vertices (the end points of oriented edges), which we will usually put on a horizontal line (or various horizontal lines), called the \textit{cut line}, and that we usually picture by a dotted line, e.g. such a $\mathfrak{gl}_n$-web is shown below for $n=4$.
\[
\scalebox{.7}{$\xy
(0,0)*{\includegraphics[scale=.75]{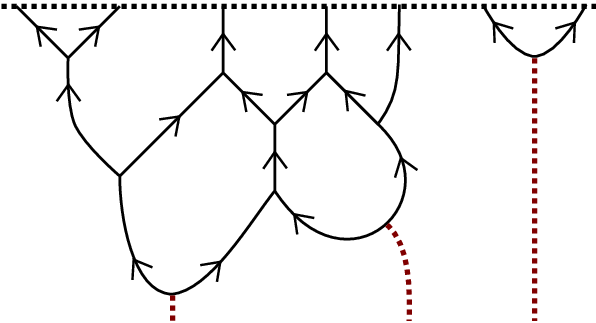}};
(-36,22)*{1};
(-22.5,22)*{1};
(-9.5,22)*{2};
(3.5,22)*{3};
(13,22)*{1};
(23.5,22)*{2};
(37,22)*{2};
(-35,16)*{\scriptstyle 1};
(-23.5,16)*{\scriptstyle 1};
(-26.5,8)*{\scriptstyle 2};
(-14,4)*{\scriptstyle 1};
(-6,10)*{\scriptstyle 1};
(0,10)*{\scriptstyle 2};
(8,10)*{\scriptstyle 1};
(-7.5,16)*{\scriptstyle 2};
(5.25,16)*{\scriptstyle 3};
(14.75,16)*{\scriptstyle 1};
(23.5,16)*{\scriptstyle 2};
(36.5,16)*{\scriptstyle 2};
(15.5,0)*{\scriptstyle 2};
(3,-7.5)*{\scriptstyle 2};
(-17.75,-13.5)*{\scriptstyle 3};
(-8.5,-13.5)*{\scriptstyle 1};
(-1,0)*{\scriptstyle 3};
\endxy$}.
\]
In this way, the boundary of a $\mathfrak{gl}_n$-web can be identified with a $\vec{k}$ as above. The $\mathfrak{gl}_n$-webs without boundary (that is $k_i\in\{0,n\}$) are called \textit{closed} $\mathfrak{gl}_n$-webs. 

Important convention: we tend to think in pictures and, by abuse of notation, sometimes call \textit{only} the $\bC(q)$-linear generators of $\fspidnn{n}{n}$ (i.e. no formal $\bC(q)$-sums, but all possible pictures) $\mathfrak{gl}_n$-webs. Of course, by linearity, these suffice for our purposes. 

Moreover, we will write $v^*$ to denote the $\mathfrak{gl}_n$-web obtained by reflecting a given $\mathfrak{gl}_n$-web $v$ horizontally and reversing all orientations but keeping the labels fixed. By $v^*u$ we shall mean the closed $\mathfrak{gl}_n$-web obtained by gluing $v^*$ on top of $u$, whenever such a construction is possible. That is, whenever the number of strands, the labels and the orientation match at the cut line.
\begin{align}\label{eq-dual}
\scalebox{.7}{$\xy
(0,0)*{\includegraphics[scale=0.4]{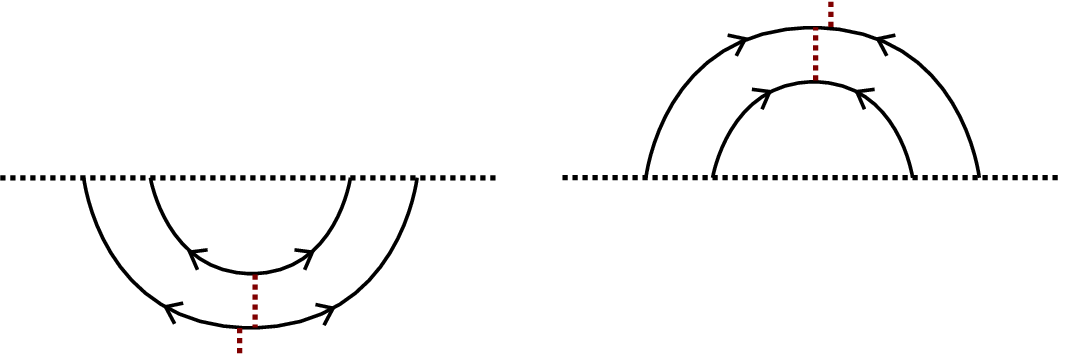}};
(-18.5,2)*{v};
(20,-2)*{v^*};
\endxy$}
\quad\rightsquigarrow\quad
\scalebox{.7}{$\raisebox{-0.0ex}{\xy
(0,0)*{\includegraphics[scale=0.4]{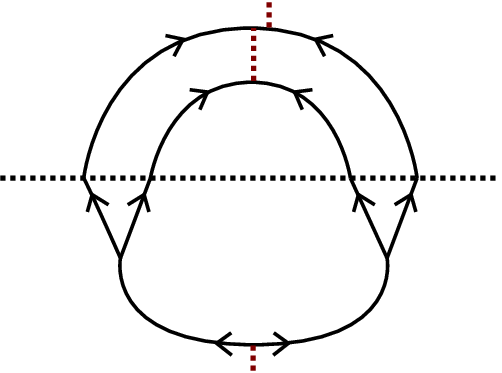}};
(12.5,-5)*{u};
(13,5)*{v^*};
\endxy}$}.
\end{align}

\begin{defn}\label{defn-kupform}(\textbf{Kuperberg form})
Given $u,v\in W_n(\Lambda)$ we define the \textit{Kuperberg form}
\[
\langle\cdot,\cdot\rangle_{\mathrm{Kup}}\colon W_n(\Lambda)\times W_n(\Lambda)\to \bC(q),
\quad
\langle u,v\rangle_{\mathrm{Kup}}=q^{d(\vec{k})}\mathrm{ev}(v^*u),
\]
where the evaluation map $\mathrm{ev}(\cdot)\colon\mathrm{End}_{\dot{\Uu}_q(\mathfrak{gl}_m)}(n^{\ell})\to\bC(q)$ is obtained as follows. First, interpret the closed $\mathfrak{gl}_n$-web $v^*u$ using Theorem \ref{thm-pifunctor} as an intertwiner with normalization factor $d(\vec{k})$ given by
\begin{equation}\label{eq-shift}
d(\vec{k})=\frac{1}{2}\left(n(n-1)\ell-\sum_{i=1}^mk_i(k_i-1)\right).
\end{equation}
Then extend this definition such that $\langle\cdot,\cdot\rangle_{\mathrm{Kup}}$ becomes 
$q$-antilinear in the first and $q$-linear in the second entry.
\end{defn}

\begin{prop}(\cite[Corollary 4.10]{my})\label{prop-kupshap}
The Kuperberg form on $W_{\Lambda}$ is, under $q$-skew Howe duality from Proposition \ref{prop-qhowe}, exactly the $q$-Shapovalov form $\langle\cdot,\cdot\rangle_{\mathrm{Shap}}$.\qed
\end{prop}

(We do not need the $q$-Shapovalov form in this paper and only refer to e.g. the part before \cite[Corollary 4.10]{my} for the definition.)

\subsubsection{Flow lines}\label{subsub-flows}

Given a $\mathfrak{gl}_n$-web $u$, we denote its vertex and edge sets by $V(u)$ and $E(u)$.

\begin{defn}(\textbf{$\mathfrak{gl}_n$-flow lines})\label{defn-slnflow}
Let $u\in W_n(\vec{k})$ be a $\mathfrak{gl}_n$-web. The set of possible \textit{edge colors} is
\[
\mathcal S=\mathfrak{P}(\{n,\dots,1\})=\mathfrak{P}^0(\{n,\dots,1\})\cup\dots\cup\mathfrak{P}^n(\{n,\dots,1\}),
\]
that is we identify the allowed edge colors with the subsets of $\{n,\dots,1\}$ where we order these colors by the number of their elements. We write $S_j\in\mathcal S$ with $S_j=\{s_1,\dots,s_j\}$ if $S_j$ has $j$ elements and its elements are ordered decreasing.

An \textit{$\mathfrak{gl}_n$-flow line $f$ for $u$} is a coloring of the edges of $u$ such that the following is satisfied.
\begin{itemize}
\item If the edge $e\in E(u)$ of $u$ has a label $j$, then the color has to be a subset with $j$ elements.
\item Recall that at each vertex there are either two incoming or outgoing edges. The colors for these two edges $S,S^{\prime}$ have to be disjoint, i.e. $S\cap S^{\prime}=\emptyset$.
\item The unique outgoing or incoming edge $S^{\prime\prime}$ has to satisfy $S^{\prime\prime}=S\cup S^{\prime}$.
\end{itemize}

For each vertex $v$ define the \textit{weight} $\mathrm{wt}^v(u_f)$ to be $\ell(S,S^{\prime})=|\{(i,j)\mid i\in S,j\in S^{\prime}, i<j\}|$ if and only if $S,S^{\prime}$ are the two upper edges and $-\ell(S^{\prime},S)$ if and only if $S,S^{\prime}$ are the two lower edges (in both cases ordered from left to right). Here, and in the following, $u_f$ denotes a $\mathfrak{gl}_n$-web $u$ together with a fixed flow $f$ for the $\mathfrak{gl}_n$-web $u$.

In the dual cases, that is with all arrows reversed, we flip all the sign conventions from above.

\textit{The (total) weight $\mathrm{wt}(u_f)$} is defined to be the sum over all local weights, i.e.
\[
\mathrm{wt}(u_f)=\sum_{v\in V(u)}\mathrm{wt}^v(u_f).
\]

The \textit{state string $\vec{S}_{u_f}$} given by $u_f$ is defined to be the ordered tuple of the colors of $u_f$ that touch the cut line. As we will see in Section \ref{sec-tabwebs}, state strings correspond 
bijectively to $n$-multipartitions, while flows on webs correspond bijectively to $n$-multitableaux.
\end{defn}

\begin{ex}\label{exa-flow}
For example, if $n=4$, $\vec{k}=(1,1,0,2,3,1,2,2)$ and the $\mathfrak{gl}_n$-web $u$ is the one from above, then a $\mathfrak{gl}_n$-flow line for $u$ is for example
\[
\scalebox{.7}{$\xy
(0,0)*{\includegraphics[scale=.75]{figs/slnwebs/exaweb1}};
(-36,16)*{\scriptstyle \{2\}};
(-22.5,16)*{\scriptstyle \{1\}};
(-24.5,8)*{\scriptstyle \{2,1\}};
(-13,4)*{\scriptstyle \{4\}};
(-8.5,5)*{\scriptstyle \{3\}};
(-1.5,10)*{\scriptstyle \{2,1\}};
(8.75,10)*{\scriptstyle \{3\}};
(-4.5,16)*{\scriptstyle \{4,3\}};
(8.25,16.5)*{\scriptstyle \{3,2,1\}};
(15.75,16)*{\scriptstyle \{4\}};
(21.5,16)*{\scriptstyle \{3,1\}};
(38.75,16)*{\scriptstyle \{4,2\}};
(18.75,-1)*{\scriptstyle \{4,3\}};
(5.5,-7.5)*{\scriptstyle \{2,1\}};
(-25.75,-13.5)*{\scriptstyle \{4,2,1\}};
(-6.5,-13.5)*{\scriptstyle \{3\}};
(2.5,0)*{\scriptstyle \{3,2,1\}};
\endxy$}.
\]
Moreover, the weight in this case is $9$.
\end{ex}

Let us denote by $Fl(u)$ the set of all possible flow lines of $u$.

\begin{thm}\label{thm-kkele}
Let $\vec{k}\in\Lambda(m,n\ell)_n$ for some $\ell\in\bN$. Fix a $\mathfrak{gl}_n$-web $u\in W_n(\vec{k})$. Then
\begin{equation}\label{eq-kkele}
u=\sum_{u_f\in Fl(u)}(-q)^{\mathrm{wt}(u_f)}x_{\vec{S}_{u_f}}\;\;\text{ with }x_{\vec{S}_{u_f}}\in\Lambda^{\vec{k}}\bC^n,
\end{equation}
where the pair $(\vec{S}_{u_f},\mathrm{wt}(u_f))$ is the state string and weight of $u_f$ and $x_{\vec{S}_{u_f}}$ is the corresponding elementary tensor.
\end{thm}

\begin{proof}
This is just the assembling of pieces now. To be more precise, we can use induction on the number of vertices of $u$ where it is easy to check for all small cases $V(u)<2$.

The main observation now is that locally our conventions match the ones given above Definition \ref{defn-pifunctor} for the intertwiners $M^{a,b}_s$ and $M^{a,b}_m$. It is worth noting that the exponents for $M^{a,b}_s$ equal exactly our definition, since for $T\subset S$ we see that $\ell(S,T)=\ell(S-T,T)$ and that our convention how flow lines add around vertices also match exactly with the cases where the intertwiner map to a non-trivial element. Thus, summing over all possibilities is the same as taking all possible flows.
We proceed by induction from a smaller $\mathfrak{gl}_n$-web to a bigger $\mathfrak{gl}_n$-web by adding one vertex. This is the same as composing the intertwiner for the smaller $\mathfrak{gl}_n$-web with one of the maps from above. Note that the coefficients will be multiplied. Hence, their powers add and this happens in the same way as for the total weight.
\end{proof}

\begin{ex}\label{exa-flow2}
In the case of the flow given in Example \ref{exa-flow} we see that the weight is $9$ and the state string is $\vec{S}_{u_f}=(\{2\},\{1\},\emptyset,\{4,3\},\{3,2,1\},\{4\},\{3,1\},\{4,2\})$. Hence, the corresponding elementary tensor is
\[
x_{\vec{S}_{u_f}}=x_2\otimes x_1\otimes 1\otimes (x_4\wedge x_3)\otimes (x_3\wedge x_2 \wedge x_1)\otimes x_4\otimes (x_3\wedge x_1) \otimes (x_4\wedge x_2).
\]
It is an element of $\Lambda^{\vec{k}}\bC^4=\bC^4\otimes \bC^4\otimes \bC\otimes\Lambda^2\bC^4\otimes\Lambda^3\bC^4\otimes\bC^4\otimes\Lambda^2\bC^4\otimes\Lambda^2\bC^4$, since we have $\vec{k}=(1,1,0,2,3,1,2,2)$. The Theorem \ref{thm-kkele} ensures that it appears in the decomposition of $u$ as a sum of elementary tensors at least once with multiplicity $(-q)^{\mathrm{wt}(u_f)}=-q^9$. In order to find the full coefficient for $x_{\vec{S}_{u_f}}$ one has to know all flows with the same state string as $u_f$ and their weights.
\end{ex}
\subsection{KLR algebras, categorification of \texorpdfstring{$\mathfrak{gl}_n$}{gln}-webs and categorified \texorpdfstring{$q$}{q}-skew Howe duality}\label{sec-higherstuff}
\subsubsection{The general linear quantum 2-algebras}\label{subsub-q2alg}

We briefly recall the (diagrammatic) categorification of the idempotented quantum groups
${\mathcal{U}}(\mathfrak{gl}_m)={\mathcal{U}}_Q(\mathfrak{gl}_m)$ in this section, see \cite{kl5} or \cite{rou}. 
We fix the following possible choices in the notation of \cite{cala}: the scalars $Q$ are given by $t_{ij}=-1$ if $j=i+1$, $t_{ij}=1$ otherwise, $r_i=1$ and $s_{ij}^{pq}=0$ (this corresponds to the signed version in \cite{kl5} and \cite{kl4}).  

Note that we work with $\mathfrak{gl}_m$ 
on our Howe dual side and 
all appearing roots and weights are roots and weights of the general linear Lie algebra.

\begin{defn}\label{defn-KL}
The $2$-category $\Ucatm$ is defined as follows.
\begin{itemize}
\item The objects in $\Ucatm$ are the weights $\vec{k}\in\bZ^{m}$.
\end{itemize}
For any pair of objects $\vec{k}$ and $\vec{k}^{\prime}$ in $\Ucatm$, the hom category 
$\Ucatm(\vec{k},\vec{k}^{\prime})$ is the $\bZ$-graded, additive $\bC$-linear category consisting of the following data.
\begin{itemize}
\item Objects (or $1$-morphisms), that is finite formal sums of the form $\mathcal{E}_{\ui}{\idm}_{\vec{k}}\{t\}$ and $\mathcal{F}_{\ui}{\idm}_{\vec{k}}\{t\}$ where $t\in\bZ$ is a grading shift and $\ui$ is string of $i\in\{1,\dots,m-1\}$ such that $\vec{k}^{\prime}=\vec{k}+\sum_{a=1}^{l}\epsilon_ai_{a}^{\prime}$.
\item The spaces of $1$-morphisms (or $2$-morphisms) are the $\bZ$-graded, $\bC$-vector spaces generated by compositions of diagrams shown below. Here $\{t\}$ denotes a degree shift up by $t$ and we use the shorthand notations $\alpha^{ii^{\prime}}=(\alpha_i,\alpha_{i^{\prime}})$ and $\alpha^{\vec{k} i}=2\frac{(\vec{k},\alpha_i)}{(\alpha_i,\alpha_i)}$.
\[
\phi_1=\xy
(0,1)*{\includegraphics[width=09px]{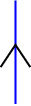}};
(1.5,-4)*{\scriptstyle i};
(3,1)*{\scriptstyle\vec{k}};
(-5,1)*{\scriptstyle\vec{k}+\alpha_i};
\endxy\,\hspace{8mm}
\,
\phi_2=\xy
(0,1)*{\includegraphics[width=09px]{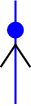}};
(1.5,-4)*{\scriptstyle i};
(3,1)*{\scriptstyle\vec{k}};
(-5,1)*{\scriptstyle\vec{k}+\alpha_i};
\endxy\,\hspace{8mm}
\,\phi_3=\xy
(0,1)*{\includegraphics[width=25px]{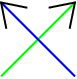}};
(-5,-3)*{\scriptstyle i};
(5.25,-3)*{\scriptstyle i^{\prime}};
(5.5,0)*{\scriptstyle\vec{k}};
\endxy\,\hspace{8mm}\,
\phi_4=\xy
(0,0)*{\includegraphics[width=25px]{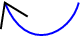}};
(0,-3)*{\scriptstyle i};
(5,0)*{\scriptstyle\vec{k}};
\endxy\hspace{8mm}\,
\phi_5=\xy
(0,0)*{\includegraphics[width=25px]{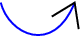}};
(0,-3)*{\scriptstyle i};
(5,0)*{\scriptstyle\vec{k}};
\endxy,
\]
\[
\psi_1=\xy
(0,1)*{\includegraphics[width=09px]{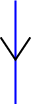}};
(1.5,-4)*{\scriptstyle i};
(3,1)*{\scriptstyle\vec{k}};
(-5,1)*{\scriptstyle\vec{k}-\alpha_i};
\endxy\,\hspace{8mm}
\,
\psi_2=\xy
(0,1)*{\includegraphics[width=09px]{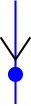}};
(1.5,-4)*{\scriptstyle i};
(3,1)*{\scriptstyle\vec{k}};
(-5,1)*{\scriptstyle\vec{k}-\alpha_i};
\endxy\,\hspace{8mm}
\,\psi_3=\xy
(0,1)*{\includegraphics[width=25px]{figs/higherstuff/downcross}};
(-5,-3)*{\scriptstyle i};
(5.25,-3)*{\scriptstyle i^{\prime}};
(5.5,0)*{\scriptstyle\vec{k}};
\endxy\,\hspace{8mm}\,
\psi_4=\xy
(0,0)*{\includegraphics[width=25px]{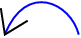}};
(0,-3)*{\scriptstyle i};
(5,0)*{\scriptstyle\vec{k}};
\endxy\hspace{8mm}\,
\psi_5=\xy
(0,0)*{\includegraphics[width=25px]{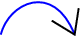}};
(0,-3)*{\scriptstyle i};
(5,0)*{\scriptstyle\vec{k}};
\endxy.
\]
The convention for reading these diagrams is from right to left and bottom to top.
\item The degree can be fixed as follows. First, a $2$-morphism $\alpha\colon X\to Y$ of degree $d$ is of degree 
$d-a+b$ seen as a $2$-morphism $\alpha\colon X\{a\}\to Y\{b\}$. 
Now, the generating $2$-morphisms above are homogeneous of degree $0$ if considered between the following shifts.
$\phi_1=\mathrm{id}_{\mathcal{E}_{i}\onel}$, $\phi_2\colon\mathcal{E}_{i}\onel\Rightarrow\mathcal{E}_{i}\onel\{\alpha^{ii}\}$, $\phi_3\colon\mathcal{E}_{i}\mathcal{E}_{i^{\prime}}\onel\Rightarrow\mathcal{E}_{i^{\prime}}\mathcal{E}_{i}\onel\{\alpha^{ii^{\prime}}\}$ and $\phi_4\colon\onel\{\frac{1}{2}\alpha^{ii}+\alpha^{\vec{k} i}\}\Rightarrow\mathcal{E}_{i}\mathcal{F}_{i}\onel$ and $\phi_5\colon\onel\{\frac{1}{2}\alpha^{ii}-\alpha^{\vec{k} i}\}\Rightarrow\mathcal{F}_{i}\mathcal{E}_{i}\onel$, and $\psi_1=\mathrm{id}_{\mathcal{F}_{i}\onel}$, $\psi_2\colon\mathcal{F}_{i}\onel\Rightarrow\mathcal{F}_{i}\onel\{\alpha^{ii}\}$, $\psi_3\colon\mathcal{F}_{i}\mathcal{F}_{i^{\prime}}\onel\Rightarrow\mathcal{F}_{i^{\prime}}\mathcal{F}_{i}\onel\{\alpha^{ii^{\prime}}\}$ and $\psi_4\colon\mathcal{F}_{i}\mathcal{E}_{i}\onel\Rightarrow\onel\{\frac{1}{2}\alpha^{ii}+\alpha^{\vec{k} i}\}$ and $\psi_5\colon\mathcal{E}_{i}\mathcal{F}_{i}\onel\Rightarrow\onel\{\frac{1}{2}\alpha^{ii}-\alpha^{\vec{k} i}\}$.
\item There are relations imposed onto $2$-morphisms, where we take the ones from the signed version in \cite{kl5} and \cite{kl4}. 
(We will not recall here since we do not need them explicitly.)
\end{itemize}
\end{defn}

Recall that, given a $1$-category $\mathcal{C}$, then the objects of the \textit{Karoubi envelope $\KAR(\mathcal{C})$} are pairs $(O,e)$ where $O\in\Ob(\mathcal{C})$ is an object and $e\colon O\to O$ is a projector $e^2=e$. For the case we are interested in, that is the Karoubi envelope of $\Ucatt$, one can define analogs of the \textit{divided powers} $E_i^{(j)}$ and $F_i^{(j)}$ as follows.

Fix a \textit{color} $j\in\bN$ and set $O=\mathcal{F}^j\onel$, where 
$\vec{k}\in\bZ^2$. Define $e_j\colon O\to O$ to be the idempotent obtained by any reduced presentation of the longest braid word on $j$ strands together with a certain, fixed dot placement (see \cite[(2.18)]{klms}). 
Then $\mathcal{F}^{(j)}\onel=(O\{\frac{j(j-1)}{2}\},e_j)$ and one can define $\mathcal{E}^{(j)}\onel$ similarly.

The category $\Ucattd$ can be described by using \textit{thick calculus}, cf. \cite{klms}. 
The (for us) most important $2$-morphisms are then given by (the right face should carry the label $\vec{k}$)
\[
\xy
(0,0)*{\includegraphics[width=09px]{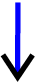}};
(1.75,0)*{\scriptstyle j};
\endxy\colon \mathcal{F}^{(j)}\onel\to \mathcal{F}^{(j)}\onel,\hspace*{0.4cm}
\xy
(0,0)*{\includegraphics[width=25px]{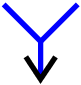}};
(-5,2)*{\scriptstyle j};
(5,2)*{\scriptstyle j^{\prime}};
(5,-2)*{\scriptstyle j+j^{\prime}};
\endxy\colon \mathcal{F}^{(j+j^{\prime})}\onel\to \mathcal{F}^{(j)}\mathcal{F}^{(j^{\prime})}\onel,\hspace*{0.4cm}
\xy
(0,0)*{\includegraphics[width=25px]{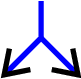}};
(-5,-2)*{\scriptstyle j};
(5.5,-2)*{\scriptstyle j^{\prime}};
(3.5,2)*{\scriptstyle j+j^{\prime}};
\endxy\colon \mathcal{F}^{(j)}\mathcal{F}^{(j^{\prime})}\onel\to \mathcal{F}^{(j+j^{\prime})}\onel
,
\]
called \textit{thick identity}, \textit{split} and \textit{merge}, the latter two being of degree $jj^{\prime}$. The \textit{thick crossing} is then a composite of (first) the merge and (then) the split
\begin{equation}\label{eq-createcrossing}
\xy
(0,1)*{\includegraphics[width=25px]{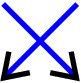}};
(-5,-2)*{\scriptstyle j};
(5.75,-2)*{\scriptstyle j^{\prime}};
\endxy\colon \mathcal{F}^{(j)}\mathcal{F}^{(j^{\prime})}\onel\to \mathcal{F}^{(j^{\prime})}\mathcal{F}^{(j)}\onel=
\xy
(0,0)*{\includegraphics[width=25px]{figs/higherstuff/split}};
(-5,2)*{\scriptstyle j};
(5,2)*{\scriptstyle j^{\prime}};
(5,-2)*{\scriptstyle j+j^{\prime}};
\endxy
\circ
\xy
(0,0)*{\includegraphics[width=25px]{figs/higherstuff/merge}};
(-5,-2)*{\scriptstyle j};
(5.5,-2)*{\scriptstyle j^{\prime}};
(3.5,2)*{\scriptstyle j+j^{\prime}};
\endxy
.
\end{equation}
The $2$-category consisting of these diagrams is denoted by $\Ucattc$, which can then be extended to a graphical calculus for $\Ucattd$ by introducing generalized versions of the dot $2$-morphisms: for each each $j$-labeled thick strand one allows a symmetric polynomial $p\in\bZ[X_1,\dots,X_j]^{S_j}$ which satisfy certain relations, see \cite{klms}.

We define $\Ucatmc$ to be the full $2$-subcategory of $\Ucatmd$ with the same objects $\vec{k}$, but with $1$-morphisms generated by the divided powers $\mathcal{E}^{(j)}_i\onel$ and $\mathcal{F}^{(j)}_i\onel$ from above for each $i\in\{1,\dots,m-1\}$.

\subsubsection{The cyclotomic KLR algebras}\label{subsub-KLR}

Let $\Lambda$ be a dominant $\mathfrak{gl}_m$-weight, $V_{\Lambda}$ the irreducible $\Um$-module of highest 
weight $\Lambda$ and $P_{\Lambda}$ the set of weights in $V_{\Lambda}$.

\begin{defn}\label{defn-cyclKLR}
The \textit{cyclotomic KLR (Khovanov--Lauda, Rouquier) algebra} $R_{\Lambda}$ is defined as the $2$-subquotient of $\Ucatm$ consisting of all diagrams with only downward oriented strands and rightmost region labeled $\Lambda$ modded out by the $2$-ideal generated by all diagrams of the form
\begin{align}\label{eq-cyclorel}
\xy
(0,0)*{\includegraphics[width=75px]{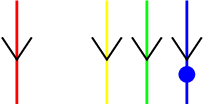}};
(11.5,-8.5)*{\scriptstyle i_1};
(6,-8.5)*{\scriptstyle i_2};
(0.5,-8.5)*{\scriptstyle i_3};
(-5,0)*{\dots};
(-11,-8.5)*{\scriptstyle i_p};
(17,-4.75)*{\scriptstyle \Lambda_{i_1}\text{-dots}};
(14,0)*{\scriptstyle \Lambda};
\endxy
,
\end{align}
where $i_k\in\{1,\dots,m-1\}$ and $p\in\bN$.
The relation \eqref{eq-cyclorel} is known as the \textit{cyclotomic relation}.
\end{defn}

Note that 
$R_{\Lambda}=\bigoplus_{\vec{k}\in P_{\Lambda}} R_{\Lambda}(\vec{k})$,
where $R_{\Lambda}(\vec{k})$ is the subalgebra generated by all diagrams 
whose left-most region is labeled $\vec{k}$. The algebra $R_{\Lambda}$ is finite dimensional, see \cite{bk1}.

If we draw pictures for the cyclotomic KLR algebra, then we do not need orientations anymore, that is pictures will look like
\[
\xy(0,0)*{\includegraphics[scale=.75]{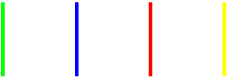}}\endxy\quad\text{or}\quad\xy(0,0)*{\includegraphics[scale=.75]{figs/higherstuff/HM-strings1}}\endxy
\]

In \cite{hm} Hu and Mathas defined a graded cellular basis of the cyclotomic KLR algebra $R_{\Lambda}$. We do not recall their definition here, since it is not short and we give an alternative definition in our language later. The reader is encouraged to take a look at their great paper. We call their basis \textit{HM basis}. We only mention that their basis (in the form we need it) is parameterized by $\vec{\lambda}\in\Lambda^+(c,c(\vec{k}),c^{\prime})$, i.e. all $c^{\prime}$-multipartitions of $c(\vec{k})$ for all suitable $c,c^{\prime}$, and $\vec{T},\vec{T}^{\prime}\in\mathrm{Std}(\vec{\lambda})$, i.e. standard $c^{\prime}$-multitableaux. They denote their basis by
\begin{align}\label{defn-hmbasis}
\{\psi^{\vec{\lambda}}_{\vec{T}^{\prime},\vec{T}}\mid \vec{\lambda}\in \mathfrak{P}_{c(\vec{k})}\text{ and }\vec{T},\vec{T}^{\prime}\in\mathrm{Std}(\vec{\lambda})\},
\end{align}
where $\mathfrak{P}_{c(\vec{k})}$ is the set of all multipartitions of $c(\vec{k})$. Moreover, the basis is homogeneous with degree
\[
\mathrm{deg}_{\mathrm{BKW}}(\psi^{\vec{\lambda}}_{\vec{T}^{\prime},\vec{T}})=\mathrm{deg}_{\mathrm{BKW}}(\vec{T})+\mathrm{deg}_{\mathrm{BKW}}(\vec{T}^{\prime}).
\]

To make the connection to webs: we fix $c^{\prime}=n$, and $c(\vec{k})$ is a constant that only depends on the weight $\vec{k}$. It could be written in an explicit formula as the author has done in \cite{tub3} for $\mathfrak{sl}_3$, but we do not do it here since we do not use the formula and it is rather cumbersome. We only note that it just counts the number of $F$ one has to apply (as an $\dot{\Uu}_q(\mathfrak{gl}_m)$-action) to go from $(n^{\ell})$ to $\vec{k}$. And the constant $c=c(\vec{S})$ depends only on the $\mathfrak{gl}_n$-flows at the cut line (and can be also written down explicitly, but we do not need any explicit formula). To summarize, we have two fixed numbers $n$ and $c(\vec{k})$ and consider the set of all $n$-multipartitions of $c(\vec{k})$.

\begin{defn}(\textbf{Thick cyclotomic KLR})\label{defn-thickcyclKLR}
The \textit{thick cyclotomic KLR algebra}, denoted by $\check{R}_{\Lambda}$, is the $2$-subquotient of $\Ucatmc$ defined by the $2$-subcategory of all diagrams with only downward oriented strands and rightmost region labeled $\Lambda$ and modded out by the cyclotomic relation \eqref{eq-cyclorel}.
\end{defn}

We will define a HM basis for $\check{R}_{\Lambda}$ later on.

\subsubsection{Matrix factorizations and categorification of $\mathfrak{gl}_n$-webs}\label{subsub-mfs}

Our main sources are \cite{my} and \cite{mack1} where the reader can find much more details. We keep our notation close to theirs (e.g. we suppress the shifts in homology degree) and the corresponding algebraic definitions can be found therein.

All the reader needs to know about matrix factorizations on the level of $\mathfrak{gl}_n$-webs is that a $\mathfrak{gl}_n$-web $u$ can be seen as a matrix factorization denoted by $\widehat{u}$. Such matrix factorizations are $(\bZ/2\bZ,\bZ)$-graded where the latter degree is called the \textit{$q$-grading}. Shifting in the first grading is indicated by $\langle\cdot\rangle$ and shifts in the $q$-grading by $\{\cdot\}$. For example, there is a \textit{dual} matrix factorization $\widehat{u}_{\bullet}$ and one can check that 
$\widehat{u}_{\bullet}\cong \widehat{u^*}\langle 1\rangle\{d(\vec{k})\}$ for $u\in W_n(\vec{k})$. (Note that taking duals in this context does not invert arrows on webs, but is rather the operation from \eqref{eq-dual}.)

Very important for us in the following are the ones that correspond to an $E^{(j)}_i$ or to an $F^{(j)}_i$. Both of them are indecomposable. We denote them by $\widehat{E}^{(j)}_{(k_i,k_{i+1})}$ and $\widehat{F}^{(j)}_{(k_i,k_{i+1})}$, respectively. 
Furthermore, we denote the one that corresponds to the identity by $\widehat{1}_{\vec{k}}$.

We freely switch between the notions of $\mathfrak{gl}_n$-webs and their corresponding matrix factorizations (e.g. we tend to write $F^{(j)}_i$ instead of $\widehat{F}^{(j)}_{(k_i,k_{i+1})}$).

In short, on the level of $1$-morphism we usually use the language of $\mathfrak{gl}_n$-webs, but on the level of $2$-morphism we use the language explained below, i.e. using certain $\mathrm{EXT}$-spaces which are isomorphic to certain $\langle\cdot\rangle$-shifted $\mathrm{HOM}$-spaces (modulo null-homotopic maps) between matrix factorizations (see \cite[Proposition 5.6]{my}). Thus, we can loosely call them \textit{homomorphisms of matrix factorizations}.

\subsubsection{The $\mathfrak{gl}_n$-web-algebra}\label{subsub-webalg}

Now we recall the definition of the $\mathfrak{gl}_n$-web algebra $H_n(\vec{k})$ from \cite{mack1}.

\begin{defn}\label{defn-slnwebalgebra}
Choose a fixed monomial basis $B(W_n(\vec{k}))$ of $W_n(\vec{k})$. That is, any basis vector $u\in B(W_n(\vec{k}))$ can be obtained from a fixed highest weight vector using $q$-skew Howe duality. We do not recall the exact definition here and refer to Example \ref{exa-fstring} instead. It should be noted that this includes that any basis vector is one fixed $\mathfrak{gl}_n$-web without any quantum factors.

For any pair $u,v\in B(W_n(\vec{k}))$, define (for $d(\vec{k})$ as in \eqref{eq-shift}) 
\[
{}_vH_n(\vec{k})_u=\mathrm{EXT}(\widehat{u},\widehat{v})\cong 
H(\widehat{v^*u})\{d(\vec{k})\}.
\]
The $\mathfrak{gl}_n$-web algebras $H_n(\vec{k})$ and $H_n(\Lambda)$ are defined by 
\[
H_n(\vec{k})=\bigoplus_{u,v\in B(W_n(\vec{k}))} {}_vH_n(\vec{k})_u\hspace*{0.25cm}\text{ and }\hspace*{0.25cm}H_n(\Lambda)=\bigoplus_{\vec{k}\in\Lambda(m,n\ell)_n}H_n(\vec{k}),
\]
with multiplication induced by the composition of maps between the corresponding matrix 
factorizations.
\end{defn}

It should be noted that $H_n(\vec{k})$ is a $\bZ$-graded, finite dimensional, unital, associative algebra. Moreover, the algebra is a $\bZ$-graded, symmetric Frobenius algebra of Gorenstein parameter $2d(\vec{k})$, that is, $H_n(\vec{k})\{-2d(\vec{k})\}$ is graded isomorphic (as $H_n(\vec{k})$-bimodules) to its graded dual. The trace $\tau$ is given by pairing elements of $H_n(\vec{k})$ with the identity $1=\sum_{u\in W_n(\vec{k})}\mathrm{id}(\widehat{u})$.

\begin{rem}\label{rem-basis}
In \cite{mack1} Mackaay has chosen a certain monomial basis called \textit{LT-basis}. This basis is obtained from a $q$-skew Howe analog of an \textit{intermediate crystal basis} defined by Leclerc--Toffin \cite{leto}. We note that all of Mackaay's constructions that are important for us only depend on the fact that this basis is monomial. In fact, Mackaay's arguments in \cite[Lemma 7.5]{mack1} show that, for all choices of bases, all the possibly different $\mathfrak{gl}_n$-web algebras will be Morita equivalent.
\end{rem}

\subsubsection{Categorified $q$-skew Howe duality}\label{subsub-cathowe}
As a last ingredient we are going to recall now how these constructions can be used to categorify an instance of $q$-skew Howe duality. We should note that this is in fact one of our main ingredients, but since the definition of the $2$-action of $\Ucatm$ on $\dot{\mathcal W}^{\circ}_{\Lambda}\cong \mathcal W^p_{\Lambda}$ (the first is a $2$-category of matrix factorizations and the second is a $2$-category of $H_n(\Lambda)$-representations, see \cite[Definition 7.1]{mack1}) is not short in any sense, we only recall it very briefly, i.e. by an example of the action on $2$-morphisms. The full list can be found in \cite[Section 9]{my}.
The point is that categorified $q$-skew Howe duality also defines a $2$-action of $\Ucatm$ on $\mathcal W^p_{\Lambda}$.

\begin{thm}\label{thm-qhowe}(\textbf{Categorified pictorial $q$-skew Howe duality}, see \cite[Theorem 9.7]{my})
The $2$-functor
\begin{equation}\label{eq-klr}
\Gamma_{m,n\ell,n}\colon\Ucatm\to\mathcal W^p_{\Lambda},
\end{equation}
defined on objects and $1$-morphisms similarly as in Proposition \ref{prop-qhowe} and on $2$-morphisms by the list of cases in \cite[Section 9]{my}, is a well-defined $2$-action of $\Ucatm$ on $\mathcal W^p_{\Lambda}$ giving latter the structure of a strong $\mathfrak{gl}_m$-$2$-representation in the sense of \cite{cala}. This strong $\mathfrak{gl}_m$-$2$-representation induces an additive equivalence of $2$-categories
\begin{equation}\label{eq-klrequi}
\tilde{\Gamma}_{m,n\ell,n}=\tilde{\Gamma}\colon\RPMOD\to\mathcal W^p_{\Lambda},
\end{equation}
i.e. from the category of finite dimensional, $\bZ$-graded, projective $R_{\Lambda}$-modules to $\mathcal W^p_{\Lambda}$.
\end{thm}

All the reader needs to know to understand this paper about the list for the $2$-action is that there are certain homomorphisms between matrix factorizations associated to the for us most important pieces
\begin{gather}\label{eq-mf-homs}
\xy
(0,1)*{\includegraphics[width=25px]{figs/higherstuff/downcross}};
(-5,-3)*{\scriptstyle i};
(5,-3)*{\scriptstyle j};
(6.5,1)*{\scriptstyle \vec{k}};
\endxy \mapsto \begin{cases}\widehat{CR}_{ji}\colon \widehat{F}_i\widehat{F}_{i\pm 1}\to \widehat{F}_{i\pm 1}\widehat{F}_i, &\text{if }j= i\pm 1,\\
\widehat{I}_{ii}\widehat{D}_{ii}\colon \widehat{F}_i\widehat{F}_{i}\to \widehat{F}_i\widehat{F}_{i}, &\text{if }i= j,\\
\widehat{s}_{ji}\colon \widehat{F}_i\widehat{F}_{j}\to \widehat{F}_{j}\widehat{F}_i, &\text{if }|i-j|>1,\end{cases}\quad\text{and}\quad
\xy
(0,1)*{\includegraphics[width=09px]{figs/higherstuff/downsimpledot}};
(1.5,-4)*{\scriptstyle i};
(3,1)*{\scriptstyle\vec{k}};
(-5,1)*{\scriptstyle\vec{k}-\bar{\alpha}_i};
\endxy \mapsto \widehat{t}_i\colon \widehat{F}_i\to\widehat{F}_i,
\end{gather}
of $q$-degree $1$, $-2$, $0$ and $2$ respectively. We will not recall the definition of these morphisms of matrix factorizations, see \cite{my} for the definitions, we only need to know their existence and that they satisfy the relations of the categorified general linear quantum group. Let us however briefly sketch how to think about these morphism. For the case $n=2$ these correspond in the familiar \textit{cobordism language} (see for example \cite{lqr1}) to a \textit{saddle}, a \textit{cup} followed by a \textit{cap} and a \textit{shift}. In the $n=3$ case these can also be translated to natural pictures, see for example \cite{lqr1} or \cite{mpt}. Moreover, the homomorphism $\widehat{t}_i$ is of $q$-degree $2$ and can be thought of as placing a dot on the corresponding ladder. To make the notation cumbersome we use sub- and superscripts like $\widehat{F}^{(j)}_{p,i,\vec{k}}$ to indicate the position $p$ (read from right to left in the KLR picture and from bottom to top in the $\mathfrak{gl}_n$-web picture), the (possible divided) power $j$, the residue (or color) $i$ and the weight $\vec{k}$. We sometimes skip some of them and hope that it is clear from the context in those cases.

The $2$-action works roughly as follows. Given one of the $2$-cell generators of $\Ucatm$, one has an object given by the $\vec{k}$ and two $\mathfrak{gl}_n$-webs at the bottom $u_b$ and top $u_t$ by reading from right to left and apply an $E_i$ for each upwards pointing string with label $i$ one passes and an $F_i$ for each downwards pointing string with label $i$. Then assign a certain homomorphism between the matrix factorization $\widehat{u}_b$ and $\widehat{u}_t$ as a $2$-morphism. For example, for $n=3$ and position $p=1$
\[
\psi_3=\xy
(0,1)*{\includegraphics[width=25px]{figs/higherstuff/downcross}};
(-5,-3)*{\scriptstyle 1};
(5,-3)*{\scriptstyle 2};
(8.5,0)*{\scriptstyle (1,2,0)};
\endxy \mapsto \widehat{CR}_{1,21}\colon u_b=F_1F_2v_{(1,2,0)}\to F_2F_1v_{(1,2,0)}=u_t.
\]
In pictures:
\[
\widehat{CR}_{1,21}\colon
\scalebox{.7}{$\xy
(0,0)*{\includegraphics[scale=.75]{figs/higherstuff/braid2bexa.eps}};
(-26,-13)*{1};
(2.5,-13)*{2};
(30,-13)*{0};
(14,-4)*{F_2};
(-26,0)*{1};
(2.5,0)*{1};
(30,0)*{1};
(-14,10)*{F_1};
(-26,13)*{0};
(2.5,13)*{2};
(30,13)*{1};
\endxy$}
\to
\scalebox{.7}{$\xy
(0,0)*{\includegraphics[scale=.75]{figs/higherstuff/braid3a.eps}};
(-26,-13)*{1};
(2.5,-13)*{2};
(30,-13)*{0};
(14,10)*{F_2};
(-26,0)*{0};
(2.5,0)*{3};
(30,0)*{0};
(-14,-4)*{F_1};
(-26,13)*{0};
(2.5,13)*{2};
(30,13)*{1};
\endxy$}.
\]
For the reader familiar with the corresponding 
foamation (see \cite{lqr1}, \cite{mpt} or \cite{qr1}) we note that this is like zipping certain edges away.

\begin{thm}\label{thm-thick}
The $2$-functor $\Gamma_{m,n\ell,n}$ extends to a $2$-functor
\[
\check{\Gamma}_{m,n\ell,n}\colon\Ucatmc\to\mathcal W^p_{\Lambda}.
\]
\end{thm}

\begin{proof}
Given any two $1$-categories and a $1$-functor $\mathcal{FUN}\colon\mathcal{C}\to\mathcal{D}$, there exists (by the universal property of the Karoubi envelope) an extension $\overline{\mathcal{FUN}}\colon\KAR(\mathcal{C})\to\KAR(\mathcal{D})$. Moreover, any $1$-category $\mathcal{C}$ embeds via $O\mapsto (O,\mathrm{id})$ fully faithful into $\KAR(\mathcal{C})$. Both statements are still true in the $2$-categorical setting.

Thus, it suffices to show that
\[
\overline{\Gamma}_{m,n\ell,n}(\mathcal{F}^{(j)}_i\onel)\cong (\widehat{F}^{(j)}_{(k_i,k_{i+1})},\mathrm{id}(\widehat{F}^{(j)}_{(k_i,k_{i+1})})),\;\text{ with }\vec{k}=(\dots,k_i,k_{i+1},\dots).
\]
(And the same for $\mathcal{E}^{(j)}_i\onel$.)
On the level of the $\mathfrak{gl}_n$-webs this means we need to prove
\[
\left(
\scalebox{.7}{$\xy(0,0)*{\includegraphics[scale=.75]{figs/slnwebs/squarea.eps}};(-11,-11)*{\scriptstyle k_i};(18.5,-11)*{\scriptstyle k_{i+1}};(2,8.5)*{\scriptstyle 1};(2,-5)*{\scriptstyle 1};(-9.5,11)*{\scriptstyle k_i-j};(20,11)*{\scriptstyle k_{i+1}+j};(-0.5,-0.95)*{\vdots};(-0.5,3.45)*{\vdots};(-12,-0.95)*{\vdots};(-12,3.45)*{\vdots};(12,-0.95)*{\vdots};(12,3.45)*{\vdots};\endxy$},
\widehat{I}^{(j)}_i\circ\widehat{D}^{(j)}_i\circ\widehat{t}^{\mathrm{sym}}\right)\cong \left(
\scalebox{.7}{$\xy(0,0)*{\includegraphics[scale=.75]{figs/slnwebs/laddera.eps}};(-11,-4)*{\scriptstyle k_i};(18.5,-4)*{\scriptstyle k_{i+1}};(2,2)*{\scriptstyle j};(-10,4)*{\scriptstyle a-j};(18,4)*{\scriptstyle b+j}\endxy$},\mathrm{id}(\widehat{F}^{(j)}_{(k_i,k_{i+1})})\right),
\]
where the ladders labeled $1$ are repeated $j$-times. Here we introduce some notation. We define
\[
\widehat{I}^{j^{\prime}}_i\colon \widehat{F}_i^{(j^{\prime}+1)}\to\widehat{F}_i^{(j^{\prime})}\widehat{F}_i,\;\widehat{I}^{(j)}_i=\widehat{I}_i^{1}\circ\dots\circ\widehat{I}_i^{j-1},\;
\widehat{D}_i^{j^{\prime}}\colon \widehat{F}_i\widehat{F}_i^{(j^{\prime})}\to\widehat{F}_i^{(j^{\prime}+1)}\text{ and }\widehat{D}^{(j)}_i=\widehat{D}^{j-1}_i\circ\dots\circ\widehat{D}^{1}_i.
\]
The steps $\widehat{I}^{j^{\prime}}$ and $\widehat{D}^{j^{\prime}}$ should be composites of $\widehat{CR}$ and $\widehat{t}$ exactly as the and $(j^{\prime},1)$-splitters and $(1,j^{\prime})$-merges are defined in Section 2 of \cite{klms}. The subscript $\mathrm{sym}$ should indicate a symmetric spread of dots starting with $j-1$ for the top edge to no dots for the bottom.

Now comes the good part about matrix factorizations: a lot of calculations are already done. So we do not need to redo them. In fact, the isomorphism above follows from work of Mackaay and Yonezawa \cite{my} (we also mention Wu \cite{wu} and Yonezawa \cite{yo1}, \cite{yo2} here) without any extra calculations. To be precise, Theorem \ref{thm-qhowe} implies that $\overline{\Gamma}_{m,n\ell,n}(\mathcal{F}^{(j)}_i\onel)$ is given as above and \cite[Corollary 9.8]{my} implies that \eqref{eq-square1} is satisfied in $K_0^{\oplus}(W^p_{\Lambda})$ 
(meaning the additive Grothendieck group). Thus, there has to be a suitable isomorphism which finishes the proof.
\end{proof}
\section{The uncategorified story}\label{sec-uncat}
\subsection{Multitableaux and \texorpdfstring{$\mathfrak{gl}_n$}{gln}-webs}\label{sec-tabwebs}

\subsubsection{Pictorial \texorpdfstring{$q$}{q}-skew Howe duality: an example}\label{subsub-example}
Before we start let us recall by an example how the translation of a string of $F^{(j)}_i$ acting on a highest weight vector $v_h$ to a $\mathfrak{gl}_n$-web $u$ works. The reader unfamiliar with this process, which is crucial for everything that follows, is encouraged to take a look at e.g. \cite{ckm}, \cite{mack1} or \cite{tub3} for a more detailed discussion.

\begin{ex}\label{exa-fstring}
Let $n=4$, $\ell=1$ and let $v_h=v_{(4)}$ be the highest weight vector for the partition $(4^1)$. 
Assume that we have the two stings
\[
\mathrm{qH}(u_1)=F_1F_2F_1\;\;\text{ and }\;\;\mathrm{qH}(u_2)=F_1F_2^{(2)}F_1^{(2)}.
\]
Then $\mathrm{qH}(u_{1,2})v_h$ will generate the following $\mathfrak{gl}_4$-webs $u_1$ and $u_2$ under $q$-skew Howe duality.
\begin{gather}\label{eq-not-same}
u_1=
\scalebox{.7}{$\xy
(0,-1.5)*{\includegraphics[scale=0.65]{figs/exgrowth/theta-exa-a.eps}};
(-12,-11)*{F_1};
(-12,13)*{F_1};
(12,1)*{F_2};
(-22,-19)*{4};
(-22,-7.3)*{3};
(-22,5.7)*{3};
(-22,16)*{2};
(2,-19)*{0};
(2,-7.3)*{1};
(2,5.7)*{0};
(2,16)*{1};
(26,-19)*{0};
(26,-7.3)*{0};
(26,5.7)*{1};
(26,16)*{1};
\endxy$},
\;\;\hspace*{0.5cm}\;\;u_2=\scalebox{.7}{$\xy
(0,-1.5)*{\includegraphics[scale=0.65]{figs/exgrowth/theta-exa-a.eps}};
(-11.5,-9.75)*{F^{(2)}_1};
(-12,13)*{F_1};
(12.5,2.5)*{F^{(2)}_2};
(-22,-19)*{4};
(-22,-7.3)*{2};
(-22,5.7)*{2};
(-22,16)*{1};
(2,-19)*{0};
(2,-7.3)*{2};
(2,5.7)*{0};
(2,16)*{1};
(26,-19)*{0};
(26,-7.3)*{0};
(26,5.7)*{2};
(26,16)*{2};
\endxy$}.
\end{gather}
Recall hereby that one can read off the corresponding $\mathfrak{gl}_m$-weight $\vec{k}$ for a fixed (horizontal) level by taking the numbers in order from left to right as $k_j$.
Note that the $\mathfrak{gl}_4$-webs in \eqref{eq-not-same} are different labels.
\end{ex}

\subsubsection{The extended growth algorithm}\label{subsub-exgrowth}

Denote by $W_n(\vec{k},\vec{S})$ the set of all possible $\mathfrak{gl}_n$-webs $u$ that can be obtained by a string of divided powers of $F$ acting on a highest weight vector $v_h=v_{(n^{\ell})}$ (without taking any $\mathfrak{gl}_n$-web relations in account at the moment) together with a flow $f$ on $u$ with boundary datum $\vec{S}$.

We start now by defining a map $\iota\colon W_n(\vec{k},\vec{S})\to\mathrm{Std}(\vec{\lambda})$.

\begin{defn}\label{defn-webtotab}(\textbf{Flows to fillings})
Given a fixed pair $(\vec{k},\vec{S})$ and a $\mathfrak{gl}_n$-web $u_f\in W_n(\vec{k},\vec{S})$ and a string that generates $u$, i.e. $\mathrm{qH}(u)=F^{(j_{m^{\prime}})}_{i_{m^{\prime}}}\dots F^{(j_1)}_{i_1}$.

We associate to it inductively a standard $n$-multitableaux $\iota(u_f)\in\mathrm{Std}(\vec{\lambda})$ as follows.
\begin{enumerate}
\item At the initial stage set $\vec{T}_0=(\emptyset,\dots,\emptyset)$.
\item At the $k$-th step use $F^{(j_{k})}_{i_{k}}$ and the local flow on the corresponding ladder to determine the operation performed on $\vec{T}_{k-1}$. We give the rule together with the operation
$\mathbf{k}\colon\vec{T}_{k-1}\mapsto\vec{T}_{k}$
below.
\item Repeat until $k=m^{\prime}$.
\item Then set $\iota(u_f)=\vec{T}_{m^{\prime}}$.
\end{enumerate}

Assume that the ladder that corresponds to the $k$-th move $F^{(j_{k})}_{i_{k}}$ with local flow is
\[
\scalebox{.7}{$F^{(j_{k})}_{i_{k}}\colon
\xy(0,0)*{\includegraphics[scale=.75]{figs/slnwebs/laddera.eps}};(-11.5,-4)*{\scriptstyle a};(16.5,-4)*{\scriptstyle b};(3,2)*{\scriptstyle j_k};(-9.5,4)*{\scriptstyle a-j_k};(18.5,4)*{\scriptstyle b+j_k}\endxy$},\quad
\scalebox{.7}{$\text{a flow on $F^{(j_{k})}_{i_{k}}$}\colon\xy(0,0)*{\includegraphics[scale=.75]{figs/slnwebs/laddera.eps}};(-11,-4)*{\scriptstyle S_1};(17,-4)*{\scriptstyle S_2};(2,2)*{\scriptstyle T};(-9,4)*{\scriptstyle S_1-T};(19,4)*{\scriptstyle S_2\cup T}\endxy$},
\]
for suitable subsets $S_1,S_2,T\subset\{n,\dots,1\}$. The set $T$ will be, by our conventions, of the form $T=\{t_{j_k},\dots,t_{1}\}$ for $t_{1}<\dots<t_{j_k}$. Then the operation $\mathbf{k}\colon\vec{T}_{k-1}\mapsto\vec{T}_{k}$ should add a node of residue $i_k$ and filling $k$ to the $t_{k^{\prime}}$-th part of $\vec{T}_k$ for all $t_{k^{\prime}}$.
\end{defn}

Let us give an example before we show the non-trivial fact that the algorithm is well-defined.
\begin{ex}\label{ex-flowtotab}
Given $n=5$, $v_h=v_{(5^2)}$ and $\mathrm{qH}(u)=F_1F_2F_3^{(2)}F_2^{(2)}$, we obtain a $\mathfrak{gl}_5$-web $u$ using $q$-skew Howe duality 
and we choose a flow $\vec{S}=(\{5,4,2,1\},\{5,3,2\},\{1\},\{4,3\})$ for it:
\[
\scalebox{.7}{$\xy
(0,0)*{\includegraphics[scale=.75]{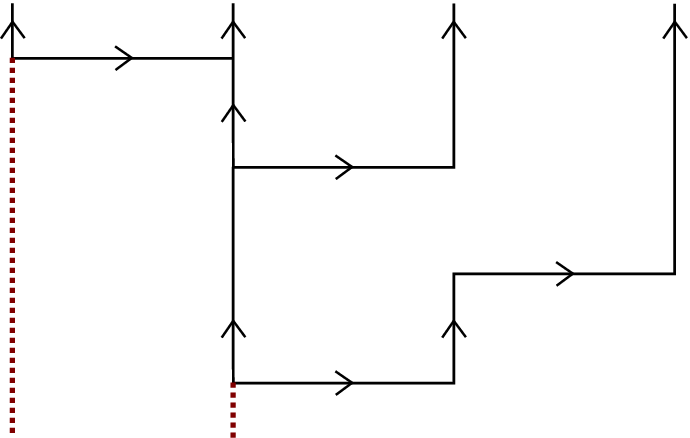}};
(0.5,-16.5)*{F_2^{(2)}};
(29,-3)*{F_3^{(2)}};
(0.5,9.75)*{F_2};
(-27.5,23.5)*{F_1};
(-40,-27.5)*{5};
(-12,-27.5)*{5};
(16,-27.5)*{0};
(44,-27.5)*{0};
(-40,-13.75)*{5};
(-12,-13.75)*{3};
(16,-13.75)*{2};
(44,-13.75)*{0};
(-40,0)*{5};
(-12,0)*{3};
(16,0)*{0};
(44,0)*{2};
(-40,13.75)*{5};
(-12,13.75)*{2};
(16,13.75)*{1};
(44,13.75)*{2};
(-40,27.5)*{4};
(-12,27.5)*{3};
(16,27.5)*{1};
(44,27.5)*{2};
\endxy$},
\quad
\scalebox{.7}{$\xy
(0,0)*{\includegraphics[scale=.75]{figs/exgrowth/webflow-exa-a.eps}};
(-0.5,-17.5)*{\scriptstyle \{4,3\}};
(-0.5,9.75)*{\scriptstyle \{1\}};
(-28.5,23.5)*{\scriptstyle \{3\}};
(-42,29)*{\scriptstyle \{5,4,2,1\}};
(-14,29)*{\scriptstyle \{5,3,2\}};
(-9.5,13.5)*{\scriptstyle \{5,2\}};
(-8.5,-14)*{\scriptstyle \{5,2,1\}};
(14,29)*{\scriptstyle \{1\}};
(27.5,-3.75)*{\scriptstyle \{4,3\}};
(42,29)*{\scriptstyle \{4,3\}};
\endxy$}.
\]
The algorithm performs five steps, i.e.
\begin{align*}
\vec{T}_0=\left(\;\emptyset\;,\;\emptyset\;,\;\emptyset\;,\;\emptyset\;,\;\emptyset\;\right) &\mapsto \vec{T}_1=\left(\;\emptyset\;,\;\xy (0,0)*{\begin{Young}1\cr \end{Young}}\endxy\;,\;\xy (0,0)*{\begin{Young}1\cr \end{Young}}\endxy\;,\;\emptyset\;,\;\emptyset\;\right)\\
&\mapsto \vec{T}_2=\left(\;\emptyset\;,\;\xy (0,0)*{\begin{Young}1& 2\cr \end{Young}}\endxy\;,\;\xy (0,0)*{\begin{Young}1& 2\cr \end{Young}}\endxy\;,\;\emptyset\;,\;\emptyset\;\right)\\
&\mapsto \vec{T}_3=\left(\;\emptyset\;,\;\xy (0,0)*{\begin{Young}1& 2\cr \end{Young}}\endxy\;,\;\xy (0,0)*{\begin{Young}1& 2\cr \end{Young}}\endxy\;,\;\emptyset\;,\;\xy (0,0)*{\begin{Young}3\cr \end{Young}}\endxy\;\right)\\
&\mapsto \vec{T}_4=\left(\;\emptyset\;,\;\xy (0,0)*{\begin{Young}1& 2\cr \end{Young}}\endxy\;,\;\xy (0,0)*{\begin{Young}1& 2\cr 4\cr \end{Young}}\endxy\;,\;\emptyset\;,\;\xy (0,0)*{\begin{Young}3\cr \end{Young}}\endxy\;\right)=\iota(u_f).
\end{align*}
\end{ex}

\begin{lem}\label{lem-allworks}
The algorithm of Definition \ref{defn-webtotab} is well-defined. Moreover, we have
\[
\iota(u_f)=\iota(v_{f^{\prime}})\Leftrightarrow u=v\text{ and }f=f^{\prime},
\]
where the equality of $\mathfrak{gl}_n$-webs and flows is not taking any $\mathfrak{gl}_n$-web relations (including isotopies) into account.
\end{lem}

\begin{proof}
We use induction on the total number $\ell(\mathrm{qH}(u))$ of $F_i$ of the string of $F_i^{(j)}$ that generate the $\mathfrak{gl}_n$-web $u$. The induction step is to remove the last, i.e. leftmost, factor $F_i^{(j)}$, to create a smaller $\mathfrak{gl}_n$-web $u^{<}$ for which the statement is already known by the hypothesis. To summarize, assume that $\ell(\mathrm{qH}(u))=r$. Then we let
\[
u=F_{i_r}^{(j_r)}\prod_{k=1}^{r-1}F_{i_k}^{(j_k)}v_h\;\;\text{ and }\;\;u^{<}=\prod_{k=1}^{r-1}F_{i_k}^{(j_k)}v_h,
\]
and check what the last step could do. 

The induction start includes all cases of total length $\ell_{\mathrm{t}}(\mathrm{qH}(u))=\sum j_k\leq n$, since the divided power can go up to $n$.
That everything is well-defined follows for these cases, because all cases with total length $\leq n$ are just the first ladder steps given by $F_{i_1}^{(j_1)}$ which can not run into ambiguities, since we fill the empty $n$-multitableaux with at most $n$ nodes and all of the correct residue due to our residue normalization. Moreover, the possible addable nodes of residue $i_2$ are given by $S^1_{i_2}-S^1_{i_2+1}$, where $\vec{S}^1$ is the flow at the top of the first ladder move. 

Otherwise, assume that it is well-defined for $u^{<}$ and the possible addable nodes of residue $i_r$ are given by $\vec{S}^<$. Observe now that the given flow on the middle edge of the ladder for $F_{i_r}^{(j_r)}$ is determined by the smaller one $f^<$ at the boundary of $u^{<}$. Moreover, by construction, it has to be disjoint to the two incoming flows at the boundary. That is, $T\subset S^<_{i_r}-S^<_{i_r+1}$.
This shows that the last step can perform a legal move and hence, the algorithm is well-defined and gives a standard $n$-multitableaux, the possible addable nodes will now be determined by $\vec{S}$. 

That the algorithm gives different results for different $\mathfrak{gl}_n$-webs $u,v$ or different flows $f,f^{\prime}$ on one $\mathfrak{gl}_n$-web $u$ follows in the same vein, i.e. it is clear by construction that the first step will give a different result for different inputs. By induction, we then only have to ensure that the first place where either $u$ and $v$ are different or where $f$ and $f^{\prime}$ are different give different results. The first follows directly, since already the boundary vectors $\vec{k}_u$ and $\vec{k}_v$ will be different for $u$ and $v$ and hence, the whole shape will be different. The second follows because different flows with the same boundary datum have to be different on the middle edge of the last ladder. But in this case the rules tell us to place the new nodes in different parts of the $n$-multitableaux.
\end{proof}

The whole procedure also works the other way around: given a fixed $n$-multitableaux $\vec{T}\in\mathrm{Std}(\vec{\lambda})$, one can generate a $\mathfrak{gl}_n$-web $u_f\in W_n(\vec{k},\vec{S})$ together with a flow on it as we describe now. 

\begin{defn}\label{defn-exgrowth}(\textbf{Extended $\mathfrak{gl}_n$-growth algorithm})
The \textit{extended $\mathfrak{gl}_n$-growth algorithm}
\[
\mathrm{g}\colon\mathrm{Std}(\vec{\lambda})\to W_n(\vec{k},\vec{S})
\]
is given inductively as follows.

Let $\vec{T}\in\mathrm{Std}(\vec{\lambda})$ be a standard $n$-multitableaux with nodes labeled from $1,\dots,s$. We assign to it a $\mathfrak{gl}_n$-web $u$ given by a sequence of divided powers of $F_{i_k}^{(j_k)}$ (under $q$-skew Howe duality) by
\[
u=\prod_{k=1}^{s} F_{i_k}^{(j_k)}v_{(n^{\ell})},
\]
where $i_k$ is the residue of the node(s) with entry $k$ and $j_k$ is their multiplicity.

Denote for $k^{\prime}=0,\dots,s$ the $\mathfrak{gl}_n$-web $u^{k^{\prime}}$ obtained by
\[
u^{k^{\prime}}=\prod_{k=1}^{k^{\prime}} F_{i_k}^{(j_k)}v_{(n^{\ell})}.
\]
The flow $f$ on $u$ is given inductively starting with a flow $f_0$ on the $\mathfrak{gl}_n$-web $u^{0}$ that has only some leashes for entries with label $n$ given by the full set $\{n,\dots,1\}$ on all leashes and nothing else.

Assume $0<k^{\prime}$ and that the flow $f_{k^{\prime}-1}$ on $u^{k^{\prime}-1}$ is given. Then extend the flow to $f_{k^{\prime}}$ on $u^{k^{\prime}}$ by extending the flow $f_{k^{\prime}-1}$ on $u^{k^{\prime}-1}$ such that the horizontal line in the ladder corresponding to the last move given by $F_{i_{k^{\prime}}}^{(j_{k^{\prime}})}$ is labeled with the set
\[
S=
\left\{
\epsilon_n,\dots,\epsilon_1\right\}-\{0\},\epsilon_{\tilde m}=
\begin{cases}
\tilde m, &\text{if the number }k^{\prime}\text{ appears in the }n\text{-multitableaux }T_{\tilde m},
\\0,&\text{else}.
\end{cases}
\]
(If well-defined this determines the labels on the two upper edges of the ladder.) Finally set $u_f=u^{s}_{f_{s}}$.
\end{defn}

It is again not a priori clear that this algorithm is well-defined. But before proving this we give an example.

\begin{ex}\label{ex-tabtoflow}
Given the $5$-multitableaux
\[
\vec{T}=(T_5,T_4,T_3,T_2,T_1)=\left(\;\emptyset\;,\;\xy (0,0)*{\begin{Young}1& 2\cr \end{Young}}\endxy\;,\;\xy (0,0)*{\begin{Young}1& 2\cr 4\cr \end{Young}}\endxy\;,\;\emptyset\;,\;\xy (0,0)*{\begin{Young}3\cr \end{Young}}\endxy\;\right),
\]
which is $\vec{T}_4$ from Example \ref{ex-flowtotab}, one gets exactly the same result as therein.
\end{ex}

\begin{lem}\label{lem-allworks2}
The algorithm of Definition \ref{defn-exgrowth} is well-defined. Moreover, we have
\[
\mathrm{forget}(\mathrm{g}(\vec{T}))=\mathrm{forget}(\mathrm{g}(\vec{T}^{\prime}))\Leftrightarrow r(\vec{T})=r(\vec{T}^{\prime}),
\]
where $\mathrm{forget}(\cdot)$ forgets the flow line and
\[
\mathrm{g}(\vec{T})=\mathrm{g}(\vec{T}^{\prime})\Leftrightarrow \vec{T}=\vec{T}^{\prime},
\]
where the equalities are again not taking any $\mathfrak{gl}_n$-web relations (including isotopies) into account.
\end{lem}

\begin{proof}
The proof that the algorithm is well-defined and gives always different results for different $n$-multitableaux follows the same idea as in the proof of Lemma \ref{lem-allworks}, i.e. induction on the length $s$ of the $n$-multitableaux. We obtain $\vec{T}^<$ from $\vec{T}$ by removing all nodes with the biggest entry such that the biggest entry of $\vec{T}^<$ is $s-1$.

For both claims it is easy to verify all small cases, i.e. all cases with length $s=1$, by hand. Our residue convention ensures that the corresponding divided power does not kill the highest weight vector. Moreover, a full $n$-multitableaux corresponds to a leash shift with a full flow, that is
\[
\left(\;\xy (0,0)*{\begin{Young}1\cr \end{Young}}\endxy\;,\dots,\;\xy (0,0)*{\begin{Young}1\cr\end{Young}}\endxy\;\right)
\mapsto
\scalebox{.7}{$\xy
(0,0)*{\includegraphics[scale=.75]{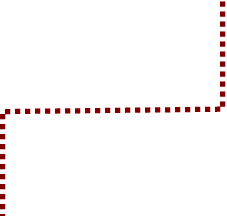}};
(0,3)*{\{n,\dots,1\}};
\endxy$}.
\]
To see that the algorithm is well-defined note that we get a legal step from $\vec{T}^<$ to $\vec{T}$, i.e. a flow, because if we add a ladder at the $i$-th position, then the values of $S_i$ and $S_{i+1}$ are determined by the same observation as above in the proof of Lemma \ref{lem-allworks}. Moreover, to see that the string of $F_i^{(j)}$ does not kill the highest weight vector in the last step from $\vec{T}^<$ to $\vec{T}$, we note that the action of $F_{i_s}^{(j_s)}$ is determined by $\vec{k}^<$. And this is encoded in $\vec{T}^<$ by the residue sequence and multiplicities of the entries. If $F_{i_s}^{(j_s)}$ would kill the vector, then the configuration could not have been legal in the first place. 

To see that $n$-multitableaux with a different residue sequence already give different $\mathfrak{gl}_n$-webs is because of the definition of the string of $F_i^{(j)}$. That different fillings give different flows follows, because the position of the nodes with the same label that are at different positions will give a different flow on the middle edge of the corresponding ladder.

On the other hand, that equal $n$-multitableaux give the same $\mathfrak{gl}_n$-webs with the same flow follows immediately and $r(\vec{T})=r(\vec{T}^{\prime})$ forces the underlying $\mathfrak{gl}_n$-webs to be the same follows because we obtain the string of $F_i^{(j)}$ that generates the $\mathfrak{gl}_n$-webs only from the residue sequence. 
\end{proof}

Because the two algorithms given in Definitions \ref{defn-webtotab} and \ref{defn-exgrowth} are inverse procedures we note the following proposition.

\begin{prop}\label{prop-extgrowth}
We have
\[
\iota\circ\mathrm{g}=\mathrm{id}_{\mathrm{Std}(\vec{\lambda})}\;\;\text{ and }\;\;\mathrm{g}\circ\iota=\mathrm{id}_{W_n(\vec{k},\vec{S})},
\]
where we again not taking any $\mathfrak{gl}_n$-web relations (including isotopies) into account.
\end{prop}

\begin{proof}
We use the two Lemmas \ref{lem-allworks} and \ref{lem-allworks2}, i.e. scrutiny of the inductive steps given in Definitions \ref{defn-webtotab} and \ref{defn-exgrowth} shows that they reverse each other.
\end{proof}

We state now in an important lemma how one can write any $\mathfrak{gl}_n$-web $u\in W_n(\vec{k})$ explicitly as a string of $F_i^{(j)}$.

\begin{lem}\label{lem-webasF}
Any $u\in W_n(\vec{k})\subset W_n(\Lambda)$, for all $\vec{k}$, can be written, using $q$-skew Howe duality, as
\[
u=\prod_{k=1}^{s}F_{i_k}^{(j_k)}v_{(n^{\ell})}
\]
for some $s\in\bN$. Moreover, this can be done in such a way that none of the $F_i^{(j)}$ connects two nested and not connected components into a single connected component.
\end{lem}

\begin{proof}
We prove the first statement by induction on the number of vertices of the $\mathfrak{gl}_n$-webs $u$. We use $1$ here as the position index without loss of generality.

If $u$ has no vertices at all, then we see that we have to check exactly five cases, i.e. cup and cap
\[
\scalebox{.7}{$\xy
(0,0)*{\includegraphics[scale=.75]{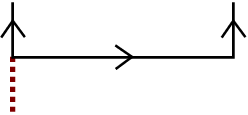}};
(-14,-9)*{n};
(14,-9)*{0};
(-14,9)*{n-a};
(14,9)*{a};
(1,4)*{F_1^{(a)}};
\endxy$},\;\;\hspace*{0.5cm}\;\;
\scalebox{.7}{$\xy
(0,0)*{\includegraphics[scale=.75]{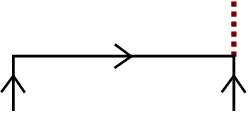}};
(14,9)*{n};
(-14,9)*{0};
(14,-9)*{n-a};
(-14,-9)*{a};
(1,4)*{F_1^{(a)}};
\endxy$},
\]
and three shifts, i.e. the left, right and the empty shift
\[
\scalebox{.7}{$\xy
(0,0)*{\includegraphics[scale=.75]{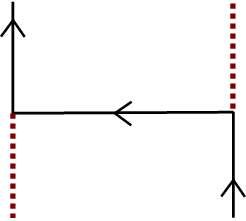}};
(-14,-15)*{n};
(14,-15)*{n-a};
(-14,15)*{n-a};
(14,15)*{n};
(1,4)*{F_1^{(a)}};
\endxy$},\;\;\hspace*{0.5cm}\;\;
\scalebox{.7}{$\xy
(0,0)*{\includegraphics[scale=.75]{figs/slnwebs/lineup.eps}};
(14,15)*{a};
(-14,15)*{0};
(14,-15)*{0};
(-14,-15)*{a};
(1,4)*{F_1^{(a)}};
\endxy$},\;\;\hspace*{0.5cm}\;\;
\scalebox{.7}{$\xy
(0,0)*{\includegraphics[scale=.75]{figs/exgrowth/leashright.eps}};
(14,15)*{n};
(-14,15)*{0};
(14,-15)*{0};
(-14,-15)*{n};
(1,4)*{F_1^{(n)}};
\endxy$}.
\]
Here we can use any $0\leq a\leq n$. This shows that any $\mathfrak{gl}_n$-web with no vertices can be obtained from $v_{(n^{\ell})}$ by an explicit sequence of $F_i^{(j)}$ starting from a suitable weight at the bottom which can be chosen as a highest weight in the closed cases.

Now assume that $u$ has at least one vertex. Take the leftmost of the vertices of $u$ with two outgoing edges (including leashes) that connects to the cut line. Cut it away by changing the cut line a little bit as illustrated below. The boundary data changes accordingly (we allow an arbitrary, finite number of zeros to the left).
\[
\scalebox{.7}{$\xy
(0,0)*{\includegraphics[scale=.75]{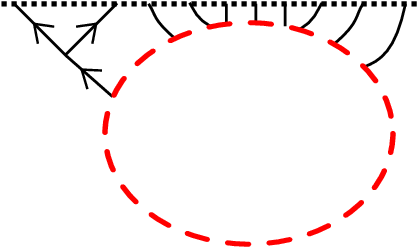}};
(5,-1)*{u^{\prime}};
\endxy$}\longmapsto
\scalebox{.7}{$\xy
(0,0)*{\includegraphics[scale=.75]{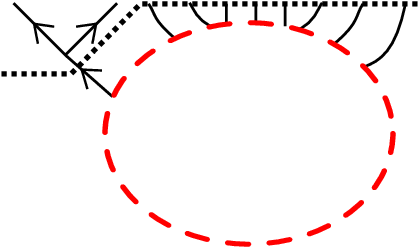}};
(5,-1)*{u^{\prime}};
\endxy$}.
\]
Since $u^{\prime}$ has fewer vertices than $u$, we can use induction and the observation that the last step can be realized as an $F_i^{(j)}$ depending on how we read the tripod, e.g. for suitable $0\leq a,b\leq n$
\[
\scalebox{.7}{$\xy
(0,0)*{\includegraphics[scale=.75]{figs/slnwebs/leash.eps}};
(-14,-9)*{n};
(-14,9)*{n-a};
(14,-9)*{b};
(14,9)*{a+b};
(1,4)*{F_1^{(a)}};
\endxy$}.
\]
Hence, $u$ can be realized as a string of suitable chosen $F_i^{(j)}$.

To see the second statement we note that we can freely use isotopies as illustrated below.
\[
\scalebox{.7}{$\xy
(0,0)*{\includegraphics[scale=.5]{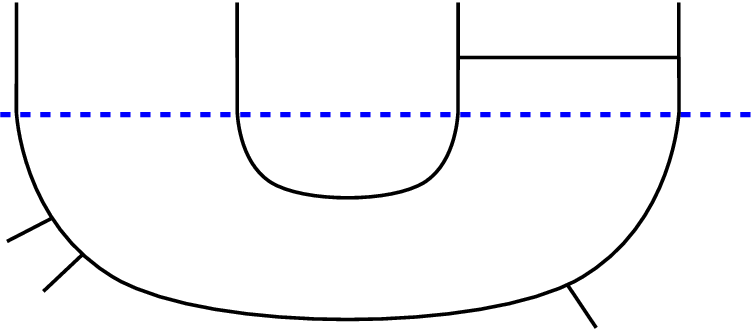}};
\endxy$}
\rightsquigarrow
\scalebox{.7}{$\xy
(0,0)*{\includegraphics[scale=.5]{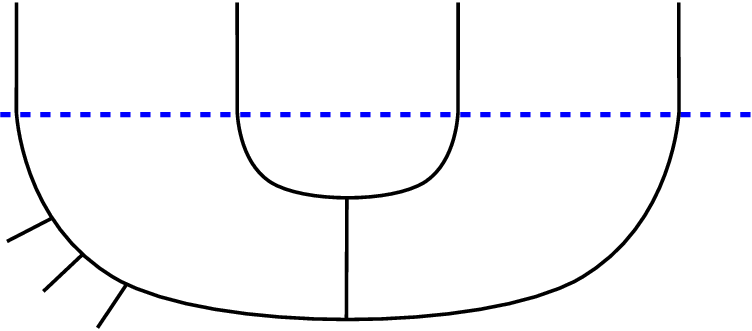}};
\endxy$}.
\]
That is, we can always avoid to connect nested parts by shifting the $F_i^{(j)}$-ladder around.
\end{proof}

\begin{ex}\label{ex-websasF}
For example a $\mathfrak{gl}_n$-web $u$ with a local dumbbell and $n>4$
\[
\scalebox{.7}{$\xy
(0,0)*{\includegraphics[scale=.75]{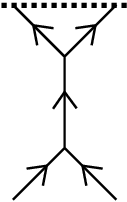}};
(-6.5,-14.25)*{2};
(6.5,-14.25)*{2};
(-6.5,14)*{2};
(6.5,14)*{2};
\endxy$}\longmapsto
\scalebox{.7}{$\xy
(0,0)*{\includegraphics[scale=.75]{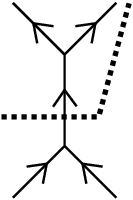}};
(-6.5,-14.25)*{2};
(6.5,-14.25)*{2};
(-6.5,0)*{0};
(3,0)*{4};
\endxy$}
\leftrightsquigarrow
\scalebox{.7}{$\xy
(0,0)*{\includegraphics[scale=.75]{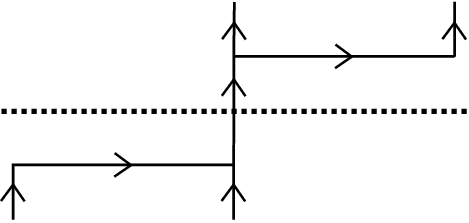}};
(-28,-16)*{2};
(0,-16)*{2};
(28,-16)*{0};
(-28,2)*{0};
(2.5,2)*{4};
(28,2)*{0};
(-28,16)*{0};
(0,16)*{2};
(28,16)*{2};
(28,16)*{2};
(-14,-10.5)*{F_{1}^{(2)}};
(14.5,11)*{F_{2}^{(2)}};
\endxy$}
.
\]
Thus, in the notation of Lemma \ref{lem-webasF}, the $\mathfrak{gl}_n$-web $u^{\prime}$ has a $F_1^{(2)}$ as a leftmost factor in its product of $F_i^{(j)}$. Hence, we have
\[
u^{\prime}=F_1^{(2)}\prod_{k} F_{i_k}^{(j_k)}v_h\rightsquigarrow u=F_2^{(2)}F_1^{(2)}\prod_{k} F_{i_k}^{(j_k)}v_h.
\]

As another example is that the $\mathfrak{gl}_4$-web $u$ from Example \ref{exa-flow} can be generated by
\[
u=F_7^{(2)}F_3F_1F_2F_1F_3F^{(2)}_4F^{(2)}_3F_4F_5F_4F_2F_1F_3^{(2)}F^{(4)}_2F^{(4)}_6F^{(4)}_5F^{(4)}_4F^{(4)}_3v_{(4^3)}.
\]
If we use this string to generate the $\mathfrak{gl}_4$-web $u$, then the flow $f$ from Example \ref{exa-flow} will be converted to the following $4$-multitableau.
\[
\iota(u_f)=\left(\;\xy (0,0)*{\begin{Young}1& 2 & 3 & 4 & 19\cr 5 & 6 & 9 & 10\cr 15 & 16 & 18\cr\end{Young}}\endxy\;,\;\xy (0,0)*{\begin{Young}1& 2 & 3 & 4\cr 5 & 6 & 11\cr 7 & 8 & 14\cr\end{Young}}\endxy\;,\;\xy (0,0)*{\begin{Young}1& 2 & 3 & 4 & 19\cr 5 & 16 & 17\cr\end{Young}}\endxy\;,\;\xy (0,0)*{\begin{Young}1& 2 & 3 & 4\cr 5 & 12 & 13\cr 17\cr\end{Young}}\endxy\;\right).
\]
\end{ex}

The Proposition \ref{prop-extgrowth} together with Lemma \ref{lem-webasF} imply that any ``reasonable'' basis of the $\mathfrak{gl}_n$-web space $W_n(\vec{k})$ is \textit{monomial}, i.e. given by a sequence of $F^{(j)}_i$ acting on a highest weight vector $v_h$. In fact, given a spanning set of $\mathfrak{gl}_n$-webs of $W_n(\vec{k})$, the hardest part is to show linear independence.

Some ``reasonable'' bases of $W_n(\vec{k})$ are the basis given by all $\mathfrak{sl}_2$-arc diagrams (here $n=2$), Kuperberg's basis of non-elliptic $\mathfrak{sl}_3$-webs (here $n=3$), intermediate crystal bases in the sense of Leclerc--Toffin \cite{leto} under $q$-skew Howe duality (see \cite{tub3} or \cite{mack1}) and Fontaine's basis \cite{fon}.

\begin{cor}\label{cor-extgrowth}
All of the bases of $W_n(\vec{k})$ mentioned above are monomial.\qed
\end{cor}

\subsubsection{Degree and the weight of flows}\label{subsub-exgrowthdegree}

We are going to show now that the result of Proposition \ref{prop-extgrowth} can be strengthened. To be more precise, both $W_n(\vec{k})$ and $\mathrm{Std}(\vec{\lambda})$ are graded. The first one by the weight of the flows and the second one by Brundan--Kleshchev--Wang's degree for multitableaux.

\begin{prop}\label{prop-extgrowthdeg}
Both maps
\[
\iota\colon W_n(\vec{k},\vec{S})\to\mathrm{Std}(\vec{\lambda})\;\;\text{ and }\;\;\mathrm{g}\colon\mathrm{Std}(\vec{\lambda})\to W_n(\vec{k},\vec{S})
\]
preserve the degree.
\end{prop}

\begin{proof}
First lets us take a look how to read off the weight for a ladder. Assume that the flow on the top of a ladder is given by $\vec{S}=(S_1,\dots,S_m)$, at the bottom by $\vec{S}^{<}=(S^<_1,\dots,S^{<}_m)$ and at its horizontal edge by $T$. Moreover, assume for simplicity that the ladder comes from an action of $F_1$, i.e. that it is a ladder at position $1$. Then, by our convention how to draw ladders, we have
\[
\scalebox{.7}{$\xy(0,0)*{\includegraphics[scale=.75]{figs/slnwebs/ladderc.eps}};(-2.5,-10.5)*{\scriptstyle S^<_1};(17,-10.5)*{\scriptstyle S^<_2};(2,-1)*{\scriptstyle T};(-9.5,10.5)*{\scriptstyle S_1};(10,10.5)*{\scriptstyle S_2}\endxy$}.
\]
The weight $\mathrm{wt}(u)$ is now given by $\ell(S_1,T)-\ell(T,S_2^<)$, that is, by counting how many pairs of the set $T\times S^<_2$ are strictly ordered and subtract the number of strictly ordered pairs of $S_1\times T$. Since $S_1=S^<_1\cup T$, this is the same as
\begin{equation}\label{eq-degshifta}
\mathrm{wt}(u)=\ell(S_1,T)-\ell(T,S_2^<)=\ell(S_1^<,T)-\ell(T,S_2^<)-\frac{1}{2}|T|(|T|-1).
\end{equation}
We are going to show that the map $\iota$ preserves the degree. The other direction follows in a similar vein, since both algorithm are inverses, and is omitted.

To proof that $\iota$ preserves the degree we can use a similar induction as in the Lemmas \ref{lem-allworks} and \ref{lem-allworks2} before. One easily verifies that the small cases, i.e. the empty shift and all possible flows on caps and cups, preserve the degree. The shift of the degree
\begin{equation}\label{eq-degshiftb}
a=\sum_{i=0}^{N^j-1}i
\end{equation}
from Definition \ref{defn-degpart1} is exactly the shift by $\frac{1}{2}|T|(|T|-1)$, because $|N^j|=|T|$. For example, if the first step is an empty shift, then $S^<_1=T=\{n,\dots,1\}$ and $S^<_2=\emptyset$ which gives the desired answer.

For a $\mathfrak{gl}_n$-web with a flow $u_f$ and $\iota(u_f)=\vec{T}$, we can assume that the degree is preserved for $u^<_{f^<}$. Hence, we only have to verify that the degree is still preserved in the last step of the algorithm. To see this we note that the three terms $\ell(S^<_1,T)$, $\ell(T,S^<_2)$ and $\frac{1}{2}|T|(|T|-1)$ from \eqref{eq-degshifta} are the three numbers from Definition \ref{defn-degpart1}, i.e.
\[
\ell(S_1,T)=|\mathsf{A}^{k\succ N}(\vec{T}^j)|,\ell(T,S_2^<)=|\mathsf{R}^{k\succ N}(\vec{T}^j)|\;\text{ and }\; \frac{1}{2}|T|(|T|-1)=a.
\]
The proof completes: both, $\mathrm{wt}$ and $\mathrm{deg}_{\mathrm{BKW}}$ are locally the same and are both defined inductively.
\end{proof}

\subsubsection{The evaluation algorithm}\label{subsub-evaluation}
We conclude this part by giving an algorithm to evaluate closed $\mathfrak{gl}_n$-webs $w$.

\begin{defn}\label{defn-evaluation}(\textbf{Evaluation of $\mathfrak{gl}_n$-webs}) Given a $\mathfrak{gl}_n$-web $u\in W_n(\vec{k})\cong\mathrm{Inv}_{\dot{\Uu}_q(\mathfrak{gl}_n)}(\Lambda^{\vec{k}}\bC^n)$ together with a sequence 
of $F_i^{(j)}$ generating it, i.e.
\[
u=\prod_{k=1}^{s}F_{i_k}^{(j_k)}v_{(n^{\ell})},
\]
we assign to it a set $\mathrm{ev}_u=\{\vec{T}_1,\dots,\vec{T}_a\}$ of standard $n$-multitableaux $\vec{T}_{b}$ inductively as follows.
\begin{itemize}
\item[(1)] Set $\mathrm{ev}_u^0=\{\emptyset\}$, where $\emptyset$ denotes the empty $n$-multitableaux.
\item[(2)] In each step $1\leq k\leq s$ add certain (explained below) new $n$-multitableaux $\vec{T}^k$ to $\mathrm{ev}_u^{k-1}$ and obtain a new set $\mathrm{ev}_u^{k}$.
\item[(3)] After each step $1\leq k\leq s$ remove all old $n$-multitableaux $\vec{T}^{k-1}$ from $\mathrm{ev}_u^{k}$.
\item[(4)] Repeat (2)+(3) until $k=s$. Set $\mathrm{ev}_u=\mathrm{ev}_u^{s}$.
\end{itemize}
The way to decide which $n$-multitableaux $\vec{T}^k$ should be added in the $k$-th step is to take all possible ways to add $j_k$ nodes with residue $i_k$ labeled $k$ to a $\vec{T}^{k-1}$ such that the result is again a standard $n$-multitableaux. Do this for all possible $\vec{T}^{k-1}$.
The \textit{evaluation} of a closed $\mathfrak{gl}_n$-web $w\in\mathrm{End}_{\dot{\Uu}_q(\mathfrak{gl}_n)}((n^{\ell}))$ is then defined to be
\[
\mathrm{ev}(w)=\sum_{\vec{T}\in\mathrm{ev}_w}q^{\mathrm{deg}_{\mathrm{BKW}}(\vec{T})}\in\bN[q,q^{-1}].
\]
\end{defn}

\begin{ex}\label{ex-evaluation}
Consider two circles as a $\mathfrak{gl}_2$-web $w$ in the $\mathfrak{gl}_2$ case. We know in this case that the evaluation should give $[2]^2=q^2+2+q^{-2}\in\bN[q,q^{-1}]$. We can write $w$ as a string of $F_i^{(j)}$ as follows.
\[
\scalebox{.7}{$\xy
(0,0)*{\includegraphics[scale=.75]{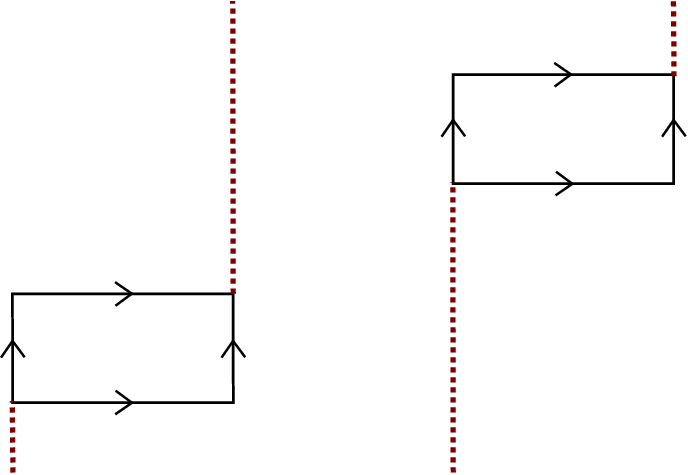}};
(-40,-28)*{2};
(-12.5,-28)*{0};
(16.5,-28)*{2};
(44,-28)*{0};
(-27.5,-18)*{F_1};
(-40,-14)*{1};
(-12.5,-14)*{1};
(16.5,-14)*{2};
(44,-14)*{0};
(-27.5,-4.25)*{F_1};
(-40,0)*{0};
(-12.5,0)*{2};
(16.5,0)*{2};
(44,0)*{0};
(27.5,10)*{F_3};
(-40,14)*{0};
(-12.5,14)*{2};
(16.5,14)*{1};
(44,14)*{1};
(27.5,24)*{F_3};
(-40,28)*{0};
(-12.5,28)*{2};
(16.5,28)*{0};
(44,28)*{2};
\endxy$}.
\]
Hence, because we also have an empty shift at the bottom, we get $F_3F_3F_1F_1F_2^{(2)}$ for $w$. Recall that we have a shift of residues given by the number of twos at the bottom. We get the four $2$-multitableaux
\[
\vec{T}_1=\left(\;\xy (0,0)*{\begin{Young}1& 2\cr 4\cr\end{Young}}\endxy\;,\;\xy (0,0)*{\begin{Young}1& 3\cr 5\cr \end{Young}}\endxy\;\right)\quad\text{and}\quad
\vec{T}_2=\left(\;\xy (0,0)*{\begin{Young}1& 2\cr 5\cr\end{Young}}\endxy\;,\;\xy (0,0)*{\begin{Young}1& 3\cr 4\cr \end{Young}}\endxy\;\right),
\]
\[
\vec{T}_3=\left(\;\xy (0,0)*{\begin{Young}1& 3\cr 4\cr\end{Young}}\endxy\;,\;\xy (0,0)*{\begin{Young}1& 2\cr 5\cr \end{Young}}\endxy\;\right)\quad\text{and}\quad
\vec{T}_4=\left(\;\xy (0,0)*{\begin{Young}1& 3\cr 5\cr\end{Young}}\endxy\;,\;\xy (0,0)*{\begin{Young}1& 2\cr 4\cr \end{Young}}\endxy\;\right),
\]
because in the first step (the one for $F_2^{(2)}$) we have exactly one option where we can add two nodes with residue $2$ to the empty $2$-multitableaux. Then we have two choices to add nodes for the two $F_1$ and the same happens for the two $F_2$. The reader should check that the degrees for the $2$-multitableaux from $\vec{T}_1$ to $\vec{T}_4$ are $2,0,0,-2$. These are exactly the powers of the $q$ in $[2]^2$.

Note that the way to obtain $w$ as a string of $F_i^{(j)}$ is far from being unique. For example
\[
\scalebox{.7}{$\xy
(0,0)*{\includegraphics[scale=.75]{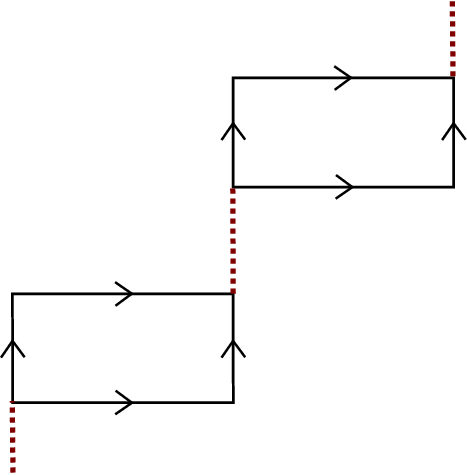}};
(-26,-28)*{2};
(2,-28)*{0};
(30,-28)*{0};
(-14,-18)*{F_1};
(-26,-14)*{1};
(2,-14)*{1};
(30,-14)*{0};
(-14,-4)*{F_1};
(-26,0)*{0};
(2,0)*{2};
(30,0)*{0};
(14,9.5)*{F_2};
(-26,14)*{0};
(2,14)*{1};
(30,14)*{1};
(14,23)*{F_2};
(-26,28)*{0};
(2,28)*{0};
(30,28)*{2};
\endxy$}.
\]
This time we get
\[
\vec{T}_{1,2}=\left(\;\xy (0,0)*{\begin{Young}1& -\cr\end{Young}}\endxy\;,\;\xy (0,0)*{\begin{Young}2& -\cr\end{Young}}\endxy\;\right)\quad\text{or}\quad
\vec{T}_{3,4}=\left(\;\xy (0,0)*{\begin{Young}2& -\cr\end{Young}}\endxy\;,\;\xy (0,0)*{\begin{Young}1& -\cr \end{Young}}\endxy\;\right),
\]
where the $-$ should be filled with either $3$ in the first and $4$ in the second or vice versa.

A crucial difference (also from the viewpoint of the $\mathfrak{gl}_n$-link polynomials) is to change the sequence for the two circles $w=F_2F_2F_1F_1$ to $w^{\prime}=F_2F_1F_2F_1$. This gives the following $\mathfrak{gl}_2$-web.
\[
\scalebox{.7}{$\xy
(0,0)*{\includegraphics[scale=.75]{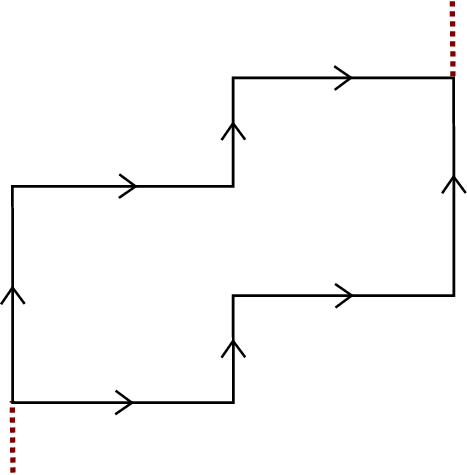}};
(-26,-28)*{2};
(2,-28)*{0};
(30,-28)*{0};
(-14,-18)*{F_1};
(-26,-14)*{1};
(2,-14)*{1};
(30,-14)*{0};
(14,-4.5)*{F_2};
(-26,0)*{1};
(2,0)*{0};
(30,0)*{1};
(-14,9.5)*{F_1};
(-26,14)*{0};
(2,14)*{1};
(30,14)*{1};
(14,23)*{F_2};
(-26,28)*{0};
(2,28)*{0};
(30,28)*{2};
\endxy$}.
\]
The algorithm gives now only the two $2$-multitableaux
\[
\vec{T}_{1}=\left(\;\xy (0,0)*{\begin{Young}1& 2\cr\end{Young}}\endxy\;,\;\xy (0,0)*{\begin{Young}3& 4\cr\end{Young}}\endxy\;\right)\text{ or }
\vec{T}_{2}=\left(\;\xy (0,0)*{\begin{Young}3& 4\cr\end{Young}}\endxy\;,\;\xy (0,0)*{\begin{Young}1& 2\cr \end{Young}}\endxy\;\right),
\]
because the nodes with labels $2$ and $3$ switch their residue. The two $2$-multitableaux are of degree $1$ and $-1$ giving the evaluation $q+q^{-1}=[2]\in\bN[q,q^{-1}]$ as expected.
\end{ex}

\begin{thm}\label{thm-evaluation}
The evaluation of $\mathfrak{gl}_n$-webs is independent of the choices involved. Moreover, for any two $\mathfrak{gl}_n$-webs $u,v\in W_n(\vec{k})$ the evaluation in Definitions \ref{defn-kupform} and \ref{defn-evaluation} satisfy ($w=v^*u$)
\[
\mathrm{ev}(v^*u)=\sum_{\vec{T}\in\mathrm{ev}_w}q^{\mathrm{deg}_{\mathrm{BKW}}(\vec{T})}=q^{-d(\vec{k})}\langle u, v\rangle_{\mathrm{Kup}}=q^{-d(\vec{k})}\langle u,v\rangle_{\mathrm{Shap}},
\]
i.e. the evaluation using $n$-multitableaux gives (up to a shift by $-d(\vec{k})$) the Kuperberg bracket 
$\langle\cdot,\cdot\rangle_{\mathrm{Kup}}$.
\end{thm}

\begin{proof}
To prove that the algorithm is well-defined we observe that the procedure is deterministic, i.e. the algorithm itself can not run into ambiguities.

To see that it is independent of the involved choices note that the algorithm is just a way to find possible flow lines on $u$ under the interpretation given in Definition \ref{defn-webtotab}. That it is independent of the choices, i.e. how to write a certain local move, and isotopies follows now from the Lemmas \ref{lem-allworks} and \ref{lem-allworks2}. To be more precise, if we start with two different $n$-multitableaux that correspond to the same flow on a fixed $\mathfrak{gl}_n$-web $u$ (including isotopies). Then we can convert both to the $\mathfrak{gl}_n$-web framework and we can use the isotopy invariance to see that they agree.

That it is also independent of the highest weight vector follows from Theorem \ref{thm-kkele} and the observation that we have normalized the degree in such a way that all empty shifts are of degree zero. Hence, since tensor products of the trivial representation have an, up to a scalar, unique basis vector, Theorem \ref{thm-kkele} and our normalization imply that the resulting evaluation $\mathrm{ev}(u)$ is a fixed element in $\bN[q,q^{-1}]$.

The third equality is a consequence of Proposition \ref{prop-kupshap}. Hence, it only remains to show the second equality. This equality can be proven using Theorem \ref{thm-kkele} again.

That is, one needs to show that the coefficients in the relations given in Definition \ref{def-spid} are given by the weight of the local flows. Furthermore, one has to take the change of $\vec{k}$ into account to see how the shift $d(\vec{k})$ changes stepwise. This is a straightforward, but exhausting, calculation and is omitted (although, because of the Lemmas \ref{lem-allworks} and \ref{lem-allworks2}, we do not have to check the isotopy relations). For example, if $n=3$, then a closed circle (i.e. \eqref{eq-digon1} with $a+b=3$) has three flows of degree $2,0,-2$ giving $q^2+1+q^{-2}=[3]$.
\end{proof}

\subsubsection{An application: dual canonical bases and \texorpdfstring{$\mathfrak{gl}_n$}{gln}-webs}\label{subsub-dualcan}

As an application of Theorem \ref{thm-evaluation} we will conclude this section by giving an explicit and algorithmic if-and-only-if-condition for a $\mathfrak{gl}_n$-web $u$ to be dual canonical. Dual canonical for $\mathfrak{gl}_n$-webs means canonical on the $q$-skew Howe dual side, see e.g. \cite[ Corollary 4.21]{mack1}. Thus, in our notation, having positive exponent properties.
The reader interested in a more detailed discussion about these bases can check for example \cite{bk1}, \cite{bs3} or \cite{lu} and a 
discussion related to $\mathfrak{gl}_n$-webs can be found in \cite{mack1}.

Recall that there is a unique $q$-antilinear \textit{bar involution} $\phi$ on $W_n(\Lambda)$ determined by $\phi(v_{\Lambda})=v_{\Lambda}$ and $\phi(Xv_{\Lambda})=\overline{X}v_{\Lambda}$ for a vector $v_{\Lambda}$ of highest weight $\Lambda$ and any $X\in\dot{\Uu}_q(\mathfrak{gl}_m)$, with $\overline{\cdot}$ being the usual bar involution on $\dot{\Uu}_q(\mathfrak{gl}_m)$. We can use the $q$-Shapovalov form $\langle\cdot,\cdot\rangle_{\mathrm{Shap}}$ on $W_n(\Lambda)$ (see e.g. \cite{my} before Corollary 4.10) to define \textit{Lusztig's symmetric bilinear form} by setting $(\cdot,\cdot)_{\mathrm{Lusz}}=\overline{\langle\cdot,\phi(\cdot)\rangle}_{\mathrm{Shap}}$. 

Moreover, it is known that $W_n(\Lambda)$ is parameterized by semistandard tableaux of shape $(n^{\ell})$, which we denote by $\mathrm{Std}^s((n^{\ell}))\subset\mathrm{Col}((n^{\ell}))$. For a column strict tableaux $T$ we can define the \textit{column word} $co(T)=(c_1,\dots,c_{n\ell})$ to be a sequence of the entries of the columns of $T$ read from top to bottom and then from left to right. Note that this sequence has length $n\ell$. Then the set $\mathrm{Col}((n^{\ell}))$ is partial order by
\[
T\leq T^{\prime}\Leftrightarrow c(T^{\prime})-c(T)\in\bN^{n\ell}\;\text{ with }\;c(T^{(\prime)})=(c^{(\prime)}_1,c^{(\prime)}_1+c^{(\prime)}_2,\dots,c^{(\prime)}_1+\dots+c^{(\prime)}_{n\ell}).
\]
Since we tend to use $n$-multipartitions and $n$-multitableaux instead let us state what this means in our notation. A column strict tableaux $T$ of shape $(n^{\ell})$ corresponds to a $n$-multipartition $\vec{\lambda}$ by subtracting from each row the row number and obtain a new column strict tableaux $\tilde T$. Read the $k$-th column from bottom to top to obtain in this way the $n{+}1{-}k$-th partition $\lambda_{n+1-k}$ of the $\vec{\lambda}=(\lambda_n,\dots,\lambda_1)$. It is easy to see that this process is in fact invertible (the usage $n+1-k$ instead of $k$ due to our reading convention for $n$-multipartitions).

Write $\vec{\lambda}_T$ for the corresponding $n$-multipartition. Then $T\leq T^{\prime}$ if and only if $\vec{\lambda}_{T}\trianglelefteq\vec{\lambda}_{T^{\prime}}$, where $\trianglelefteq$ is the dominance order from Definition \ref{defn-dominnancelambda}. As a small example consider the following.
\[
\xy (0,0)*{\begin{Young}1& 3\cr 2 & 4 \cr\end{Young}}\endxy\leq\xy(0,0)*{\begin{Young}1& 2\cr 3 & 4\cr\end{Young}}\endxy\;\;\text{ and }\;\;\left(\;\emptyset\;,\;
\xy (0,0)*{\begin{Young}&\cr &\cr\end{Young}}\endxy\;\right)\trianglelefteq \left(\;\xy (0,0)*{\begin{Young}\cr\end{Young}}\endxy\;,\;\xy (0,0)*{\begin{Young}&\cr \cr\end{Young}}\endxy\;\right).
\]

Note that the conversion of a column strict tableaux $T$ to a $\vec{S}=(S_1,\dots,S_k)$ is given by counting the multiplicities of the entry $r$ and obtain an $r$-element subset $S_r\subset\{n,\dots, 1\}$ by taking the column numbers in which the entry appears as elements of $S_r$. Our Proposition \ref{prop-extgrowth} is actually stronger: for each boundary condition $\vec{S}$ there exists a $\mathfrak{gl}_n$-web $u_f$ that realizes this condition. To see this note that, as explained above, one can covered $\vec{S}$ to a $n$-multipartition $\vec{\lambda}$, then fill $\vec{\lambda}$ in any standard way and use Proposition \ref{prop-extgrowth} to generate a $\mathfrak{gl}_n$-web $u_f$. Thus, it makes sense to write $x_T$ since this corresponds 1:1 to the elementary tensors $x_{\vec{S}}$ from Section \ref{subsub-reptheo}.

A standard argument shows that a canonical basis, if it exists, is unique for a given precanonical structure. For a more general discussion see e.g. \cite{web3}. Moreover, Lusztig and Kashiwara proved that there exists a canonical basis $\{b_T\mid T\in\mathrm{Std}^s((n^{\ell}))\}$ of $W_n(\Lambda)$ with respect to the precanonical structure given by the elementary tensors $\{x_T\mid T\in\mathrm{Col}((n^{\ell}))\}$, the bar involution $\phi$ and Lusztig's symmetric bilinear form $(\cdot,\cdot)_{\mathrm{Lusz}}$.

In order to state the condition we need to extend the notion of a \textit{canonical flow} $f_c$ for a fixed $\mathfrak{gl}_n$-web $u\in W_n(\vec{k})$. To understand the notion recall that, e.g. by Lemma \ref{lem-webasF}, any $\mathfrak{gl}_n$-web $u$ can be obtained from a string of $F_i^{(j)}$ acting on a suitable highest weight vector $v_h$. While the elements of $W_n(\Lambda)$ are indexed by semistandard (meaning only weakly increasing along columns, but strictly along rows) tableaux of shape $(n^{\ell})$, the elements of the tensor product $\Lambda^{k_1}\bC^n\otimes\dots\otimes\Lambda^{k_m}\bC^n$ are indexed by column strict tableaux of shape $(n^{\ell})$ and $W_n(\Lambda)$ is a direct summand of it. Let us denote by $\mathrm{sh}\in\bZ$ some shift. Then Theorem \ref{thm-kkele} says that
\begin{gather}\label{eq-dualcan1}
\begin{aligned}
u&=q^{\mathrm{sh}}x_T+\sum_{T\leq T^{\prime}}c(u,T^{\prime})x_{T^{\prime}},c(u,T^{\prime})\in\bN[q,q^{-1}],T,T^{\prime}\in\mathrm{Col}((n^{\ell}))
\\ \phantom{u}=&q^{\mathrm{sh}}x_{\vec{\lambda}_T}+\sum_{\vec{\lambda}_{T}\unlhd\vec{\lambda}_{T^{\prime}}}c(u,\vec{\lambda}_{T^{\prime}})x_{\vec{\lambda}_{T^{\prime}}},c(u,\vec{\lambda}_{T^{\prime}})\in\bN[q,q^{-1}],\vec{\lambda}_T,\vec{\lambda}_{T^{\prime}}\in\Lambda^+(c(\vec{\lambda}_{T^{(\prime)}}),c(\vec{k}),n).
\end{aligned}
\end{gather}
We do not have a positive exponent property in general. Note that we are mostly interested in the case when the inequalities are strict and the leading coefficient is $1$, because it is one condition for a vector to be (dual) canonical.

By Theorem \ref{thm-kkele} the flows encode the coefficients of $u$ in terms of elementary tensors. The canonical flow now should be the flow that encodes the leading coefficient in the decomposition above. Recall from the previous sections that a flow $f$ can be translated to a string $\vec{S}_f$ of elements of $\mathfrak{P}(\{n,\dots,1\})$ by looking at the boundary and to a $n$-multipartition $\vec{\lambda}_f$ by removing all numbers from its $n$-multitableaux $\vec{T}_f$ from Section \ref{sec-tabwebs}.

It is very important in the following that we assume that the strings that generate our $\mathfrak{gl}_n$-webs are not arbitrary, but in such a way that they do not connect nested, unconnected components. This is always possible as explained in Lemma \ref{lem-webasF}.

\begin{defn}\label{defn-canflow}(\textbf{Canonical flow})
Fix a $\mathfrak{gl}_n$-web $u$ and a sequence of $F_i^{(j)}$ generating $u$. The \textit{canonical flow} $f_c$ for $u$ is the flow that corresponds (via Proposition \ref{prop-extgrowth}) to the $n$-multitableaux $\vec{T}_c$ obtained inductively by placing $j_k$ nodes with residue $i_k$ in the rightmost possible position. We denote the corresponding $n$-multipartition by $\vec{\lambda}_c$.
\end{defn}

\begin{lem}\label{lem-canflow}
Given a fixed $\mathfrak{gl}_n$-web $u$. Then the canonical flow $f_c$ on $u$ exists, i.e. the algorithm from Definition \ref{defn-canflow} is well-defined.
Moreover, $\mathrm{deg}_{\mathrm{BKW}}(\vec{T}_c)=\mathrm{wt}(u_{f_c})=\mathrm{sh}$ for some constant $\mathrm{sh}\leq 0$ and for all flows $f$ on $u$ the corresponding $\vec{\lambda}_f$ are bigger in the dominance order.
Hence, the $\vec{\lambda}_c=\vec{\lambda}_T$ and $\mathrm{sh}$ is the shift from \eqref{eq-dualcan1}. This inequality is strict if and only if $\mathrm{sh}=0$.
\end{lem}

\begin{proof}
That the algorithm is well-defined, i.e. in each step one can place the correct number of nodes at the correct positions, follows again by induction on the number of vertices $V(u)$. The induction step is, as before, removing the last $F_i^{(j)}$ of the string that generates $u$. Then it is true for $u^<$ and we can check locally that it still works.

In fact, we prove something stronger. Recall that $u$ has a boundary string $\vec{k}=(k_1,\dots,k_m)$ and $\vec{S}_{u_c}=(S_1,\dots,S_m)$ denotes the boundary of the canonical flow on $u$ (if it exists) and the $S_i$ are subsets of $\{n,\dots,1\}$. We show that $|S_i-S_{i+1}|<\mathrm{min}(k_i,n-k_{i+1})$ if and only if $S_k$ and $S_{k+1}$ are not connected and belong to two nested components of $u$. Moreover, we also want to show at the same time that $u$ has a canonical flow in the sense of Definition \ref{defn-canflow}.

First we note that we are only interested in the boundary, that is we can ignore internal closed components and that the statement is certainly true for all shifts. So let $u$ be a collection of arcs, i.e. $V(u)=0$. We have to check three cases. These are
\[
\scalebox{.7}{$\xy
(0,0)*{\includegraphics[scale=.75]{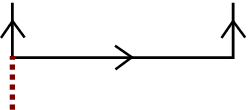}};
(-14,8)*{S_i};
(14,8)*{S_{i+1}};
(0,2.5)*{\scriptstyle \{k_i,\dots,1\}};
\endxy$},
\quad
\scalebox{.7}{$\xy
(0,0)*{\includegraphics[scale=.75]{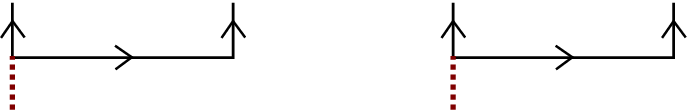}};
(-14,8)*{S_i};
(14,8)*{S_{i+1}};
(-28,2.5)*{\scriptstyle \{k_i,\dots,1\}};
(27,2.5)*{\scriptstyle \{k_{i+1},\dots,1\}};
\endxy$},
\quad
\scalebox{.7}{$\xy
(0,0)*{\includegraphics[scale=.75]{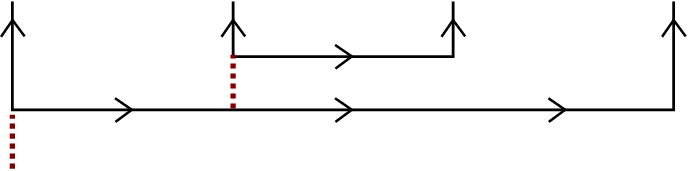}};
(14,12)*{S_i};
(42,12)*{S_{i+1}};
(0,6.5)*{\scriptstyle \{k_i,\dots,1\}};
(0,0)*{\scriptstyle \{k_{i+1},\dots,1\}};
\endxy$}.
\]
In all these cases the canonical flow is displayed above. Hence, the canonical exists and satisfies the extra condition from above (recall that leashes have flow $\{n,\dots,1\}$ which splits into two disjoint flows at the top). Note that $\{k_i,\dots,1\}-\{n,\dots,n-k_{i+1}+1\}=\{\mathrm{min}(k_i,n-k_{i+1}),\dots,1\}$.

Moreover, that the statement is true if $u$ has exactly one vertex follows in the same fashion by checking three extra cases involving a component that looks like a theta-web (we need this case too, because a ladder can have two vertices).

The main observation now is that one can always apply every non-killing divided power of $F_i$ in the first two cases and the canonical flow will carry over, but one could run into problems in the last case.

Now assume $|V(u)|>1$. Remove the last ladder from $u$ and obtain a $\mathfrak{gl}_n$-web $u^<$. Note that it is clear by the case-by-case check above that the statement will carry over from $u^<$ to $u$ if this last ladder was an arc. Thus, we can freely assume that the last ladder has at least one vertex and we can use the induction hypothesis on $u^<$. But then the statement follows also for $u$, since we know by Lemma \ref{lem-webasF} that the last $F_i^{(j)}$ does not connect nested, unconnected components of $u^<$. But then, since the last $F_i^{(j)}$ does not kill $u^<$, we can apply the procedure from Definition \ref{defn-canflow} to the canonical flow on $u^<$, because of the translation between flows and $n$-multitableaux from Section \ref{sec-tabwebs}. Moreover, the other statement also carries over. Thus, the algorithm is well-defined.

We observe that the second statement can in fact be strengthen. That is, each local step is of degree lower or equal zero (and therefore of course also the total result). To see this note that if a step would have addable nodes of the same residue to the right, then we would have placed them differently. Thus, the only contributions to the degree comes from removable nodes which always lower the degree and the total degree will be some constant $\mathrm{sh}\leq 0$.
That all other flows give bigger $n$-multipartitions follows immediately from the definition of the dominance order, since we place the nodes in the rightmost possible positions. But in general there can be non-canonical flows $f$ with the same $n$-multipartition $\vec{\lambda}_{f}=\vec{\lambda}_{f_c}$, e.g. if $u$ has a connected, internal, closed $\mathfrak{gl}_n$-web as for example a closed circle.

But if $\mathrm{sh}=0$, then this inequality has to be strict. This follows because the residue sequence of the $n$-multitableaux $\vec{T}$ have to be the same for all flows on $u$. That is $\vec{\lambda}_{f}=\vec{\lambda}_{f_c}$ and $f\neq f_c$ implies the existence of removable nodes, because $f\neq f_c\Leftrightarrow\vec{T}_f\neq\vec{T}_{f_c}$ and, by the argument above, $\vec{T}_{f_c}$ does not have addable nodes. But then $\mathrm{sh}<0$.

In the same vein, if $\mathrm{sh}<0$, then the existence of removable nodes allows use to define another $n$-multitableaux $\vec{T}_f\neq\vec{T}_{f_c}$ with $\vec{\lambda}_{f}=\vec{\lambda}_{f_c}$ by switching the corresponding entries of the nodes.
\end{proof}

\begin{ex}\label{ex-canflow}
The reader is invited to check that our notion of canonical flow for arc-diagrams in the case $n=2$ gives counter-clockwise oriented circles in the notation of Brundan and Stroppel \cite{bs1} and in the case $n=3$ our definition gives exactly Khovanov and Kuperberg's notion of canonical flows for non-elliptic $\mathfrak{sl}_3$-webs \cite{kk}.

The $\mathfrak{sl}_2$-webs that do not satisfy $\mathrm{sh}=0$ will be all $\mathfrak{sl}_2$-webs with internal circles (aka closed $\mathfrak{sl}_2$-\textit{subwebs}) and all $\mathfrak{sl}_3$-webs with internal digons or closed $\mathfrak{sl}_3$-subwebs.

A bigger example is the $\mathfrak{gl}_4$-web from the Examples \ref{exa-flow} and \ref{ex-websasF}. Here the resulting $4$-multitableaux is
\[
\vec{T}_c=\left(\;\xy (0,0)*{\begin{Young}1& 2 & 3 & 4\cr 5 & 14 \cr\end{Young}}\endxy\;,\;\xy (0,0)*{\begin{Young}1& 2 & 3 & 4\cr 5 & 12 & 13\cr 17\cr\end{Young}}\endxy\;,\;\xy (0,0)*{\begin{Young}1& 2 & 3 &4 & 19\cr 5 & 6 & 11\cr 15 & 16 & 18\cr\end{Young}}\endxy\;,\;\xy (0,0)*{\begin{Young}1& 2 & 3 & 4 & 19\cr 5 & 6 & 9 & 10\cr 7 & 8 & 12 & 13\cr\end{Young}}\endxy\;\right).
\]
Thus, by the Theorem \ref{thm-dualcanwebs} below, this $\mathfrak{gl}_4$-web is not dual-canonical because the degree of $\vec{T}_c$ is $-1$. In fact, only the node labeled $13$ is not of degree zero, but of degree $-1$.
\end{ex}

We are now ready to state the condition for a $\mathfrak{gl}_n$-web to be dual canonical. It is worth noting that the conditions (b) and (c) can be checked by the algorithm from Definition \ref{defn-evaluation}. Recall the shift $d(\vec{k})$ in the definition of the Kuperberg bracket, see \eqref{eq-shift}.

\begin{thm}\label{thm-dualcanwebs}
Given a $\mathfrak{gl}_n$-web $u\in W_n(\vec{k})$. The following are equivalent.
\begin{itemize}
\item[(a)] The $\mathfrak{gl}_n$-web $u$ is a dual canonical basis element.
\item[(b)] The evaluation of $w=u^*u$ satisfies $\mathrm{ev}(w)=q^{-d(\vec{k})}(1+\mathrm{rest}(w))$ with $\mathrm{rest}(w)\in q\bN[q]$ (positive exponent property).
\item[(c)] $\mathrm{ev}_u$ does not contain $n$-multitableaux $\vec{T}$ with $\mathrm{deg}_{\mathrm{BKW}}(\vec{T})\leq 0$ except the canonical $n$-multitableaux $\vec{T}_c$ which is of degree zero.
\end{itemize}
Moreover, a $\mathfrak{gl}_n$-web $u\in W_n(\vec{k})$ that does contain a closed $\mathfrak{gl}_n$-subweb is never dual canonical.
\end{thm}

\begin{proof}
(b)$\Leftrightarrow$(c). The difference hereby is that $\mathrm{ev}_u$ contains all flows on $u$, while $\mathrm{ev}_{u^*u}$ contains all possible ways to glue flows on $u$ together. Still (b) and (c) are equivalent: the weight of a flow $f$ on $w=u^*u$ is given by the sum of the weights of two flows $f_b$ and $f_t$ on the bottom and top part, respectively. But by Theorem \ref{thm-evaluation}, Proposition \ref{prop-kupshap} and the properties of the $q$-Shapovalov form $\langle\cdot,\cdot\rangle_{\mathrm{Shap}}$ we see that (b)$\Leftrightarrow$(c).
To be precise, we have
\[
\mathrm{deg}_{\mathrm{BKW}}(u_f^*)=\mathrm{deg}_{\mathrm{BKW}}(u_f)-d(\vec{k})
\]
by duality. Thus, (c)$\Rightarrow$(b) since, under the assumption that (c) is true, there can be only one flow of degree $-d(\vec{k})$ on $u^*$, namely the dual of the canonical flow on $u$.
Furthermore, the existence of a non-canonical flow $f$ on $u$ with degree $\leq 0$ gives, again by duality, a non-canonical flow on $w=u^*u$ of degree $\leq 0$ even after shifting everything by $d(\vec{k})$. Thus, by Theorem \ref{thm-evaluation}, (b) can not be true. Moreover, a canonical flow $f_c$ always exists and has degree lower or equal zero by Lemma \ref{lem-canflow}. That is, if $f_c$ has negative degree, then, by Theorem \ref{thm-evaluation} and duality again, (b) can not be true. Hence, $\neg\text{(c)}\Rightarrow\neg\text{(b)}$.

(a)$\Rightarrow$(b). This follows from Theorem \ref{thm-evaluation}, because the evaluation $\mathrm{ev}(w)$ is (up to a shift) the $q$-Shapovalov form $\langle u,u\rangle_{\mathrm{Shap}}$. By the discussion above the unique precanonical structure is given by the bar involution $\phi$, the elementary tensors and Lusztig's bilinear form $(\cdot,\cdot)_{\mathrm{Lusz}}=\overline{\langle\cdot,\phi(\cdot)\rangle}_{\mathrm{Shap}}$. Hence, a $\mathfrak{gl}_n$-web $u$ that does not satisfy (b) can not satisfy the positive exponent property.

(b)$\Rightarrow$(a). Recall that we already know that the $q$-Shapovalov form is the Kuperberg form is the evaluation result from Theorem \ref{thm-evaluation}. Thus, we only need to check that $u$ is bar invariant and satisfies \eqref{eq-dualcan1} with $\mathrm{sh}=0$, $c(u,\vec{\lambda}_{T^{\prime}})\in q\bN[q]$ and a strict inequality for the sum. Then, because a dual canonical structure is unique (if it exists), we can conclude that the $\mathfrak{gl}_n$-web $u$ is dual canonical.

We observe that Lemma \ref{lem-webasF} ensures that $u$ can be written as a sequence of $F_i^{(j)}$ acting on a highest weight vector. Hence, since $\phi(F_i^{(j)})=F_i^{(j)}$, the bar invariance follows.

Moreover, the second condition follows from Lemma \ref{lem-canflow} (because (b)$\Leftrightarrow$(c)) together with Theorem \ref{thm-kkele}. Thus, (b) is a sufficient condition for $u$ to be dual canonical.

If $u$ has a closed $\mathfrak{gl}_n$-subweb $w$, then, since this corresponds to a multiplication by $\mathrm{ev}(w)$ by Theorem \ref{thm-evaluation} and the canonical flow corresponds to a negative degree of $\mathrm{ev}(w)$, the condition (c) can not be satisfied.  
\end{proof}
\subsection{Connection to colored \texorpdfstring{$\mathfrak{gl}_n$}{gln}-link polynomials}\label{sec-linkswebs}

\begin{rem}\label{rem-tangles} 
It is not hard to adapt the discussion in this section to tangles. While the result for a link is a quantum number in $\bZ[q,q^{-1}]$ (a Laurent polynomial in $q$ with integer coefficients), the result for a tangle is a matrix of quantum numbers.

To see this note that the invariant is an intertwiner of $\Uu_q(\mathfrak{gl}_n)$-representations which we, under $q$-skew Howe duality, see as a certain string of $F_i^{(j)}$ acting on a $\dot{\Uu}_q(\mathfrak{gl}_m)$-weight space $W_n(\vec{k}_b)$ at the bottom to another $\dot{\Uu}_q(\mathfrak{gl}_m)$-weight space $W_n(\vec{k}_t)$ at the top. In the case of a link the bottom one will be the highest $\dot{\Uu}_q(\mathfrak{gl}_m)$-weight space and the top the lowest $\dot{\Uu}_q(\mathfrak{gl}_m)$-weight space of the $\dot{\Uu}_q(\mathfrak{gl}_m)$-highest weight module $W_n(\Lambda)$. Both are of dimension $1$. Hence, the whole results is a certain quantum number. For a tangle the weight spaces $W_n(\vec{k}_b)$ and $W_n(\vec{k}_t)$ do not have to be one dimensional.
\end{rem}

\subsubsection{The MOY-calculus}\label{subsub-moy}
We start by recalling the \textit{colored Reshetikhin--Turaev $\mathfrak{gl}_n$-link polynomial $\langle L_D\rangle_n$} of a colored link diagram $L_D$ following the approach of Murakami--Ohtsuki--Yamada from \cite{moy}, i.e. using the so-called \textit{MOY graph polynomial} $\langle w\rangle_{\mathrm{MOY}}$ of a closed $\mathfrak{gl}_n$-web $w$. To fix notation, we call a crossing $\xy(0,0)*{\includegraphics[scale=0.4]{figs/linkpoly/positive.eps}};\endxy$ \textit{positive} and a crossing $\xy(0,0)*{\includegraphics[scale=0.4]{figs/linkpoly/negative.eps}};\endxy$ \textit{negative} and the difference of their total numbers $|\overcrossing|$ and $|\undercrossing|$ the \textit{writhe} $\mathrm{w}(L_D)=|\overcrossing|-|\undercrossing|$ of the diagram.

\begin{defn}\label{defn-MOYpoly}(\textbf{MOY graph polynomial}) Let $w$ be a closed $\mathfrak{gl}_n$-web and let $V(w)$ and $E(w)$ be the sets of its vertices and edges. Let $c\colon E(w)\to\bN$ be the function that assigns to edges $e\in E(w)$ its \textit{label} (or color) $c(e)\in\bN$. Moreover, for a fixed flow $f$ on $w$ let $\mathrm{f}\colon E(w)\to\mathfrak{P}(\{n,\dots,0\})$ be the function that assigns to each edges $e\in E(w)$ its \textit{flow} (or state) $\mathrm{f}(e)\in\mathfrak{P}(\{n,\dots,0\})$.

Recall that for each vertex $v\in V(w)$ and a fixed flow $w_f$ the notation $\mathrm{wt}^v(w_f)$ denotes the weight of the vertex $v$ with respect to $w_f$ (see Definition \ref{defn-slnflow}). Define the \textit{(total) shifted weight} $\mathrm{wt}(v,w_f)$ and $\mathrm{wt}^t(v,w_f)$ by
\[
\mathrm{wt}(v,w_f)=q^{\frac{\mathrm{c}(e_1)\mathrm{c}(e_2)}{2}-\mathrm{wt}^v(w_f)}\;\;\text{ and }\;\;\mathrm{wt}^t(v,w_f)=\prod_{v\in V(w)}\mathrm{wt}(v,w_f),
\]
where $e_1,e_2\in E(w)$ are the two unique incoming or outgoing edges at $v$.

Define for a fixed flow $f$ on $w$ a graph by replacing each edge $e\in E(w)$ by $\mathrm{c}(e)$ parallel edges. Then assign to each of these edges a different element of $\mathrm{f}(e)$. Then connect the new edges with the same element of $\mathfrak{P}(\{n,\dots,0\})$. From this we get a collection of embedded, oriented, labeled circles that we denote by $\mathcal{C}$, and we denote the label of each $C\in\mathcal{C}$ by $\mathrm{f}(C)$. Moreover, denote by $\mathrm{rot}(C)$ the \textit{orientation} of the circle $C$, i.e. $\mathrm{rot}(C)=1$ if the orientation is counter-clockwise and $\mathrm{rot}(C)=-1$ otherwise. Note that there are some for us unimportant technicalities how to obtain these circles, see \cite{moy}.

The \textit{rotation number} $\mathrm{rot}(w_f)$ is then defined by
\[
\mathrm{rot}(w_f)=\sum_{C\in\mathcal{C}}\mathrm{rot}(C)\mathrm{f}(C).
\]
Then the \textit{$\mathfrak{gl}_n$-MOY graph polynomial} of $w$ is defined by
\[
\langle w\rangle_{\mathrm{MOY}}=\sum_{f\in Fl(w)}\mathrm{wt}^t(v,w_f)q^{\mathrm{rot}(w_f)}\in\bN[q,q^{-1}],
\]
where $Fl(w)$ denotes the set of all flow lines on $w$.
\end{defn}

\begin{thm}\label{thm-MOYrel}(\cite{moy})
The polynomial $\langle\cdot\rangle_{\mathrm{MOY}}$ satisfies the relations of the $\Uu_q(\mathfrak{gl}_n)$-spider $\spid{n}$.\qed
\end{thm}

\begin{thm}(\cite[Theorem 2.4]{wu})\label{thm-MOYuniq}
The MOY graph polynomial $\langle\cdot\rangle_{\mathrm{MOY}}$ is uniquely determined by the relations of the $\Uu_q(\mathfrak{gl}_n)$-spider $\spid{n}$ from Definition \ref{def-spid}.\qed
\end{thm}

Hence, our notions are the same, something that is not clear from Definition \ref{defn-MOYpoly} above and follows only from the Theorems \ref{thm-MOYrel} and \ref{thm-MOYuniq}. Because of this we use our notation in the following.

\begin{cor}\label{cor-kupismoy}
Let $w=v^*u$ be a closed $\mathfrak{gl}_n$-web. Then $\langle w\rangle_{\mathrm{Kup}}=\langle w\rangle_{\mathrm{MOY}}$.\qed
\end{cor}

\begin{defn}\label{defn-cRTpoly}(\textbf{Colored Reshetikhin--Turaev $\mathfrak{gl}_n$-link polynomial}) 
Let $L_D$ be a colored link diagram. Then the \textit{colored Reshetikhin--Turaev $\mathfrak{gl}_n$-link polynomial} $\langle L_D\rangle_n$ of $L_D$ is defined by applying the following to all crossings of $L_D$. We use
\[
\left\langle
\xy
(0,0)*{\includegraphics[scale=.75]{figs/linkpoly/positive.eps}};
(2,-4.1)*{\scriptstyle b};
(-2,-4.35)*{\scriptstyle a};
\endxy
\right\rangle_n=\sum_{k=0}^b (-1)^{k+(a+1)b}q^{-b+k}
\sbox0{
\scalebox{.7}{$\xy(0,0)*{\includegraphics[scale=0.74]{figs/slnwebs/squarec.eps}};(-11.5,-11)*{\scriptstyle a};(11.5,-11)*{\scriptstyle b};(6,8)*{\scriptstyle a+k-b};(2,-5)*{\scriptstyle k};(-9.8,0)*{\scriptstyle a+k};(10,0)*{\scriptstyle b-k};(-11.5,10)*{\scriptstyle b};(11.5,10)*{\scriptstyle a}\endxy$}}
\mathopen{\resizebox{\width}{1.15\ht0}{$\Bigg\langle$}}
\raisebox{1.2ex}{\usebox{0}}
\mathclose{\resizebox{\width}{1.15\ht0}{$\Bigg\rangle$}}_{\mathrm{Kup}},
\]
if $b\leq a$, and for $a<b$ we use
\[
\left\langle
\xy
(0,0)*{\includegraphics[scale=.75]{figs/linkpoly/positive.eps}};
(2,-4.1)*{\scriptstyle b};
(-2,-4.35)*{\scriptstyle a};
\endxy
\right\rangle_n=\sum_{k=0}^a (-1)^{k+(b+1)a}q^{-a+k}
\sbox0{
\scalebox{.7}{$\xy(0,0)*{\reflectbox{\includegraphics[scale=0.74]{figs/slnwebs/squarec.eps}}};(-11.5,-11)*{\scriptstyle a};(11.5,-11)*{\scriptstyle b};(6,8)*{\scriptstyle a+k-b};(2,-5)*{\scriptstyle k};(-9.8,0)*{\scriptstyle a+k};(10,0)*{\scriptstyle b-k};(-11.5,10)*{\scriptstyle b};(11.5,10)*{\scriptstyle a}\endxy$}}
\mathopen{\resizebox{\width}{1.15\ht0}{$\Bigg\langle$}}
\raisebox{1.2ex}{\usebox{0}}
\mathclose{\resizebox{\width}{1.15\ht0}{$\Bigg\rangle$}}_{\mathrm{Kup}}
\]
for a positive $\overcrossing_{a,b}$ and almost the same for a negative $\undercrossing_{a,b}$ with the same colors $a,b$, but the powers of $q$ above are minus the ones for the positive $\overcrossing_{a,b}$.

Moreover, for each positive crossing $\overcrossing_{a,b}$ we 
need the shift
\[
s\left(\xy
(0,0)*{\includegraphics[scale=.75]{figs/linkpoly/positive.eps}};
(2,-4.1)*{\scriptstyle b};
(-2,-4.35)*{\scriptstyle a};
\endxy\right)=\begin{cases}(-1)^{b+1}q^{b(n+1-b)}, &\text{if }a=b,\\1,&\text{else,}\end{cases}
\]
and the same again up to a multiplication with $-1$ in the exponent of $q$ for a negative crossing with the same colors.
The \textit{normalized}, colored Reshetikhin--Turaev $\mathfrak{gl}_n$-link polynomial of $L_D$ is then defined by
\begin{equation}\label{eq-norRT}
\mathrm{RT}_n(L_D)=\langle L_D\rangle_n\cdot\prod_{c_{a,b}}s(c_{a,b}),
\end{equation}
where the product runs over all colored crossings. 
\end{defn}

\begin{thm}\label{thm-RTinv}(\cite[Theorem 5.1]{moy})
The colored Reshetikhin--Turaev $\mathfrak{gl}_n$-link polynomial $\langle\cdot\rangle_n\in\bZ[q,q^{-1}]$ is invariant under the second and third Reidemeister moves.
The normalized, colored Reshetikhin--Turaev $\mathfrak{gl}_n$-link polynomial $\mathrm{RT}_n(\cdot)\in\bZ[q,q^{-1}]$ is an invariant of links.\qed
\end{thm}

Note that already $\langle\cdot\rangle_n$ is invariant under the Reidemeister moves up to a normalization, i.e. it gives an invariant of framed links. We ignore the normalization in the following.

\subsubsection{Our setup}\label{subsub-oursetup}

The rest of the section is intended to explain how our approach can be used to calculate $\langle L_D\rangle_n$ for all colorings using the language of $n$-multitableaux. Thus, we have explain how a colored link diagram $L_D$ can be translated to our framework using $q$-skew Howe duality and actions of $F_i^{(j)}$ on some highest weight vector $v_{(n^{\ell})}$. We start by defining the \textit{colored braiding} operators. Recall that we assume that $\Lambda$ denotes $n$-times the $\ell$-th fundamental $\dot{\Uu}_q(\mathfrak{gl}_m)$-weight and that $W_n(\Lambda)$ denotes the irreducible $\dot{\Uu}_q(\mathfrak{gl}_m)$-representation of highest weight $\Lambda$. Recall that we use notations such as $\vec{k}$ for $\mathfrak{gl}_m$-weights.

\begin{defn}\label{defn-colbraid}
For $a,b\in\{0,\dots,n\}$ let $\vec{k}=(\dots,a,b,0,\dots)$ and $\vec{k}^{\prime}=(\dots,0,a,b,\dots)\in\bN^{m}$ be $\dot{\Uu}_q(\mathfrak{gl}_m)$-weights where $a$ is the $i$-th entry of $\vec{k}$ and the $i+1$-th entry of $\vec{k}^{\prime}$.
For all $k=0,\dots,\mathrm{min}(a,b)$ the \textit{$k$-th colored braiding operator $T_{a,b,i}^k$} acts on the $\vec{k}$-weight space $W_n(\vec{k})$ of $W_n(\Lambda)$ by
\[
T_{a,b,i}^k\colon W_n(\vec{k})\to W_n(\vec{k}^{\prime}), v_{\vec{k}}\mapsto \begin{cases}F_{i+1}^{(a+k-b)}F_{i}^{(a)}F_{i+1}^{(b-k)}v_{\vec{k}}, &\text{if } b\leq a,\\F_{i}^{(a-k)}F_{i+1}^{(a)}F_{i}^{(k)}v_{\vec{k}}, &\text{if } a< b,\end{cases}
\]
for $v_{\vec{k}}\in W_n(\vec{k})$. Or in pictures with $T_{a,b,i}^k=\xy(0,0)*{\includegraphics[scale=0.4]{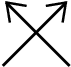}};\endxy$
\[
\xy
(0,0)*{\includegraphics[scale=.75]{figs/linkpoly/braid.eps}};
(2,-4.1)*{\scriptstyle b};
(-2,-4.35)*{\scriptstyle a};
(0,4.35)*{\scriptstyle k};
\endxy 
v_{\vec{k}}
=
\scalebox{.7}{$\xy
(0,0)*{\includegraphics[scale=.75]{figs/linkpoly/colbraidb.eps}};
(17,18.5)*{F_{i+1}^{(a+k-b)}};
(-26,-20)*{a};
(2,-19.5)*{b};
(30,-19.5)*{0};
(-12,4)*{F_{i}^{(a)}};
(-26,-6.5)*{a};
(2.5,-7)*{k};
(33.5,-6.5)*{b-k};
(15,-9.5)*{F_{i+1}^{(b-k)}};
(-26,6.5)*{0};
(6,6.5)*{a+k};
(33.5,6)*{b-k};
(-26,19.5)*{0};
(2,19.5)*{b};
(30,19)*{a};
\endxy$}
\quad
\text{or}
\quad
\scalebox{.7}{$\xy
(0,0)*{\includegraphics[scale=.75]{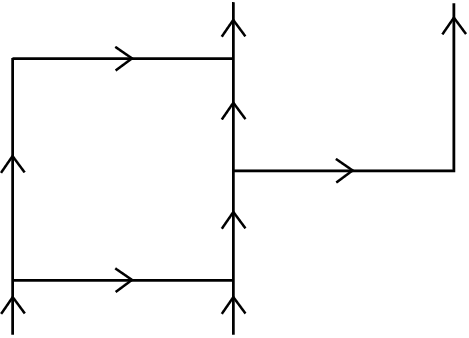}};
(-12,-10.5)*{F_{i}^{(k)}};
(-26,-20)*{a};
(2,-19.5)*{b};
(30,-19.5)*{0};
(18.5,3.5)*{F_{i+1}^{(a)}};
(-23,-6.5)*{a-k};
(6,-7)*{b+k};
(30,-6.5)*{0};
(-11,18)*{F_{i}^{(a-k)}};
(-23,6.5)*{a-k};
(10,6.5)*{b+k-a};
(30,6)*{a};
(-26,19.5)*{0};
(2,19.5)*{b};
(30,19)*{a};
\endxy$}.
\]
Note that, if the weights have values $<0$ or $>n$, then the corresponding diagram is zero due to our convention. The same is true for the action, since it factors through $\Lambda^{<0}\bC^n$ or $\Lambda^{>n}\bC^n$.

We define the \textit{left} ${}_l T_{a,b,i}^k=\xy(0,0)*{\rotatebox{90}{\includegraphics[scale=0.4]{figs/linkpoly/braid.eps}}};\endxy$, \textit{right} ${}_r T_{a,b,i}^k=\xy(0,0)*{ \rotatebox{270}{\includegraphics[scale=0.4]{figs/linkpoly/braid.eps}}};\endxy$ and \textit{downwards} ${}_d T_{a,b,i}^k=\xy(0,0)*{ \rotatebox{180}{\includegraphics[scale=0.4]{figs/linkpoly/braid.eps}}};\endxy$ versions by
\[
\scalebox{.7}{$\xy(0,0)*{\rotatebox{90}{\includegraphics[scale=.75]{figs/linkpoly/braid.eps}}};(2,-4.6)*{\scriptstyle b};
(0,4.35)*{\scriptstyle k};
(-2,-4.9)*{\scriptstyle a};\endxy=\xy(0,0)*{\includegraphics[scale=.75]{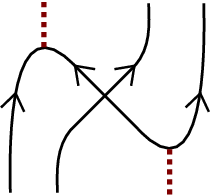}};
(-7,-10.6)*{\scriptstyle b};
(-11,-10.9)*{\scriptstyle a};
(6.75,10.9)*{\scriptstyle b};
(0,4.35)*{\scriptstyle k};
(11.25,10.5)*{\scriptstyle a};\endxy$}
\quad
\text{and}
\quad
\scalebox{.7}{$\xy(0,0)*{\rotatebox{270}{\includegraphics[scale=.75]{figs/linkpoly/braid.eps}}};(2,-4.6)*{\scriptstyle b};
(0,4.35)*{\scriptstyle k};
(-2,-4.9)*{\scriptstyle a};\endxy=\xy(0,0)*{\includegraphics[scale=.75]{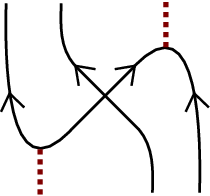}};(-6.75,10.5)*{\scriptstyle a};
(-11.25,10.9)*{\scriptstyle b};
(7,-10.9)*{\scriptstyle a};
(0,4.35)*{\scriptstyle k};
(11,-10.6)*{\scriptstyle b};\endxy$}
\quad
\text{and}
\quad
\scalebox{.7}{$\xy(0,0)*{\rotatebox{180}{\includegraphics[scale=.75]{figs/linkpoly/braid.eps}}};(-2,-4.6)*{\scriptstyle b};
(0,4.35)*{\scriptstyle k};
(2,-4.9)*{\scriptstyle a};\endxy=\xy(0,0)*{\includegraphics[scale=.75]{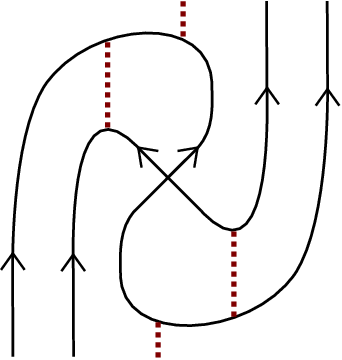}};(-13.5,-20.6)*{\scriptstyle b};
(-19,-20.9)*{\scriptstyle a};
(0,4.35)*{\scriptstyle k};
(19,20.9)*{\scriptstyle b};
(13.5,20.5)*{\scriptstyle a};\endxy$}.
\]

These three definitions correspond to
\[
{}_l T_{a,b,i}^k v_{\vec{k}_l}=F_{i+1}^{(a)}F_{i}^{(a)}T_{b,n-a,i+1}^{k}F_{i+3}^{(a)}F_{i+2}^{(a)}v_{\vec{k}_l}\hspace*{0.15cm}\text{ and }\hspace*{0.15cm}{}_r T_{a,b,i}^k=F_{i+1}^{(n-b)}F_{i-2}^{(b)}T_{n-b,a,i-1}^{k}F_{i-2}^{(n-b)}F_{i+1}^{(b)}v_{\vec{k}_r}
\]
with the new weights $\vec{k}_l=(\dots,a,b,n,0,0,\dots)$ and $\vec{k}_r=(\dots,n,0,a,b,0,\dots)$ and
\[
{}_d T_{a,b,i}^k v_{\vec{k}_l}=F_{i+2}^{(a)}F_{i+3}^{(a)}F_{i+1}^{(a)}F_{i+2}^{(b)}F_{i}^{(a)}F_{i+1}^{(b)}T_{n-a,n-b,i+2}^{k}F_{i+4}^{(b)}F_{i+3}^{(b)}F_{i+2}^{(a)}F_{i+5}^{(a)}F_{i+4}^{(a)}F_{i+3}^{(a)}v_{\vec{k}_d}
\]
with $\vec{k}_d=(\dots,a,b,n,n,0,0,0,\dots)$ with $a$ always in the $i$-th position and the $v_{\vec{k}}$ are all vectors in the corresponding weight modules for the three $\vec{k}$.

The \textit{positive full braiding operator} $T_{a,b,+i}$ is then defined to be the $q$-weighted sum
\begin{equation}\label{eq-posbraid}
T_{a,b,+i}=\begin{cases}\sum_{k=0}^b (-1)^{k+(a+1)b}q^{-b+k}T_{a,b,i}^k, &\text{if } b\leq a,\\\sum_{k=0}^a (-1)^{k+(b+1)a}q^{-a+k}T_{a,b,i}^k, &\text{if } a<b.\end{cases}
\end{equation}
Moreover, the \textit{negative full braiding operator} $T_{a,b,-i}$ is defined similar but with all powers of $q$ multiplied by the factor $-1$.
\end{defn}

\begin{ex}\label{ex-braidop}
Let us consider a small $\mathfrak{gl}_2$ example. Let $a=b=1$ and therefore $k=0$ or $k=1$. Then we have essentially two pictures.
\[
k=0:
\scalebox{.7}{$\xy
(0,0)*{\includegraphics[scale=.75]{figs/linkpoly/braid1a.eps}};
(-26,-20)*{1};
(2,-19.5)*{1};
(30,-19.5)*{0};
(-14,3.5)*{F_{i}};
(-26,-6.5)*{1};
(2.5,-7)*{0};
(30,-6.5)*{1};
(14,-10.5)*{F_{i+1}};
(-26,6.5)*{0};
(2.5,6.5)*{1};
(30,6)*{1};
(-26,19.5)*{0};
(2,19.5)*{1};
(30,19)*{1};
\endxy$}
\hspace*{0.2cm}
\text{ and }
\hspace*{0.2cm}
k=1:
\scalebox{.7}{$\xy
(0,0)*{\includegraphics[scale=.75]{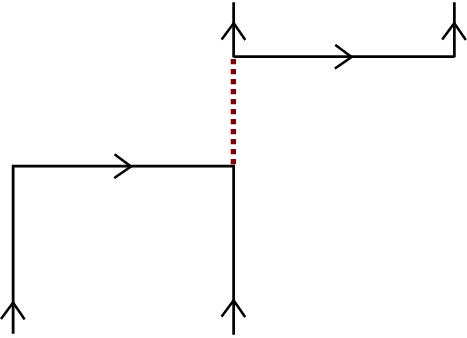}};
(14,17.5)*{F_{i+1}};
(-26,-20)*{1};
(2,-19.5)*{1};
(30,-19.5)*{0};
(-14,3.5)*{F_{i}};
(-26,-6.5)*{1};
(2.5,-7)*{1};
(30,-6.5)*{0};
(-26,6.5)*{0};
(2.5,6.5)*{2};
(30,6)*{0};
(-26,19.5)*{0};
(2,19.5)*{1};
(30,19)*{1};
\endxy$}.
\]
These are exactly the two terms in the Kauffman calculus for the Jones polynomial.
\end{ex}

Let $T_D$ denote a colored, oriented diagram of a tangle. We assume that $T_D$ is in a general Morse position. By this we mean that strands of $T_D$ are locally either \textit{identities}, \textit{cups}, \textit{caps}, \textit{shifts}, \textit{overcrossings} or \textit{undercrossings} (with all possible orientations) as illustrated below.
\[
\scalebox{.7}{$\xy
(0,0)*{\includegraphics[scale=.75]{figs/slnwebs/pointup.eps}};
\endxy$},
\hspace*{0.5cm}
\scalebox{.7}{$\xy
(0,0)*{\includegraphics[scale=.75]{figs/slnwebs/cupgen.eps}};
\endxy$},
\hspace*{0.5cm}
\scalebox{.7}{$\xy
(0,0)*{\includegraphics[scale=.75]{figs/slnwebs/capgen.eps}};
\endxy$},
\hspace*{0.5cm}
\scalebox{.7}{$\xy
(0,0)*{\includegraphics[scale=.75]{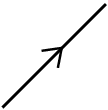}};
\endxy$},
\hspace*{0.5cm}
\scalebox{.7}{$\xy
(0,0)*{\reflectbox{\includegraphics[scale=.75]{figs/linkpoly/leftup.eps}}};
\endxy$},
\hspace*{0.5cm}
\scalebox{.7}{$\xy
(0,0)*{\includegraphics[scale=.75]{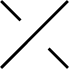}};
\endxy$},
\hspace*{0.5cm}
\scalebox{.7}{$\xy
(0,0)*{\includegraphics[scale=.75]{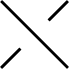}};
\endxy$}.
\]
Our approach for calculation is to use the evaluation algorithm.

\begin{lem}\label{lem-linkasF}
Any colored, oriented tangle diagram $T_D$ can be written, using $q$-skew Howe duality, as
\[
T_D=\prod_{k=1}^{s} \tilde{F}_{i_k}^{(j_k)}v_{(n^{\ell})},\hspace*{0.25cm}\tilde{F}_{i_k}^{(j_k)}=\begin{cases}F_{i_k}^{(j_k)}, &\text{for some }i_k\in\{1,\dots,m-1\},j_k\in\{0,\dots,n\},\\T_{a_k,b_k,\pm i_k}, &\text{for some }a_k,b_k\in\{0,\dots,n\},i_k\in\{1,\dots,m-1\},\end{cases}
\]
for some $s\in\bN$, some highest weight vector $v_{(n^{\ell})}$ and marked braiding operators $T_{a_k,b_k,\pm i_k}$ (where the signs should indicate if the corresponding crossing is positive $\xy(0,0)*{\includegraphics[scale=0.4]{figs/linkpoly/positive.eps}};\endxy$ or negative $\xy(0,0)*{\includegraphics[scale=0.4]{figs/linkpoly/negative.eps}};\endxy$).

Hence, each such tangle diagram $T_D$ can be realized as
\[
T_D=\prod_{k=1}^{s} \tilde{F}_{i_k}^{(j_k)}v_{(n^{\ell})}=\sum_{j=1}^{t}(-1)^{\mathrm{sgn}_j}q^{\mathrm{sh}_j}\prod_{k_j=1}^{s_j} F_{i_{k_j}}^{(j_{k_j})}v_{(n^{\ell})}=\sum_{j=1}^{t}(-1)^{\mathrm{sgn}_j}q^{\mathrm{sh}_j}u_j
\]
where $\mathrm{sgn}_j$ and $\mathrm{sh}_j$ are some constants and all summands are of the same total length $\sum j_{k_j}$. The $u_j$ are certain $\mathfrak{gl}_n$-webs. Moreover, if $T_D$ is a link diagram, then the $u_j$ are all closed $\mathfrak{gl}_n$-webs.
\end{lem}

\begin{proof}
All the statements are easy to verify following the proof of Lemma \ref{lem-webasF} and we omit the details.
\end{proof}

Using the last part of Lemma \ref{lem-linkasF} we can therefore define the \textit{evaluation} $\mathrm{ev}(L_D)$ of a colored, oriented link diagram $L_D$ to be
\[
\mathrm{ev}(L_D)=\sum_{j=1}^{t}(-1)^{\mathrm{sgn}_j}q^{\mathrm{sh}_j}\mathrm{ev}(w_j),
\]
where $\mathrm{ev}(w_j)$ denotes our evaluation algorithm from Definition \ref{defn-evaluation}.

\begin{thm}\label{thm-evoflinks}
Let $L_D$ be a colored, oriented link diagram. The evaluation $\mathrm{ev}(L_D)$ is invariant under the second and third Reidemeister moves and isotopies. Moreover,
\[
\mathrm{ev}(L_D)=\langle L_D\rangle_n,
\]
i.e. the evaluation algorithm gives the colored Reshetikhin--Turaev $\mathfrak{gl}_n$-link polynomial. The normalized colored Reshetikhin--Turaev $\mathfrak{gl}_n$-link polynomial can be obtained by a shift.
\end{thm}

\begin{proof}
This is only an assembling of pieces: the claim follows from Theorem \ref{thm-evaluation} and Corollary \ref{cor-kupismoy}.
\end{proof}

There is an alternative way to prove the statement in our setup which we sketch here. Because of Theorem \ref{thm-evaluation} we note that we already have the isotopy invariance. Thus, it suffices to restrict to braids (the braid is oriented upwards). We sketch how to show the invariance for the second Reidemeister move by restricting to the uncolored case $a=b=1$. It will be a consequence of the Serre relations from Definition \ref{def-serre}. The same is true for the uncolored third Reidemeister move as we invite the reader to check. The invariance in the colored case follows in the same vein using the higher Serre relations as in e.g. \cite[Chapter 7]{lu}.

The invariance under the second Reidemeister move in our case can be proven by checking that
\begin{equation}\label{eq-serreinv}
T_{1,1,\mp i+1}T_{1,1,\pm i}v_{\dots,1,1,0,0,\dots}=F_{i+1}F_{i+2}F_iF_{i+1}v_{\dots,1,1,0,0,\dots},\hspace*{0.15cm}\text{with the first }1\text{ in the }i\text{-th entry.}
\end{equation}
Or in pictures (the other possibility can be proven analogously): the move
\[
\scalebox{.7}{$\xy
(0,0)*{\includegraphics[scale=.75]{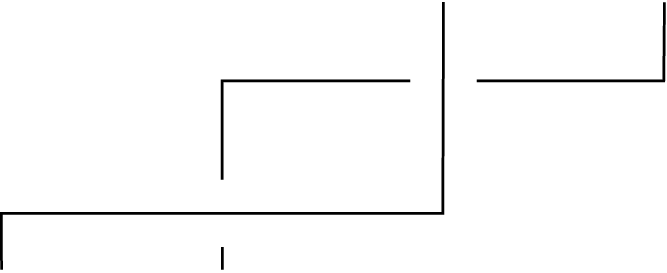}};
(0,10)*{T_{1,1,-i+1}};
(-27.5,-7.5)*{T_{1,1,+i}};
(-40.5,-16)*{1};
(-12.5,-16)*{1};
(15.5,-16)*{0};
(43.5,-16)*{0};
(-40.5,0)*{0};
(-12.5,0)*{1};
(15.5,0)*{1};
(43.5,0)*{0};
(-40.5,16)*{0};
(-12.5,16)*{0};
(15.5,16)*{1};
(43.5,16)*{1};
\endxy$}
\]
has to be
\[
\scalebox{.7}{$\xy
(0,0)*{\includegraphics[scale=.75]{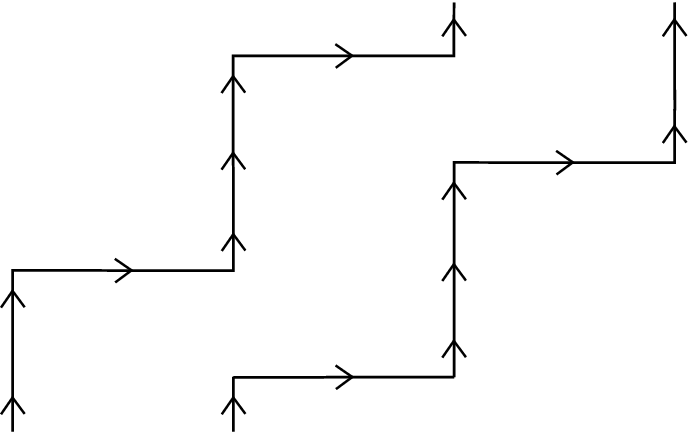}};
(0,24)*{F_{i+1}};
(-27.5,-3.75)*{F_{i}};
(0,-17)*{F_{i+1}};
(27.5,10.5)*{F_{i+2}};
(-40.5,-27)*{1};
(-12.5,-27)*{1};
(15.5,-27)*{0};
(43.5,-27)*{0};
(-40.5,-13.5)*{1};
(-12.5,-13.5)*{0};
(15.5,-13.5)*{1};
(43.5,-13.5)*{0};
(-40.5,0)*{0};
(-12.5,0)*{1};
(15.5,0)*{1};
(43.5,0)*{0};
(-40.5,13.5)*{0};
(-12.5,13.5)*{1};
(15.5,13.5)*{0};
(43.5,13.5)*{1};
(-40.5,27)*{0};
(-12.5,27)*{0};
(15.5,27)*{1};
(43.5,27)*{1};
\endxy$}.
\]
Factoring the left side of \eqref{eq-serreinv} using the definition from \eqref{eq-posbraid} gives the term (we use $v=v_{\dots,1,1,0,0,\dots}$)
\[
(F_{i+1}F_{i+2}F_{i}F_{i+1}-q^{+1}\cdot F_{i+1}F_{i+2}F_{i+1}F_{i}-q^{-1}\cdot F_{i+2}F_{i+1}F_{i}F_{i+1}+ F_{i+2}F_{i+1}F_{i+1}F_{i})v.
\]
Therefore, it suffices to show that
\[
F_{i+2}F_{i+1}F_{i+1}F_{i}v\stackrel{!}{=}q^{+1}\cdot F_{i+1}F_{i+2}F_{i+1}F_{i}v+q^{-1}\cdot F_{i+2}F_{i+1}F_{i}F_{i+1}v.
\]
Since $F^2_{i+1}v=0$, we see by using the Serre relations on the right three $F$ that
\[
q^{-1}\cdot F_{i+2}F_{i+1}F_{i}F_{i+1}v=\frac{q^{-1}}{[2]}\cdot F_{i+2}F^2_{i+1}F_{i}v.
\]
Using the Serre relations on the three left $F$ of the other term gives
\[
\frac{q^{+1}}{[2]}\cdot F_{i+2}F^2_{i+1}F_{i}v+\frac{q^{-1}}{[2]}\cdot F_{i+2}F^2_{i+1}F_{i}v=F_{i+2}F^2_{i+1}F_{i}v=F_{i+2}F_{i+1}F_{i+1}F_{i}v.
\]
The other cases follow similar.

\subsubsection{Two examples}\label{subsub-twoexamples}
Since empty shifts do not change anything interesting, we sometimes do not use them in the following, e.g. in order to go from the highest to the lowest weight one would have to do empty shifts at the end to order all non-zero entries to the right.

\begin{ex}\label{ex-linkasF}
Let us consider a certain diagram of the unknot $U_D$ as such a sum of $F_i^{(j)}$. Here we use $n=2$ and strands are only colored with color $1$. Note that this example belongs to Example \ref{ex-evaluation}.
\[
\scalebox{.7}{$\xy
(0,0)*{\rotatebox{270}{\includegraphics[scale=.75]{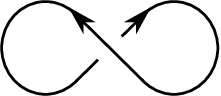}}};
(8,10)*{1};
(8,-10)*{1};
\endxy$}
\hspace*{0.25cm}
\rightsquigarrow
\hspace*{0.25cm}
\scalebox{.7}{$\xy
(0,0)*{\includegraphics[scale=.75]{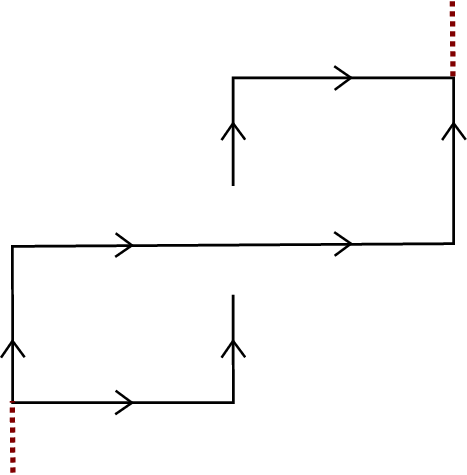}};
(-26,-28)*{2};
(2,-28)*{0};
(29.5,-28)*{0};
(-14,-17.5)*{F_1};
(-26,-8.5)*{1};
(2,-8.5)*{1};
(29.5,-8.5)*{0};
(-8,2)*{T_{1,1,2}};
(-26,8.5)*{0};
(2,8.5)*{1};
(29.5,8.5)*{1};
(14,23)*{F_2};
(-26,28)*{0};
(2,28)*{0};
(29.5,28)*{2};
\endxy$}.
\]
Hence, we can write the unknot as (beware that it has an undercrossing)
\[
U_D=F_2T_{1,1,2}F_1v_{(2^1)}=qF_2F_2F_1F_1v_{(2^1)}-F_2F_1F_2F_1v_{(2^1)}.
\]
We should note that we are cheating a little bit here, since, if we would strictly follow the algorithm, then we would have to 
rewrite the right pointing crossing as in Definition \ref{defn-colbraid} and we would get
\begin{align*}
U_D &=F_4F_2F_1T_{1,1,2}F_1F_4F_3F^{(2)}_2v_{(2^2)}\\ &=qF_4F_2F_1F_2F_3F_1F_4F_3F^{(2)}_2v_{(2^2)}-F_4F_2F_1F_3F_2F_1F_4F_3F^{(2)}_2v_{(2^2)}.
\end{align*}

Hence, as we have already calculated in Example \ref{ex-evaluation} before, the left summand gives four $2$-multitableaux of degrees $2,0,-2$ and the right summand two of degrees $1,-1$. Thus,
\[
\mathrm{ev}(U_D)=q(q^2+2+q^{-2})-(q+q^{-1})=q^{3}+q=q^{2}[2],
\]
which is, up to a normalization, the polynomial $[2]$ of the trivial diagram. The normalization factor given in Definition \ref{defn-cRTpoly} is indeed $q^{-2}=q^{2\mathrm{w}(U_D)}$ in this case.
\end{ex}

\begin{ex}\label{ex-Hopf}
A more demanding, but also more interesting, example is the Hopf link given below. Our space here is limited, so we only sketch the calculation.
\[
\scalebox{.7}{$\xy
(0,0)*{\includegraphics[scale=.75]{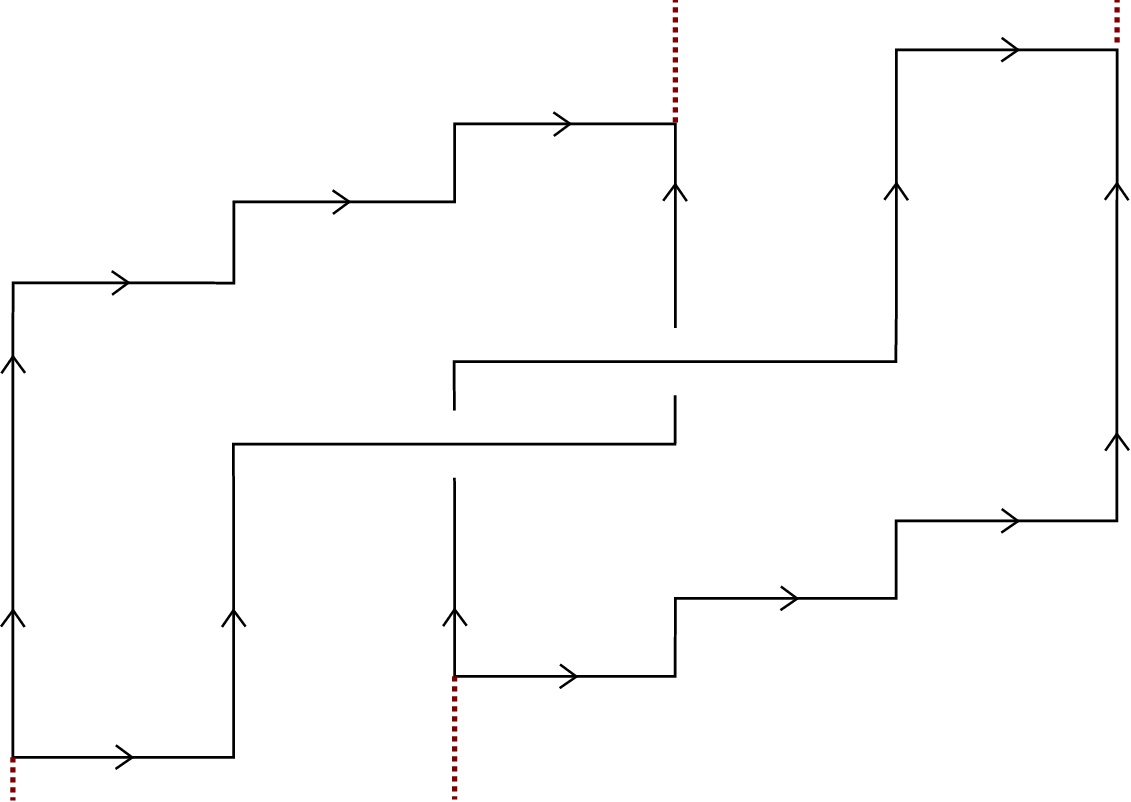}};
(-68,-50)*{3};
(-40,-50)*{0};
(-12,-50)*{3};
(16,-50)*{0};
(44,-50)*{0};
(72,-50)*{0};
(-56,-42)*{F_1};
(-68,-40)*{2};
(-40,-40)*{1};
(-12,-40)*{3};
(16,-40)*{0};
(44,-40)*{0};
(72,-40)*{0};
(0,-32)*{F_3};
(-68,-30)*{2};
(-40,-30)*{1};
(-12,-30)*{2};
(16,-30)*{1};
(44,-30)*{0};
(72,-30)*{0};
(28,-22)*{F_4};
(-68,-20)*{2};
(-40,-20)*{1};
(-12,-20)*{2};
(16,-20)*{0};
(44,-20)*{1};
(72,-20)*{0};
(56,-12)*{F_5};
(-68,-10)*{2};
(-40,-10)*{1};
(-12,-10)*{2};
(16,-10)*{0};
(44,-10)*{0};
(72,-10)*{1};
(-27.5,-1.5)*{T_{1,2,2}};
(-68,0)*{2};
(-40,0)*{0};
(-12,0)*{2};
(16,0)*{1};
(44,0)*{0};
(72,0)*{1};
(0.5,8.5)*{T_{2,1,3}};
(-68,10)*{2};
(-40,10)*{0};
(-12,10)*{0};
(16,10)*{1};
(44,10)*{2};
(72,10)*{1};
(-55,18.5)*{F^{(2)}_1};
(-68,20)*{0};
(-40,20)*{2};
(-12,20)*{0};
(16,20)*{1};
(44,20)*{2};
(72,20)*{1};
(-27,28.5)*{F^{(2)}_2};
(-68,30)*{0};
(-40,30)*{0};
(-12,30)*{2};
(16,30)*{1};
(44,30)*{2};
(72,30)*{1};
(1,38.5)*{F^{(2)}_3};
(-68,40)*{0};
(-40,40)*{0};
(-12,40)*{0};
(16,40)*{3};
(44,40)*{2};
(72,40)*{1};
(57,48.5)*{F^{(2)}_5};
(-68,50)*{0};
(-40,50)*{0};
(-12,50)*{0};
(16,50)*{3};
(44,50)*{0};
(72,50)*{3};
\endxy$}.
\]
In this case we want to calculate the colored $\mathfrak{gl}_3$-link polynomial using $\dot{\Uu}_q(\mathfrak{gl}_6)$-weight representations, i.e. $n=3$ and $m=6$. Moreover, the colors are illustrated above, that is we have the two braiding operators $T_{1,2,2}$ (bottom) and $T_{2,1,3}$ (top). Thus, we choose the orientations of the Hopf link to point upwards, i.e. both crossings should be $\xy(0,0)*{\includegraphics[scale=0.4]{figs/linkpoly/positive.eps}};\endxy$. Both of them correspond to two summands. Thus, we have four summands in total. Moreover, we see that (in the picture above we skipped the empty shift $F_2^{(3)}$ at the bottom and we can ignore the empty shift at the top)
\[
\mathrm{Hopf}=F_5^{(2)}F_3^{(2)}F_2^{(2)}F_1^{(2)}T_{2,1,3}T_{1,2,2}F_5F_4F_3F_1F_2^{(3)}v_{(3^2)}.
\]
The first operator gives the summands $-q^{-1}F_2F_3\mathbbm{1}$ and $\mathbbm{1}F_3F_2$ and the top gives $-q^{-1}F_4F_3^{(2)}F_4$ and $F_4^{(2)}F_3^{(2)}\mathbbm{1}$. Recall from Definition \ref{defn-colbraid} that the braiding operators will have three terms $F_i^{(j)}$ and we have indicated the trivial one by $\mathbbm{1}$. Or in local pictures:
\[
F_2F_3\mathbbm{1}:
\scalebox{.7}{$\xy
(0,0)*{\includegraphics[scale=.75]{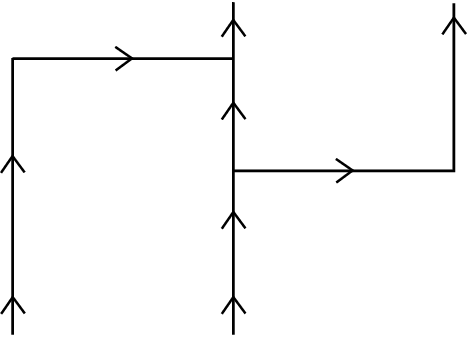}};
(-26,-20)*{1};
(2,-19.5)*{2};
(30,-19.5)*{0};
(-14,17)*{F_{2}};
(-26,-6.5)*{1};
(2.5,-7)*{2};
(30,-6.5)*{0};
(14,3)*{F_{3}};
(-26,6.5)*{1};
(2.5,6.5)*{1};
(30,6)*{1};
(-26,19.5)*{0};
(2,19.5)*{2};
(30,19)*{1};
\endxy$}
\quad
\text{ and }
\quad
F_4F_3^{(2)}F_4:
\scalebox{.7}{$\xy
(0,0)*{\includegraphics[scale=.75]{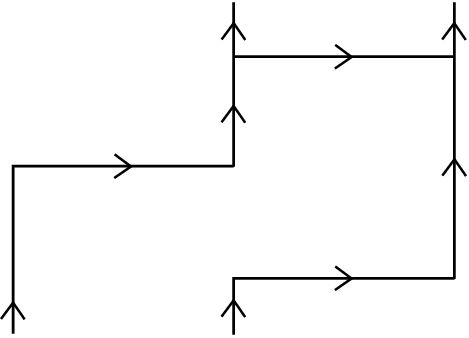}};
(14,17.5)*{F_{4}};
(-26,-20)*{2};
(2,-19.5)*{1};
(30,-19.5)*{0};
(-14,4.5)*{F_{3}^{(2)}};
(-26,-6.5)*{2};
(2.5,-7)*{0};
(30,-6.5)*{1};
(-26,6.5)*{0};
(2.5,6.5)*{2};
(30,6)*{1};
(14,-11.5)*{F_{4}};
(-26,19.5)*{0};
(2,19.5)*{1};
(30,19)*{2};
\endxy$}.
\]
The other two look like the right case in Example \ref{ex-braidop} with different numbers. We now follow the algorithm to generate for each of the four possibilities the sets of $3$-multitableaux. Note that the string of $F_i^{(j)}$ before the first braiding operator, denoted by $F_b$, opens two components that will eventually connect later. Both correspond to three possible flows and the evaluation algorithm will generate nine $3$-multitableaux. They will be
\[
\left(\;\xy (0,0)*{\begin{Young}1& - & - & -\cr $\cdot$\cr\end{Young}}\endxy\;,\;\xy (0,0)*{\begin{Young}1& - & - & -\cr $\cdot$\cr\end{Young}}\endxy\;,\;\xy (0,0)*{\begin{Young}1& - & - & -\cr $\cdot$\cr \end{Young}}\endxy\;\right),\hspace*{0.5cm}\hspace*{0.5cm}\left(\;\xy (0,0)*{\begin{Young}1& 3 & 4 & 5\cr\end{Young}}\endxy\;,\;\xy (0,0)*{\begin{Young}1\cr \end{Young}}\endxy\;,\;\xy (0,0)*{\begin{Young}1\cr 2\cr \end{Young}}\endxy\;\right),
\]
where the number $2$ is allowed to appear in the node marked $\cdot$ and the numbers $3,4$ and $5$ (in order) are allowed to appear in the nodes marked $-$. An explicit example is illustrated above.

The four possibilities how the two braiding operators can be composed will kill some of them and create new ones while we follow the evaluation algorithm. For example, the evaluation of $F_2F_3F_bv_{(3^2)}$ will raise this number to twelve $3$-multitableaux because the $F_3$ can be place in two different positions for each of the nine $3$-multitableaux. But the $F_2$ will kill some of them, since to place a node of residue $2$ is only possible if we see a hook. For example, the left of the possible two extensions of the upper right example does not have such a hook, while the right one has
\[
\left(\;\xy (0,0)*{\begin{Young}1& 3 & 4 & 5\cr\end{Young}}\endxy\;,\;\xy (0,0)*{\begin{Young}1 & 6\cr \end{Young}}\endxy\;,\;\xy (0,0)*{\begin{Young}1\cr 2\cr \end{Young}}\endxy\;\right),\hspace*{0.5cm}\hspace*{0.5cm}\left(\;\xy (0,0)*{\begin{Young}1& 3 & 4 & 5\cr\end{Young}}\endxy\;,\;\xy (0,0)*{\begin{Young}1\cr \end{Young}}\endxy\;,\;\xy (0,0)*{\begin{Young}1 & 6\cr 2 & 7\cr \end{Young}}\endxy\;\right).
\]
If we extend the string now by $F_4^{(2)}F_3^{(2)}$, then we see that the first will kill most of the possibilities. For example it is not possible to add two nodes of residue $3$ to the right $3$-multitableaux above. This is due to the fact that $F_3^{(2)}$ corresponds to a cap. The $F_4^{(2)}$, which corresponds to a cup, will then create new possibilities. Following this process to the end and calculate the degrees we see that we will get
\[
\langle\mathrm{Hopf}\rangle_3=q^{-2}[2]^2[3]-2q^{-1}[2][3]+[3]^2,
\]
which is the corresponding colored quantum polynomial.
\end{ex}
\section{Its categorification}\label{sec-cat}
\subsection{A cellular basis for matrix factorizations}\label{sec-catpart1}

\subsubsection{The dotted identities}\label{subsub-dotidem}
We start by giving the definition of the idempotent for $\vec{\lambda}$, denoted by $e(\vec{\lambda})$. Recall that we choose and fix $n$ and $\ell$ and that there is a constant $c(\vec{k})$ that only depends on the $\mathfrak{gl}_m$-weight $\vec{k}$. Note that, since $\vec{\lambda}$ corresponds to a state string $\vec{S}$ which includes the $\vec{k}$, the $\vec{\lambda}$ determines $c(\vec{k})$.

\begin{defn}\label{defn-idem}(\textbf{Idempotent associated to $\vec{\lambda}$})
Given an $n$-multipartition $\vec{\lambda}$ with $c(\vec{k})$ nodes filled with non-repeating $k\in\{1,\dots,c(\vec{k})\}$, we can associate to it a certain idempotent, denoted by $e(\vec{\lambda})$, using the following rules. Define a sequence of $F_k$ for $\vec{\lambda}$ by (with $r(\vec{\lambda})$ as in Definition \ref{defn-rsequence})
\[
\mathrm{qH}(\vec{\lambda})= \prod_{k=1}^{c(\vec{k})}F_{r(\vec{\lambda})_k}=F_{r(\vec{\lambda})_{c(\vec{k})}}\dots F_{r(\vec{\lambda})_1},\quad\text{with }r(\lambda)=(r(\lambda)_1,\dots,r(\lambda)_{c(\vec{k})}).
\]
Define a $\mathfrak{gl}_n$-web $u_{\vec{\lambda}}$ to be the $\mathfrak{gl}_n$-web generated by applying $\mathrm{qH}(\vec{\lambda})$ to a highest weight vector $v_{(n^{\ell})}$ (here $\ell$ is as in Definition \ref{defn-tabcomb}) and use $q$-skew Howe duality. Then
\[
e(\vec{\lambda})=\mathrm{id}\colon\widehat{u}_{\vec{\lambda}}\to\widehat{u}_{\vec{\lambda}},
\]
that is, the identity between two copies of the matrix factorization $\widehat{u}_{\vec{\lambda}}$ associated to $u_{\vec{\lambda}}$.
\end{defn}

Recall that $\widehat{t}_{i}$ denote homomorphisms of matrix factorizations corresponding to dots, cf. \eqref{eq-mf-homs}.

\begin{defn}\label{defn-idemdots}(\textbf{Dot placement associated to $\vec{\lambda}$})
Given a $n$-multipartition $\vec{\lambda}$ as in Definition \ref{defn-idem} together with its associated idempotent $e(\vec{\lambda})$, and denote by $m(k)=\mathsf{A}^{^{r\succ N}}(T_{\vec{\lambda}^k})$ the number of addable nodes after the node $N$ with entry $k$ in $T_{\vec{\lambda}^k}$ with the same residue $r$ as the node $N$. We define  $e(\vec{\lambda})d(\vec{\lambda})=e(\vec{\lambda})\circ d(\vec{\lambda})\colon\widehat{u}_{\vec{\lambda}}\to\widehat{u}_{\vec{\lambda}}$, where
\[
d(\vec{\lambda})=\widehat{t}_{c(\vec{k})}^{\phantom{.}m(c(\vec{k}))}\circ\dots\circ \widehat{t}_{1}^{\phantom{.}m(1)}\colon\widehat{u}_{\vec{\lambda}}\to\widehat{u}_{\vec{\lambda}}.
\]
We call it \textit{the dotted identity associated to $\vec{\lambda}$}.
\end{defn}

\begin{lem}\label{lem-welldefidem}
The dotted identity $e(\vec{\lambda})d(\vec{\lambda})$ is always non-zero, and an idempotent if and only if $d(\vec{\lambda})=\mathrm{id}$ and nilpotent otherwise. For all $n$-multipartitions $\vec{\lambda},\vec{\mu}$ we have
\[
e(\vec{\lambda})e(\vec{\mu})=e(\vec{\mu})e(\vec{\lambda})=\delta_{\vec{\lambda},\vec{\mu}}e(\vec{\lambda})=\delta_{\vec{\lambda},\vec{\mu}}e(\vec{\mu}),\text{ with }\delta_{\vec{\lambda},\vec{\mu}}=\begin{cases}1, &\text{if } r(\vec{\lambda})=r(\vec{\mu}),\\ 0,&\text{ otherwise}.\end{cases}
\]
Moreover, we have
\[
e(\vec{\lambda})\circ d(\vec{\lambda})=d(\vec{\lambda})\circ e(\vec{\lambda})\hspace*{0.2cm}\text{ and }\hspace*{0.2cm}d(\vec{\lambda})\circ d(\vec{\mu})=d(\vec{\mu})\circ d(\vec{\lambda}).
\]
That is, the dotted identities for $\vec{\lambda}$ and $\vec{\mu}$ commute.
\end{lem}

\begin{proof}
To see that $e(\vec{\lambda})d(\vec{\lambda})$ is well-defined we need two ingredients. The first ingredient is that we have to make the equivalence $\tilde{\Gamma}$ from \eqref{eq-klrequi} explicit. That is, we are going to argue that $e(\vec{\lambda})d(\vec{\lambda})$ is the image of a certain cyclotomic KLR diagram under $\tilde{\Gamma}$ as illustrated below (the numbers $i$ and colors should illustrate the corresponding $\mathcal{F}_i$).
\[
\tilde{\Gamma}\colon
\xy
(0,0)*{\includegraphics[scale=.75]{figs/intro/HM-strings-mid}};
(15,-4.5)*{\scriptstyle 2};
(6,-4.5)*{\scriptstyle 2};
(-3.5,-4.5)*{\scriptstyle 3};
(-13,-4.4)*{\scriptstyle 1};
\endxy
\mapsto
e\left(\left(\;\xy(0,0)*{\begin{Young}1\cr\end{Young}}\endxy\;,\xy(0,0)*{\begin{Young}2&3\cr 4\cr\end{Young}}\endxy\;\right)\right)d\left(\left(\;\xy(0,0)*{\begin{Young}1\cr\end{Young}}\endxy\;,\xy(0,0)*{\begin{Young}2&3\cr 4\cr\end{Young}}\endxy\;\right)\right),
\]
where the residue sequence of $r(\vec{\lambda})$ is (recall our shift) given by $(2,2,1,3)$ and only the node with entry $1$ has an addable node (the node with entry $2$).

To see that everything works out we need our second ingredient, namely the HM basis from \cite{hm}. More explicitly, we use their definition of the dotted identity given in \cite[Definitions 4.9 and 4.15]{hm}. We denote their diagram associated to $\vec{\lambda}$, by abuse of notation, also by $e(\vec{\lambda})d(\vec{\lambda})$.

We consider now the lift of $e(\vec{\lambda})d(\vec{\lambda})$ to the KLR algebra, i.e. without taking the cyclotomic quotient (and again use the same notation). Then, by comparing their definition to Definition \ref{defn-idemdots}, we see that
\[
\Gamma_{m,n\ell,n}\colon e(\vec{\lambda})d(\vec{\lambda})\mapsto e(\vec{\lambda})d(\vec{\lambda}),
\]
since the action from \eqref{eq-klr} is given explicitly: it sends a dot 
to a dot $\widehat{t}$ and an idempotent as above is sent to the identity between the $\mathfrak{gl}_n$-web that can 
be read off from $r(\vec{\lambda})$.

To see that it is also the image under the cyclotomic equivalence we note that the definition of $\tilde{\Gamma}$ comes from the equivalence given in \cite[Proposition 5.6]{rou}. Comparing Rouquier's definition (beware that he uses lowest weight notation) with our conventions shows that
\[
\tilde{\Gamma}\colon e(\vec{\lambda})d(\vec{\lambda})\mapsto e(\vec{\lambda})d(\vec{\lambda}).
\]
Thus, applying \cite[Corollary 4.16]{hm}, we have that the dotted identity is well-defined and non-zero since $\tilde{\Gamma}$ is faithful. The other statements follow now directly from the corresponding ones in the cyclotomic KLR setting using the equivalence from \eqref{eq-klrequi}.
\end{proof}

\begin{ex}\label{ex-idemfoam}
To give an explicit example assume that $n=4$ and $\ell=2$ and let us consider two $4$-multipartition $\vec{\lambda}=((2,1),(0),(0),(1))$ and $\vec{\mu}=((0),(2,1),(1),(0))$. We have
\[
T_{\vec{\lambda}}=\left(\xy(0,0)*{\begin{Young}1&2\cr3\cr\end{Young}}\endxy\;,\;\emptyset\;,\;\emptyset\;,\;\xy(0,0)*{\begin{Young}4\cr\end{Young}}\endxy \right)\hspace*{0.2cm}\text{ and }\hspace*{0.2cm}T_{\vec{\mu}}=\left(\;\emptyset\;,\xy(0,0)*{\begin{Young}1&2\cr3\cr\end{Young}}\endxy\;,\;\xy(0,0)*{\begin{Young}4\cr\end{Young}}\endxy\;,\;\emptyset\;\right)
.
\]
We get that $r(T_{\vec{\lambda}})=r(T_{\vec{\mu}})=(2,3,1,2)$ and therefore $u=u_{\vec{\lambda}}=u_{\vec{\mu}}$ will be
\[
\scalebox{.7}{$\xy
(0,0)*{\includegraphics[scale=.75]{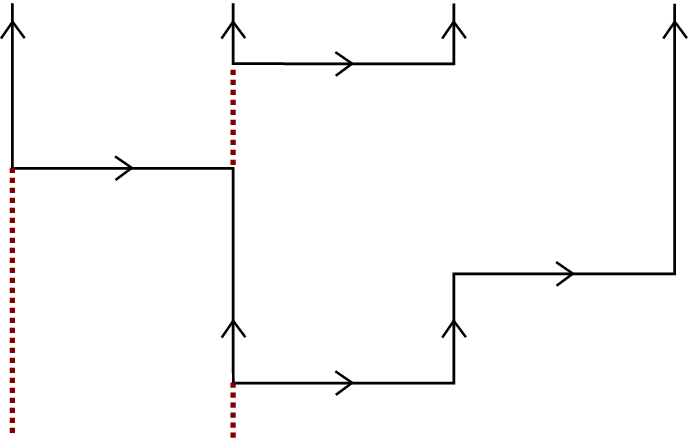}};
(0.5,-17.75)*{F_2};
(28.5,-4)*{F_3};
(0.5,22.75)*{F_2};
(-27.5,9.5)*{F_1};
(-40,-27.5)*{4};
(-12,-27.5)*{4};
(16,-27.5)*{0};
(44,-27.5)*{0};
(-40,-13.75)*{4};
(-12,-13.75)*{3};
(16,-13.75)*{1};
(44,-13.75)*{0};
(-40,0)*{4};
(-12,0)*{3};
(16,0)*{0};
(44,0)*{1};
(-40,13.75)*{3};
(-12,13.75)*{4};
(16,13.75)*{0};
(44,13.75)*{1};
(-40,27.5)*{3};
(-12,27.5)*{3};
(16,27.5)*{1};
(44,27.5)*{1};
\endxy$}.
\]
Its associated matrix factorization is $\widehat{u}=\widehat{F}_{4,2,(3,4,0,1)}\widehat{F}_{3,1,(4,3,0,1)}\widehat{F}_{2,3,(4,3,1,0)}\widehat{F}_{1,2,(4,4,0,0)}$. The idempotent for both $4$-multipartitions is therefore the identity homomorphism $\mathrm{id}\colon\widehat{u}\to\widehat{u}$. But the dot placement will be different, because $\vec{\lambda}$ has only three addable nodes for the first $F_2$, while $\vec{\mu}$ has two addable nodes for the first $F_2$ and one for the second. Thus, we have
\[
e(\vec{\lambda})d(\vec{\lambda})=\widehat{t}_{1}^{\,3}\colon\widehat{u}\to\widehat{u}\hspace*{0.1cm}\text{ and }\hspace*{0.1cm}e(\vec{\mu})d(\vec{\mu})=\widehat{t}^{\phantom{2}}_{4}\widehat{t}_{1}^{\,2}\colon\widehat{u}\to\widehat{u}.
\]
These correspond to three dots on the bottom horizontal ladder respectively to two dots on the bottom horizontal ladder and one on the top horizontal ladder.
\end{ex}

\subsubsection{The symmetric group and homomorphisms of matrix factorizations}\label{subsub-hommat}

\begin{rem}\label{rem-action}
Fix a $n$-multipartition $\vec{\lambda}$ with $c(\vec{k})$ nodes. Recall that the set $\mathrm{Std}(\vec{\lambda})$ denotes the set of all standard fillings of $\vec{\lambda}$. Now the symmetric group $S_{c(\vec{k})}$ makes its appearance because it acts on the subset $\mathrm{Std}_1(\vec{\lambda})$ of all standard fillings where every entry appears just once. The action for the simple transpositions $\tau_k$ is defined by the exchange of $k$ and $k+1$, if possible, and by doing nothing else.

Moreover, $S_{c(\vec{k})}$ acts on the set of strings of $F$ of length $c(\vec{k})$ with a fixed number of occurrences of the $F$ by defining the action of the $k$-th transposition $\tau_k$ by exchanging the neighboring entries $k$ and $k+1$ reading from right to left (as usual). In order to remember the residue as well, we denote a transposition $\tau_k$ that exchanges $F_{i}$ and $F_{i^{\prime}}$ by $\tau_k(i,i^{\prime})$. That is,
\[
\tau_k(i,i^{\prime})(F_{c(\vec{k})} \dots \underbrace{F_{i^{\prime}}F_{i}}_{\text{pos. }k} \dots F_1)=F_{c(\vec{k})}\dots \underbrace{F_{i}F_{i^{\prime}}}_{\text{pos. }k} \dots F_1.
\]
Note that these two actions agree. To see this recall that, by our discussion in Section \ref{sec-tabwebs}, an element of $\mathrm{Std}_1(\vec{\lambda})$ gives rise to a string of $F_i$ by reading the nodes ordered by their number and turn them into a string of $F_i$ by setting the $i$ for the $k$-th (from right to left) $F$ to be the residue of the node with label $k$.

We define $\tau_k(i,i^{\prime})^{*}=\tau_k(i^{\prime},i)$ and $\sigma^*=(\tau_{k_r}(i_r,i_r^{\prime})\dots\tau_{k_1}(i_1,i_1^{\prime}))^*=\tau_{k_1}(i_1^{\prime},i_1)\dots\tau_{k_r}(i_r^{\prime},i_r)$. 
\end{rem}

Before the following definition recall that we have homomorphisms of matrix factorizations $\widehat{CR}$, $\widehat{I}\widehat{D}$, and $\widehat{s}$ corresponding to zipping, cap-cup and shifting, respectively, cf. \eqref{eq-mf-homs}.

\begin{defn}\label{defn-foamLT}(\textbf{Homomorphisms between matrix factorizations}) Given two strings of $F$
\[
\mathrm{qH}_1=\prod_{k=1}^{c(\vec{k})} F_{i_k}=F_{i_{c(\vec{k})}}\dots F_{i_1}\;\;\text{ and }\;\;\mathrm{qH}_2=\prod_{k=1}^{c(\vec{k})} F_{i^{\prime}_k}=F_{i^{\prime}_{c(\vec{k})}}\dots F_{i^{\prime}_1}.
\]
Let $\widehat{u}_1$ and $\widehat{u}_2$ denote the two matrix factorizations that we associate to the corresponding $\mathfrak{gl}_n$-webs $u_1=\mathrm{qH}_1v_{(n^{\ell})}$ and $u_2=\mathrm{qH}_2v_{(n^{\ell})}$.
We assume that $\mathrm{qH}_1$ and $\mathrm{qH}_2$ differ only by a permutation $\sigma\in S_{c(\vec{k})}$ of their $F$ and that $\sigma$ is already decomposed into a string of transpositions
\[
\sigma=\tau_{k_l}\dots\tau_{k_1},
\]
such that $\sigma\cdot\mathrm{qH}_1=\mathrm{qH}_2$. Then we associate to the triple $\mathrm{qH}_1$, $\mathrm{qH}_2$ and $\sigma$ a homomorphism of matrix factorizations
\begin{equation}\label{eq-hommat}
\phi_{\sigma}(\mathrm{qH}_1,\mathrm{qH}_2)\colon \widehat{u}_1\to \widehat{u}_2,\;\;\phi_{\sigma}(\mathrm{qH}_1,\mathrm{qH}_2)=\phi(\tau_{k_l}(i_l,i^{\prime}_l))\circ\dots \circ\phi(\tau_{k_1}(i_1,i^{\prime}_1))\circ\mathrm{Id}_{\widehat{u}_1}
\end{equation}
by composing the identity $\mathrm{Id}_{\widehat{u}_1}$ on $\widehat{u}_1$ from the left with the homomorphisms of matrix factorizations
\[
\phi(\tau_{k_r}(i_r,i^{\prime}_r)) \mapsto \begin{cases}\widehat{CR}_{k_r,i_ri_r\pm 1}, &\text{if }i^{\prime}_r= i_r\pm 1,\\
\widehat{I}_{k_r,i_ri_r}\widehat{D}_{k_r,i_ri_r}, &\text{if }i_r= i^{\prime}_r,\\
\widehat{s}_{k_r,i_ri^{\prime}_r}, &\text{if }|i_r- i^{\prime}_r|>1,\end{cases}
\]
where all the other parts should be the identity.

We note that this depends on the choice of the decomposition of $\sigma$ into transpositions. We choose a certain decomposition in the following. We should point out that this choice only makes sense in a special case where the $n$-multitableaux $\vec{T}_1$ and $\vec{T}_2$ associated to $u_1$ and $u_2$ are of the same shape $\vec{T}_1,\vec{T}_2\in\mathrm{Std}_1(\vec{\lambda})$ for some $\vec{\lambda}$. Moreover, since we always want to factor through an idempotent associated with something ``canonical'', we only fix such a decomposition for the special case where $\vec{T}_2=T_{\vec{\lambda}}$ (recall that $T_{\vec{\lambda}}$ was defined in Definition \ref{defn-dominnancelambda}). 

For a fixed $n$-multipartition $\vec{\lambda}$ and a corresponding $n$-multitableau $\vec{T}\in\mathrm{Std}_1(\vec{\lambda})$, we choose a fixed permutation $\sigma\in S_{c(\vec{k})}$ satisfying 
\[
\sigma\cdot\vec{T}=T_{\vec{\lambda}}
\]
by searching for the lowest $k\in\{1,\dots,c(\vec{k})\}$ such that the node $N$ with entry $k$ in $\vec{T}$ is not the same as the node $N^{\prime}$ with entry $k$ in $T_{\vec{\lambda}}$. Apply a minimal sequence of transpositions until they match and repeat the process until $\sigma\cdot\vec{T}=T_{\vec{\lambda}}$. By construction, the permutation $\sigma\in S_{c(\vec{k})}$ will be of minimal length with respect to the property $\sigma\cdot\vec{T}=T_{\vec{\lambda}}$. We denote the homomorphism of matrix factorizations associated to this permutation $\sigma$ by $\phi_{\sigma}$.
\end{defn}

We point out that this combinatorial construction of the homomorphisms can not be read off
directly from a (cyclotomic) KLR diagram as the following example illustrates.
\[
\xy
(0,0)*{\includegraphics[scale=.75]{figs/intro/HM-strings-midcr}};
(15,-4.5)*{\scriptstyle 2};
(7,-4.5)*{\scriptstyle 2};
(-3.5,-4.5)*{\scriptstyle 3};
(-13,-4.4)*{\scriptstyle 1};
(30,8)*{\left(\;\xy(0,0)*{\begin{Young}2\cr\end{Young}}\endxy\;,\;\xy(0,0)*{\begin{Young}1 & 3\cr 4\cr\end{Young}}\endxy\;\right)};
(30,-8)*{\left(\;\xy(0,0)*{\begin{Young}1\cr\end{Young}}\endxy\;,\;\xy(0,0)*{\begin{Young}2 & 3\cr 4\cr\end{Young}}\endxy\;\right)};
(30,0)*{\uparrow};
\endxy
\]
In the example above the two $\mathfrak{gl}_2$-webs $u_1,u_2$ are the same $u_1=u_2=F_1F_3F_2F_2v_{(2^1)}$ and there is a non-trivial diagram that we can not see by just looking at the boundary. But one can associate different $2$-multitableaux to them, as illustrated above.

This procedure is well-defined, i.e. one does not run into ambiguities and the resulting homomorphism is between $\widehat{u}_1$ and $\widehat{u}_2$, by Remark \ref{rem-action}.

Furthermore, it is easy to see that
\[
\sigma\cdot\vec{T}_1=\vec{T}_2
\Leftrightarrow\sigma^{-1}\cdot\vec{T}_2=\vec{T}_1\;\;\text{for all }\vec{T}_1,\vec{T}_2\in\mathrm{Std}_1(\vec{\lambda}),\;\sigma\in S_{c(\vec{k})}.
\]

\begin{lem}\label{lem-fromcyclklr}
Given the setup from Definition \ref{defn-foamLT}, there exists an element in $R_{\Lambda}$, denoted by $\phi_{\sigma}(\mathrm{qH}_1,\mathrm{qH}_2)$, such that
$\tilde{\Gamma}\colon\phi_{\sigma}(\mathrm{qH}_1,\mathrm{qH}_2)\in R_{\Lambda}\mapsto\phi_{\sigma}(\mathrm{qH}_1,\mathrm{qH}_2)$,
i.e. all elements of the form $\phi_{\sigma}(\mathrm{qH}_1,\mathrm{qH}_2)$ come from $R_{\Lambda}$.
\end{lem}

\begin{proof}
The proof works essentially as the proof of Lemma \ref{lem-welldefidem}, i.e. we show that there exists an element of the KLR part of $\Ucatm$, that we denote again by the same expression, such that
\[
\Gamma_{m,n\ell,n}\colon \phi_{\sigma}(\mathrm{qH}_1,\mathrm{qH}_2)\in \Ucatm\mapsto\phi_{\sigma}(\mathrm{qH}_1,\mathrm{qH}_2).
\]
Comparing again the definition before Lemma 5.4 in \cite{rou} to our convention, we see that this proves the lemma.

The element of the KLR part of $\Ucatm$ is obtained by putting the string of $F$ for $\mathrm{qH}_1$ at the bottom and the one for $\mathrm{qH}_2$ at the top and then draw a diagram consisting of crossings given by the procedure from Definition \ref{defn-foamLT} in between. For example
\[
\xy
(0,0)*{\includegraphics[scale=.75]{figs/intro/HM-strings-topcr2}};
(15,-4.5)*{\scriptstyle 2};
(5.5,-4.5)*{\scriptstyle 3};
(-2.5,-4.5)*{\scriptstyle 2};
(-13,-4.4)*{\scriptstyle 1};
\endxy\leftrightsquigarrow\tau_2(3,2)\colon F_1F_2F_3F_2\to F_1F_3F_2F_2.
\]
This shows the existence of the $\phi_{\sigma}(\mathrm{qH}_1,\mathrm{qH}_2)\in R_{\Lambda}$ we need.
\end{proof}

\subsubsection{The categorified growth algorithm}\label{subsub-catgrowth}
We are now able to give the definition of the categorification of our extended growth algorithm.

To define the basis for the $\mathfrak{gl}_n$-web algebra ${}_vH_n(\vec{k})_u$ for any $\mathfrak{gl}_n$-webs $u$ and $v$ we need to use certain isomorphisms of matrix factorizations between the left and right sides of the square removal \eqref{eq-square1}. That is, we have to go to the thick cyclotomic KLR $\check{R}_{\Lambda}$ from Definition \ref{defn-thickcyclKLR} and have to associate something to the split and merge from Section \ref{subsub-q2alg}. 

Thus, we need to substitute all divided powers $F_i^{(j)}$ in the sequence associated to $u$ by $j$-times $F_i$. This means in pictures that we replace (here $j=2$)
\begin{equation}\label{eq-rest}
\scalebox{.7}{$\xy(0,0)*{\includegraphics[scale=.75]{figs/slnwebs/laddera.eps}};(-11,-4)*{\scriptstyle k_i};(18.5,-4)*{\scriptstyle k_{i+1}};(1,-2)*{\scriptstyle 2};(0.5,4.5)*{F_i^{(2)}};(-9.5,4)*{\scriptstyle k_i-2};(20,4)*{\scriptstyle k_{i+1}+2}\endxy$}\hspace*{0.2cm}\xy(0,3)*{\stackrel{\widehat{I}_i}{\longmapsto}};(0,-3)*{\stackrel{\widehat{D}_i}{\longmapsfrom}};\endxy\hspace*{0.2cm}
\scalebox{.7}{$\xy(0,0)*{\includegraphics[scale=.75]{figs/slnwebs/squarea.eps}};(-11,-11)*{\scriptstyle k_i};(18.5,-11)*{\scriptstyle k_{i+1}};(1.5,5)*{\scriptstyle 1};(0,10)*{F_i};(1.5,-8.5)*{\scriptstyle 1};(0,-3.5)*{F_i};(-9,0)*{\scriptstyle k_i-1};(20,0)*{\scriptstyle k_{i+1}+1};(-9.5,11)*{\scriptstyle k_i-2};(20,11)*{\scriptstyle k_{i+1}+2}\endxy$}.
\end{equation}
The morphisms $\widehat{I}_i$ and $\widehat{D}_i$ are not isomorphisms and are of $q$-degree $-1$, as the split and merge. Thus, we have to choose: starting with a flow on the right picture in \eqref{eq-rest}, we choose one flow for the left. Our choice will ensure that the whole process preserves the $q$-degree because the chosen flow will be of weight one lower than the starting flow. The precise definitions of $\widehat{I}_i$ and $\widehat{D}_i$ are not important for us (and long) and can be found for example in \cite[Definitions 8.11 and 8.12]{my}.

For $j>2$ we do literally the same, but use the image under $\check{\Gamma}_{m,n\ell,n}$ (see Theorem \ref{thm-thick}) of the splits $\mathcal{F}_i^{(j^{\prime})}\to \mathcal{F}_i^{(j^{\prime}-1)}\mathcal{F}_i$ repeatedly. We denote them by $\widehat{I}_i^{j^{\prime}}$. These are of degree $-{j^{\prime}}+1$. Thus, the full split is of degree $-(j+j-1+j-2+\dots+1)$. Our choice in this case will ensure that the whole process preserves the $q$-degree because the chosen flow will be of weight $j+j-1+j-2+\dots+1$ lower than the starting flow.

\begin{defn}\label{defn-foamLT2}(\textbf{Homomorphism of matrix factorizations for $\mathfrak{gl}_n$-webs $u_f$ with a flow})
Given a $\mathfrak{gl}_n$-web with a flow $u_f$, we associate to it a homomorphism of matrix factorizations
\[
\phi_{u_f}\colon \widehat{u}_f\to \widehat{u}_{\vec{\lambda}},
\]
where $\vec{\lambda}$ is the boundary datum/$n$-multipartition and $\widehat{u}_{\vec{\lambda}}$ is as in 
Definition \ref{defn-idem}, in the following way.

Change the $n$-multitableau $\iota(u_f)$ by replacing the lowest multiple entry $k$ of multiplicity $j_k$ of $\iota(u_f)$ increasing from left to right with consecutive numbers $k,\dots,k+j_k$ and shift all other entries by $j_k$. Repeat until no multiple entries occur and obtain $\iota(u_f)^{\prime}$. Set
\[
\phi_{u_f}=\phi_{\sigma}(\iota(u_f)^{\prime},T_{\vec{\lambda}})\circ \phi_{\mathrm{R}}\colon \widehat{u}_f\to \widehat{u}_{\vec{\lambda}},
\]
with $\phi_{\sigma}(\iota(u_f)^{\prime},T_{\vec{\lambda}})$ for the strings of $F_i$ $\mathrm{qH}_{1,2}$ corresponding to $\iota(u_f)^{\prime}$ and $T_{\vec{\lambda}}$ respectively.

The homomorphism $\phi_{\mathrm{R}}$ is given by composing an appropriate number of the $\widehat{I}$ from below. That is, the difference between the two corresponding $\mathfrak{gl}_n$-webs is $\dots F_i^{(j)}\dots$ for $\widehat{u}_f$ and $\dots F_i\dots F_i\dots$ for $\phi_{\sigma}(\iota(u_f)^{\prime},T_{\vec{\lambda}})$ which are replaced inductively by $\widehat{I}_i^{j^{\prime}}$: the order does not matter by the associativity of splits (see \cite[Proposition 2.2.4]{klms}) combined with Theorem \ref{thm-thick}.
\end{defn}

\begin{lem}\label{lem-welldef}
There is a diagram in $\Ucatmc$, denoted by the same symbol, such that
\[
\check{\Gamma}_{m,n\ell,n}\colon\phi_{u_f}\mapsto \phi_{u_f},
\]
where $\check{\Gamma}_{m,n\ell,n}$ is the extended functor from Theorem \ref{thm-thick}.
\end{lem}

\begin{proof}
It follows from our construction that (the color should be $i$)
\[
\check{\Gamma}_{m,n\ell,n}\colon
\xy
(0,0)*{\includegraphics[width=25px]{figs/higherstuff/split}};
(-4,2)*{\scriptstyle 1};
(4,2)*{\scriptstyle 1};
(3,-2)*{\scriptstyle 2};
\endxy\mapsto \widehat{I}_i\colon F_i^{(2)}\to F_i^{1}F_i^{1}.
\]
This induces maps for all the thick splits and shows that $\phi_R$ comes from a diagram in $\Ucatmc$. We can use Lemma \ref{lem-fromcyclklr} to see that $\phi_{\sigma}(\iota(u_f)^{\prime},T_{\vec{\lambda}})$ comes from a diagram in $\Ucatm$. Combining both we obtain the statement.
\end{proof}

We are now able go state a \textit{growth algorithm for homomorphism of matrix factorizations} which gives rise to a graded cellular basis.

\begin{defn}\label{defn-growthfoam}(\textbf{Growth algorithm for homomorphisms of matrix factorizations}) Let us denote by $B(W_n(\vec{k}))$ any monomial basis of the $\mathfrak{gl}_n$-web space $W_n(\vec{k})$. We denote by $\tilde{B}(W_n(\vec{k}))$ the set of all basis elements together with a choice of a flow line. We note that, because we have chosen a basis, none of the $\mathfrak{gl}_n$-webs in $\tilde{B}(W_n(\vec{k}))$ will be isotopic.

Given a state string $\vec{S}$, the corresponding $n$-multipartition $\vec{\lambda}$ and $u_f,v_{f^{\prime}}\in \tilde{B}(W_n(\vec{k}))$ . We define a homomorphism following Definition \ref{defn-foamLT2} by
\[
\mathcal{F}^{\vec{\lambda}}_{\iota(v_{f^{\prime}}), \iota(u_f)}\colon \widehat{u}\to \widehat{v},\hspace*{0.5cm}
\mathcal{F}^{\vec{\lambda}}_{\iota(v_{f^{\prime}}), \iota(u_f)}=\phi_{v_{f^{\prime}}}^*e(\vec{\lambda})d(\vec{\lambda})\phi_{u_f},
\]
where the ${}^*$ for $\phi_{v_{f^{\prime}}}$ is defined as
\[
\phi_{v_{f^{\prime}}}^*=(\phi_{v_{f^{\prime}}})^*=(\phi_{\sigma}(\iota(v_{f^{\prime}})^{\prime},T_{\vec{\lambda}})\circ \phi_{\mathrm{R}})^*=\phi_{\mathrm{R}}^*\circ \phi_{\sigma^{*}}(T_{\vec{\lambda}},\iota(v_{f^{\prime}})^{\prime}).
\]
Here the $\phi_{\mathrm{R}}^*$ consists of $\widehat{D}_i^{j^{\prime}}$ going in the other direction than the corresponding $\widehat{I}_i^{j^{\prime}}$, see \eqref{eq-rest}.
\end{defn}

\begin{lem}\label{lem-welldef2}
There is a diagram in $\Ucatmc$, denoted by the same symbol, such that
\[
\check{\Gamma}_{m,n\ell,n}\colon\mathcal{F}^{\vec{\lambda}}_{\iota(v_{f^{\prime}}), \iota(u_f)}\mapsto \mathcal{F}^{\vec{\lambda}}_{\iota(v_{f^{\prime}}), \iota(u_f)}.
\]
Moreover, if $\iota(u_f)=\iota(u_f)^{\prime}$ and $\iota(v_{f^{\prime}})=\iota(v_{f^{\prime}})^{\prime}$, then there is an element of the HM basis of $R_{\Lambda}$, denoted by the same symbol, such that
\[
\tilde{\Gamma}\colon\mathcal{F}^{\vec{\lambda}}_{\iota(v_{f^{\prime}}), \iota(u_f)}\mapsto \mathcal{F}^{\vec{\lambda}}_{\iota(v_{f^{\prime}}), \iota(u_f)}.
\]
This element is completely determined by $u_f,v_{f^{\prime}}$ in the sense that changing either the $\mathfrak{gl}_n$-webs or the flows will give another element of the HM basis.
\end{lem}

\begin{proof}
The first and second statement are just combinations of Lemmas \ref{lem-welldefidem}, \ref{lem-fromcyclklr} and \ref{lem-welldef}. The third statement follows from our translation in Section \ref{sec-tabwebs}, i.e. the HM basis element $\psi^{\vec{\lambda}}_{\vec{T}^{\prime},\vec{T}}$ (see Definition \ref{defn-hmbasis} or, with a slightly different notation, Definition 5.1 in \cite{hm}) with the datum
\[
(\vec{\lambda},\vec{T}=\iota(u_f)^{\prime}\in\mathrm{Std}(\vec{\lambda}),\vec{T}^{\prime}=\iota(v_{f^{\prime}})^{\prime}\in\mathrm{Std}(\vec{\lambda}))
\]
will be the one for $\mathcal{F}^{\vec{\lambda}}_{\iota(v_{f^{\prime}}), \iota(u_f)}$.
\end{proof}

\begin{rem}\label{rem-hmbasis}
One can show analogously as the author has done in Lemma 4.15 of \cite{tub3} that
\[
\mathrm{deg}_q(\mathcal{F}^{\vec{\lambda}}_{\iota(v_{f^{\prime}}), \iota(u_f)})=\mathrm{deg}_{\mathrm{wt}}(u_{f})+\mathrm{deg}_{\mathrm{wt}}(v_{f^{\prime}})=\mathrm{deg}_{\mathrm{BKW}}(\iota(u_f))+\mathrm{deg}_{\mathrm{BKW}}(\iota(v_{f^{\prime}})).
\]
The main ingredient is of course the translation from Proposition \ref{prop-extgrowthdeg}. The reader should be careful, because the homomorphisms $\phi_R$ are not of degree zero. But our convention to obtain $\iota(u_f)^{\prime}$ from $\iota(u_f)$ ensures that the shift of degree is exactly the difference of the degrees of $\iota(u_f)^{\prime}$ and $\iota(u_f)$.
\end{rem}

\begin{ex}\label{ex-foam}
As a small example consider the $\mathfrak{gl}_4$-web from Example \ref{ex-idemfoam}. We use the growth algorithm to generate elements of ${}_uH_n(\vec{k})_u$. Note that we have to choose a flow in order to give an example and, in this very special case, the flow only depends on its boundary datum. Thus, everything will be symmetric and the flows can be read off from the cut line.
The two flows that belong to $T_{\vec{\lambda}}$ and $T_{\vec{\mu}}$ are given by $S_{\vec{\lambda}}=(\{3,2,1\},\{4,3,2\},\{1\},\{4\})$ and $S_{\vec{\mu}}=(\{4,2,1\},\{4,3,1\},\{2\},\{3\})$ respectively. In these two cases the corresponding elements are just given by the dotted identities from Example \ref{ex-idemfoam}, since we do not have to let the symmetric group $S_4$ act on the $4$-multitableaux.

The flow $f$, given by $S=(\{4,3,2\},\{4,3,1\},\{2\},\{1\})$, on the other hand gives rise to
\[
\iota(u_{f})=\left(\;\emptyset\;,\;\emptyset\;,\xy(0,0)*{\begin{Young}4\cr\end{Young}}\endxy\;,\;\xy(0,0)*{\begin{Young}1 & 2\cr 3\cr\end{Young}}\endxy\;\right)\hspace*{0.2cm}\text{ and }\hspace*{0.2cm}T_{\vec{\lambda}_f}=\left(\;\emptyset\;,\;\emptyset\;,\xy(0,0)*{\begin{Young}1\cr\end{Young}}\endxy\;,\;\xy(0,0)*{\begin{Young}2 & 3\cr 4\cr\end{Young}}\endxy\;\right).
\]
Thus, the permutation $\tau_1(2,2)\tau_2(3,2)\tau_3(1,2)$ gives $\tau_1(2,2)\tau_2(3,2)\tau_3(1,2)\cdot\iota(u_{f})=T_{\vec{\lambda}_f}$. In this case we see, since $\phi_R=\mathrm{id}$ and $e(\vec{\lambda}_f)d(\vec{\lambda}_f)=\widehat{t}_{1}$, that
\[
\mathcal{F}^{\vec{\lambda}_f}_{\iota(u_f),\iota(u_f)}=\widehat{CR}_{3,21}\circ \widehat{CR}_{2,23}\circ \widehat{I}_{2,22}\widehat{D}_{2,22}\circ\widehat{t}_{1}\circ\widehat{I}_{2,22}\widehat{D}_{2,22}\circ
\widehat{CR}_{3,32}\circ\widehat{CR}_{3,12}\colon\widehat{u}\to\widehat{u}
\]
where the degree is $\mathrm{deg}(\mathcal{F}^{\vec{\lambda}_f}_{\iota(u_f),\iota(u_f)})=2=\mathrm{deg}_{\mathrm{BKW}}(\iota(u_f))+\mathrm{deg}_{\mathrm{BKW}}(\iota(u_f))$. Moreover, we invite the reader to verify that the corresponding element in the (thick) cyclotomic KLR is
\[
\xy
(0,0)*{\includegraphics[scale=.75]{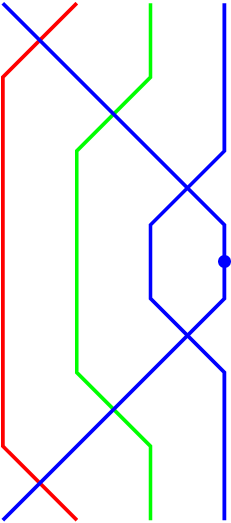}};
(15.5,0)*{\scriptstyle 2};
(6,0)*{\scriptstyle 2};
(-3.5,0)*{\scriptstyle 3};
(-13,0)*{\scriptstyle 1};
(25,0)*{\longleftrightarrow};
(40,0)*{e(\vec{\lambda}_f)d(\vec{\lambda}_f)};
(25,9)*{\longleftrightarrow};
(37.5,9)*{\tau_1(2,2)};
(25,-9)*{\longleftrightarrow};
(37.5,-9)*{\tau_1(2,2)};
(25,18)*{\longleftrightarrow};
(37.5,18)*{\tau_2(2,3)};
(25,-18)*{\longleftrightarrow};
(37.5,-18)*{\tau_2(3,2)};
(25,27)*{\longleftrightarrow};
(37.5,27)*{\tau_3(2,1)};
(25,-27)*{\longleftrightarrow};
(37.5,-27)*{\tau_3(1,2)};
\endxy
\]
\end{ex}

\subsubsection{It is a basis}\label{subsub-basis}
We are now able to prove that the growth algorithm given in Definition \ref{defn-growthfoam} gives a basis of the $\mathfrak{gl}_n$-web algebra ${}_vH_n(\vec{k})_u$. The main ingredients are the results from Section \ref{sec-tabwebs}. 

\begin{thm}\label{thm-foambasis}
The growth algorithm from Definition \ref{defn-growthfoam} gives a homogeneous basis of ${}_vH_n(\vec{k})_u$.
\end{thm}

\begin{proof}
We will show that the growth algorithm gives a linear independent set denoted by
\[
\mathfrak{F}=\{\mathcal{F}^{\vec{\lambda}}_{\iota(v_{f^{\prime}}), \iota(u_f)}\in {}_vH_n(\vec{k})_u\mid\,(\vec{S},u_{f},v_{f^{\prime}}),\,\vec{S}\text{ is a state string},\,u_{f},v_{f^{\prime}}\in \tilde{B}(W_n(\vec{k}))\}.
\]
By a counting argument, which heavily relies on the translation from Section \ref{sec-tabwebs}, we see that this set has the right cardinality, since we know that the set of all triples
\[
(\vec{\lambda},\iota(u_f)\in\mathrm{Std}(\vec{\lambda}),\iota(v_{f^{\prime}})\in\mathrm{Std}(\vec{\lambda}))
\]
has the same size as a possible basis of ${}_vH_n(\vec{k})_u$: the set of all possible flows on $v^*u$ has the same size as $\mathrm{dim}(\mathrm{EXT}(\widehat{u},\widehat{v}))$ since the Euler form $\mathrm{dim}(\mathrm{EXT}(\cdot,\cdot))$ categorifies the Kuperberg bracket (which can be deduced from \cite[Sections 6 to 11]{wu} or \cite[Section 3]{yo1}, i.e. that matrix factorizations satisfy the $\mathfrak{gl}_n$-web relations). Hence, we conclude that the linear independence of $\mathfrak{F}$ suffices to show that the set $\mathfrak{F}$ forms a basis.

We want to consider the additive equivalence of $2$-categories $\tilde{\Gamma}$ from Theorem \ref{thm-qhowe}. The argument goes as follows. The linear independence of the set
\[
\mathfrak{F}^{\prime}=\{\mathcal{F}^{\vec{\lambda}}_{\iota(v_{f^{\prime}})^{\prime},\iota(u_f)^{\prime}}\in {}_vH_n(\vec{k})_u\mid\,(\vec{S},u_{f},v_{f^{\prime}}),\,\vec{S}\text{ is a state string},\,u_{f},v_{f^{\prime}}\in \tilde{B}(W_n(\vec{k}))\},
\]
that is, without the removals $\phi_R$, suffices to show that $\mathfrak{F}$ is also linear independent. To see this note that the homomorphisms from \eqref{eq-rest} give rise to an isomorphism between the left side and a $q$-shifted sum of the right side (they correspond to the splitters and merges and the isomorphism can be verified as in \cite[Theorem 5.1.1]{klms}). Our choice of $\phi_R$ is a restriction of this isomorphism to a certain summand (and forget the $q$-degree shift).

But the set $\mathfrak{F}^{\prime}$ comes, by our translation from Section \ref{sec-tabwebs} and Lemma \ref{lem-welldef2}, directly from a (usually strict!) subset $\mathfrak{F}^{\prime}_{\mathrm{HM}}$ of the HM basis in some cyclotomic KLR algebra, i.e. we have
\[
\tilde{\Gamma}(\mathfrak{F}^{\prime}_{\mathrm{HM}})=\mathfrak{F}^{\prime}\quad\text{and}\quad|\mathfrak{F}^{\prime}_{\mathrm{HM}}|=|\mathfrak{F}^{\prime}|.
\]
Since $\tilde{\Gamma}$ is an additive equivalence of $2$-categories and all subsets of the HM basis are linear independent, we see that $\mathfrak{F}^{\prime}$ has to be linear independent, too.

Hence, the set $\mathfrak{F}$ is linear independent and therefore, by the counting argument mentioned above, also spanning, i.e. it is a basis. This basis is clearly homogeneous by our construction as a composition of some generators of a certain degree.
\end{proof}
We immediately obtain the following corollary, since
\[
H_n(\vec{k})=\bigoplus_{u,v\in B(W_n(\vec{k}))} {}_vH_n(\vec{k})_u\hspace*{0.25cm}\text{ and }\hspace*{0.25cm}H_n(\Lambda)=\bigoplus_{\vec{k}\in\Lambda(m,n\ell)_n}H_n(\vec{k}).
\]

\begin{cor}\label{cor-foambasis}
The growth algorithm gives a homogeneous basis of $H_n(\vec{k})$ and of $H_n(\Lambda)$ respectively.\qed
\end{cor}

In order to connect the $\mathfrak{gl}_n$-web algebras to the thick cyclotomic KLR $\check{R}_{\Lambda}$ , we define
\[
\check{R}(\vec{k})=\bigoplus_{u,v\in B(W_n(\vec{k}))} e(\vec{\lambda}^v_c)\check{R}_{\Lambda}e(\vec{\lambda}^u_c)\hspace*{0.25cm}\text{ and }\hspace*{0.25cm}\check{R}(\Lambda)=\bigoplus_{\vec{k}\in\Lambda(m,n\ell)_n}\check{R}(\vec{k}),
\]
where $\vec{\lambda}^u_c$ denotes the canonical $n$-multipartition (see Definition \ref{defn-canflow}) associated to $u$ and $e(\vec{\lambda}^u_c)$ is the associated idempotent from Lemma \ref{lem-welldefidem}.

\begin{thm}\label{thm-iso}
Let $u,v\in W_n(\vec{k})$ be two $\mathfrak{gl}_n$-webs. Then
\[
e(\vec{\lambda}^v_c)\check{R}_{\Lambda}e(\vec{\lambda}^u_c)\cong {}_vH_n(\vec{k})_u\text{  (graded).}
\]
This gives rise to isomorphisms of graded algebras
\[
\check{R}(\vec{k})\cong H_n(\vec{k}) \hspace*{0.25cm}\text{ and }\hspace*{0.25cm} \check{R}(\Lambda)\cong H_n(\Lambda)
\]
which extends \eqref{eq-klrequi} to an additive equivalence of $2$-categories
\[
\check\Gamma\colon\RcPMOD\to\mathcal W^p_{\Lambda},
\]
i.e. from the category of finite dimensional, $\bZ$-graded, projective $\check{R}_{\Lambda}$-modules to $\mathcal W^p_{\Lambda}$.
\end{thm}

\begin{proof}
This is just an assembling of pieces now. By Lemma \ref{lem-welldef2} the basis of ${}_vH_n(\vec{k})_u$ that we have obtained in Theorem \ref{thm-foambasis} comes from a set of the same size in $\Ucatmc$ via our extension of the categorified q-skew Howe duality from Theorem \ref{thm-thick}. By the faithfulness of $\tilde{\Gamma}$ from Theorem \ref{thm-qhowe} and the fact that the $\phi_R$ come from certain compositions of splitters and merges, we get an inclusion of graded $\bC$-vector spaces
\[
e(\vec{\lambda}^v_c)\check{R}_{\Lambda}e(\vec{\lambda}^u_c)\hookrightarrow {}_vH_n(\vec{k})_u.
\]
Thus, a counting argument can ensure again that they are isomorphic. The graded dimension of the left side is known by Theorem 4.10 in \cite{bk2}. Using our results from Proposition \ref{prop-extgrowthdeg}, we see that the graded dimensions are the same, since the right sides graded dimension (up to a shift) can be obtained by counting all weights of flows on $v^*u$ (as already explained in the proof of Theorem \ref{thm-foambasis}). Thus, we get an isomorphism.

The other statements are now just direct consequences of the first isomorphism. 
\end{proof}

\begin{rem}\label{rem-rewrite}
We should note here (already with the computation method from Section \ref{sec-catpart2} in mind) that it follows from Theorem \ref{thm-iso} that the homomorphisms $\mathcal{F}^{\vec{\lambda}}_{\iota(v_{f^{\prime}}), \iota(u_f)}$ are local in the sense that all their factors satisfy the thick cyclotomic KLR relations. One can use these local relations to re-write the homomorphisms in a (at least for a machine) not too complicated way. A list of these relations can be found in different places, e.g. either using diagrams in \cite{kl1}, \cite{kl3} or as an algebraic list in \cite{hm}. Moreover, a list of local rules for the thick cyclotomic KLR can be deduced from the ones for splitters and merges given in Section 2 of \cite{klms}.
\end{rem}

\begin{rem}\label{rem-involution}
The definition of the ${}^*$ gives rise to an antiinvolution on the $\mathfrak{gl}_n$-web algebra $H_n(\vec{k})$ by Theorem \ref{thm-foambasis} and a small calculation shows that
\[
(\mathcal{F}^{\vec{\lambda}}_{\iota(v_{f^{\prime}}), \iota(u_f)})^*=\mathcal{F}^{\vec{\lambda}}_{\iota(u_f),\iota(v_{f^{\prime}})}.
\]
This is exactly the antiinvolution Mackaay defines before Remark 7.8 in \cite{mack1} using Brundan and Kleshchev's duality on the category of finite dimensional, projective modules of the cyclotomic KLR algebra. His definition is not explicit as Mackaay points out himself. Our definition can, on the other hand, be computed explicitly.
\end{rem}

\subsubsection{Cellularity}\label{subsub-cell}
The basis $\mathfrak{F}$ is a graded cellular basis of $H_n(\vec{k})$. Let us shortly recall the definition which is in the ungraded setting due to Graham and Lehrer \cite{grle} and in the graded setting to Hu and Mathas \cite{hm}.

\begin{defn}(\textbf{Graham--Lehrer, Hu--Mathas})\label{defn-cellular}
Suppose $A$ is a $\bZ$-graded free algebra over $R$ of finite rank. A \textit{$\bZ$-graded cell datum} is an ordered quintuple $(\mathfrak{P},\mathcal{T},C,\mathrm{i},\mathrm{deg})$, where $(\mathfrak P,\rhd)$ is the \textit{weight poset}, $\mathcal{T}(\lambda)$ is a finite set for all $\lambda \in\mathfrak P$, $\mathrm{i}$ is an antiinvolution of $A$ and $C$ is an injection
\[
C\colon\coprod_{\lambda\in\mathfrak P}\mathcal{T}(\lambda)\times \mathcal{T}(\lambda)\to A,\;(s,t)\mapsto c^{\lambda}_{st}.
\]
Moreover, the \textit{degree function} deg is given by
\[
\mathrm{deg}\colon\coprod_{\lambda\in\mathfrak P}\mathcal{T}(\lambda)\to\bZ.
\] 
The whole data should be such that the $c^{\lambda}_{st}$ form a homogeneous $R$-basis of $A$ with $\mathrm{i}(c^{\lambda}_{st})=c^{\lambda}_{ts}$ and $\mathrm{deg}(c^{\lambda}_{st})=\mathrm{deg}(s)+\mathrm{deg}(t)$ for all $\lambda\in\mathfrak{P}$ and $s,t\in\mathcal{T}(\lambda)$. Moreover, for all $a\in A$
\begin{align}\label{eq-cell1}
ac^{\lambda}_{st}=\sum_{u\in \mathcal{T}(\lambda)}r_{a}(s,u)c^{\lambda}_{ut}\;(\mathrm{mod}\;A^{\rhd\lambda}).
\end{align}
(Note that the scalar $r_{a}(s,u)$ does not depend on $t$.)
Here $A^{\rhd\lambda}$ is the $R$-submodule of $A$ spanned by the set $\{c^{\mu}_{st}\mid \mu\rhd\lambda\text{ and }s,t\in\mathcal{T}(\mu)\}$.

An algebra $A$ with such a quintuple is called a \textit{graded cellular algebra} and the $c^{\lambda}_{st}$ are called a \textit{graded cellular basis} of $A$ (with respect to the antiinvolution $\mathrm{i}$).
\end{defn}

\begin{thm}\label{thm-cellular}(\textbf{Graded cellular basis}) The algebra $H_n(\vec{k})$ is a graded cellular algebra in the sense of Definition \ref{defn-cellular} with the cell datum
\begin{align}\label{eq-celldatum}
(\mathfrak{P}_{c(\vec{k})}^n,\iota(\tilde{B}(W_n(\vec{k}))),\mathfrak{F},{}^*,\mathrm{deg}_{\mathrm{BKW}}),
\end{align}
where $\mathfrak{P}_{c(\vec{k})}^n$ is the set of all $n$-multipartitions of $c(\vec{k})$ ordered by the dominance order $\vartriangleright$ from Definition \ref{defn-dominnancelambda}, $\iota(\tilde{B}(W_n(\vec{k})))$ is the image under our translation from Definition \ref{defn-webtotab}, the antiinvolution ${}^*$ is as above in Remark \ref{rem-involution} and the degree $\mathrm{deg}_{\mathrm{BKW}}$ on the $n$-multitableaux in $\iota(\tilde{B}(W_n(\vec{k})))$. These cell data (one for each $\vec{k}\in\Lambda(m,n\ell)_n$) can be extended to $H_n(\Lambda)$.
\end{thm}

\begin{proof}
To shorten our notation we skip the $\iota(\cdot)$ in the following. Moreover, the scalars below should all depend on the left side of the multiplication, but not on the right.

We have to prove four statements to show that \eqref{eq-celldatum} is a graded cell datum for $H_n(\vec{k})$. The four statements are that $\mathfrak{F}$ is a basis of the graded algebra $H_n(\vec{k})$, the elements $\mathcal{F}^{\vec{\lambda}}_{v_{f^{\prime}},u_f}\in\mathfrak{F}$ are homogeneous of degree
\[
\mathrm{deg}_q(\mathcal{F}^{\vec{\lambda}}_{v_{f^{\prime}},u_f})=\mathrm{deg}_{\mathrm{BKW}}(u_f)+\mathrm{deg}_{\mathrm{BKW}}(v_{f^{\prime}}),
\]
the antiinvolution ${}^*$ satisfies
\[
(\mathcal{F}^{\vec{\lambda}}_{v_{f^{\prime}},u_f})^*=\mathcal{F}^{\vec{\lambda}}_{u_f,v_{f^{\prime}}}
\]
and the crucial one (which suffices to verify \eqref{eq-cell1} by linearity)
\begin{align}\label{eq-cell2}
\mathcal{F}^{\vec{\mu}}_{\tilde{v}_{\tilde{f}^{\prime}},\tilde{u}_{\tilde{f}}}\mathcal{F}^{\vec{\lambda}}_{v_{f^{\prime}},u_f}=\sum_{w_{f^{\prime\prime}}\in \tilde{B}(W_n(\vec{k}))}r_{v_{f^{\prime}},w_{f^{\prime\prime}}}\mathcal{F}^{\vec{\lambda}}_{w_{f^{\prime\prime}},u_f}\;(\mathrm{mod}\;H_n(\vec{k})^{\vartriangleright\lambda}).
\end{align}
The first statement is Corollary \ref{cor-foambasis}, the second one follows from Remark \ref{rem-hmbasis} (which is based on Proposition \ref{prop-extgrowthdeg}) and the third one follows almost directly from the definition of ${}^*$, see Remark \ref{rem-involution}.

To verify \eqref{eq-cell2} we note that the product is zero if the two $\mathfrak{gl}_n$-webs $\tilde u$ and $v$ are not the same. Thus, we can focus on the case $\tilde u=v$.

Since the ``thick cellularity'' can be more easily seen in the thick cyclotomic KLR setup (to which we can freely switch by Theorem \ref{thm-iso}) let us illustrate with thick cyclotomic KLR diagrams how we can prove \eqref{eq-cell2}. Note that it is enough to consider only the middle part (after the dotted identity $e(\vec{\lambda})d(\vec{\lambda})$ and before the dotted identity $e(\vec{\mu})d(\vec{\mu})$). Thus, this is the only part we illustrate below (the right diagram is the top of $e(\vec{\lambda})d(\vec{\lambda})$).
\[
\xy
(0,0)*{\includegraphics[scale=.75]{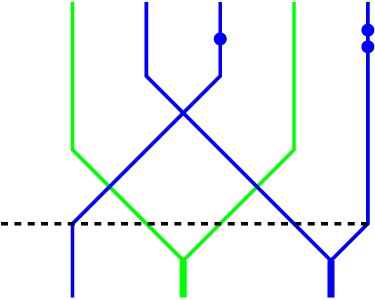}};
(-20,5)*{\text{usual}};
(-20,-15)*{\text{thick}};
\endxy\hspace*{0.5cm}\text{and}\hspace*{0.5cm}
\xy
(0,0)*{\includegraphics[scale=.75]{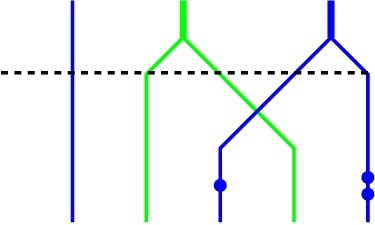}};
(-20,10)*{\text{thick}};
(-20,-5)*{\text{usual}};
\endxy
.
\]
We have illustrated two typical examples above. Everything splits into a usual and a thick part.

The main point is that, by our construction from Definition \ref{defn-foamLT2}, the assumption $\tilde{u}=v$ implies that the thick parts of both are mirrors. Thus, composing the two pictures will always create a composition of the split$\circ$merge as in \eqref{eq-createcrossing}. This will always create extra crossings which are part of the usual story. Thus, it suffices to verify \eqref{eq-cell2} in the case of the cyclotomic KLR algebra where we do not have any splits or merges at all.

We now use Lemma \ref{lem-welldef2} and the proof of cellularity by Hu--Mathas, see \cite[Theorem 5.8]{hm}, to see that \eqref{eq-cell2} holds in the usual cyclotomic KLR setup. The proof of this is essentially the same as in the $\mathfrak{sl}_3$ case and can be directly adapted from there (that is, the part after Equation 4.6 in the proof of \cite[Theorem 4.22]{tub3}). Thus, using their result and the isomorphism (which preserves the dominance order $\rhd$ by Lemma \ref{lem-welldef2}) from Theorem \ref{thm-iso}, we see that \eqref{eq-cell2} is satisfied which finishes the proof.
\end{proof}

\begin{rem}\label{rem-othermethod}
We note that there is another convention to obtain a HM basis. That is, one could also use the dual $n$-multitableau $T^*_{\vec{k}}$ of $T_{\vec{k}}$ from Definition \ref{defn-dominnancelambda}. Everything is this section works in the same vein as above. The difference is that the strings $\phi_{\sigma}$ of Definition \ref{defn-foamLT} tend to be shorter for elements of low order but longer for elements of big order. We just have chosen to take the $T_{\vec{k}}$ to stay closer to the formulation of Hu--Mathas. This basis already appears in the non-thick form in \cite[Section 6]{hm} and \cite[Theorem 6.11]{hm} shows that the dual basis is also cellular.
Let us briefly mention what the main differences in our setup of this dual basis compared to Definition \ref{defn-growthfoam} are. There are only two, namely the following.
\begin{itemize}
\item[(1)] The dotted identity $e(\vec{\lambda})d(\vec{\lambda})$ is obtained from dual $n$-multitableau $T^*_{\vec{k}}$ by counting addable boxes to the right. Same for the degree: count addable and removable nodes before (to the left), cf. Definition \ref{defn-tabcomb}.
\item[(2)] We have to rearrange our conversion from Definition \ref{defn-foamLT2} for $\iota(u_f)\to\iota(u_f)^{\prime}$ (recall that we needed this for the thick version) to $\iota(u_f)\to\tilde\iota(u_f)^{\prime}$, where latter is obtained by replacing numbers decreasing from left to right instead of increasing from left to right.
\end{itemize}
A small example for (2) is the following.
\[
\text{(usual) }\left(\;\xy(0,0)*{\begin{Young}1&3\cr\end{Young}}\endxy\;,\;\xy(0,0)*{\begin{Young}2&4\cr\end{Young}}\endxy\;\right)\leftarrow\vec{T}=\left(\;\xy(0,0)*{\begin{Young}1&2\cr\end{Young}}\endxy\;,\;\xy(0,0)*{\begin{Young}1&2\cr\end{Young}}\endxy\;\right)\rightarrow \left(\;\xy(0,0)*{\begin{Young}2&4\cr\end{Young}}\endxy\;,\;\xy(0,0)*{\begin{Young}1&3\cr\end{Young}}\endxy\;\right)\text{ (dual)}
\]
The reason for this is just that our choice has to be different for the dual since the dual turns degrees and order around. Note that $\mathrm{deg}_{\mathrm{BKW}}(\vec{T})=0$ for both conventions due to our shift.
\end{rem}

\begin{ex}\label{ex-dual}
Let us consider the following example. Compare also to Example \ref{ex-linkasF}. We want to illustrate the HM basis for $\mathrm{EXT}(\widehat{u},\widehat{v})$ for $n=2$. The $\mathfrak{gl}_2$-web $v$ should be the last one from Example \ref{ex-evaluation} which is given by $F_2F_1F_2F_1v_{(2^1)}$. The other one should be
\[
\scalebox{.7}{$\xy
(0,0)*{\includegraphics[scale=.75]{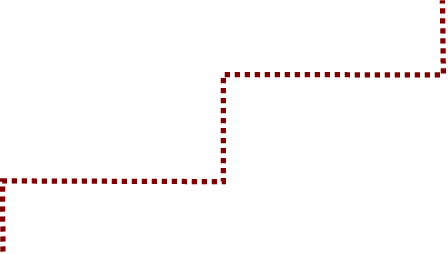}};
(-26.5,-15)*{2};
(1.5,-15)*{0};
(29,-15)*{0};
(-26.5,0)*{0};
(1.5,0)*{2};
(29,0)*{0};
(-26.5,15)*{0};
(1.5,15)*{0};
(29,15)*{2};
(-12.5,-3.5)*{F_1^{(2)}};
(15,10)*{F_2^{(2)}};
\endxy$}
\]
That is, $u=F_2^{(2)}F_1^{(2)}v_{(2^1)}$. The reader might think of elements of $\mathrm{EXT}(\widehat{u},\widehat{v})$ as dotted cups and of $\mathrm{EXT}(\widehat{v},\widehat{u})$ as dotted caps (in terms of Bar-Natan's cobordisms). As usual there is a duality: the dual of the un-dotted cup is the dotted cap. The same happens for the HM basis and its dual.

We have one $2$-multitableaux for $u$, namely $\vec{T}$ from Remark \ref{rem-othermethod}, and two for $v$, namely $\vec{T}_1$ and $\vec{T}_2$ from Example \ref{ex-linkasF}. The HM basis for $\mathrm{EXT}(\widehat{u},\widehat{v})$ is (using our isomorphism from Theorem \ref{thm-iso}) given by the two diagrams (of degree $\mathrm{deg}_{\mathrm{BKW}}(\vec{T}_1)=+1$ and $\mathrm{deg}_{\mathrm{BKW}}(\vec{T}_2)=-1$)
\[
\xy
(0,0)*{\includegraphics[scale=.75]{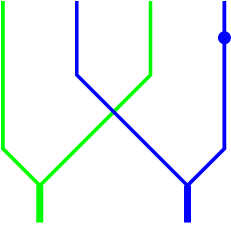}};
(-30,-10)*{\left(\;\xy(0,0)*{\begin{Young}1&2\cr\end{Young}}\endxy\;,\;\xy(0,0)*{\begin{Young}1&2\cr\end{Young}}\endxy\;\right)};
(-30,0)*{\left(\;\xy(0,0)*{\begin{Young}1&3\cr\end{Young}}\endxy\;,\;\xy(0,0)*{\begin{Young}2&4\cr\end{Young}}\endxy\;\right)};
(-30,10)*{\left(\;\xy(0,0)*{\begin{Young}1&2\cr\end{Young}}\endxy\;,\;\xy(0,0)*{\begin{Young}3&4\cr\end{Young}}\endxy\;\right)};
(-30,-5)*{\uparrow};
(-37.5,5)*{\tau_2(2,1)};
(-30,5)*{\uparrow};
(-40,-5)*{\text{unthickening}};
\endxy
\hspace*{0.25cm}\text{and}\hspace*{0.25cm}
\xy
(0,0)*{\includegraphics[scale=0.8]{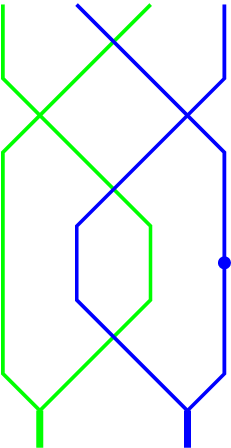}};
(31,-25)*{\left(\;\xy(0,0)*{\begin{Young}1&2\cr\end{Young}}\endxy\;,\;\xy(0,0)*{\begin{Young}1&2\cr\end{Young}}\endxy\;\right)};
(31,-15)*{\left(\;\xy(0,0)*{\begin{Young}1&3\cr\end{Young}}\endxy\;,\;\xy(0,0)*{\begin{Young}2&4\cr\end{Young}}\endxy\;\right)};
(31,-5)*{\left(\;\xy(0,0)*{\begin{Young}1&2\cr\end{Young}}\endxy\;,\;\xy(0,0)*{\begin{Young}3&4\cr\end{Young}}\endxy\;\right)};
(31,5)*{\left(\;\xy(0,0)*{\begin{Young}1&3\cr\end{Young}}\endxy\;,\;\xy(0,0)*{\begin{Young}2&4\cr\end{Young}}\endxy\;\right)};
(31,15)*{\left(\;\xy(0,0)*{\begin{Young}2&4\cr\end{Young}}\endxy\;,\;\xy(0,0)*{\begin{Young}1&3\cr\end{Young}}\endxy\;\right)};
(31,25)*{\left(\;\xy(0,0)*{\begin{Young}3&4\cr\end{Young}}\endxy\;,\;\xy(0,0)*{\begin{Young}1&2\cr\end{Young}}\endxy\;\right)};
(31,20)*{\uparrow};
(38.5,20)*{\tau_2(1,2)};
(31,10)*{\uparrow};
(49,10)*{\tau_1(1,1)\text{ and }\tau_3(2,2)};
(31,0)*{\uparrow};
(38.5,0)*{\tau_2(1,2)};
(31,-10)*{\uparrow};
(38.5,-10)*{\tau_2(2,1)};
(31,-20)*{\uparrow};
(41,-20)*{\text{unthickening}};
\endxy
,
\]
as the reader is invited to check. The left-hand side corresponds to the datum $(\vec{T},\vec{T}_1)$ and the right-hand side to $(\vec{T},\vec{T}_2)$. In the $\mathfrak{sl}_2$-cobordism language these are (up to signs) just a dotted cup (left) and a cup (right). The duals for $\mathrm{EXT}(\widehat{v},\widehat{u})$ on the other hand are given by (of dual-degree $\mathrm{deg}_{\mathrm{BKW}}(\vec{T}_1)=-1$ and $\mathrm{deg}_{\mathrm{BKW}}(\vec{T}_2)=+1$)
\[
\xy
(0,0)*{\reflectbox{\rotatebox{180}{\includegraphics[scale=0.8]{figs/catcell/dual-example2.eps}}}};
(-31,25)*{\left(\;\xy(0,0)*{\begin{Young}1&2\cr\end{Young}}\endxy\;,\;\xy(0,0)*{\begin{Young}1&2\cr\end{Young}}\endxy\;\right)};
(-31,15)*{\left(\;\xy(0,0)*{\begin{Young}2&4\cr\end{Young}}\endxy\;,\;\xy(0,0)*{\begin{Young}1&3\cr\end{Young}}\endxy\;\right)};
(-31,5)*{\left(\;\xy(0,0)*{\begin{Young}3&4\cr\end{Young}}\endxy\;,\;\xy(0,0)*{\begin{Young}1&2\cr\end{Young}}\endxy\;\right)};
(-31,-5)*{\left(\;\xy(0,0)*{\begin{Young}2&4\cr\end{Young}}\endxy\;,\;\xy(0,0)*{\begin{Young}1&3\cr\end{Young}}\endxy\;\right)};
(-31,-15)*{\left(\;\xy(0,0)*{\begin{Young}1&3\cr\end{Young}}\endxy\;,\;\xy(0,0)*{\begin{Young}2&4\cr\end{Young}}\endxy\;\right)};
(-31,-25)*{\left(\;\xy(0,0)*{\begin{Young}1&2\cr\end{Young}}\endxy\;,\;\xy(0,0)*{\begin{Young}3&4\cr\end{Young}}\endxy\;\right)};
(-31,-20)*{\downarrow};
(-38.5,-20)*{\tau_2(1,2)};
(-31,-10)*{\downarrow};
(-49,-10)*{\tau_1(1,1)\text{ and }\tau_3(2,2)};
(-31,0)*{\downarrow};
(-38.5,0)*{\tau_2(2,1)};
(-31,10)*{\downarrow};
(-38.5,10)*{\tau_2(1,2)};
(-31,20)*{\downarrow};
(-41,20)*{\text{unthickening}};
\endxy
\hspace*{0.25cm}\text{and}\hspace*{0.25cm}
\xy
(0,0)*{\reflectbox{\rotatebox{180}{\includegraphics[scale=.75]{figs/catcell/dual-example1.eps}}}};
(30,10)*{\left(\;\xy(0,0)*{\begin{Young}1&2\cr\end{Young}}\endxy\;,\;\xy(0,0)*{\begin{Young}1&2\cr\end{Young}}\endxy\;\right)};
(30,0)*{\left(\;\xy(0,0)*{\begin{Young}2&4\cr\end{Young}}\endxy\;,\;\xy(0,0)*{\begin{Young}1&3\cr\end{Young}}\endxy\;\right)};
(30,-10)*{\left(\;\xy(0,0)*{\begin{Young}3&4\cr\end{Young}}\endxy\;,\;\xy(0,0)*{\begin{Young}1&2\cr\end{Young}}\endxy\;\right)};
(30,-5)*{\downarrow};
(37.5,-5)*{\tau_2(1,2)};
(30,5)*{\downarrow};
(40,5)*{\text{unthickening}};
\endxy
.
\]
Note that composing them with the cups at the bottom gives an element of $\mathrm{EXT}(\widehat{u},\widehat{u})$ which is a number $\bQ$. Moreover, they are really duals: from the four possibilities for composition, only two give non-zero numbers.
\end{ex}

\begin{rem}\label{rem-cell}
Using the cell modules (which can be constructed explicitly from the cellular basis, see Section 2 in \cite{hm}), we get two sets
\[
\mathcal D=\{D^{\vec{\lambda}}\{k\}\mid \vec{\lambda}\in\tilde{\mathfrak{P}}_{c(\vec{k})}^n,\, k\in\bZ\}\text{  and  }\mathcal P=\{P^{\vec{\lambda}}\{k\}\mid \vec{\lambda}\in\tilde{\mathfrak{P}}_{c(\vec{k})}^n,\, k\in\bZ\},
\]
where $\tilde{\mathfrak{P}}_{c(\vec{k})}^n\subset \mathfrak{P}_{c(\vec{k})}^n$ is the subset of all $n$-multipartitions of $c(\vec{k})$ with $D^{\vec{\lambda}}\neq 0$. These form a complete set of pairwise non-isomorphic, graded, simple $H_n(\Lambda)$-modules and pairwise non-isomorphic, graded, projective indecomposable $H_n(\Lambda)$-modules respectively.

Furthermore, following the same approach as indicated in Remark 4.25 in \cite{tub3}, one can verify that these sets under the isomorphism of the (split) Grothendieck groups
\[
K^{(\oplus)}_0(\mathcal{W}_{\Lambda}^{(p)})\otimes_{\bZ[q,q^{-1}]}\bC(q)\cong W_{\Lambda}=\bigoplus_{\vec{k}\in\Lambda(m,n\ell)_n}W^{(*)}_n(\vec{k})
\]
correspond to the canonical and dual canonical basis respectively. Here the $\mathcal{W}_{\Lambda}^{(p)}$ are certain categories of modules over $H_n(\Lambda)\cong\check{R}(\Lambda)$, see Definition 7.1 in \cite{mack1}.
\end{rem}

\subsubsection{An example}\label{subsub-example2}
We conclude this section with an example - we hope that it helps the reader.

\begin{ex}\label{ex-cheating}
We will cheat a bit now in order to give a hopefully illustrating example how the graded cellular basis works. First let us fix $n=2$, $\ell=2$ and $\vec{k}=(1,1,1,1)$, i.e. we will give a $\mathfrak{gl}_2$ example with $v_h=v_{(2^2)}$. We cheat, because we do not use matrix factorizations in this example, but Bar-Natan's cobordisms \cite{bn2} (not even Blanchet's cobordisms, i.e. everything below is only true up to a sign, see \cite{bla1} and \cite{lqr1}). The reason is that the usage of these cobordisms illustrates without to many technical difficulties why the HM basis really works so well. To cheat even more: we also ignore any shifts and gradings in this example.

We use the standard arc basis which in this case consists of the two $\mathfrak{gl}_2$-webs $u=F_2F_1F_3F_2v_{(2^2)}$
\[
u=
\scalebox{.7}{$\xy
(0,0)*{\includegraphics[scale=.75]{figs/catcell/webflow-exa-b.eps}};
(0.5,-17.75)*{F_2};
(28.5,-4)*{F_3};
(0.5,22.75)*{F_2};
(-27.5,9.5)*{F_1};
(-40,-27.5)*{2};
(-12,-27.5)*{2};
(16,-27.5)*{0};
(44,-27.5)*{0};
(-40,-13.75)*{2};
(-12,-13.75)*{1};
(16,-13.75)*{1};
(44,-13.75)*{0};
(-40,0)*{2};
(-12,0)*{1};
(16,0)*{0};
(44,0)*{1};
(-40,13.75)*{1};
(-12,13.75)*{2};
(16,13.75)*{0};
(44,13.75)*{1};
(-40,27.5)*{1};
(-12,27.5)*{1};
(16,27.5)*{1};
(44,27.5)*{1};
\endxy$}
,
\]
and $v=F_1F_2F_3F_2v_{(2^2)}$
\[
v=
\scalebox{.7}{$\xy
(0,0)*{\includegraphics[scale=.75]{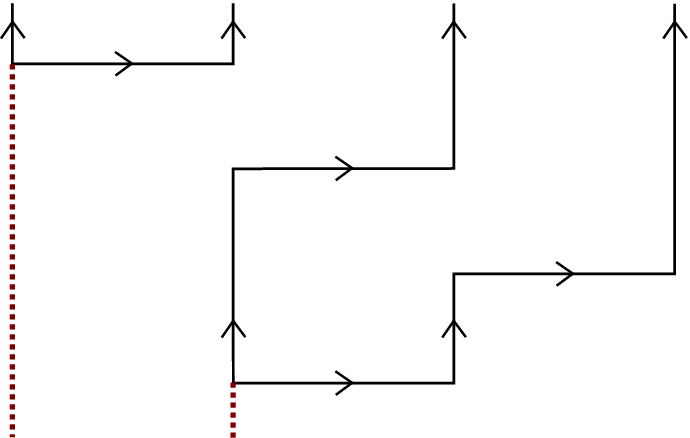}};
(0.5,-17.75)*{F_2};
(28.5,-4)*{F_3};
(0.5,9.5)*{F_2};
(-27.5,22.75)*{F_1};
(-40,-27.5)*{2};
(-12,-27.5)*{2};
(16,-27.5)*{0};
(44,-27.5)*{0};
(-40,-13.75)*{2};
(-12,-13.75)*{1};
(16,-13.75)*{1};
(44,-13.75)*{0};
(-40,0)*{2};
(-12,0)*{1};
(16,0)*{0};
(44,0)*{1};
(-40,13.75)*{2};
(-12,13.75)*{0};
(16,13.75)*{1};
(44,13.75)*{1};
(-40,27.5)*{1};
(-12,27.5)*{1};
(16,27.5)*{1};
(44,27.5)*{1};
\endxy$}.
\]
In this case, the flows on these $\mathfrak{gl}_2$-webs are completely determined by the cut line and we have six flows: the two canonical flows $S_c(u)=(\{2\},\{2\},\{1\},\{1\})$ and $S_c(v)=(\{2\},\{1\},\{2\},\{1\})$ and the two ``anticanonical'' flows $S^c(u)=(\{1\},\{1\},\{2\},\{2\})$ and $S^c(v)=(\{1\},\{2\},\{1\},\{2\})$. Moreover, the $\mathfrak{gl}_2$-web $v$ has two additional flows, namely $S_1(v)=(\{1\},\{2\},\{2\},\{1\})$ and $S_2(v)=(\{2\},\{1\},\{1\},\{2\})$.

We expect two different important idempotents $e(\vec{\lambda})$ and $e(\vec{\mu})$, since these will determine the Specht modules. And we expect different dot placements $d(\cdot)$ for them, since both, idempotent and dot placement, depend only on the cut line. And this is exactly what we get: we have six different $2$-multipartitions (one for each flow at the boundary), namely (for $S_c(u), S^c(u)$ and $S^c(v)$)
\[
\vec{\mu}=\left(\;\emptyset\;,\xy(0,0)*{\begin{Young}&\cr&\cr\end{Young}}\endxy\;\right)\hspace*{0.2cm}\vec{\mu}^{\prime}=\left(\;\xy(0,0)*{\begin{Young}&\cr&\cr\end{Young}}\endxy\;,\;\emptyset\;\right)\hspace*{0.2cm}\vec{\mu}^{\prime\prime}=\left(\;\xy(0,0)*{\begin{Young}&\cr\cr\end{Young}}\endxy\;,\xy(0,0)*{\begin{Young}\cr\end{Young}}\endxy\;\right)
,
\]
and (for $S_c(v), S_1(v)$ and $S_2(v)$ respectively)
\[
\vec{\lambda}=\left(\;\xy(0,0)*{\begin{Young}\cr\end{Young}}\endxy\;,\xy(0,0)*{\begin{Young}&\cr\cr\end{Young}}\endxy\;\right)\hspace*{0.2cm}\vec{\lambda}^{\prime}=\left(\;\xy(0,0)*{\begin{Young}\cr\cr\end{Young}}\endxy\;,\xy(0,0)*{\begin{Young}&\cr\end{Young}}\endxy\;\right)\hspace*{0.2cm}\vec{\lambda}^{\prime\prime}=\left(\;\xy(0,0)*{\begin{Young}&\cr\end{Young}}\endxy\;,\xy(0,0)*{\begin{Young}\cr\cr\end{Young}}\endxy\;\right)
,
\]
where $\vec{\mu}$, $\vec{\mu}^{\prime}$ and $\vec{\mu}^{\prime\prime}$ have the same residue sequence $r(\vec{\mu})=r(\vec{\mu}^{\prime})=r(\vec{\mu}^{\prime\prime})=(2,3,1,2)$ (recall the shift of the residue by $\ell=2$ and one fills in numbers from left to right and top to bottom with rows first). Moreover, $r(\vec{\lambda})=(2,2,3,1), r(\vec{\lambda}^{\prime})=(2,1,2,3)$ and $r(\vec{\lambda}^{\prime\prime})=(2,3,2,1)$.

Thus, we have $e(\vec{\mu})=e(\vec{\mu}^{\prime})=e(\vec{\mu}^{\prime\prime})=\mathrm{id}_{F_2F_1F_3F_2v_{(2^2)}}\neq e(\vec{\lambda})=\mathrm{id}_{F_1F_3F_2F_2v_{(2^2)}}$ and two additional $e(\vec{\lambda}^{\prime})=\mathrm{id}_{F_1F_2F_1F_2v_{(2^2)}}$ and $e(\vec{\lambda}^{\prime\prime})=\mathrm{id}_{F_1F_2F_3F_2v_{(2^2)}}$.

Moreover, since the dot placement is given by addable nodes to the right, we have no dots for $\vec{\mu}$, two dots for $\vec{\mu}^{\prime}$ and one dot for the other four $2$-multipartitions. The reader is invited to check that the Specht module for the $\vec{\mu}$, after modding out by the radical, is exactly the $P_u$. Moreover, we do not get too much: the elements for the two flows $S_1(v)$ and $S_2(v)$ will give rise to two nilpotent elements (with one dot each). Thus, they do not belong to the set $\tilde{\mathfrak{P}}_{c(\vec{k})}^n$ from Remark \ref{rem-cell} since modding out by the radical will kill them (they are ``unimportant'').

We do the other in more details now, since it illustrates how the HM basis does exactly what one would expect if one could guess the answer (as in this case), but works even if it is impossible to guess the answer (as in almost all other cases). The idempotent $e(\vec{\lambda})$ in this case is the $\mathrm{id}$ on
\[
w=\scalebox{.7}{$\xy
(0,0)*{\includegraphics[scale=.75]{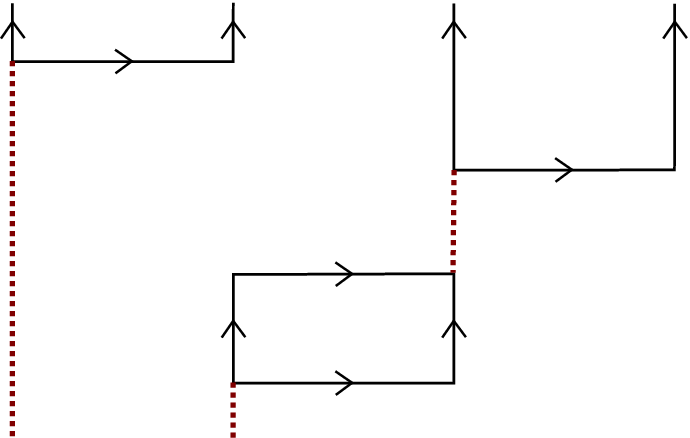}};
(0.5,-17.75)*{F_2};
(28.5,9.5)*{F_3};
(0.5,-4)*{F_2};
(-27.5,22.75)*{F_1};
(-40,-27.5)*{2};
(-12,-27.5)*{2};
(16,-27.5)*{0};
(44,-27.5)*{0};
(-40,-13.75)*{2};
(-12,-13.75)*{1};
(16,-13.75)*{1};
(44,-13.75)*{0};
(-40,0)*{2};
(-12,0)*{0};
(16,0)*{2};
(44,0)*{0};
(-40,13.75)*{2};
(-12,13.75)*{0};
(16,13.75)*{1};
(44,13.75)*{1};
(-40,27.5)*{1};
(-12,27.5)*{1};
(16,27.5)*{1};
(44,27.5)*{1};
(25,-5)*{\saddledown};
(-20,10)*{\saddleup};
(-20,-14)*{\cupcap};
\endxy$}
.
\]
The main question now can be seen as follows. The canonical flow on $v$ works not only for $v$, but also for $u$ (where it is a mixed flow). But since the dot placement and the idempotent is completely determined by the cut line and one can not distinguish between the two just on the cut line, the question is what is a ``good'' idempotent for $\vec{\lambda}$. The answer $e(\vec{\lambda})=\mathrm{id}_{F_1F_3F_2F_2v_{(2^2)}}$, that is the identity on the $\mathfrak{gl}_2$-web above, can be seen as the smallest common multiple between $u$ and $v$. That is, one can easily go from $w$ to either $u$ or $v$ by using saddle moves $F_iF_{i\pm 1}\to F_{i\pm 1}F_i$ indicated above. We note that one has to use two saddles to go to $u$: first $F_3F_2\to F_3F_2$ (bottom saddle above) and then $F_1F_2\to F_2F_1$ (top saddle above), but only the bottom one to go to $v$.

The two possible extensions of $\vec{\lambda}$ are the canonical flow on $v$ and the mixed on $u$ given by
\[
\vec{T}_c=\left(\;\xy(0,0)*{\begin{Young}3\cr\end{Young}}\endxy\;,\xy(0,0)*{\begin{Young}1&2\cr 4\cr\end{Young}}\endxy\;\right)\hspace*{0.2cm}\vec{T}_m=\left(\;\xy(0,0)*{\begin{Young}4\cr\end{Young}}\endxy\;,\xy(0,0)*{\begin{Young}1&2\cr 3\cr\end{Young}}\endxy\;\right)\hspace*{0.2cm}T_{\vec{\lambda}}=\left(\;\xy(0,0)*{\begin{Young}1\cr\end{Young}}\endxy\;,\xy(0,0)*{\begin{Young}2&3\cr 4\cr\end{Young}}\endxy\;\right),
\]
where the rightmost filling is the standard filling. Thus, in order to go from $T_{\vec{\lambda}}$ to the others, one has to use the permutations $\tau_1(2,2)\tau_2(3,2)\vec{T}_c=T_{\vec{\lambda}}$ in the first and $\tau_1(2,2)\tau_2(3,2)\tau_3(1,2)\vec{T}_m=T_{\vec{\lambda}}$ in the second case. The $\tau_k(i,j)$ correspond to a cup-cap-move (if $i=j$, see in the figure above), a \textit{saddle} (if $|i-j|=1$) or a shift (if $|i-j|>1$). Thus, if we use $\sigma=\tau_1(2,2)\tau_2(3,2)$ and $\tilde{\sigma}=\tau_1(2,2)\tau_2(3,2)\tau_3(1,2)$ as shorthand notations, we see that the four elements
\[
v^*v\rightsquigarrow\sigma^{*} e(\vec{\lambda})d(\vec{\lambda}) \sigma \hspace*{0.125cm}\hspace*{0.125cm}u^*v\rightsquigarrow \tilde{\sigma}^{*} e(\vec{\lambda})d(\vec{\lambda}) \sigma \hspace*{0.125cm}\hspace*{0.125cm}v^*u\rightsquigarrow \sigma^{*} e(\vec{\lambda})d(\vec{\lambda}) \tilde{\sigma}\hspace*{0.125cm}\hspace*{0.125cm}u^*u\rightsquigarrow \tilde{\sigma}^{*} e(\vec{\lambda})d(\vec{\lambda}) \tilde{\sigma}
\]
(here $d(\vec{\lambda})$ denotes a dot on a cylinder between the internal circle), which correspond to the four possible combinations $v^*v$, $u^*v$, $v^*u$ and $u^*u$, gives exactly the answer one would expect.

That is, all of them remove the internal circle by closing the dotted cylinder using a cap at the top and a cup at the bottom (with the Bar-Natan relations: this is a dotted sphere and hence equals $1$). Now the first one for example uses the saddle move given by $\tau_2(3,2)$ to connect the internal circle to one of the boundary sheets and the end result is just two un-dotted sheets (as one would guess). The reader is invited to draw the pictures for the other three possibilities. Note that in the last case the algorithm creates a neck (in the language of Bar-Natan's cobordism) that one can cut giving a linear combination in contrast to the case for the ``anticanonical'' which gives two dotted cylinders. Thus, they are all nilpotent except $\sigma^{*} e(\vec{\lambda})d(\vec{\lambda})\sigma$.
\end{ex}
\subsection{Connections to \texorpdfstring{$\mathfrak{gl}_n$}{gln}-link homologies}\label{sec-catpart2}

\begin{rem}\label{rem-notation} 
We will formulate everything in this section in a mixture of different notations. First we note that we freely switch between the notions $\mathfrak{gl}_n$-webs, their associated matrix factorizations, string of $F^{(j)}_i$ and string of $\mathcal{F}^{(j)}_i$. We hope that is not too confusing.

Moreover, we stay in the KLR part of $\Ucatmc$ for braids and only go to the cyclotomic quotient for the $\mathfrak{gl}_n$-link homologies. The reason is that we can not formulate the complex locally in the thick cyclotomic KLR, because, in our convention, we would have to start at a weight $(n^{\ell})$ for some $\ell$. We try to distinguish them as follows: the pictures for the KLR part of $\Ucatmc$ have orientations (in our notation they are oriented downwards) and the ones for $\check R_{\Lambda}$ do not have orientations. Finally, for the $2$-Schur quotient $\check{\mathcal{S}}(m,n\ell)_n$ (see below before Lemma \ref{lem-rickardisf}) of $\Ucatmc$ we use the same notations as for $\Ucatmc$ itself.
\end{rem}

\subsubsection{The Rickard complex}\label{subsub-rickard}

Recall that Chuang--Rouquier's version of the Rickard complex \cite{chro1} can be seen as a categorification of the quantum Weyl group action on $V_N$ from \eqref{eq-hweight} that acts by a reflection isomorphism between the $k$-th and $-k$-th weight space. We use a slightly adapted version of Cautis' variant \cite{cau1} here.

We denote by $\mathcal{T}v_{\vec{k}}$ usually the $F$-braiding complex given in Definition \ref{defn-basicT} and the Rickard version by $\mathcal{T}{\idm}_{\vec{k}}$.

Given a suitable $2$-category $\mathfrak{C}$, then the $2$-category $\Kom_{\mathrm{gr}}(\mathfrak{C})$ has the same objects as $\mathfrak{C}$, but the morphisms are complexes of $\mathfrak{C}$ and the $2$-morphisms are chain maps between these complexes. Moreover, everything should be graded and morphisms should preserve the degree.

\begin{defn}\label{defn-rickard}
For $a,b\in\bN$ let $q_k=-b+k$, if $b\leq a$, and $q_k=-a+k$, if $a<b$.
Given a $\mathfrak{gl}_m$-weight $\vec{k}$ with $a,b$ in the $i$-th and $i+1$-th entry, we define the \textit{$i$-th positive Rickard complex $\mathcal{T}^{+}_i{\idm}_{\vec{k}}$} in $\Kom_{\mathrm{gr}}(\Ucatmc)$ as
\[
\mathcal{T}^{+}_i{\idm}_{\vec{k}}=
\begin{cases}
\begin{xy}
\xymatrix{
\mathcal{F}_i^{(a-b)}{\idm}_{\vec{k}}\{q_0\} \ar[r]^/-.3em/{d^{\mathrm{R}}_0}    &   \mathcal{F}_i^{(a+1-b)}\mathcal{E}_i{\idm}_{\vec{k}}\{q_1\} \ar[r]^/-0.1em/{d^{\mathrm{R}}_1} & \mathcal{F}_i^{(a+2-b)}\mathcal{E}_i^{(2)}{\idm}_{\vec{k}}\{q_2\}\ar[r]^/1.8em/{d^{\mathrm{R}}_2} & \dots
}
\end{xy},&\text{if }b\leq a,
\\
\begin{xy}
\xymatrix{
\mathcal{E}_i^{(-a+b)}{\idm}_{\vec{k}}\{q_0\} \ar[r]^/-.5em/{d^{\mathrm{R}}_0}    &   \mathcal{E}_i^{(-a+1+b)}\mathcal{F}_i{\idm}_{\vec{k}}\{q_1\} \ar[r]^/-0.2em/{d^{\mathrm{R}}_1} & \mathcal{E}_i^{(-a+2+b)}\mathcal{F}_i^{(2)}{\idm}_{\vec{k}}\{q_2\}\ar[r]^/1.8em/{d^{\mathrm{R}}_2} & \dots
}
\end{xy},&\text{if }a<b.
\end{cases}
\]
In both cases the leftmost part is in homology degree zero. The differentials are given by
\begin{align*}
d^{\mathrm{R}}_k&=
\xy
(0,0)*{\reflectbox{\includegraphics[scale=1.0]{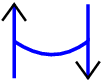}}};
(0,-1)*{\scriptstyle 1};
\endxy\colon \mathcal{F}_i^{(a+k-b)}\mathcal{E}_i^{(k)}{\idm}_{\vec{k}}\{q_k\}\to \mathcal{F}_i^{(a+k+1-b)}\mathcal{E}_i^{(k+1)}{\idm}_{\vec{k}}\{q_{k+1}\},
\\
d^{\mathrm{R}}_k&=
\xy
(0,0)*{\includegraphics[scale=1.0]{figs/linkhom/rickard}};
(0,-1)*{\scriptstyle 1};
\endxy\colon \mathcal{E}_i^{(-a+k+b)}\mathcal{F}_i^{(k)}{\idm}_{\vec{k}}\{q_k\}\to \mathcal{E}_i^{(-a+k+1+b)}\mathcal{F}_i^{(k+1)}{\idm}_{\vec{k}}\{q_{k+1}\},
\end{align*}
for the two cases, respectively. They are both invertible up to homotopy and we denote their inverses (that should correspond to our negative crossings) by ${\idm}_{\vec{k}}\mathcal{T}^{-}_i$ and call them \textit{$i$-th negative Rickard complex ${\idm}_{\vec{k}}\mathcal{T}^{-}_i$}. They are also in $\Kom_{\mathrm{gr}}(\Ucatmc)$.
\end{defn}

As an example, for $\vec{k}=(1,1)$ we have
\[
\mathcal{T}^{+}_1{\idm}_{\vec{k}}=
\begin{xy}
\xymatrix{
{\idm}_{\vec{k}}\{-1\} \ar[r]^/-.3em/{\includegraphics[width=20px]{figs/higherstuff/rightcup}}    &   \mathcal{F}_i\mathcal{E}_i{\idm}_{\vec{k}}\{0\}
}
\end{xy},
\]
which is essentially a categorification of the Kauffman bracket.

\subsubsection{The \texorpdfstring{$F$}{F}-braiding complex}\label{subsub-braiding}
In this subsection we define the categorification of the (colored) braiding operator $T_{a,b,i}^k$ from Definition \ref{defn-colbraid}. We call the categorification the (colored) $F$-braiding complex. We start with the uncolored case where we still draw the pictures. For the colored case we do not draw the $\mathfrak{gl}_n$-webs anymore but use our $F$ notation instead.

\begin{defn}\label{defn-basicT}(\textbf{Braiding complex: only $F$})
Recall that we defined in Definition \ref{defn-colbraid} the braiding operators $T_i^k$ for $k=0,1$ which acts on a weight $\vec{k}$ with $i$ and $i{+}1$ entry equal to $1$ and the $i{+}2$-th entry equal to zero. The \textit{$F$-braiding complex $\mathfrak{T}_i^{+}v_{\vec{k}}$} is then defined to be
\[
\mathfrak{T}_i^{+}v_{\vec{k}}=\begin{xy}
\xymatrix{\scalebox{.7}{$\xy
(0,0)*{\includegraphics[scale=.75]{figs/linkpoly/braid1a.eps}};
(-26,-20)*{1};
(2,-19.5)*{1};
(30,-19.5)*{0};
(-14,3.5)*{F_{i}};
(-26,-6.5)*{1};
(2.5,-7)*{0};
(30,-6.5)*{1};
(14,-10.5)*{F_{i+1}};
(-26,6.5)*{0};
(2.5,6.5)*{1};
(30,6)*{1};
(-26,19.5)*{0};
(2,19.5)*{1};
(30,19)*{1};
\endxy$}\{-1\}\ar[r]^/.4em/{d}   &    \scalebox{.7}{$\xy
(0,0)*{\includegraphics[scale=.75]{figs/linkhom/braid1c.eps}};
(14,17.5)*{F_{i+1}};
(-26,-20)*{1};
(2,-19.5)*{1};
(30,-19.5)*{0};
(-14,3.5)*{F_{i}};
(-26,-6.5)*{1};
(2.5,-7)*{1};
(30,-6.5)*{0};
(-26,6.5)*{0};
(2.5,6.5)*{2};
(30,6)*{0};
(-26,19.5)*{0};
(2,19.5)*{1};
(30,19)*{1};
\endxy$}.}
\end{xy}
\]
with differential $d=\raisebox{-0.25em}{\includegraphics[width=15px]{figs/higherstuff/downcross}}\colon F_iF_{i+1}v_{\vec{k}}\to F_{i+1}F_iv_{\vec{k}}$ and leftmost component in homology degree zero. The \textit{braiding complex $\mathfrak{T}_i^{-}v_{\vec{k}}$} is defined in the same way, but with reflected pictures, rightmost component in homology degree zero, a differential $d=\raisebox{-0.2em}{\reflectbox{\includegraphics[width=15px]{figs/higherstuff/downcross}}}\colon F_{i+1}F_iv_{\vec{k}}\to F_iF_{i+1}v_{\vec{k}}$ and a $q$-degree shift by $1$ for the rightmost component. In an algebraic notation this will be
\[
\mathfrak{T}_i^{-}v_{\vec{k}}=\begin{xy}
\xymatrix{0\ar[r] & T^0_i v_{\vec{k}}=F_{i+1}F_iv_{\vec{k}} \ar[r]^/-.5em/{d}   &    T^1_i v_{\vec{k}}=F_iF_{i+1}v_{\vec{k}}\{1\}\ar[r] & 0.}
\end{xy}
\]
We encourage the reader to draw the pictures.

Now assume that $\vec{k}$ has $a$ in the $i$-th and $b$ in the $i+1$-th entry and the $i+2$-th entry equal to zero. The colored \textit{positive $F$-braiding complex $\mathfrak{T}_{a,b,i}^{+}v_{\vec{k}}$} is then defined to be
\[
\begin{xy}
\xymatrix{
F_{i+1}^{(a-b)}F_{i}^{(a)}F_{i+1}^{(b)}v_{\vec{k}}\{q_0\} \ar[r]^/-.3em/{d_0}    &   F_{i+1}^{(a+1-b)}F_{i}^{(a)}F_{i+1}^{(b-1)}v_{\vec{k}}\{q_1\} \ar[r]^/2.3em/{d_1}  & \dots  \ar[r]^/-1.7em/{d_{b-1}}  & F_{i+1}^{(a)}F_{i}^{(a)}F_{i+1}^{(0)}v_{\vec{k}}\{q_b\}
}
\end{xy}
\]
in the case $b\leq a$, and for $a<b$ we use
\[
\begin{xy}
\xymatrix{
F_{i}^{(a)}F_{i+1}^{(a)}F_{i}^{(0)}\{q_0\} \ar[r]^/-.4em/{d_0}    &   F_{i}^{(a-1)}F_{i+1}^{(a)}F_{i}^{(1)}v_{\vec{k}}\{q_1\} \ar[r]^/2.0em/{d_1}  & \dots  \ar[r]^/-1.7em/{d_{b-1}}  & F_{i}^{(0)}F_{i+1}^{(a)}F_{i}^{(a)}v_{\vec{k}}\{q_a\}
}
\end{xy}
\]
with the leftmost term in homology degree zero. The $q$-degree shifts are $q_k=-b+k$ in the first and $q_k=-a+k$ in the second case (compare to Definition \ref{defn-colbraid}).

The differentials are given by (the thickness of the middle edge is $1$)
\[
d_k=
\xy
(0,0)*{\includegraphics[scale=1.25]{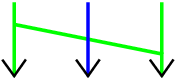}};
(6.5,0)*{\scriptstyle 1};
(2,-7)*{\scriptstyle a};
(-10,-7)*{\scriptstyle a+k-b};
(12,-7)*{\scriptstyle b-k};
(2,6.75)*{\scriptstyle a};
(-8.5,6.5)*{\scriptstyle a+k+1-b};
(11,6.5)*{\scriptstyle b-k-1};
\endxy\colon F_{i+1}^{(a+k-b)}F_{i}^{(a)}F_{i+1}^{(b-k)}v_{\vec{k}}\{q_k\}\to F_{i+1}^{(a+k+1-b)}F_{i}^{(a)}F_{i+1}^{(b-k-1)}v_{\vec{k}}\{q_{k+1}\}
\]
in the case $b\leq a$, and by (the thickness of the middle edge is $1$)
\[
d_k=
\xy
(0,0)*{\reflectbox{\includegraphics[scale=1.25]{figs/linkhom/HM-diff-KLR}}};
(-6.5,0)*{\scriptstyle 1};
(2,-7)*{\scriptstyle a};
(-12,-7)*{\scriptstyle a-k};
(13.5,-7)*{\scriptstyle k};
(2,6.75)*{\scriptstyle a};
(-10,6.5)*{\scriptstyle a-k-1};
(12.5,6.5)*{\scriptstyle k+1};
\endxy\colon F_{i}^{(a-k)}F_{i+1}^{(a)}F_{i}^{(k)}v_{\vec{k}}\{q_k\}\to F_{i}^{(a-k-1)}F_{i+1}^{(a)}F_{i}^{(k+1)}v_{\vec{k}}\{q_{k+1}\}
\] 
in the case $a<b$. (We note that the special case $a=b=1$ is the usual KLR crossing from above.)

The colored \textit{negative $F$-braiding complex $\mathfrak{T}_{a,b,i}^{-}v_{\vec{k}}$} is defined by turning 
everything around: reflected pictures, rightmost component in homology degree zero, the differentials are reflections of the ones from before and $q$-degree shifts $q_k=b-k$ in the $b\leq a$ and $q_k=a-k$ in the $a<b$ case. (The reader is encouraged to write down the complexes.)
Since the $a,b$ are encoded by $v_{\vec{k}}$ we tend not to write the $a$ and $b$ explicitly.
\end{defn}

\begin{lem}\label{lem-complex}
The $F$-braiding complex $\mathfrak{T}_{a,b,i}^{\pm}v_{\vec{k}}$ is an element of $\Kom_{\mathrm{gr}}(\Ucatm)$, i.e. the differentials preserve the degree and $d_{k+1}\circ d_k=0$.
\end{lem}

\begin{proof}
Let us skip the labels in the following. We have in the positive $b\leq a$ case
\[
\xy
(0,0)*{\includegraphics[scale=1.25]{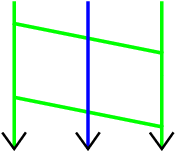}};
\endxy=\xy
(0,0)*{\includegraphics[scale=1.25]{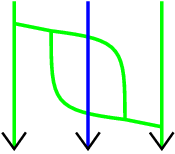}};
\endxy=\xy
(0,0)*{\includegraphics[scale=1.25]{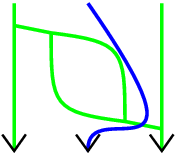}};
\endxy=0,
\]
where the first equation follows from the associativity of splitters and merges (see e.g. \cite[Proposition 2.2.4]{klms}), the second from the pitchfork relation (see e.g. \cite[Proposition 4]{st1}) and the third is a direct consequence of the definition of splitters and merges (see e.g. \cite[Equation 2.64]{klms}). We leave the positive $a<b$ case and the negative cases to the reader.

The difference between two shifts is $q_k-q_{k+1}=-1$. Thus, the differentials have to be of degree $1$ in order to be degree preserving. Recall that splits and merges are of degree $-jj^{\prime}$ (if they split $j+j^{\prime}$ into $j$ and $j^{\prime}$ or vice versa for merges). Since the middle edges are of thickness $1$, we can read off minus the degree of them by looking at the bottom left and top right boundary in the $b\leq a$ case and at the bottom right and top left boundary in the $a<b$ case. For both the sum is $a-1$. Thus, since the thick middle crossing is of degree $a$, the differentials are of degree $1$. Again, we leave the negative cases to the reader.
\end{proof}

I thank Queffelec and Rose that they pointed out that using the Rickard complex $\mathcal{T}_i^{+}{\idm}_{\vec{k}}$ is essentially equivalent to the $F$-braiding complex $\mathfrak{T}_i^{+}v_{\vec{k}}$. Part (a) can be seen as a categorification of Lemma \ref{lem-webasF}. For analogous statements, see \cite[Lemma 3.13 and Remark 3.14]{qr1}.

Before we start recall that the $q$-Schur $2$-algebra $\mathcal{S}(m,n\ell)_n$ is obtained from $\mathcal{U}(\mathfrak{gl}_n)$ by taking the quotient by setting all $2$-morphisms that have a region with a label not in $\Lambda(m,n\ell)_n$ to zero. For details see \cite{msv2}. The reader may convince herself/himself that it is in fact not a big deal to define $\check{\mathcal{S}}(m,n\ell)_n$ that we will use in the following and denote just by $\check{\mathcal{S}}$.

\begin{lem}\label{lem-rickardisf}
Denote by $\Kom^h_{\mathrm{gr}}(\check{\mathcal{S}})$ the homotopy category of complexes for $\check{\mathcal{S}}(m,n\ell)_n$ and sufficiently large $m$.
\begin{itemize}
\item[(a)] Let $u,v\in W_n(\vec{k})$ be two isotopic $\mathfrak{gl}_n$-webs with a possible different presentation under $q$-skew Howe duality $u=\mathrm{qH}v_{(n^{\ell})}$ and $v=\mathrm{qH}^{\prime}v_{(n^{\ell})}$ (here $\mathrm{qH}$ and $\mathrm{qH}^{\prime}$ consists of strings of $E_i^{(j)}$ and $F_i^{(j)}$). Then there exists an isomorphism in $\Ucatmc$ between the corresponding $\mathcal{E}_i^{(j)}$ and $\mathcal{F}_i^{(j)}$ realizing this isotopy, and all isotopies come already from isomorphisms in the KLR part of $\Ucatmc$ for sufficiently large $m$.
\item[(b)] The Rickard complex $\mathcal{F}_{i}^{(b)}\mathcal{F}_{i+1}^{(a)}\mathcal{T}^{+}_i{\idm}_{\vec{k}}$ is the same as $\mathfrak{T}^{+}_iv_{\vec{k}}$ in $\Kom^h_{\mathrm{gr}}(\check{\mathcal{S}})$ in the case $b\leq a$. Analogous statements are true for the other cases.
\end{itemize}
\end{lem}

\begin{proof}
(a). This is just a consequence of the results from the previous sections. To be more precise, by Lemma \ref{lem-webasF} and Proposition \ref{prop-extgrowth} we see that each $\mathfrak{gl}_n$-web corresponds to an equivalence class of $n$-multipartitions (taking isotopies in account). By Theorem \ref{thm-iso} and \cite[Corollary 7.6]{mack1} (that the split Grothendieck group of $\mathcal W^p_{\Lambda}$ is equivalent to the $\mathfrak{gl}_n$-web space $W_n(\Lambda)$) we see that all $\mathfrak{gl}_n$-web isotopies, if only $F_i^{(j)}$ are involved, have to come from a certain $\check{R}(\Lambda)$. If $E_i^{(j)}$ are involved, then the $\mathfrak{gl}_n$-webs still give the same on the level of Grothendieck groups, but the isotopies come from $\Ucatmc$ for a suitable $m$ (rewriting $E$ in terms of $F$ increases the $m$).

(b). We note that any isomorphism is not sufficient, since it has to give rise to a chain map. We therefore give such isomorphisms below which come from the following isomorphisms between the $\mathfrak{gl}_n$-webs
\[
\scalebox{.7}{$\xy
(0,0)*{\includegraphics[scale=.75]{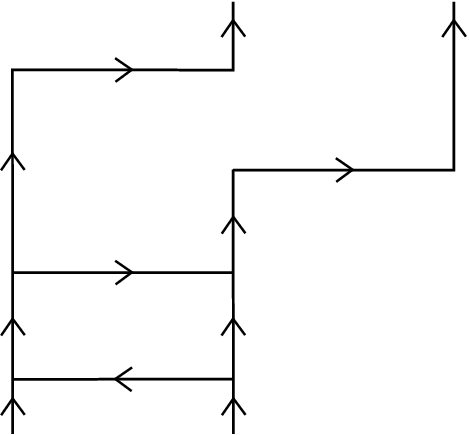}};
(-26,-26)*{a};
(2,-26)*{b};
(30,-26)*{0};
(-22,-13)*{a+k};
(6,-13)*{b-k};
(30,-13)*{0};
(-26,0)*{b};
(2,0)*{a};
(30,0)*{0};
(-26,13)*{b};
(2,13)*{0};
(30,13)*{a};
(-26,26)*{0};
(2,26)*{b};
(30,26)*{a};
(-12,22)*{F_{i}^{(b)}};
(-8.75,-3)*{F_{i}^{(a+k-b)}};
(-12,-17)*{E_{i}^{(k)}};
(15,10)*{F_{i+1}^{(a)}};
\endxy$}
\text{ and }
\scalebox{.7}{$\xy
(0,0)*{\includegraphics[scale=.75]{figs/linkpoly/colbraidb.eps}};
(17,18.5)*{F_{i+1}^{(a+k-b)}};
(-26,-20)*{a};
(2,-19.5)*{b};
(30,-19.5)*{0};
(-12,4)*{F_{i}^{(a)}};
(-26,-6.5)*{a};
(2.5,-7)*{k};
(33.5,-6.5)*{b-k};
(15,-9.5)*{F_{i+1}^{(b-k)}};
(-26,6.5)*{0};
(6.25,6.5)*{a+k};
(33.5,6)*{b-k};
(-26,19.5)*{0};
(2,19.5)*{b};
(30,19)*{a};
\endxy$},
\]
where the first $\mathfrak{gl}_n$-web is for the Rickard complex (which categorifies the rules from Definition \ref{defn-cRTpoly}) and the second is for the $F$-braiding complex (which, on the other hand, categorifies the rules from Definition \ref{defn-colbraid}).

We do not care for signs here because, if some signs for some squares as below do not work, then we can change them by multiplying with an extra sign for the right arrow for the corresponding square (starting at the leftmost). Moreover, we note that using $\check{\mathcal{S}}$ ensures that the complexes are all bounded from left and right. Thus, the sign change procedure is well-defined and terminates.

We now consider the following square where the $k$-th part of the Rickard complex is the top left and the $k$-th part of the $F$-braiding complex is the bottom left (with $\vec{k}=(,\dots,a,b,0,\dots)$).
\[
\begin{xy}
\xymatrix{
\mathcal{F}_{i}^{(b)}\mathcal{F}_{i+1}^{(a)}\mathcal{F}_{i}^{(a+k-b)}\mathcal{E}_{i}^{(k)}{\idm}_{\vec{k}}\{q_k\} \ar[r]^/-0.6em/{d^{\mathrm{R}}_k} \ar@<2pt>[d]^{g_k}    &   \mathcal{F}_{i}^{(b)}\mathcal{F}_{i+1}^{(a)}\mathcal{F}_{i}^{(a+k+1-b)}\mathcal{E}_{i}^{(k+1)}{\idm}_{\vec{k}}\{q_{k+1}\} \ar@<2pt>[d]^{g_{k+1}}  \\
F_{i+1}^{(a+k-b)}F_{i}^{(a)}F_{i+1}^{(b-k)}v_{\vec{k}}\{q_k\} \ar[r]^/-0.8em/{d_k}  \ar@<2pt>[u]^{f_k}           &   F_{i+1}^{(a+k+1-b)}F_{i}^{(a)}F_{i+1}^{(b-k-1)}v_{\vec{k}}\{q_{k+1}\} \ar@<2pt>[u]^{f_{k+1}} 
}
\end{xy}
\]
The maps $f_k$ (left) and $g_k$ (middle) are given by
\[
\xy
(0,0)*{\includegraphics[scale=.75]{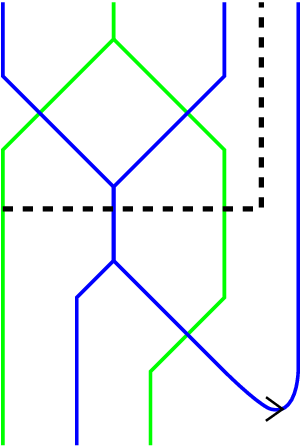}};
(21,0)*{\vec{k}};
(-14,0)*{\scriptstyle a+k-b};
(-10.5,-25)*{\scriptstyle a};
(-3,-25)*{\scriptstyle b-k};
(-1,-10)*{\scriptstyle k};
(5.5,-10)*{\scriptstyle b-k};
(17.5,-10)*{\scriptstyle k};
(-1.5,0)*{\scriptstyle a+k};
(-9.25,7.25)*{\scriptstyle b};
(3.75,7.25)*{\scriptstyle a+k-b};
(-5,18)*{\scriptstyle a+k-b};
(4.5,18)*{\scriptstyle b-k};
(-20,25)*{\scriptstyle b};
(-6,25)*{\scriptstyle a};
(4.5,25)*{\scriptstyle a+k-b};
\endxy
\hspace*{0.5cm}
\xy
(0,0)*{\includegraphics[scale=.75]{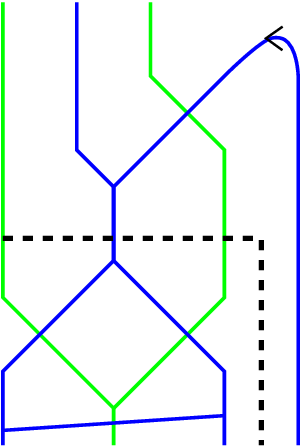}};
(21,0)*{\vec{k}};
(-14,0)*{\scriptstyle a+k-b};
(-10.5,25)*{\scriptstyle a};
(-3,25)*{\scriptstyle b-k};
(-1,10)*{\scriptstyle k};
(5.5,10)*{\scriptstyle b-k};
(17.5,10)*{\scriptstyle k};
(-1.5,0)*{\scriptstyle a+k};
(-10.5,-7.25)*{\scriptstyle b-k};
(4,-7.25)*{\scriptstyle a+2k-b};
(-5,-18)*{\scriptstyle a+k-b};
(4.5,-18)*{\scriptstyle b-k};
(-20,-27.5)*{\scriptstyle b};
(-6,-27.5)*{\scriptstyle a};
(4.5,-27.5)*{\scriptstyle a+k-b};
(4.5,-23.5)*{\scriptstyle k};
\endxy
\hspace*{0.5cm}
\xy
(0,0)*{\includegraphics[scale=.75]{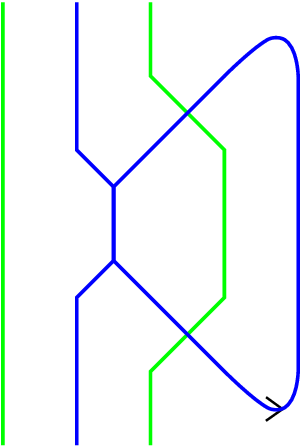}};
(21,0)*{\vec{k}};
(-14,0)*{\scriptstyle a+k-b};
(-10.5,-25)*{\scriptstyle a};
(-3,-25)*{\scriptstyle b-k};
(-1,-10)*{\scriptstyle k};
(5.5,-10)*{\scriptstyle b-k};
(17.5,-10)*{\scriptstyle k};
(-1.5,0)*{\scriptstyle a+k};
(-1,10)*{\scriptstyle k};
(-10.5,25)*{\scriptstyle a};
(-3,25)*{\scriptstyle b-k};
\endxy
\]
We have also indicated the thickness of the strands in order to help the reader. We note that part of these $2$-morphisms (the ones that we have separated) are exactly the same $2$-morphisms as in \cite[Section 4.2]{sto1}. The partition $\alpha\in P(0,k)$ in Sto\v{s}i\'{c}'s notation there will be empty. Note that the marked parts are of degree zero.

The proof that $g_k\circ f_k=\pm\mathrm{id}_{F_{i+1}^{(a+k-b)}F_{i}^{(a)}F_{i+1}^{(b-k)}v_{\vec{k}}}$ and $f_k\circ g_k=\pm\mathrm{id}_{\mathcal{F}_{i}^{(b)}\mathcal{F}_{i+1}^{(a)}\mathcal{F}_{i}^{(a+k-b)}\mathcal{E}_{i}^{(k)}{\idm}_{\vec{k}}}$ follows from calculations of Sto\v{s}i\'{c} in \cite{sto1}. For example, to see the first identity, one can use the equation in the proof of Lemma 4 in \cite{sto1} (recall that we ignore signs). This reduces the diagram to the right picture above. Then one can use the ``Opening of a thick edge'' (cf. \cite[Proposition 5]{sto1}) followed by the ``Thick R3 move'' (cf. \cite[Proposition 7]{sto1}) and apply ``Higher reduction of bubbles'' (see \cite[Proposition 5.2.9]{klms}) to see that this is just the identity (up to a sign). 
The other cases are again similar in the sense that they can be deduced from $\mathfrak{gl}_n$-web isotopies (and in the sense that they need non-trivial calculations) and left to the reader. This shows (b). 
\end{proof}

\begin{defn}\label{defn-khshoweF}(\textbf{Khovanov--Rozansky $\mathfrak{gl}_n$-braid complex only using $\mathbf{F}$}) Given an oriented, colored braid diagram $B_D$ with $\mathrm{cr}$ crossings and a fixed presentation of it using $q$-skew Howe duality
\[
B_D=\prod_{k} \tilde F^{(j_k)}_{i_k}v_{(n^{\ell})},\quad\text{with }\tilde F^{(j_k)}_{i_k}\text{ as in Lemma }\ref{lem-linkasF},
\]
with $T^{\pm}$ for the $\overcrossing$ or $\undercrossing$, we assign to it the \textit{$\mathfrak{gl}_n$-braid complex via $F$} by
\[
\llbracket B_D\rrbracket^n_{F}=\prod_{k} F^{(j_k)}_{i_k}\cdot\bigotimes_{k=1}^{\mathrm{cr}}\mathfrak{T}^{\pm}_{i_k}\cdot\prod_{k} F^{(j_k)}_{i_k}v_{(n^{\ell})},
\]
where we allow $F^{(j)}_i$ between the $T^{\pm}_{i_k}$ if they appear in the fixed presentation above between them. Moreover, the weights $\vec{k}$ for the $\mathfrak{T}$ from Definition \ref{defn-basicT} have to be suitably rearranged and the corresponding diagrams are the identities on the components $\prod_{k} F^{(j_k)}_{i_k}$.
\end{defn}

\begin{prop}\label{prop-invarianceF}
The complex $\llbracket B_D\rrbracket^n_{F}$, viewed in the corresponding homotopy category of complexes $\Kom^h_{\mathrm{gr}}(\check{S})$, gives an invariant of framed braids. That is, it does not depend on the braid moves.
\end{prop}

\begin{proof}
This is Lemma \ref{lem-rickardisf} combined with the fact that \cite[Theorem 6.3]{ck3} applies to $\check{\mathcal{S}}$.
\end{proof}

\begin{rem}\label{rem-betterway}
We point out that there is a way to prove Proposition \ref{prop-invarianceF} directly in our framework and extend it to braid like tangles. The latter are more flexible then braids and satisfy additional moves called \textit{tangle braid moves} (see e.g. \cite[Figure 2]{cau1} or \cite[Lemma X.3.5]{kassel}).

This alternative proof is based on the higher quantum Serre relations and their categorification given in \cite{sto1}. Moreover, we think that these complexes can be used for a ``divide and conquer'' strategy for computations {\`a} la Bar-Natan \cite{bn3}. But we only sketch how it should work. Compare also to our proof of Theorem \ref{thm-evoflinks}.

Given the setup as in the proof of Theorem \ref{thm-evoflinks}, we get a complex (recall that $v=v_{\dots,1,1,0,0,\dots}$)
\[
\begin{xy}
\xymatrix{
& F_{i+1}F_{i+2}F_{i}F_{i+1}v\{0\}\ar[dr]|{\phantom{-}\xy
(0,0)*{\raisebox{-0.15em}{\includegraphics[width=10px]{figs/higherstuff/downcross}}};\endxy\colon F_{i}F_{i+1}\to F_{i+1}F_{i}}\ar@{.}[dd]|{\bigoplus} & \\
F_{i+2}F_{i+1}F_iF_{i+1}v\{-1\}\ar[dr]|{\phantom{-}\xy
(0,0)*{\raisebox{-0.15em}{\includegraphics[width=10px]{figs/higherstuff/downcross}}};\endxy\colon F_{i}F_{i+1}\to F_{i+1}F_{i}}\ar[ur]|{\phantom{-}\xy
(0,0)*{\raisebox{-0.15em}{\includegraphics[width=10px]{figs/higherstuff/downcross}}};\endxy\colon F_{i+2}F_{i+1}\to F_{i+1}F_{i+2}} &  & F_{i+1}F_{i+2}F_{i+1}F_iv\{+1\}\\
& F_{i+2}{\color{red}F_{i+1}F_{i+1}}F_iv\{0\}\ar[ur]|{-\xy
(0,0)*{\raisebox{-0.15em}{\includegraphics[width=10px]{figs/higherstuff/downcross}}}\endxy\colon F_{i+2}F_{i+1}\to F_{i+1}F_{i+2}} &    
}
\end{xy}
.
\]
There is an explicit isomorphism $\mathcal{F}_{i+1}\mathcal{F}_{i+1}\cong\mathcal{F}^{(2)}_{i+1}\{-1\}\oplus\mathcal{F}^{(2)}_{i+1}\{+1\}$ in $\Ucatmc$, see \cite[Theorem 5.1.1]{klms} (the same is true in $\check{\mathcal{S}}$). This, in the $n=2$ case, is just Bar-Natan's delooping from \cite[Lemma 3.1]{bn3}.

We get from this, focusing on the bottom path of the complex above, the following complex.
\[
\begin{xy}
\xymatrix{
F_{i+2}F_{i+1}F_iF_{i+1}v\{-1\}\ar[r]^/0.5em/{d_1} &  F_{i+2}F^{(2)}_{i+1}F_iv\{-1\}\ar@{.}[d]|{\bigoplus} & \\
& F_{i+2}F^{(2)}_{i+1}F_iv\{+1\}\ar[r]^/-0.5em/{d_2} &  F_{i+1}F_{i+2}F_{i+1}F_iv\{+1\}. 
}
\end{xy}
\]
The differentials will change as usual using Gauss elimination, see e.g. \cite[Lemma 3.2]{bn3}, to
\[
d_1=\xy(0,0)*{\includegraphics[width=75px]{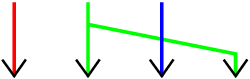}}\endxy\hspace*{0.5cm}\text{  and  }
\hspace*{0.5cm}d_2=\xy(0,0)*{\includegraphics[width=75px]{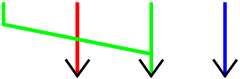}}\endxy
.
\]
The top line is part of the null-homotopic complex defined in Theorem 7 in \cite{sto1} (or a variant of it by exchanging $E$ to $F$ and indices) for $a=2,b=1$. We note that, due to weight reasons, most terms of Sto\v{s}i\'{c}'s complex will be zero. On the other hand, the bottom line is part of the null-homotopic complex defined in \cite[Theorem 6]{sto1} (again slightly rearranged).
As explained in \cite{bn3}, the complex will collapse and the starting complex is homotopic to the trivial complex which shows the invariance under the second Reidemeister move.
\end{rem}

\subsubsection{Colored \texorpdfstring{$\mathfrak{gl}_n$}{gln}-link homology using \texorpdfstring{$F$}{F}}\label{subsub-algocalc}
We are now ready to define our version of the colored Khovanov--Rozansky $\mathfrak{gl}_n$-link homology.

\begin{defn}\label{defn-cansequence}
Given a weight $(n^{\ell})$, we associate to it a \textit{canonical sequence of $F_i^{(j)}$}, denoted by $F^c_{(n^{\ell})}$, by applying $F_i^{(n)}$ to shift all $n$ to the right by shifting always the rightmost pair of the form $(\dots,n,0,\dots)$ to $(\dots,0,n,\dots)$.  
\end{defn}

\begin{ex}\label{ex-canseq}
The canonical sequence associated to $(3,3,0,0)$ is $F_2^{(3)}F_1^{(3)}F_3^{(3)}F_2^{(3)}$. Another example is given in Example \ref{ex-dual}.
\end{ex}

\begin{defn}\label{defn-khshoweF2}(\textbf{Khovanov--Rozansky $\mathfrak{gl}_n$-link homology only using $\mathbf{F}$}) Given an oriented, colored link diagram $L_D$ with $\mathrm{cr}$ crossings $c_{a,b}$ and a fixed presentation of it using $q$-skew Howe duality
\[
L_D=\prod_{k} F^{(j_k)}_{i_k}\cdot T_{i_{\mathrm{cr}}}^{\pm}\dots T_{i_1}^{\pm}\cdot\prod_{k} F^{(j_k)}_{i_k}v_{(n^{\ell})}
\]
with $T^{\pm}$ for the $\overcrossing$ or $\undercrossing$ (as before, we allow extra $F$ between the different $T^{\pm}$), we assign to it the \textit{colored Khovanov--Rozansky $\mathfrak{gl}_n$-link homology via $F$} by
\[
\llbracket L_D\rrbracket^n_{F}=\mathrm{hom}_{\check{R}}(F^c_{(n^{\ell})},\prod_{k} F^{(j_k)}_{i_k}\cdot\bigotimes_{k=1}^{\mathrm{cr}}\mathfrak{T}^{\pm}_{i_k}\cdot\prod_{k} F^{(j_k)}_{i_k}v_{(n^{\ell})})
\]
(we write shortly $\mathrm{hom}_{\check{R}}$ for $\mathrm{hom}_{\check{R}(\Lambda)}$) and
\[
\KR(L_D)^n_{\mathrm{F}}=\llbracket L_D\rrbracket^n_{\mathrm{F}}\{\mathrm{power}(q)\}
\]
where the shift in the $q$-degree $\{\mathrm{power}(q)\}$ is the same as power of the $q$ in the product from \eqref{eq-norRT}. Moreover, the weights $\vec{k}$ for the $\mathfrak{T}$ from Definition \ref{defn-basicT} have to be suitably rearranged for the tensor product to make sense.
\end{defn}

\begin{thm}\label{thm-same2}
The complex $\KR(L_D)^n_{\mathrm{F}}$ is the same as $\KR(L_D)^n$ viewed as objects in the homotopy category of complexes $\Kom^h_{\mathrm{gr}}(\mathcal W^p_{\Lambda})$. Thus, it is an invariant of colored links and therefore invariant under the three Reidemeister moves and isotopies.
\end{thm}

A similar result can be concluded for the complex $\llbracket L_D\rrbracket^n_{\mathrm{F}}$, but one has to be very careful with possible degree shifts. We do not do it here. Moreover, we would like to prove the invariance directly in our setup.

\begin{proof}
One part of the argument is very similar to the one used by Lauda, Queffelec and Rose to prove that their complex agrees with the Khovanov--Rozansky $\mathfrak{gl}_n$-link homology (for $n=2,3$), see Proposition 4.3 in \cite{lqr1}. One part of their argument is that the differentials in their complex are, up to a sign, the same for both complexes. Then they use an argument similar to \cite{ors}. A very similar argument works for the complex $\KR(L_D)^n_{\mathrm{F}}$. Thus, we can ignore these signs in the following.

The rest is also easy to verify with our results from the previous sections. To be more precise, using Theorem \ref{thm-iso}, we see that our modules $\mathrm{hom}_{\check{R}}(F^c_{(n^{\ell})},\cdot)$ are graded isomorphic to modules over the $\mathfrak{gl}_n$-web algebra $H_n(\Lambda)$ defined by Mackaay. Thus, they are certain EXT-spaces of matrix factorizations associated to the underlying $\mathfrak{gl}_n$-webs (that we obtain from the string of $F_i^{(j)}$ via the translation from Section \ref{sec-tabwebs}).

Checking the definition of the differentials for $\KR(L_D)^n$ (that can be found in \cite[Section 7]{kr1} or in the colored case in \cite[Definition 12.4]{wu} or alternatively in \cite[Sections 5 and 6]{yo2}) we see that they all can be obtained by applying the extended $2$-functor from Theorem \ref{thm-thick} to the Rickard complex from \ref{defn-rickard}.

Now comes an important point that we like to prove in our setting directly. Using the isotopy invariance of $\KR(L_D)^n$ (see \cite[Theorem 2]{kr1}, or, in the colored case, see \cite[Theorem 1.1]{wu} or alternatively \cite[Theorem 1.3]{yo2}) together with Lemma \ref{lem-rickardisf} we see that this induces a homotopy between $\KR(L_D)^n_{\mathrm{F}}$ and $\KR(L_D)^n$ which shows the first statement.
Since $\KR(L_D)^n$ is invariant under the Reidemeister moves, the same holds for $\KR(L_D)^n_{\mathrm{F}}$ as well. This completes the proof.
\end{proof}

And $\KR(L_D)^n_{\mathrm{F}}$ categorifies the colored Reshetikhin--Turaev $\mathfrak{gl}_n$-link polynomial $\mathrm{RT}_n$.

\begin{cor}\label{cor-categorification}
Let $L_D$ be an oriented, colored link diagram. The	
graded Euler characteristic of the complex $\KR(L_D)^n_{\mathrm{F}}$ gives $\mathrm{RT}_n(L_D)$.
\end{cor}

\begin{proof}
This is just a combination of Theorem \ref{thm-same2} 
and e.g. \cite[Theorem 1.3]{wu}.
\end{proof}

\begin{rem}\label{rem-tanglehom}
An analog of Definition \ref{defn-khshoweF2} and Theorem \ref{thm-same2} can be formulated and proven for braid-like tangles (tangles with a fixed number of bottom and top boundary points) as well: just close the bottom/top of the tangle in all possible ways (one needs a bigger $m$ for this) and proceed as above. This realizes the complex as bimodules/bimodule maps over $\check{R}(\Lambda)$ as in the original formulation of Khovanov for his arc algebra, see \cite{kh4}.

A good question would be to extend Lemma \ref{lem-rickardisf} to braid-like tangles by checking the braid tangle moves (see for example \cite[Lemma X.3.5]{kassel}) in our setup.
\end{rem}

\subsubsection{The calculation algorithm}\label{subsub-algocalchom}
We now define an algorithm to compute the local differentials (that is, the ones from one resolution to another) of the complex $\KR(L_D)^n_{\mathrm{F}}$ using the HM basis. We start by simplifying the notation: since the canonical sequence from Definition \ref{defn-cansequence} is fixed by $(n^{\ell})$ and therefore by our presentation of the link diagram using $q$-skew Howe duality, we suppress to write $\mathrm{hom}_{\check{R}}(F^c_{(n^{\ell})},\cdot)$ in the following.

\begin{ex}\label{ex-hopfnew}
Let us give the complex associated to Hopf link from Example \ref{ex-Hopf} as an example. Recall that we have colored it with $1$ and $2$ and the presentation via $F_i^{(j)}$ was
\[
\mathrm{Hopf}=F_4^{(3)}F_5^{(2)}F_3^{(2)}F_2^{(2)}F_1^{(2)}T_{2,1,3}T_{1,2,2}F_5F_4F_3F_1F_2^{(3)}v_{(3^2)}.
\]
Let us shortly write $F_t$ and $F_b$ for the string of $F_i^{(j)}$ after (at the top) and before (at the bottom) the crossings $T_{2,1,3}T_{1,2,2}$ and $v$ for $v_{(3^2)}$. Then the chain complex associated to it is, in simplified notation, given by
\[
\begin{xy}
\xymatrix{
& F_t{\color{red}F^{(2)}_4F^{(2)}_3}{\color{green}F_2F_3}F_bv\{-1\}\ar[dr]|{\phantom{-}\xy
(0,0)*{\raisebox{-0.15em}{\includegraphics[width=10px]{figs/higherstuff/HM-strings1}}};\endxy\colon F_2F_3\to F_3F_2}\ar@{.}[dd]|{\bigoplus} & \\
F_t{\color{red}F_4F^{(2)}_3F_4}{\color{green}F_2F_3}F_bv\{-2\}\ar[ur]|{\phantom{-}\xy
(0,0)*{\raisebox{-0.15em}{\includegraphics[width=20px]{figs/linkhom/HM-diff}}};\endxy\colon F_4F^{(2)}_3F_4\to F^{(2)}_4F^{(2)}_3}\ar[dr]|{\phantom{-}\xy
(0,0)*{\raisebox{-0.15em}{\includegraphics[width=10px]{figs/higherstuff/HM-strings1}}};\endxy\colon F_2F_3\to F_3F_2} &  & F_t{\color{red}F^{(2)}_4F^{(2)}_3}{\color{green}F_3F_2}F_bv\{0\}\\
& F_t{\color{red}F_4F^{(2)}_3F_4}{\color{green}F_3F_2}F_bv\{-1\}\ar[ur]|{-\xy
(0,0)*{\raisebox{-0.15em}{\includegraphics[width=20px]{figs/linkhom/HM-diff}}}\endxy\colon F_4F^{(2)}_3F_4\to F^{(2)}_4F^{(2)}_3} &    
}
\end{xy}
\]
with leftmost part in homological degree zero. Moreover, there is no extra shift for the $q$-degree.
\end{ex}

We point out that every step in the following definition is given by an algorithm.

\begin{defn}\label{defn-hmcomplex}(\textbf{Computation algorithm}) Given a oriented, colored link diagram $L_D$ with $\mathrm{cr}$ crossings $c_{a,b}$ and a fixed presentation of it using $q$-skew Howe duality
\begin{align}\label{eq-linkasF}
L_D=\prod_{k} F^{(j_k)}_{i_k}\cdot T_{i_{\mathrm{cr}}}^{\pm}\dots T_{i_1}^{\pm}\cdot\prod_{k} F^{(j_k)}_{i_k}v_{(n^{\ell})}
\end{align}
with $T^{\pm}$ for the $\overcrossing$ or $\undercrossing$ (as before, we allow extra $F$ between the different $T^{\pm}$), we assign to it a complex $\KR(L_D)^n_{\mathrm{F}}$ as in Definition \ref{defn-khshoweF}.

Fix two vertices $v_1,v_2$ in the Khovanov cube associated to $L_D$ that are connected by an edge and assume that $v_1$ is in lower homological degree. For both vertices we have a string of $F_i^{(j)}$ associated to it that we denote by $F_{v_1},F_{v_2}$. We also denote the associated $\check{R}(\Lambda)$-modules by $M_1,M_2$. Then there is local differential $d\colon M_1\to M_2$ of the form as in Definition \ref{defn-basicT}. 

Then the local differential $d\colon M_1\to M_2$ of $\KR(L_D)^n_{\mathrm{F}}$ can be computed in the following way.
\begin{itemize}
\item Compute the thick HM basis for $M_1$ that we have defined in Definition \ref{defn-growthfoam}. Denote the elements of this basis by $m_1^{1},\dots,m_1^{k_1}$. These elements are given by string diagrams from $F^c_{(n^{\ell})}$ at the bottom to $F_{v_1}$ at the top.
\item Compute the dual thick HM basis for $M_2$ that we have defined in Remark \ref{rem-othermethod}. Denote the elements of this basis by $m_2^{1},\dots,m_2^{k_2}$. These elements are given by string diagrams from $F_{v_2}$ at the bottom to $F^c_{(n^{\ell})}$ at the top.
\item The differential $d$ is a diagram with $F_{v_1}$ at the bottom and $F_{v_2}$ at the top.
\item Thus, the composition $m_2^{k_{r^{\prime}}}\circ d\circ m_1^{k_r}$ for each pair $r,r^{\prime}$ is $d_{rr^{\prime}}\in\mathrm{hom}_{\check{R}}(F^c_{(n^{\ell})},F^c_{(n^{\ell})})$.
\item Define a matrix $d=(d_{rr^{\prime}})$ consisting of these $d_{rr^{\prime}}$ for $r=1\dots, k_1$ and $r^{\prime}=1,\dots,k_2$ scaled by the values that come from pairing the duals $m_2^{1},\dots,m_2^{k_2}$ with the usual basis.
\end{itemize}
\end{defn}

\begin{thm}\label{thm-algworks}
The algorithm from Definition \ref{defn-hmcomplex} gives a way to compute the homology of $\KR(L_D)^n_{\mathrm{F}}$ and thus, the colored Khovanov--Rozansky $\mathfrak{gl}_n$-link homology $\KR(L_D)^n$.
\end{thm}

\begin{proof}
To simplify notation: let us denote by $\widehat{\cdot}$ the associated matrix factorizations (for strings of $F_i^{(j)}$) or homomorphisms of matrix factorization (for $\check{R}(\Lambda)$-diagrams) using Theorem \ref{thm-iso}.

First we note that we can use the local differentials from Definition \ref{defn-hmcomplex} to define the differentials of $\KR(L_D)^n_{\mathrm{F}}$ by taking sums as usual if the local differentials of the algorithm coincide with the local ones from $\KR(L_D)^n_{\mathrm{F}}$. Then, by Theorem \ref{thm-same2}, we see that the complexes will have the same homology. The rest is linear algebra: compute the kernels and images of the matrices, keep track of the gradings and obtain this way the homology of $\KR(L_D)^n_{\mathrm{F}}$. Hence, we have to ensure that the local differential agree. But this is also linear algebra:
\begin{itemize}
\item The two $\bC$-vector spaces $M_1$ and $M_2$ are $\check{R}(\Lambda)-\check{R}(\Lambda)$-bimodules. Here the action from left (or right) is given by multiplying from the bottom (or top) by pre(or post)composing.
\item Thus, by Theorem \ref{thm-iso}, they are also $H_n(\Lambda)-H_n(\Lambda)$-bimodules and the action is given in the same way. We see this way that $\mathrm{hom}_{\check{R}}(F^c_{(n^{\ell})},F^c_{(n^{\ell})})$ is one dimensional and the $d_{rr^{\prime}}$ can therefore be seen as elements of $\bC$ by choosing the evident basis of the diagram that only points upwards.
\item The local differentials from Definition \ref{defn-basicT} are exactly given by composing the corresponding $\widehat{d}$ to the left. Hence, $\widehat{d}\circ \widehat{m}_1^{r}$ is an element of $\mathrm{EXT}(\widehat{F}^c_{(n^{\ell})},\widehat{F}_2)$.
\item Since the thick HM basis is a basis that works in this generality, see Theorem \ref{thm-foambasis},  one can re-write $\widehat{d}\circ \widehat{m}_1^{r}$ in terms of the basis for $\widehat{M}_2$.
\item But using the dual basis as in Definition \ref{defn-hmcomplex} as above is nothing else then using the trace that we have recalled in Definition \ref{defn-slnwebalgebra}. This is nothing else than taking the inner product $\langle \widehat{d}\circ \widehat{m}_1^{r},\widehat{m}_2^{r^{\prime}}\rangle$. Thus, the $d_{rr^{\prime}}$ count the multiplicity of $\widehat{m}_2^{r^{\prime}}$ if one re-writes $\widehat{d}\circ \widehat{m}_1^{r}$ in terms of the thick HM basis for $\widehat{M}$ (and scales the result as above).
\end{itemize}
Thus, we obtain the statement by Theorem \ref{thm-same2}.
\end{proof}

\begin{ex}\label{ex-honestcalc}
Recall Example \ref{ex-linkasF} from before. We note that we cheat below, since, if we would strictly follow the algorithm, then we would have to write $U_D$ using a longer string of $F^{(j)}_i$.

We write just $v=v_{(2^1)}$. We get the following chain complex for the diagram of the unknot $U_D$. 
Here the right part is homology degree zero.
\[
\begin{xy}
\xymatrix{
F_2{\color{red}F_1F_2}F_1v\{-2\}\ar[rr]^{\xy(0,0)*{\raisebox{-0.15em}{\includegraphics[width=10px]{figs/higherstuff/HM-strings1}}};\endxy\colon F_1F_2\to F_2F_1} & & F_2{\color{red}F_2F_1}F_1v\{-1\} \\    
}
\end{xy}.
\] 
Thus, we need to calculate the thick HM basis for $\mathrm{hom}_{\check{R}}(F_2^{(2)}F_1^{(2)},F_2F_1F_2F_1)$ and, analogously, the dual 
thick HM basis for $\mathrm{hom}_{\check{R}}(F_2F_2F_1F_1,F_2^{(2)}F_1^{(2)})$. We have already done the first in Example \ref{ex-dual}.

Note now that the $2$-multitableaux for $F_2^{(2)}F_1^{(2)}$ is still $\vec{T}$ from Example \ref{ex-dual}. Moreover, we have four for $F_2F_2F_1F_1$, namely $\vec{T}_{1,2}$ and $\vec{T}_{3,4}$ from Example \ref{ex-evaluation}. Recall that the dot placement is just given by the associated dual standard filling $T^*_{\vec{\lambda}}$ where $\vec{\lambda}$ is the shape of the $\vec{T}$.

From this we get a sequence of transpositions $\tau$ from $\vec{T}_k$ to $T^*_{\vec{\lambda}}$. For the first two $2$-multitableaux we have $\tau_2(1,2)\tau_3(2,2)\tau_1(1,1)\vec{T}_1=T^*_{\vec{\lambda}}$ and $\tau_2(1,2)\tau_1(1,1)\vec{T}_2=T^*_{\vec{\lambda}}$ and $\tau_2(1,2)\tau_3(2,2)\vec{T}_3=T^*_{\vec{\lambda}}$ and $\tau_2(1,2)\vec{T}_4=T^*_{\vec{\lambda}}$ for the last two. Thus, we have the four dual basis elements
\[
\xy
(0,0)*{\includegraphics[scale=.75]{figs/linkhom/dual-example7}};
(0,-26)*{\vec{T}_1=\left(\;\xy (0,0)*{\begin{Young}1& 3\cr\end{Young}}\endxy\;,\;\xy (0,0)*{\begin{Young}2& 4\cr\end{Young}}\endxy\;\right)};
\endxy\hspace*{0.25cm}
\xy
(0,0)*{\includegraphics[scale=.75]{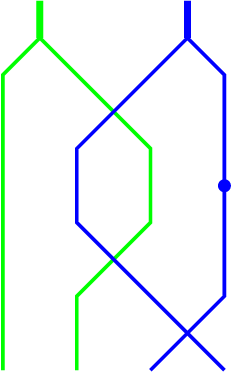}};
(0,-26)*{\vec{T}_2=\left(\;\xy (0,0)*{\begin{Young}1& 4\cr\end{Young}}\endxy\;,\;\xy (0,0)*{\begin{Young}2& 3\cr\end{Young}}\endxy\;\right)};
\endxy\hspace*{0.25cm}
\xy
(0,0)*{\includegraphics[scale=.75]{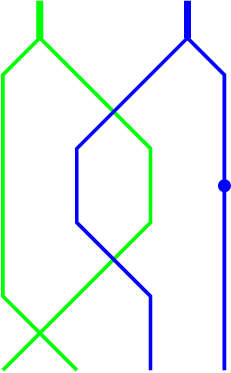}};
(0,-26)*{\vec{T}_3=\left(\;\xy (0,0)*{\begin{Young}2& 3\cr\end{Young}}\endxy\;,\;\xy (0,0)*{\begin{Young}1& 4\cr\end{Young}}\endxy\;\right)};
\endxy\hspace*{0.25cm}
\xy
(0,0)*{\includegraphics[scale=.75]{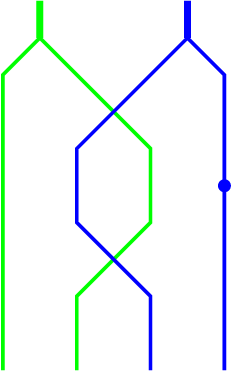}};
(0,-26)*{\vec{T}_4=\left(\;\xy (0,0)*{\begin{Young}2& 4\cr\end{Young}}\endxy\;,\;\xy (0,0)*{\begin{Young}1& 3\cr\end{Young}}\endxy\;\right)};
\endxy
.
\]
Applying the isomorphism to the $\mathfrak{gl}_2$-web algebra (and cheating again using Bar-Natan's cobordisms as in Example \ref{ex-cheating}) we see that these corresponds from right to left to a pair of undotted caps, a pair of caps where one has a dot and a pair of caps where both have a dot. To make connections to Definition \ref{defn-hmcomplex}, let us denote them by $m_2^1,m_2^2,m_2^3$ and $m_2^4$.
Moreover, the basis of the source from Example \ref{ex-dual} can be read as a cup with a dot (denoted by $m_1^1$) and an undotted cup (denoted by $m_1^2$) and the differential $d$ is the usual comultiplication. Thus, we expect that $d\circ m_1^1$ will pair with everything except one element of the dual basis to zero.

So let us evaluate the pictures which are just given by stacking now. We have
\[
\xy
(0,0)*{\includegraphics[scale=.75]{figs/linkhom/dual-example7}};
(0,-29)*{\includegraphics[scale=.75]{figs/linkhom/dual-example8}};
(0,-50)*{\includegraphics[scale=.75]{figs/catcell/dual-example1}};
(0,-24)*{\circ};
(0,-35)*{\circ};
(10,-40)*{m_1^1};
(10,-29)*{d};
(10,0)*{m_2^1};
\endxy\hspace*{0.35cm}
\xy
(0,0)*{\includegraphics[scale=.75]{figs/linkhom/dual-example6}};
(0,-29)*{\includegraphics[scale=.75]{figs/linkhom/dual-example8}};
(0,-50)*{\includegraphics[scale=.75]{figs/catcell/dual-example1}};
(0,-24)*{\circ};
(0,-35)*{\circ};
(10,-40)*{m_1^1};
(10,-29)*{d};
(10,0)*{m_2^2};
\endxy\hspace*{0.35cm}
\xy
(0,0)*{\includegraphics[scale=.75]{figs/linkhom/dual-example5}};
(0,-29)*{\includegraphics[scale=.75]{figs/linkhom/dual-example8}};
(0,-50)*{\includegraphics[scale=.75]{figs/catcell/dual-example1}};
(0,-24)*{\circ};
(0,-35)*{\circ};
(10,-40)*{m_1^1};
(10,-29)*{d};
(10,0)*{m_2^3};
\endxy\hspace*{0.35cm}
\xy
(0,0)*{\includegraphics[scale=.75]{figs/linkhom/dual-example4}};
(0,-29)*{\includegraphics[scale=.75]{figs/linkhom/dual-example8}};
(0,-50)*{\includegraphics[scale=.75]{figs/catcell/dual-example1}};
(0,-24)*{\circ};
(0,-35)*{\circ};
(10,-40)*{m_1^1};
(10,-29)*{d};
(10,0)*{m_2^4};
\endxy
.
\]
Note that it is exactly as we expected: all of the diagrams above give a $\bC$ multiple of the trivial diagram with only two upwards pointing thick strands. And all with the exception of the left one are zero. To see this note that the rightmost two diagrams are on the nose zero because of two dots on the same strand (we are in $n=2$). The second is zero which can be deduced from the thick calculus rules (see e.g. \cite{klms} or \cite{sto1}). That is, opening the bottom Reidemeister 2 moves gives two terms: $\pm$ one with a dot on the green (left) strand $\mp$ one with a dot on the blue (right) strand. The second term is always zero, since the middle crossing is a composition of a split$\circ$merge. Thus, at the bottom we have a merge$\circ$split with two dots - this is always zero for $n=2$.
But the same holds for the top now: only a dot on the green (left) strand can survive after opening the Reidemeister 2 move. But then we have two dots on the green (left) strand which is zero in $n=2$. Thus, the whole composition is zero.

The first one on the other hand gives $\pm 1$: only one term survives the opening of the Reidemeister 2 moves and it has exactly one dot between each merge$\circ$split-pair. Thus, they can be reduced to a line (up to a sign), see e.g. \cite[Corollary 2.4.2]{klms}. This shows that $d(m_1^1)=\pm m_2^1$.

Doing the same for $m_1^2$ (which has two surviving, namely $m_2^2$ and $m_2^3$) we see that $d$ is given (up to a sign) by Khovanov's original comultiplication map which comes from the algebra $\bC[X]/X^2$, see \cite{kh1}, namely $1\mapsto 1\otimes X+X\otimes 1$ and $X\mapsto X\otimes X$. This map is injective which shows that the homology is trivial.
\end{ex}

\end{document}